\documentclass{amsart}

  \RequirePackage{tikz-cd}
\usepackage{amsmath,amsfonts,amssymb,amscd,verbatim,multicol}
\usepackage{tikz}
  \usepackage{hyperref}

\usepackage{url}

  \usepackage{color}
\definecolor{note}{rgb}{0.1,0.1,0.4}

\newcommand{\fff}{{\color{red}{\footnote{\color{red}\textsl{\textcolor{note}{{\color{red} Change of citation}}}}}}}

\numberwithin{equation}{section}
 \theoremstyle{plain}
 
\newtheorem{theorem}[equation]{Theorem}

\newtheorem{sublemma}[equation]{Sublemma}
\newtheorem{lemma}[equation]{Lemma}

\newtheorem{corollary}[equation]{Corollary}
\newtheorem{proposition}[equation]{Proposition}

\newtheorem{condition}[equation]{Condition}

\theoremstyle{definition}
\newtheorem{definition}[equation]{Definition}
\newtheorem{notation}[equation]{Notation}

\newtheorem{remark}[equation]{Remark}
\newtheorem{remarks}[equation]{Remarks}
\newtheorem{hypothesis}[equation]{Hypothesis}
\newtheorem{hypotheses}[equation]{Hypotheses}
\newtheorem{example}[equation]{Example}
\newtheorem{question}[equation]{Question}
\newtheorem{conjecture}[equation]{Conjecture}

 \makeatletter
\DeclareRobustCommand{\qed}{%
  \ifmmode
    \eqno \def\@badmath{$$}
    \let\eqno\relax \let\leqno\relax \let\veqno\relax
    \hbox{\openbox}%
  \else
    \leavevmode\unskip\penalty9999 \hbox{}\nobreak\hfill
    \quad\hbox{\openbox}%
  \fi
}

\makeatother  

\usepackage[mathscr]{euscript} 
\DeclareMathAlphabet{\euls}{U}{eus}{m}{n}

\newcommand{\eC}{{\euls{C}_{\chi}}}
\newcommand{\eCdash}{{\euls{C}_{\chi'}}}
\newcommand{\eCop}{\euls{C}_{\chi}^{\mathrm{op}}}
 \newcommand{\eO}{{\euls{O}}}
 
 \newcommand{\eD}{{\euls{D}}}
 
\newcommand{\eG}{{\euls{G}}} \newcommand{\eH}{{\euls{H}}}
\newcommand{\eGt}{{\widetilde{\eG}}}
\newcommand{\eK}{{\euls{K}}}
\newcommand{\eKt}{{\widetilde{\euls{K}}}} 
 
\newcommand{\eM}{{\euls{M}}}
\newcommand{\eMt}{{\widetilde{\euls{M}}}}
 
\newcommand{\eP}{{\euls{P}}} \newcommand{\eQ}{{\euls{Q}}}
 \newcommand{\eS}{{\euls{S}}}
\newcommand{\eT}{{\euls{T}}} 
\newcommand{\eX}{{\euls{X}}} 
\newcommand{\eY}{{\euls{Y}}} 
\newcommand{\eu}{{\mathrm{eu}}}
\newcommand{\llambda}{{\lambda}}
\newcommand{\vt}{\vartheta}
 
\DeclareMathOperator{\Biad}{{Bi^{\ad}}}
\DeclareMathOperator{\Bi}{{Bi}}

\DeclareMathOperator{\Lie}{{Lie}}
\DeclareMathOperator{\Der}{{Der}}

\DeclareMathOperator{\ad}{\mathbf{ad}}

\DeclareMathOperator{\lmod}{\text{-mod}}
\DeclareMathOperator{\rmod}{\text{mod-}}
\newcommand{\Cdim}{\operatorname{Cdim}}
\newcommand{\too}{\longrightarrow}

\newcommand{\C}{\mathbb{C}}
\newcommand{\D}{\euls{D}}

\newcommand{\NN}{{\mathbb N}}
\DeclareMathOperator{\Hom}{{Hom}}
\DeclareMathOperator{\End}{{End}}
\DeclareMathOperator{\Ext}{{Ext}}
\DeclareMathOperator{\Tor}{{Tor}}
 \newcommand{\DD}{{\mathbb D}}
  \newcommand{\Dvar}{{\widetilde{\mathbb D}}}
  \newcommand{\BH}{\mathbb{H}}
\newcommand{\HHleft}{{{}^\perp\BH}}
\newcommand{\HHright}{{\BH{}^\perp}}

\DeclareMathOperator{\gr}{{gr}}
\DeclareMathOperator{\Sym}{{Sym}}
\DeclareMathOperator{\GKdim}{{GKdim}}

\newcommand{\BD}{\mathbb{D}}
\newcommand{\BC}{\mathbb{C}}
\newcommand{\BZ}{\mathbb{Z}}
\newcommand{\Osph}{\euls{O}^{\mathrm{sph}}}
\newcommand{\Osphop}{\euls{O}^{\mathrm{sph,op}}}

\newcommand{\git}{/\!/}
\newcommand{\Ver}{\mathbb{V}}
\newcommand{\shT}{\mathbb{T}}

\newcommand{\bsq}{{\boldsymbol{q}}}

\newcommand{\rr}{\varrho}
\newcommand{\rrp}{\rr^*}
\newcommand{\rrpp}{\xi}
\newcommand{\gtilde}{\widetilde{\mathfrak{g}}}
\newcommand{\Gtilde}{\widetilde{G}}

 \newcommand{\rad}{\operatorname{rad}}
 \newcommand{\rann}{\operatorname{r-ann}}
  \newcommand{\lann}{\operatorname{\ell-ann}}
\newcommand{\Kdim}{\operatorname{Kdim}}
\newcommand{\gldim}{\operatorname{gldim}}
\newcommand{\Ker}{\operatorname{ker}}
\renewcommand{\Im}{\operatorname{Im}}
\renewcommand{\o}{\otimes}
\newcommand{\injdim}{\operatorname{injdim}}
\newcommand{\pd}{\operatorname{pd}}
\newcommand{\reg}{\operatorname{reg}}

\newcommand{\grade}{j}
\newcommand{\g}{\mathfrak{g}}
\newcommand{\p}{\mathfrak{p}}
\newcommand{\h}{\mathfrak{h}}
\newcommand{\mf}{\mathfrak}
\newcommand{\ds}{\dots}

\newcommand{\mr}{\mathrm}

\newcommand{\R}{\mathbb{R}}

\newcommand{\cdag}{\kappa^{\dagger}}
\newcommand{\Hk}{H_{\kappa}}
\newcommand{\Ak}{A_{\kappa}}
\newcommand{\Aak}{A_{\kappa}}
\newcommand{\dd}{\eD}
\newcommand{\ddd}{\eD}
\newcommand{\bpsi}{\boldsymbol{\psi}}

\newcommand{\Z}{\mathbb{Z}}
\newcommand{\vs}{\varsigma}
\newcommand{\isom}{\stackrel{\sim\,\,}{\longrightarrow}}

\newcommand{\CJ}{{\euls{C}_J}} 
\newcommand{\CI}{{\euls{C}_I}} 
\newcommand{\dalpha}{\alpha}
\newcommand{\gainly}{robust}
\DeclareMathOperator{\Ch}{Ch}

 \newcommand{\deltah}{{h}}
 \newcommand{\deltav}{{\delta}}

\begin{document}

\title[Invariant holonomic systems]{Invariant holonomic systems on symmetric spaces and other polar representations}
\author{G. Bellamy, T. Nevins and J. T. Stafford }
\address{(Bellamy)
	School of Mathematics and Statistics, University of Glasgow, University Gardens, Glasgow G12 8QW, Scotland. }\email{gwyn.bellamy@glasgow.ac.uk}
\address{(Nevins)  
	Department of Mathematics, University of Illinois at Urbana-Champaign, Urbana, IL 61801, USA.}
\address{(Stafford) School of Mathematics,  The University of Manchester,   Manchester M13 9PL,
	England.}
\email{Toby.Stafford@manchester.ac.uk}

\dedicatory{\it Dedicated to the memory of our friend and coauthor Tom Nevins}  


\begin{abstract}     
	Let $\g$ be a complex reductive  Lie algebra, with adjoint group $G$, acting on a symmetric space $V$, with associated 
	little Weyl group $W$ and  discriminant $\delta$.
	Then  $G$ also acts on the ring of differential operators $\eD(V)$  and we write $\tau : \g\to \eD(V)$ for the differential of	this action. Consider
	the \emph{invariant holonomic system}
	$$ 
	\eG =\eD(V)\Big/ \Bigl(\eD(V) \tau(\g) +  \eD(V)(\Sym V)^G_+ \Bigr).
	$$ 
	In the diagonal case, when $V=\g$, this module has been intensively studied.  For example, the fact that $\eG$ has no 
	$\delta$-torsion factor module     lies at the heart of Harish-Chandra's regularity theorem, while Hotta and Kashiwara have shown that $\eG$ is semisimple, which has important consequences for the geometric representation theory of $\g$. 

	We study analogous problems for a symmetric space and, more generally, for a visible stable polar $G$-representation $V$. By work of Levasseur and the present authors, there exists a \emph{radial parts map} 
	$$\rad: \eD(V)^G/(\eD(V)\tau(\g))^G\to \Ak(W),$$  where $\Ak(W)$ is the spherical subalgebra  of a Cherednik algebra $\Hk(W)$. When  $\Ak(W)$ is simple, $\rad$ is surjective and  we generalise work of Sekiguchi and Galina-Laurent by proving that $\eG$ has no factor nor submodule that is $\delta$-torsion.  This answers a conjecture of Sekiguchi. Moreover we show that $\eG$ is semisimple if and only if the Hecke algebra $\euls{H}_q(W)$ associated to $\Ak(W)$ is semisimple, thereby answering a conjecture of Levasseur-Stafford. 
	
	By twisting the radial parts map, we study a family of invariant holonomic systems and, more generally, families of admissible modules on $V$. We introduce shift functors that allow us to pass between different twists. We show that the image of each simple summand of $\eG$ under these shift functors is described by Opdam's KZ-twist.

 \end{abstract}

\maketitle  

\clearpage
   \tableofcontents 
     
\section{Introduction}  
Fix a complex connected reductive algebraic group $G$ with Lie algebra $V:=\g=\Lie G$.  Then $G$ acts via the adjoint action on   the coordinate ring $\C[V]$  and hence on the ring of differential operators $\eD(V)$. We identify $\Sym V$ with the constant coefficient differential operators on $V$ and let $\tau: \g\to \eD(V)$ be the differential of the action of $G$ on $\C[V].$  Fundamental objects in the application of  $\eD$-modules to Lie theory are the  Harish-Chandra modules,  defined as follows.  Given $\lambda\in \h^*$, where $\h$ is a Cartan subalgebra of $\g$,   the corresponding  \emph{Harish-Chandra $\eD(V)$-module} is defined to be 
\begin{equation}\label{HC-intro}
\eGt_\lambda = \eD(V)/\bigl( \eD(V)\tau(\g)+ \sum_{p\in (\Sym V)^G}   \eD(V)(p-p(\lambda))\bigr).
\end{equation}
These modules are important for several reasons. First, the space of \emph{invariant eigen\-distri\-butions}  on a real form $\g_0$ of $\g$ can be identified with $\Hom_{\eD(V)}\left(\eGt_\lambda,\, \text{Dist}(\g_0)\right)$, where $ \text{Dist}(\g_0)$ denotes the space of distributions on   $\g_0$.   If $\delta\in \C[V]$ denotes the discriminant and $V_{\reg}=(\delta\not=0)$ the regular locus, then the essence of Harish-Chandra's fundamental result on     the regularity of 
invariant eigendistributions \cite{HC3}  lies in the theorem that \emph{$\eGt_\lambda$ has no nonzero quotient supported on the singular  locus $V\smallsetminus V_{\reg}$}; see,  \cite{HC3} and  \cite[Theorem~6.7.2]{HK}. More generally,  $\eGt_\lambda$ has no nonzero quotient supported on the complement to the open set $(d\not=0)$ for any $0\not=d\in \C[V]^G$; see,  \cite{HC3} or   \cite[Theorem~8.3.5]{Wa1}.  This, in turn, leads to Harish-Chandra's regularity theorem for  characters of a real reductive group; see \cite[Section~8.4]{Wa1}.  Secondly,  the study of $\eGt_\lambda$  is integral to the representation theory of $\g$. Here Hotta and Kashiwara have obtained deep results on the module $\eGt_\lambda$ and its corresponding system of differential equations. Central to their theory is the fact that \emph{$\eGt_\lambda$ is a semisimple $\eD(V)$-module, with known simple factors}; see \cite[Theorem~5.3]{HK}. 

\subsection*{Symmetric spaces} 

We are interested in generalising these results to symmetric spaces $V$ and, more generally,  to polar $G$-representations $V$ satisfying certain necessary conditions. We begin with symmetric spaces, defined as follows.  Let $\gtilde = \mathrm{Lie} \, \Gtilde$ be a second  complex reductive Lie algebra with involution $\theta$. Set $\g=\gtilde^\theta$, with corresponding algebraic group $G$ the connected component of $\Gtilde^{\theta}$,  and  $(-1)$ $\theta$-eigenspace $V$. Then $(\gtilde,\g)$ is a \emph{symmetric pair} and  $V$ is the corresponding  \emph{symmetric space}; it is often denoted $\mf{p}$ in the literature. The adjoint representation of $G$ on $\g$ is a particular symmetric space, by taking $\gtilde=\g\times\g$ with the obvious involution. The basic theory for such spaces  is described in Section~\ref{Sec:examples} or \cite{He1}.  

Generalising Harish-Chandra's work, Sekiguchi \cite{Se} studied equivariant eigendistributions on a real form $V_{\R}$ of the symmetric space $V$. The definition \eqref{HC-intro} of the  Harish-Chandra  modules $\eGt_{\lambda}$ works here as well. There is again an analogue of the discriminant $\delta$ with regular locus $V_{\reg}=(\delta\not=0)\subseteq V$; see  Section~\ref{Sec:polarreps} for the details. Under the assumption that the symmetric space $V$ is \textit{nice}, (the formal definition is given in Definition~\ref{nice-space}, but it suffice to say here that about a third of the infinite families of symmetric spaces satisfy this condition) Sekiguchi made the following conjectures.
\begin{enumerate}
	\item[(C1)] \cite[Conjecture~7.1]{Se} $\eGt_{\lambda}$ is regular holonomic (locally, as an analytic $\mathcal{D}$-module, as in \cite{Se}). 
		\item[(C2)]  \cite[Conjecture~7.2]{Se} $\eGt_{\lambda}$ is the minimal extension of its restriction to the locus of regular semisimple elements in $V$.   Equivalently    $\eGt_{\lambda}$ has no nonzero  quotient nor    submodule that is supported on  $V\smallsetminus V_{\reg}$.
\end{enumerate}
 We note  the following  standard consequence of (C2).
\begin{enumerate}
	\item[(C2$'$)]  the only $G$-equivariant eigendistribution  $T$ supported on the real form of $V \smallsetminus V_{\reg}$ is $T=0$.
\end{enumerate}
It was also conjectured in  \cite[Conjecture~C4]{LS3} that
\begin{enumerate}
	\item[(C3)] $\eGt_{\lambda}$ is a semisimple $\eD(V)$-module. 
\end{enumerate}

Conjecture (C1)  has since been proved in \cite[Theorem~2.2.1]{La2}  for any symmetric space.

For a nice symmetric  space $V$, one of the main goals of this paper is  to prove (C2) and (C2$'$),  to determine precisely when (C3) holds and  to give a  detailed analysis of the structure of $\eGt_{\lambda}$ and  related $\eD$-modules. In fact, we do both of these not only for nice spaces, but  for the more  general \textit{\gainly \ spaces}, as introduced in \cite{BLNS} and recalled in Definition~\ref{nice-space}.

 For nice symmetric spaces, it  was shown in \cite{LS3}   that $\eGt_{\lambda}$ has no nonzero  quotient module supported on $V\smallsetminus V_{\reg}$, with  a different proof using the theory of $b$-functions given in \cite[Corollary~1.6.3]{GL}. 
Much harder is to show that $\eGt_{\lambda}$ has no submodule supported on this complement, but as the next result shows,  we are able to completely answer Sekiguchi's conjecture  (C2). In fact, we are able to prove the following  much stronger torsion-free statement that also gives  an analogue of  Harish-Chandra's 
theorem on invariant eigendistributions  with nilpotent support \cite[Theorem~5]{HC2}; see the comments after \cite[Theorem~5.2]{LS}.

\begin{theorem} \label{intro-torsionfree1}   {\rm (Theorem~\ref{torsionfree} and Corollary~\ref{LS3-corollary})}
	Assume that  $V$ is a \gainly \ symmetric space and let $d=\delta$ or, more generally, take any $0\not=d\in \C[V]^G$. Then $\eGt = \eGt_\lambda$ has no nonzero $d$-torsion  submodule, nor $d$-torsion factor module.
\end{theorem}

Here, a module $M$ is \emph{$d$-torsion} if, for each $m\in M$ there exists $j\in \mathbb{N}$ such that $d^j m=0$; equivalently $M$  is supported on $(d=0)$.  As an easy consequence we obtain:

\begin{corollary}\label{cor:intro-torsionfree1}	
	Assume that  $V$ is a \gainly \ symmetric space. 
\begin{enumerate}
		\item {\rm (Corollary~\ref{torsionfree-corollary})}  $\eGt$ is the minimal extension of $\euls{L} := \eGt |_{V_{\mathrm{reg}}}$; 
		equivalently,  $\eGt = j_{!*} \euls{L}$   for the inclusion $j \colon V_{\mathrm{reg}} \hookrightarrow V$.  
		\item {\rm (Corollary~\ref{cor:distributions})}   The only $G$-invariant eigendistribution  $T$ supported on a real form of $V \smallsetminus V_{\reg}$ is $T=0$.   	
	\end{enumerate}
\end{corollary}

In particular, both Theorem~\ref{intro-torsionfree1} and Corollary~\ref{cor:intro-torsionfree1} apply to the nice symmetric space introduced by Sekiguchi and therefore answer Conjectures (C2) and (C2$'$). We are also able to  completely  answer  Conjecture (C3), see Theorem~\ref{intro-thm:semi-simplicity},  but for the moment we  will simply note that the answer is ``not often.''

\subsection*{Polar representations and quantum Hamiltonian reduction}	
One of the achievements of this paper is  to  extend this result beyond the case of nice symmetric spaces to a significant class of polar representations. 
 In order to explain these results, and the ideas behind them, we return to Harish-Chandra's work.  In \cite{Wallach}, Wallach  gave a second approach to Harish-Chandra's theory by making extensive  use of Harish-Chandra's homomorphism $\rad: \dd(V)^G\to \dd(\h)^W$, where $\h$ is a Cartan subalgebra of $\g$ with Weyl group $W$.  This homomorphism generalises the Chevalley isomorphism $\gamma: \C[\g]^G \isom \C[\h]^W$.  Wallach's theory, complemented by \cite{LSSurjective},   shows that $\rad$ is surjective and, as a consequence,  one obtains  significantly simpler proofs of Harish-Chandra's results, as well as  applications to 
Weyl group representations and the Springer correspondence.

 The radial parts map  $\rad$ is a special case of quantum Hamiltonian reduction that has   been used extensively in other areas of representation theory, notably in the theory of rational Cherednik algebras
  \cite{BG,  EG, EGGO, GGS,  Lo0, MN, OblomkovHC}. 
  As  we will see, there also exists a  radial parts map $\rad$ for polar representations, although now the image will be the spherical subalgebra of a Cherednik algebra. Nevertheless the representation theory of that image will be key to our generalisations of the above results.

  Thus one should look for classes of representations for which there is an analogue of Chevalley's Theorem. A natural class is that of polar representations.    Following Dadok and Kac \cite{DadokKac},  a 
$G$-representation $V$  is called  \textit{polar}\label{defn:polar}  if there exists a semisimple element $v \in V$ such that 
$$
\h := \{ x \in V \, | \, \mf{g} \cdot x \subset \mf{g} \cdot v \}
$$
satisfies $\dim \h = \dim V \git G$.  One of the fundamental properties of polar representations is that the finite group $W = N_G(\h)/ Z_G(\h)$ acts as a complex reflection group on $\h$ and, by restriction, defines a Chevalley isomorphism $V \git G \cong \h / W$; see  \cite[Theorem~2.9]{DadokKac}. In this context, one has a discriminant $\delta$ in $\C[V]^G \cong \C[\h]^W$ and we define the regular locus $V_{\reg}$ to be the open set $(\delta\not=0)$.  The group  $W$ is   called \emph{the Weyl group} associated to this data. Thus, polar representations are a natural situation---perhaps the most general one---where one can hope to have a radial parts map.

It is immediate that invariant differential operators act on invariant functions, and so there is a natural map $\dd(V)^G\to \dd(V\git G) \cong \dd(\h\git W)$, but this almost never has image in $\dd(\h)^W$. However, for stable polar representations, as defined in  Section~\ref{Sec:polarreps}, restricting the above map to the regular locus $V_{\reg}=(\delta\not=0)$   does at least give an  isomorphism 

\begin{equation}\label{eq:intro-localizatioradiso}
	(\dd(V_{\reg}) / \dd(V_{\reg}) \tau(\g))^G  \ \isom \  \dd(\h_{\reg})^W, 
\end{equation}
   where $\tau:\g\to \dd(V)$ is the differential of the $G$-action on $V$.
   
In the companion paper \cite{BLNS}, we show that radial parts maps exist for any polar representation, although the image   lives not in
         $\dd(\h)^W$ but  inside the \emph{spherical subalgebra} $\Ak(W)$ of a rational Cherednik algebra $\Hk(W,\h)$, in the sense of \cite{EG}.
 
\begin{theorem} {\rm (\cite[Theorem~5.1, Corollary~6.10, Theorem~7.8]{BLNS})}  \label{thm:intro-radial-exists}  
Let $V$ be a polar representation. There exists a  spherical algebra $\Ak(W)$, for some  parameter $\kappa $, 
  such that  that there is a radial parts map 
 $$
 \rad:  \dd(V)^G \to \Ak(W).
 $$
 
Moreover, the morphism $\rad$ has the following properties.
\begin{enumerate}
\item The   map
$\rad  $ restricts to give
filtered isomorphisms $\rr \colon \C[V]^G
\stackrel{\sim}{\longrightarrow} \C[\h]^W$ and $\rrpp \colon
(\Sym \, V)^G \stackrel{\sim}{\longrightarrow} (\Sym \, \h)^W$. 
 \item If $V$ is stable then the restriction of $\rad$ to $V_{\reg}$ induces \eqref{eq:intro-localizatioradiso}.  
 \end{enumerate}
\end{theorem}

 \begin{remark}\label{intro-twisting-rad} In applications it is often important to be able to twist the radial parts map $\rad$ by a linear character $\chi$ of $\g$. Although we will not discuss this generalisation in the introduction, both Theorem~\ref{thm:intro-radial-exists} and the other  results of this paper are proved  in this generality.
\end{remark}

It is also proved in \cite{BLNS} that $\rad$ is surjective under very mild conditions, of which the one that is relevant to this paper is the following.

\begin{theorem}\label{thm:intro-radial-surjective}  {\rm (\cite[Theorem~7.10]{BLNS})} The morphism $\rad \colon \dd(V)^G \to \Ak(W)$ is surjective whenever 
 $\Ak(W)$ is a simple algebra. 
\end{theorem}

Note that $\Ak(W)$ is a (flat) deformation of the fixed ring $\dd(\h)^W$ and so generically  it will be simple. However, this is not always true and, indeed, its failure produces many of the most interesting Cherednik algebras.

The main results of this paper are proved under the following assumption, \emph{which will be assumed for the rest of the introduction}:

 \begin{condition}\label{intro-hypothesis} $V$ is a visible,  stable polar representation for which $\Ak(W) $ is a simple ring.  
\end{condition}

Visibility and stability are defined in Section~\ref{Sec:polarreps} and 
the main results of this paper fail without   these assumptions; see Section~\ref{sec:otherexamples} in particular. 
So these conditions are quite natural and, moreover, they hold in many  significant  cases. In particular, both symmetric spaces and  examples such as representation spaces for the cyclic quiver (see Section~\ref{Sec:Quivers}) are visible,  stable polar representations. Moreover, the algebra $\Ak(W)$ is simple for these quiver representations; see Proposition~\ref{prop:cyclcirad0simple}. In contrast,   $\Ak(W)$ is \textit{not} simple for all symmetric spaces; indeed, it is shown in \cite[Theorem~8.23]{BLNS} that $\Ak(W)$ is simple if and only if the symmetric space is \gainly, with  a complete classification of the \gainly\ symmetric spaces  given by the tables in \cite[Appendix~B]{BLNS}. There are also well-known examples of (necessarily non-\gainly)  symmetric spaces for which  Theorem~\ref{intro-torsionfree1} and Corollary~\ref{cor:intro-torsionfree1} fail; see  \cite[(6.2)]{Se} and Remark~\ref{intro-torsionfree-remark}(2).

   \subsection*{The main results.}
Before stating the main  results,  we have to slightly modify the definition of the Harish-Chandra module so that it also works in the generality of \eqref{intro-hypothesis}. It is almost immediate from the construction that
\[
\ker(\rad)\ = \ \{\theta\in \eD(V)^G : \theta * \C[V]^G=0\} \ \supseteq \ [\eD(V)\tau(\g)]^G.
\]
If $\mf{m}_{\lambda}$ denotes the maximal ideal in $(\Sym V)^G \cong (\Sym \h)^W$ corresponding to the orbit of $\lambda \in \h^*$ in $\h^*/W$ then set 
\[\eG_\lambda = \eD(V)/ \Big(\eD(V)\tau(\g)  + \eD \ker(\rad) + \eD \mf{m}_{\lambda}\Big).
\]
Alternatively, if   $\eMt=\eD(V)/\eD(V)\tau(\g)$, with factor module 
\[
\eM = \eD(V)/(\eD\tau(\g) + \eD(V)\ker(\rad)), 
\] 
then one can identify  
\begin{equation}\label{intro-rem:semiprime2}
\eG_\lambda\ = \  \eMt\otimes_{\Ak(W)} \left( \Ak(W)  / \Ak(W) \mf{m}_\lambda \right) \ = \  \eM\otimes_{\Ak(W)} \left(\Ak(W)  / \Ak(W) \mf{m}_\lambda\right); 
\end{equation} 
see Remark~\ref{rem:semiprime2} for the equivalence.  

The main tool for interrelating the representation theory of $\dd(V)$ to  that of $\Ak(W)$ is the 
 $\bigl(\dd(V),\, \Ak(W)\big)$-bimodule $\eM$. The difference between this module and $\eMt$ is also not too serious since, as a bimodule,  $\eMt=\eM\oplus \euls{N}$, for a $\delta$-torsion module $\euls{N}$; see Proposition~\ref{semiprime2}.

  We will be largely interested in    the  following categories of representations over $\Ak(W)$,
   respectively $\dd(V)$.
First, and by analogy with the category $\euls{O}$ for a semisimple Lie algebra, we have  the category $\Osph$ for $\Ak(W)$. This is defined to be the full subcategory 
of the category of all finitely generated $\Ak(W)$-modules on which $(\Sym  \h)^W$ acts locally finitely.    
Similarly, 	a finitely generated $G$-equivariant  left  $\dd(V)$-module $N$ is \emph{admissible} if  $(\Sym  V)^G$ acts locally finitely on $N$; see Definition~\ref{defn:admissible}. The  full subcategory of all such modules is written $\euls{C}$, and it is readily checked that $\eG_\lambda\in \euls{C}$. Via the (microlocal) functor of Hamiltonian reduction, admissible $\ddd$-modules have been shown to play a key role in understanding geometric category $\euls{O}$ associated to quantisations of Higgs branches (in particular, quiver varieties), as described in \cite{BLPWAst,GenCatOWebster}. 

Our main result  generalising Theorem~\ref{intro-torsionfree1} states:

\begin{theorem} \label{intro-torsionfree} {\rm (Theorem~\ref{torsionfree} and Corollary~\ref{torsionfree-corollary})}  
	Let $\eG =\eG_\lambda$  or, indeed,  let $\eG=\eM\otimes_{\Ak(W)}\eP$ for some  projective object $\eP\in \Osph$.  Let $0\not=d\in \C[V]^G$; for example, $d=\delta$.
	Then:
	\begin{enumerate}
		\item  $\eG$ is a $d$-torsionfree  holonomic $\eD(V)$-module;   
		\item  $\eG$ has no nonzero factor module $T$ that is $d$-torsion;
		 		\item $\eG$ is the minimal extension of $\euls{L} := \eG |_{V_{\mathrm{reg}}}$.
		 	\end{enumerate}
\end{theorem}
 
\begin{remark}\label{intro-torsionfree-remark}
	(1) By Lemma~\ref{lem:Vregintconnection}, $\euls{L} = \eG |_{V_{\mathrm{reg}}}$ is  an integrable connection. 
	 	
	(2)  The theorem  fails if $\Ak(W)$ is not simple; indeed if $\Ak(W)$ is not simple then there 
	always exists a nonzero 
	$\delta$-torsion factor of $\eG_0$; see Proposition~\ref{torsionfree-converse}.
 
 (3)  If $V$ is a \gainly \ symmetric space, then $\eGt_\lambda=\eG_\lambda$ by Theorem~\ref{LS3-theorem} and so 
 Theorem~\ref{intro-torsionfree} implies   Theorem~\ref{intro-torsionfree1}  and Corollary~\ref{cor:intro-torsionfree1} .
 \end{remark}
  
We now turn to Conjecture (C3). In the diagonal case (when $V=\g$)   one of the main results of  Hotta and Kashiwara 
\cite[Theorem~5.3]{HK}   shows that $\eG_0$ is semisimple, while  \cite[Theorem~B]{AJM} shows, more generally, that  each $\eG_\lambda$ is semisimple. As Hotta and Kashiwara show,   this    has significant 
applications to 
the representation theory of $\g$, which is one reason for the importance of Conjecture~(C3). 

Our next main result gives a strikingly complete answer to Conjecture (C3), written  in terms of    the Hecke algebra  $\euls{H}_{q}(W)$ associated to $\Ak(W)$; see Definition~\ref{Hecke-defn}.

\begin{theorem}\label{intro-thm:semi-simplicity}   {\rm (Theorem~\ref{thm:semi-simplicity})}   	The module  $\eG_{\lambda}$ is semisimple if and only if the Hecke algebra 
	$\euls{H}_{q}(W_{\lambda})$ is semisimple. 
\end{theorem}

 This result also holds when $\Ak(W)$ is not simple; see Corollary~\ref{thm:semi-simplicity2} for the
	details. However, the result definitely requires the stability hypothesis since the analogue of  Theorem~\ref{intro-thm:semi-simplicity} fails for representation spaces over framed quivers; see Example~\ref{rem:framed-quiver} for the details.

 It is easily determined  from   the literature   whether   a given  Hecke algebra  $\euls{H}_q(W_{\lambda})$ is semisimple.  Applying this theory to the Hecke algebras arising from  symmetric spaces gives 
 the following result.   Here, we are using  the classification of  symmetric spaces from \cite[Chapter~X]{He1}, as explained in more detail in Section~\ref{Sec:examples}.

\begin{corollary}\label{intro-cor:symmetricspacessHecke} {\rm (Theorem~\ref{cor:symmetricspacessHecke}) }
	The Harish-Chandra module $\eG_0$ of a symmetric space $V$ is semi\-simple if and only if  each simple factor of $V$ is either of diagonal  type or of type 
	\[\begin{aligned} \text{$\mathsf{AII}_n$} \  & \ \text{
	where $(\widetilde{\g},\g)= \mathfrak{sl}(2n), \mathfrak{sp}(n)$, }\\
	\text{$\mathsf{DII}_{p}$}  \ &\ \text{ 
where $(\widetilde{\g},\g)= \mathfrak{so}(2p), \mathfrak{so}(2p-1)$, or }\\
 \text{$\mathsf{EIV}$} \, \ &\ \text{  where $(\widetilde{\g},\g)= \mathfrak{e}(6), \mathfrak{f}(4)$.  }
 \end{aligned}\]\end{corollary}

 When $\euls{H}_{q}(W_{\lambda})$  is semisimple, it follows from Theorem~\ref{intro-thm:semi-simplicity} that the simple summands of $\eG_{\lambda}$ are labelled by the irreducible representations of $\euls{H}_q(W_{\lambda})$,  with the latter acting on the multiplicity space of each simple summand. The following result is the natural extension of the (algebraic) Springer correspondence for $V=\g$  proved in 
  \cite[Theorem~5.3]{HK} and \cite[Theorem~B]{AJM}. 

\begin{corollary}\label{intro-cor:simplicity}  {\rm (Corollary~\ref{cor:simplicity})}   	 If  $\euls{H}_{q}(W_{\lambda})$ is semisimple then 
	$$
	\eG_{\lambda} = \bigoplus_{\rho \in \mr{Irr} \euls{H}_q(W_{\lambda})} \eG_{\lambda,\rho} \o \rho^*
	$$
	as a $(\eD(V),\euls{H}_q(W_{\lambda}))$-bimodule. Moreover,  each $\eG_{\lambda,\rho}$ is an  irreducible  $\eD(V)$-module 
	and the $\eG_{\lambda,\rho}$ are non-isomorphic for distinct $(\lambda,\rho).$ 
\end{corollary}

 \subsection*{Specific examples}
 With hindsight it is   easy to construct  counterexamples to Conjecture~(C3); indeed this happens for the rank one symmetric space corresponding to the symmetric pair $(\mathfrak{sl}(2),\,\mathfrak{so}(2))$.  This 
is described in detail in  Section~\ref{HC-example}, and  most of the results of this paper can be proved from first principles
in this special case--this example can be viewed an alternative introduction to the paper.
 
 The example in Section~\ref{HC-example} can also be regarded as a representation space for a cyclic quiver. 
 This is considerably generalised in  Section~\ref{Sec:Quivers},  where we consider  the representation space $V=\mathrm{Rep}(Q_{\ell}, n\mathfrak{d})$, where $\mathfrak{d}=(1,\dots,1)$ is the minimal imaginary root over the cyclic quiver $Q_{\ell}$ with $\ell $ vertices.  Hypothesis~\ref{intro-hypothesis} always holds here; see 
Proposition~\ref{prop:cyclcirad0simple}.   Moreover, for these quiver representations  $\eMt=\eM$ and so $\eGt_0=\eG_0$.  However, $\eG_0$ has a complicated structure: even when $n=1$,  Corollary~\ref{thm:neq1cyclicminext} shows that $\eG_0$ has simple socle and simple top, both isomorphic to the polynomial representation $\C[V]$, yet  in any composition series of $\eG_0$ there are $\ell(2^{\ell-1}-1)$ simple subfactors that are $\delta$-torsion.

\subsection*{Admissible modules and KZ-twists}

In  the course of proving Theorems~\ref{intro-torsionfree} and~\ref{intro-thm:semi-simplicity} we also obtain a considerable amount of extra information about the   category 
$\euls{C}$ of admissible $\eD(V)$-modules. Most notably, the Harish-Chandra module $\eG_{\lambda}$ possesses another remarkable property when considered as an admissible module. 

\begin{theorem}\label{intro-G-is-injective}{\rm (Theorem~\ref{G-is-injective})} 
\begin{enumerate}	\item  $\eG_{\lambda}$ is both projective and injective in $\euls{C}$.  
	 
\item More generally, if $\eG = \eM\otimes_{\Aak} \eP$ for a projective object $\eP \in \Osph$, then $\eG$ is both projective and injective as an object in $\euls{C}$. 
\end{enumerate}
\end{theorem}

As a consequence, in the cases where $\euls{H}_q(W_{\lambda})$ is also semisimple,  Theorem~\ref{intro-G-is-injective} implies that the simple Springer modules $\eG_{\lambda,\rho}$ of Corollary~\ref{intro-cor:simplicity} span semisimple blocks of $\euls{C}$.

Next, we explain the relevance of Opdam's KZ-twist to the decomposition of the Harish-Chandra module. First, we note that the radial parts map allows us to relate equivariant $\ddd$-modules to representations of a spherical Cherednik algebra. By twisting the radial parts map, we can relate twisted equivariant (or monodromic) $\ddd$-modules to other spherical Cherednik algebras. This produces a family of admissible categories $\euls{C}_{\kappa}$ depending on the parameter $\kappa$. 

The full subcategory of $\euls{C}_{\kappa}$ consisting of modules for which  the action of $(\Sym V)^G_+$ is locally nilpotent is denoted $\euls{C}_{\kappa,0}$. As we show in Section~\ref{Sec:Shiftfunctors}, Theorem~\ref{intro-G-is-injective} enables one to define   a geometric analogue   $\mr{KZ}_G \colon 	\euls{C}_{\kappa, 0} \to  \euls{H}_q(W) \lmod$ of the usual Knizhnik-Zamolodchikov   functor $\mr{KZ}$. In particular, we can  define twists $\euls{C}_{\kappa,0} \to \euls{C}_{\kappa',0}$ that have  properties
 analogous  to those of the KZ-twists defined by Opdam \cite{Opdamlectures} for Weyl groups, and studied more generally for complex reflection groups in \cite{BerestChalykhQuasi}.

The construction is as follows.
By analogy with the shift functors $\euls{T}_{\kappa,\kappa'} \colon \Osph_{\kappa,0} \to \Osph_{\kappa',0}$ for Cherednik algebras, as introduced in \cite{BerestChalykhQuasi},  we define left exact shift functors $\mathbb{T}_{\kappa,\kappa'} \colon \euls{C}_{\kappa,0} \to \euls{C}_{\kappa',0}$. Similarly, the classical KZ functor for Cherednik algebras can be used to define a twist functor $\mr{kz}_{\kappa,\kappa'} \colon \euls{H}_{q}(W) \lmod \rightarrow \euls{H}_{q'}(W) \lmod$. All these functors are compatible in the expected way; see Proposition~\ref{prop:KZtwistdiagram}.

The parameter  $\kappa$ is said to be \emph{regular} if the associated Hecke algebra $\euls{H}_q(W)$ is semisimple.  If $\kappa,\kappa'$ are regular then $\mr{kz}_{\kappa,\kappa'}$ is an equivalence, indeed  if one further assumes that 
   $q = q' = 1$, then $\mr{kz}_{\kappa,\kappa'}$    is just Opdam's KZ-twist
  \cite[Corollary~3.8]{Opdamlectures}. For regular parameters, Theorem~\ref{intro-G-is-injective} implies that this KZ-twist encodes the action of the shift functor on the simple summands of the Harish-Chandra module $\eG_0$ as follows.   
\begin{theorem}\label{intro-thm:KZGtwistsimples} {\rm (Theorem~\ref{thm:KZGtwistsimples})}
	Assume that $\kappa$ and $\kappa'$ are regular. Then, in the notation of Corollary~\ref{intro-cor:simplicity}, 
	$$
	\shT_{\kappa,\kappa'}(\eG_{0,\rho}) \cong \eG_{0,\rho'}, \quad \textrm{where} \quad \mr{kz}_{\kappa,\kappa'}(\rho) \cong \rho'.
	$$
\end{theorem}

\subsection*{The proofs}   
The  broad outline of the proofs in this paper are similar to those in \cite{AJM} and \cite{LS3}, although with some significant differences. 
First, unlike the case of \gainly \ symmetric spaces we   do not  know whether $\ker(\rad)=\eD(V)\tau(\g)$, which is why we needed to distinguish between $\widetilde{\eG}_\lambda$ and $\eG_{\lambda}$. To circumvent this problem we  show that   \emph{  $R=\eD(V)^G/[\eD(V)\tau(\g)]^G $ 
	splits as a 
	sum of rings, with one factor being $\Im(\rad)=\Ak(W)$ and the other being $\delta$-torsion}; see Proposition~\ref{semiprime2}. 
With this result in hand, the difference between $R$ and $\Ak(W)$, or between $\eMt$ and $\eM$ is not especially significant. The   proof of 
Theorem~\ref{intro-torsionfree} then consists of controlling the $d$-torsion subfactors  of  $\eG_\lambda$ and $\eM$. 
Here it is important to be able to pass between sub- and factor-modules, and for this the following intertwining result, proved using the techniques of \cite{Bj},  
forms a crucial step. 

\begin{theorem}\label{intro-intertwining}  {\rm (Theorem~\ref{intertwining})} 
	Let $N={\Aak}\otimes_{(\Sym  \h)^W}L$ for some finite dimensional $(\Sym  \h)^W$-module $L$, or let 
	$N$ be a projective object in $\Osph$. Then
	\[ \Ext_{\eD(V)}^{\dim V} \left( \eM\otimes_{{\Aak}} N \right) \ \cong \ \Ext^{\dim \h}_{{\Aak}}  (N,\,{\Aak} ) \otimes_{{\Aak}} 
	\eM',\]
	where $\eM'= \Ext^{m}_{\eD(V)}(\eM, \, \eD(V)) $, for $m= \dim V-\dim\h$.      
\end{theorem}

Similarly,  in  proving Theorem~\ref{intro-G-is-injective}, it is not hard to prove that $\eG_\lambda$ is a projective object in $\euls{C}$, but injectivity is much harder. The key duality between injectivity and projectivity is obtained by means of the following result. 

\begin{theorem}  {\rm (Theorem~\ref{projectivity-theorem-stable})} \label{thm:introM'}
 The right $\dd(V)$-module $\eM'=\Ext^m_{\dd(V)}(\eM, \dd(V))$ is projective as an object in the category of $G$-equivariant right $\dd(V)$-modules. 
\end{theorem}

Theorem~\ref{intro-intertwining} also plays a key r\^ole in the proof of Theorem~\ref{intro-thm:semi-simplicity}.
In particular, combined with the next result, Theorem~\ref{intro-intertwining} allows one to move between 
$\Osph$-modules and $\euls{C}$-modules. This is important since some aspects of the above theorems are easier to prove in one category, and some in the other.

\begin{proposition}\label{intro-cor:hom(ML)}  {\rm (Corollary~\ref{cor:hom(ML)})}
	If $L$ is a nonzero left $\dd(V)$-submodule of $\eG_\lambda$, then 
	$ L^G\not=0$.
\end{proposition}

\medskip
\noindent
{\bf Acknowledgements:}  We would like to  thank  Thierry Levasseur  for  numerous helpful conversations. We would also like to thank Kari Vilonen and Ting Xue for conversations about character sheaves. 

The second  author was partially supported by NSF grants DMS-1502125 and DMS-1802094 and a Simons Foundation Fellowship.
The third author is partially supported by a Leverhulme Emeritus Fellowship.

Part of this material is based upon work supported by the National Science Foundation under   Grant   DMS-1440140, while the third 
author was   in residence at the Mathematical Sciences Research Institute in Berkeley, California 
during the Spring 2020 semester.

\medskip
Finally, Gwyn Bellamy and Toby Stafford would like to acknowledge the great debt they hold for Tom Nevins who tragically passed away while this paper was in preparation. He was a dear friend and a powerful mathematician who taught us so much about so many parts of mathematics. He is sorely missed. 



\section{A Illustrative Example}\label{HC-example}

In this short section we  provide an example that exemplifies the results described in Theorems~\ref{intro-torsionfree} and \ref{intro-thm:semi-simplicity} for both symmetric spaces and quiver representations, while at the same time being elementary  enough to understand from first principles. 

We take $V=\C^2$ with $\C[V]=\C[x_0, x_1]$ and $G =\C^\times$ acting on $V$ such that $x_0$ has weight one and $x_1$ has weight $-1$. Explicit computations allow us to describe completely the submodule structure of the Harish-Chandra module $\eG$. Write $\eD= \eD(V) =\C\langle x_0,x_1,\partial_0,\partial_1\rangle$, for 
$ \partial_j=\frac{\partial}{\partial x_j}$. We identify $\g=\mathrm{Lie}\, G $ with its image in $\eD$ where it is spanned by $\nabla= x_0\partial_0-x_1\partial_1$.  
The  discriminant is $\delta = x_0x_1$. 

The space $V$ can be identified with the space of representations of the cyclic quiver $Q$ with two vertices and dimension vector $(1,1)$, as above. Alternatively, it can be regarded as the symmetric 
space $V=\mathfrak{p}$ corresponding to the symmetric pair $\left(\mathfrak{sl}(2), \mathfrak{so}(2)\right)$ of type AI$_2$, in the notation of \cite[Appendix~B]{BLNS} or \cite[Chapter~X]{He1}.

Let $I=I(\eG)$ denote the left ideal in $\eD$ defining $\eG$. For simplicity we write a left ideal $\sum_{i=1}^r \eD m_i$ as $(m_1,\dots, m_r)$. Then, by definition, $\eG = \eD/(\nabla,\Delta)$, for $\Delta=\partial_0\partial_1$ and so 
$I=(x_0\partial_0-x_1\partial_1, \, \partial_0\partial_1).$ There are then four obvious ideals containing $I$: 
$$
J_0=( x_0\partial_0 ,\,  x_1   \partial_1 ,\, \partial_0 \partial_1),  \
J_1=(  \partial_0 ,\,  x_1 \partial_1 ), \ J_2=( x_0 \partial_0 ,\, \partial_1) \ \text{and}
\  J_\infty = ( \partial_0,\, \partial_1 ).
$$
These are arranged as 
{\small \begin{equation} \label{HC-diag}
\begin{tikzcd}
	& \eD \ar[d,dash] & \\
	& J_{\infty} 
	\ar[dl,dash]  \ar[dr,dash]   &  \\
	J_1 \ar[dr,dash]  & & J_2\ar[dl,dash] \\
	& J_0\ar[d,dash] & \\
	&I& 
\end{tikzcd}
\end{equation}}
We also have isomorphisms:
$$
\eD / ( \partial_0, \partial_1) \stackrel{\sim}{\longrightarrow}J_0/I , \quad 1 \mapsto x_0 \partial_0 = x_1 \partial_1,
$$
$$
\eD / ( x_0, \partial_1 ) \stackrel{\sim}{\longrightarrow}   J_1/J_0, \quad 1 \mapsto \partial_0,\quad\text{ and }\quad
\eD / ( x_1, \partial_0 ) \stackrel{\sim}{\longrightarrow}   J_2/J_0, \quad 1 \mapsto \partial_0. 
$$
To see this,   note in each case that the module on the left hand side is trivially a simple $\eD$-module and maps onto the right hand side. Since the right hand side is clearly nonzero, this gives the required identity. This also implies that each factor in \eqref{HC-diag} is a simple $\eD$-module.
We claim:

\begin{proposition}\label{HC1-lemma}  
	\begin{enumerate} 
		\item The left ideals described by \eqref{HC-diag} are the only left ideals containing $I$ and the inclusions given there are the only inclusions between them. 
		
		\item In particular, $\eG$ is a module of length 4 with a simple socle and simple top, both isomorphic to $ \eD/(\partial_0,\partial_1)$. 
		The middle term $J_\infty/J_0$ is a direct sum of two simple $\delta$-torsion modules.
		
		\item The spherical subalgebra ${\Aak}=\Ak(W)$ is isomorphic to the factor ring $U(\mathfrak{sl}_2)/(\Omega+1) $ 
		of $U(\mathfrak{sl}_2)$, where $\Omega$ is the Casimir element. It is simple ring  of infinite global homological dimension. 
		  The corresponding Hecke algebra is isomorphic to $\C[T]/(T-1)^2$, which is clearly not semisimple.   
	\end{enumerate}
\end{proposition}

\begin{remark} 
	This proposition shows the remarkable nature of Theorem~\ref{intro-torsionfree}: despite the fact that $\eG$ has several $\delta$-torsion subfactors, they can never appear in the socle or top of the module. 
\end{remark}

\begin{proof} 
	(1,2) The discussion above shows that the $J_i$ are indeed distinct left ideals with the claimed inclusions and simple factors. Thus it only remains to show that there are no other left ideals. This is not hard to prove directly. Alternatively, by Part~(3) and  Theorem~\ref{torsionfree}, $\eD/I(\eG)$ has a $\delta$-torsionfree socle and top. Thus $J_0/I$ is necessarily the socle of $\eD(V)/I$ while $J_\infty/I$ is necessarily its unique maximal submodule. This suffices to prove the claims. 
	
	(3)   We will just sketch the argument.   The basic observation  is that $\C[V]^G = \C[z]$, for $z=x_0x_1$ and so $\C[\h]^W= \C[z]$, as well. Since  $\nabla\ast z=0$, it follows that $\nabla\in \ker(\rad_0)$. Now, the factor ring  $A=\dd(V)^G/(\nabla)$ is generated by the images of $E=x_0x_1,$ $ F= -\Delta$ and   $H=[E,F]$ and  one can then check that $A\cong U(\mathfrak{sl}_2)/(\Omega+1)$ where, as usual, 
	 $\Omega=H^2+2H+4FE$.  As such $A$ is a domain that is simple; see \cite{St}. Therefore, it must equal $\Ak(W)$ by Theorem~\ref{thm:intro-radial-surjective}. The fact that $A$ has infinite global homological dimension is \cite[Theorem~B]{St}.  We  omit the explicit computations since  they are easy but not particularly illuminating. These assertions are also special cases of the more general results in Section~\ref{Sec:Quivers}. 
\end{proof}
 

\section{A General  Intertwining Theorem}\label{Sec-Intertwining}
In this section we prove a 
general, abstract  intertwining theorem, relating  the (co)homology groups for two rings having a bimodule in common.
This result  (necessarily) requires some quite stringent restrictions  on the two rings (see Theorem~\ref{mainresult}) but these hypotheses will  hold in our setting. 

The context for this section is that of  Auslander-Gorenstein rings,  defined as follows.

\begin{definition}\label{Aus-defn}   Let $R$ be a noetherian ring  of  finite  injective dimension, written 
$\injdim R<\infty$.  Then  $R$ is \emph{Auslander-Gorenstein}  if $R$  satisfies the following condition: 
for  any finitely generated (left or right) $R$-module $L$ and submodule
$N\subseteq \Ext_R^j(L,R)$, one has  $\Ext^i_R(N, R)=0$ for all $i<j$.
The ring $R$ is \emph{Auslander-regular} if, in addition,  $R$ has finite  global homological dimension, written 
 $\gldim R<\infty$.  The \emph{grade}\label{grade-defn} of an $R$-module $L$ is 
 \[  \grade_R(L) \ = \\min \{ j : \Ext^j_R(L,R)\not=0\}.\]
 The $R$-module $L$ is called \emph{$p$-Cohen-Macaulay} or \emph{$p$-CM}, for $p\in \mathbb{N}$, if $\Ext^i_R(L,R)=0$ for $i\not=p$. In this case, by definition, $p=j(L)$.
 
  Finally, let $k$ be a field and assume that $R$ is  also a $k$-algebra  of finite Gelfand-Kirillov dimension, $\GKdim(R) <\infty$. Then $R$ is called \emph{Cohen-Macaulay} or \emph{CM} if \label{CM-defn}
$$
\grade_R(L)+\GKdim(L)=\GKdim(R),
$$
for all finitely generated $R$-modules $L$.
\end{definition}

For the rest of the section, we make the following hypotheses.

\begin{hypotheses} \label{notation-vanishing}
	Fix noetherian $k$-algebras  $S$ and $A$, where  $S$ is   Auslander-regular and CM while   
	$A$ is  Auslander-Gorenstein and CM. Fix $n,m\in \mathbb{N}$.
	We then  fix  a finitely generated left $A$-module $N$ such that  
	\begin{equation}\label{vanishing0} \text{$N$ is CM with  projective dimension $\pd(N)=n$, and hence $\grade_A(N)=n$.}
	\end{equation} 
	We also specify a  $(S,A)$-bimodule $M$ 
	such that   
	\begin{equation}\label{vanishing0.5} \text{ ${}_SM $ is  finitely generated and CM, 
			with $\grade_S(M)=m$.}
	\end{equation}
	
	Finally, we add   an  Ext-vanishing condition: 
	\begin{equation}\label{vanishing1}
		\Ext^i_S(\Tor_j^A(M,N),S) = 0  \quad \forall \ i \not= n + m.
	\end{equation}
\end{hypotheses}

We might remark that these conditions are quite natural in the context of quantum Hamiltonian reduction, as will become 
apparent in the later parts of this paper. 

We now turn to the intertwining theorem. The following result will prove useful and we would like to thank Thierry Levasseur for the reference.

\begin{proposition}\label{prop:Bjorksequence}   {\rm \cite[Satz~1.14]{Is} }
	Let $A$ be a left noetherian ring, and $N$ a finitely generated left $A$-module of finite projective dimension $n$  such that $\Ext^i_A(N,A) = 0$ for all $i \neq n$. 
	Let $K$ be a left $A$-module. Then $\Ext^i_A(N,K) = 0$ for $i > n$, while 
	\[\Ext^{n-j}_A(N,K) = \Tor_j^A(\Ext^n_A(N,A),K) \qquad \text{for all}\  	 0 \le j \le n.    \qed\] 
\end{proposition}

We can now prove the main result of this section.

\begin{theorem}\label{mainresult} Keep the assumptions from Hypotheses~\ref{notation-vanishing}. Then 
	\begin{enumerate}
		\item  $\quad
		\Ext^{n+m}_S(M \otimes_A N,S) = \Ext^{n}_A(N,A) \otimes_A \Ext^m_S(M,S).  $
		
		\smallskip
		\item  $\quad \Tor_{i}^A(M,N) = 0 $ for all $i > 0$.  
	\end{enumerate}
\end{theorem}

\begin{proof}  We use  the left-right analogue of the  spectral sequences given on 
	\cite[Section~4, page 345]{CE}, with the following changes of notation. We take rings 
	$\Gamma = k$, $\Lambda = A$ and $\Sigma = S$ together with modules 
	${}_{\Lambda}A = {}_AN$,   ${}_{\Sigma} B_{\Lambda} = {}_S M_A$ and ${}_{\Sigma}C = {}_SS$.
	Then the spectral sequences (2) and (3) of \textit{loc. cit.} say that 
	\begin{equation}\label{CE2}
		\Ext^p_A(N,\Ext^q_S(M,S)) \ \stackrel{p}{\Longrightarrow} \ R^r T(N,S),
	\end{equation}
	respectively
	\begin{equation}\label{CE3}
		\Ext^q_S(\Tor_p^A(M,N),S)\  \stackrel{q}{\Longrightarrow} \ R^r T(N,S).
	\end{equation}
	In both cases, $r = p + q$. 
	
	By assumption, ${}_SM$ is $m$-CM, and so $\Ext^p_A(N,\Ext^q_S(M,S)) = 0$ for $q \not= m$. 
	Thus sequence \eqref{CE2} collapses to give
	\begin{equation}\label{CE4}
		R^rT(N,S) = \Ext^{r-m}_A(N,\Ext^m_S(M,S)).
	\end{equation}
	Since $N$ has finite projective dimension,   Proposition~\ref{prop:Bjorksequence} implies that 
	\[\Ext^{n-j}_A(N,K) = \Tor_j^A(\Ext^n_A(N,A),\, K)\] for any left $A$-module $K$ and $0\leq j\leq n$. 
	If $n - j = r- m$, then $j = n + m - r$ and  \eqref{CE4} becomes
	$$
	R^rT(N,S) = \Tor_{n+m-r} (\Ext^{n}_A(N,A),\Ext^m_S(M,S)).
	$$
	Similarly, by Hypothesis~\ref{vanishing1} the terms of the sequence  \eqref{CE3}  are all zero unless $q = n + m$. Thus sequence \eqref{CE3} also collapses and gives, for $p = r - (n+m)$, the equality 
	$$R^rT(N,S) = \Ext^{n+m}_S(\Tor_{r-(n+m)}^A(M,N),S)
	$$
	Since both sequences converge to $R^r T(N,S)$, we deduce that 
	\begin{equation}\label{CE5}
		\Tor_{n+m-r} (\Ext^{n}_A(N,A),\Ext^m_S(M,S)) = \Ext^{n+m}_S(\Tor_{r-(n+m)}^A(M,N),S)
	\end{equation}
	for all $r \in \mathbb{Z}$. Obviously these terms are zero for  $r-(n+m)<0$.

	When $r = m+n$  \eqref{CE5} gives
	\[
	\Ext^{n+m}_S(M \otimes_A N ,\, S) = \Ext^n_A(N,A) \otimes_{A} \Ext^m_S(M,S).
	\] 
	
	On the other hand, suppose that     $\Tor_{i}^A(M,N) \not= 0$ for some $i>0$, and consider \eqref{CE5} for 
	$r=i+(m+n)>0$. By \cite[Theorem~2.2]{Le},  and   the fact that $S$ is Auslander-Gorenstein, $X=\Ext^{j}_S(\Tor_{i}^A(M,N),S)\not=0$ for some $j\geq0$. Hypothesis~\ref{vanishing1}  forces  $j=n+m$ and so $X=\Ext^{n+m}_S(\Tor_{i}^A(M,N),S)\not=0$. But  \eqref{CE5} then  implies that  
	$X= \Tor_{-i}(N,\, M)$, which is of course automatically zero.
\end{proof}

More compactly,  Theorem~\ref{mainresult}(1) can be written as follows. 

\begin{notation}\label{main-notation} Keep the notation from Hypotheses~\ref{notation-vanishing}. Set 
	\[{}^\perp\BH(J)=M\otimes_AJ \qquad\text{and}\qquad \BH^\perp (L)=L\otimes_A\Ext^m_S(M,\, S)\] 
	for a left $A$-module $J$ and right $A$-module $L$. Similarly, set 
	\[   \BD_S(J_1)=\Ext^{n+m}_S(J_1,\, S) \qquad\text{and}\qquad \BD_A(L_1) = \Ext_A^n(L_1,\, A)\]
	for a left $S$-module $J_1$ and left $A$-module $L_1$.    \end{notation}

Then Theorem~\ref{mainresult} becomes:

\begin{corollary}\label{maincorollary}
	Keep the assumptions from  Hypotheses~\ref{notation-vanishing}. Then
	\[ \left( \BD_S\circ{}^\perp\BH\right) (N) \ \cong \ \left(\BH^\perp\circ \BD_A\right)(N).\qed\]
	
\end{corollary}


\section{Spherical Subalgebras of Rational Cherednik Algebras}\label{Sec:Cherednik}

In this section we define and prove some basic facts, many of which are well-known, about  rational Cherednik algebras and their spherical subalgebras. 
  We emphasise that our notation will be the same as that of  the companion paper \cite{BLNS} and  so for many standard concepts we will refer the reader to that paper rather than repeating the definitions here. A survey on Cherednik algebras can also be found in \cite{BellamySRAlecturenotes} and our notation is (usually) consistent with that given there. Henceforth, the base field will be $\C$.

We follow \cite{GGOR} or \cite{EG} for the definition of  the Cherednik algebra $\Hk(W)$, with the precise definition being given in \cite[Definition~2.1]{BLNS}. 
  We   identify $\C[\h^*] =\Sym \h$  using the natural inner product and hence regard $\h^*$  as the natural space of generators for $\C[\h]$. However, it will be more convenient to use $\Sym \h$ in place of $\C[\h^* ]$ and, similarly, $\Sym V $ in place of $\C[V^*]$. Let $e=\frac{1}{|W|} \sum_{w\in W } w$ be the trivial idempotent in $\C W$. The \emph{spherical subalgebra}\label{spherical-defn} of $\Hk(W)$ is $\Ak(W)= e \Hk(W) e$.  Since we rarely consider the Cherednik algebra itself, we will simply call $\Ak(W)$ a \emph{spherical algebra}  and  for brevity,  we usually  write $\Ak(W)={\Aak}$. The Poincar\'e-Birkhoff-Witt Theorem \cite[Theorem~1.3]{EG} implies that  
  $
\C[\h]^W\cong e\C[\h]e \subset {\Aak} $ and $  (\Sym  \h)^W\cong e(\Sym  \h)e \subset {\Aak}.
$
 
\begin{definition}\label{defn:Osph}
 The category $\Osph=\Osph_{\kappa}(W)$ is defined to be the full subcategory of ${\Aak}\lmod$  of  finitely generated modules on which $(\Sym  \h)^W$ acts locally finitely.  These modules are finitely generated over $\C[\h]^W$.  For $\lambda \in \h^*$, let $\mf{m}_{\lambda}$ \label{defn:em-lambda} denote the maximal ideal of $(\Sym  \h)^W$ defined by the coset $W \lambda \in \h^*/W$; thus $\mf{m}_{\lambda} = \mf{m}_{w(\lambda)}$ for any $w \in W$. Let $\Osph_{\llambda } = \Osph_{\kappa,\llambda}(W)$ denote the full 
subcategory of ${\Aak}\lmod$ consisting of all modules $N$ on which the action of $\mf{m}_{\lambda}$ is locally nilpotent.  The most important of these subcategories is $\Osph_0$, consisting of the modules for which    $\mf{m}_0= (\Sym\h)^W_+$ acts locally nilpotently. The analogous categories of right modules will be written $\Osphop$, respectively $\Osphop_{\lambda}$.

Perhaps the most basic  objects in $\Osph$  are the 
\begin{equation}\label{M-definition}    
\eQ_{\lambda} = {\Aak} / {\Aak} \mf{m}_{\lambda}\in   \Osph_{\llambda }.
\end{equation}
Finally, the fact that every element in $(\Sym  \h)^W$ acts locally ad-nilpotently on ${\Aak}$  
 implies that
$$
\Osph = \bigoplus_{\llambda W  \in \h^* / W} \Osph_{\llambda }.
$$
 \end{definition}

We next prove some basic properties of spherical algebras and their modules.

 \begin{lemma}\label{CM for spherical}
	The algebra ${\Aak} $ is Auslander-Gorenstein and CM.    
\end{lemma}

\begin{proof}
	Use \cite[Theorem~1.5(1)]{EG}  and its proof.
	\end{proof}

\begin{lemma}\label{lem:GKdimOsph}
If $L \in \Osph$ then $\mathrm{GKdim}_{\Aak}(L) = \mathrm{GKdim}_{\C[\h]^W}(L)$. 
\end{lemma} 

 \begin{proof}  This is standard.  
 As explain in \cite[p.262]{EG},  ${\Aak}$ has a finite dimensional filtration $\{\Gamma_n{\Aak}\}$ with $\gr_\Gamma {\Aak}=\C[\h\times\h^*]^W$. Pick a good filtration 	$\{\Gamma_nL\}$ for $L$; thus $\gr_\Gamma L$ is a finitely generated $\gr_{\Gamma}{\Aak}$-module. Since $\C[\h^*]^W=(\Sym \h)^W$ acts locally finitely on $L$, clearly it also acts locally finitely on $\gr_\Gamma L.$ As $\gr_\Gamma L$ is a finitely generated $\C[\h]^W\otimes\C[\h^*]^W$-module, this implies that $\gr_\Gamma L$ is also finitely generated as a $\C[\h]^W$-module.  Regarding $\gr_\Gamma L$ as a $(\gr_\Gamma {\Aak},\, \C[\h]^W)$-bimodule,  \cite[Proposition~8.3.14(ii)]{MR}  therefore implies that  	$\GKdim_{\gr {\Aak}} (\gr L) =\GKdim_{\C[\h]^W}(\gr L)$. By  \cite[Proposition~8.6.5]{MR} this pulls back to give the claimed result.
 \end{proof}

\begin{proposition}\label{pdQ} 
	Set $E=(\Sym \h)^W$ and  $n = \dim \h$.  
	\begin{enumerate}
		\item ${\Aak}$ is a free $\left(\C[\h]^W, E\right)$-bimodule of finite rank.
		\item  Let $0\not=L$ be any finite-dimensional $E$-module. Then  ${\Aak} \o_{E} L$ is $n$-CM. In particular, $ \pd_{\Aak}({\Aak} \o_{E} L) = n$.
		\item The ${\Aak}$-module $\eQ_{\lambda}$ is $n$-CM, and hence  $\pd_{\Aak}(\eQ_{\lambda})=n$. 
		\item Each projective object $\eQ \in \Osph$ is $n$-CM and thus   $\pd_{\Aak}(\eQ)=\dim \h$.   
	\end{enumerate}  
\end{proposition}

\begin{proof}  Part (1) is \cite[Lemma~2.4]{BLNS}.

	(2,3)	 
	 Here, $\pd_EL= n$ and hence $\Ext^i_{E}(L,K) = 0$
	for all left ${\Aak}$-modules $K$ and all $i > n$. Since ${\Aak}$ is a  free  $E$-module,  \cite[Corollary~10.65]{Rot} implies that
	$$
	\Ext^i_{\Aak}({\Aak} \o_{E} L,K) \cong \Ext^i_{E}(L,K) = 0
	$$
	for all $i > \dim \h$. Thus, $\pd_{\Aak}({\Aak} \o_{E} L) \le n$. 
	
	Conversely, by Part~(1), ${\Aak} \o_{E} L$ is a free $\C[\h]^W$-module of finite rank and hence  $\mathrm{GKdim}_{\C[\h]^W}( {\Aak} \o_{E} L) = n$. Therefore $\mathrm{GKdim}_{{\Aak}}( {\Aak} \o_{E} L) = n$  by Lemma~\ref{lem:GKdimOsph}.  Thus, by the CM condition,  $\Ext_{\Aak}^{n}({\Aak} \o_{E} L,{\Aak}) \neq 0$  and so $\pd({\Aak} \o_{E} L) = n$. Finally, since ${\Aak}$ is a free $\C[\h]^W$-module and $L$ is $n$-CM as a $\C[\h]^W$-module, 
	\begin{equation}\label{eq:indCMn}
	\Ext^i_{\Aak}({\Aak} \o_{E} L,{\Aak}) \cong \Ext^i_{E}(L,{\Aak}) = 0
	\end{equation}
	for all $i \neq n$. Thus, ${\Aak} \o_{E} L$ is $n$-CM. 
	 	
	\medskip
	
	(4) There exists a   surjection $\phi: {\Aak} \o_{E} L \twoheadrightarrow \eQ$, for some  finite-dimensional $E$-module $L$. Since both objects belong to $\Osph$ and  $\eQ$ is projective, $\phi$  splits and so  $\eQ$ is a summand of ${\Aak} \o_{E} L$. As in the proof of  Parts~(2,3), this implies that $\eQ$ is free of finite rank over $\C[\h]^W$ and hence  that $\pd(\eQ) = \dim \h$. Finally, since   $\eQ$ is a summand
	 of ${\Aak} \o_{E} L$,  the fact that $\eQ$ is $n$-CM follows from \eqref{eq:indCMn}.  
\end{proof}

Since we wish to apply Corollary~\ref{maincorollary} to  $\eQ_{\lambda} $,   the following  precise 
description of $\BD_{\Aak}(\eQ_{\lambda})=\Ext_{\Aak}^n(\eQ_\lambda,\, {\Aak})$, as defined in  Notation~\ref{main-notation}, will be useful.

\begin{corollary}\label{Q'-description} 
	If $\eQ_{\lambda}'$ is the right ${\Aak}$-module ${\Aak}/ \mf{m}_{\lambda} {\Aak}$, then $\BD_{\Aak}(\eQ_{\lambda})=\eQ_{\lambda}'$. 
\end{corollary}

\begin{proof}
	By Proposition~\ref{pdQ}(1),  ${\Aak}$ is a free module  over  $E=(\Sym  \h)^W$. If $U \subset \mf{m}_{\lambda}$ is 
	a minimal generating subspace  (of dimension $n=\dim \h$) 
	then the Koszul resolution 
	$$  	
	0\to E \o_\C \Lambda^{n} U \to E \o_\C \Lambda^{n- 1} U\to \cdots\to E \o_\C \Lambda^{0} U \to E/\mf{m}_{\lambda} \to 0,
	$$
	is a projective resolution of  $E$-modules. Tensoring by ${\Aak}$  gives a projective resolution 
	\begin{equation}\label{eq:projresQlambda}
	0\to {\Aak} \otimes_{\C} \Lambda^{n} U \to {\Aak} \otimes_{\C} \Lambda^{n- 1} U\to \cdots\to {\Aak} \otimes_{\C} \Lambda^{0} U \to \eQ_{\lambda} \to 0
	\end{equation}
	of the ${\Aak}$-module $\eQ_{\lambda}$. Thus,
	\[\begin{aligned} 
	\BD_{\Aak}(\eQ_{\lambda})  \  =\  \text{Coker}\Bigl( \Hom_{\Aak}({\Aak}\otimes_{\C}  \Lambda^{n-1}U,\, {\Aak})  
	\to&	\Hom_{\Aak}({\Aak}\otimes_{\C}  \Lambda^{n} U,\, {\Aak})\Bigr) \\
	& =\   {\Aak}/\mf{m}_{\lambda} {\Aak}  ,   
	\end{aligned}\]
	as required.
\end{proof}

Finally, we give some stronger results in the case when ${\Aak}$ is simple. A finitely generated 
${\Aak}$-module $L$ is defined  to be \emph{holonomic} \label{defn:holonomic} if $\GKdim (L)=\frac{1}{2} \GKdim ({\Aak})$. By \cite[Introduction]{Lo},  these modules have the same
well-known properties of holonomic modules over the Weyl algebra (where the same definition applies).
		
\begin{lemma}\label{CM for spherical2}		
	Let $n = \dim \h$ and assume that ${\Aak}$ is simple.
\begin{enumerate}

\item ${\Aak}$ has injective dimension $\injdim {\Aak}= n$.
		
\item Every nonzero module in $\Osph$ is holonomic and  $n$-CM.

\item Let  $L$ be a left ${\Aak}$-module on which $(\Sym  \h)^W$ acts locally finitely and such that some $0 \not= d\in \C[\h]^W$
acts locally nilpotently on $L$. Then $L=0$.  
	\end{enumerate}
\end{lemma}

\begin{proof}  
	(1) Recall from Lemma~\ref{CM for spherical} that ${\Aak} $ is Auslander-Gorenstein. 
	Let  $\mu=\injdim {\Aak}<\infty$ and pick a finitely generated left ${\Aak}$-module $M$ that satisfies  
	$\Ext_{\Aak}^\mu(M,{\Aak})\not=0$. Then $\grade\bigl( \Ext^\mu(\Ext^\mu(M,{\Aak}),{\Aak})\bigr)=\mu$ by \cite[Theorem~2.4(a)]{Le}. Also, note that $\GKdim ({\Aak})=2n$. Thus,  using  the CM condition,  in order to 
	prove that $\injdim {\Aak}=n$ it suffices to prove that $n$  is the minimum Gelfand-Kirillov dimension of  nonzero  finitely generated left ${\Aak}$-modules.
		
	 By  \cite[Section~2.4 and Theorem~1.2(1)]{Lo},   $\GKdim (N) \geq \dim\mathfrak{h}$ holds for every finitely generated ${\Aak}$-module $ N\not=0$. Conversely, pick    $0\not=L\in \euls{O}^{sph}$. Then  $L$ is finitely generated as a $\C[\h]^W$-module  and so Lemma~\ref{lem:GKdimOsph} implies that 
	 $$\GKdim_{\Aak} L =\GKdim_{\C[\h]^W}L\leq \GKdim \C[\h]^W = n.$$ Thus 
	 $\GKdim_{\Aak} (L)=n$, as required.

	(2)   If $0\not=L\in \Osph$,  the proof of Part~(1) shows that  $L$ is holonomic. Thus, by  the CM condition,  $\Ext^j(L,{\Aak}) =0$ if  $j<n$, while  $\Ext^j(L,{\Aak}) =0$ by definition  if  $j> n=\injdim {\Aak}$.
 	
	(3) Replacing $L$ by a finitely generated ${\Aak}$-submodule, we can assume that $L\in \Osph$. Now, as in (1),
	$L$ is finitely generated as a $\C[\h]^W$ but it is now a torsion  $\C[\h]^W$-module. As such, 
	Lemma~\ref{lem:GKdimOsph} implies that $\GKdim {}_{\Aak} (L) <\GKdim (\C[\h]^W)$.
	This contradicts Part~(2). 
\end{proof}


\subsection*{Projective objects in Category $\Osph$}
 
In later sections we will be assuming that ${\Aak}=\Ak(W)$ is simple, and   this has strong consequences for the structure of projective objects from the category  $\Osph$. Many of these results are proved in \cite{LosevTotally}    and in this subsection we summarise the results of Losev we will need later.
 
 We begin with some relevant definitions.  As before, let  $\Aak=\Ak(W) $ denote a spherical algebra
 although, since $\kappa$ and $W$  will now vary, we will typically write $\Osph_{\kappa,\llambda}(W)$ in place of  $\Osph_\lambda$. By Steinberg's Theorem, a parabolic subgroup $W_\nu\subset W$ (that is, the stabilizer of $\nu \in \h$) is generated by the reflections of $W$ that fix $\nu$. In particular, it is itself a complex reflection group. Taking only those $\kappa_{H,i}$ with $\nu \in H$ defines a parameter for the rational Cherednik algebra associated to $(\h,W_{\nu})$. Abusing notation, we still denote this parameter by $\kappa$ (formally, it is the restriction of $\kappa$ to $W_{\nu}$). \label{defn:c(nu)} For this complex reflection group we have an analogous spherical algebra  $\Ak(W_{\nu})$. \label{spherical-defn2}
 
Set $E=(\Sym  \h)^W$ and write $\widehat{E}_\lambda$ for  the completion of $E$ along $\mf{m}_{\lambda}$. Set 
\begin{equation}\label{defn:completion}
\widehat{A}_{\kappa}(W)_{\lambda} := \Ak(W) \otimes_{E} \widehat{E}_{\lambda}.
\end{equation}
By \cite[Theorem~3.2]{BE}, the algebra structure on $\Ak(W)$ extends to $\widehat{A}_{\kappa}(W)_{\lambda}$ and there is an isomorphism
\begin{equation}\label{eq:isoeHecomplete}
\widehat{A}_{\kappa}(W)_{\lambda} \stackrel{\sim}{\longrightarrow} \widehat{A}_{\kappa}(W_{\lambda})_{0}
\end{equation}
that restricts to the isomorphism $\widehat{E}_{\lambda}  \stackrel{\sim}{\longrightarrow} \widehat{(\Sym  \h)}^{W_{\lambda}}_{0}$ 
which is in turn  induced by mapping  $y-  \lambda(y) \mapsto y   $, for $y \in \h$.

These observations imply the following result.

\begin{lemma}\label{lem:genOsphparabolicequi}
 {\rm (1)}
	There is an equivalence of categories $\Osph_{\kappa,\llambda}(W) \stackrel{\sim}{\longrightarrow} \Osph_{\kappa,0}(W_{\lambda})$ which sends
	$\eQ_{\lambda}$ to $\eQ_0$.

{\rm (2)}   $\Osph$ has enough projectives.
\end{lemma}

\begin{proof} Part~(1) follows from the  isomorphism \eqref{eq:isoeHecomplete} . In order to prove Part~(2), 
 note that  $\Osph_{\kappa,0}(W_{\lambda})$ has enough projectives by  \cite[Theorem~2.5(ii)]{Primitive}. Therefore Part~(1) implies the same for   $\Osph$.
 \end{proof}

We let $\euls{O}_{\kappa}$ denote the usual category $\euls{O}$ \label{defn:O-cherednik}
for the (full) rational Cherednik algebra $\Hk(W)$. The \textit{Harish-Chandra module} in this case is
$$
\mathfrak{H}_{\kappa} := \Hk (W) \otimes_{(\Sym  \h) \rtimes W} (\Sym  \h)^{\mathrm{co} W}.
$$
Multiplication by the trivial idempotent $e$ induces  an exact quotient functor $\euls{O}_{\kappa} \to \Osph_{\kappa,0}$.

\begin{lemma}\label{lem:HCprojiffQproj}
\begin{enumerate}
\item	$e\mathfrak{H}_\kappa = \eQ_0$. 

\item For $M\in \euls{O}_{\kappa}$ we have 
$\Hom_{\Hk(W)}(\mathfrak{H}_{\kappa},\, M)= \Hom_{\Ak(W)}(\eQ_0, eM)$.

\item Moreover,  $\mathfrak{H}_\kappa$ is projective in $\euls{O}_\kappa$ if and only if $\eQ_0$ is projective in $\Osph_{\kappa,0}$.
\end{enumerate}
\end{lemma}

\begin{proof} (1) 
 Note that $(\Sym  \h)^{\mathrm{co} W} = \left((\Sym  \h) \rtimes W\right) \otimes_{F} \C e$, for   $F=(\Sym  \h)^W \otimes_{\C} \C W$.  
	This implies that
	$$
	e \mathfrak{H}_{\kappa} \ \cong\  e \Hk (W) \otimes_{F} \C e
	\  \cong\  e \Hk(W) e \otimes_{E} \C \ = \  \eQ_0.
	$$

(2)	Set $C=(\Sym \h)^W$.  By \cite[(2.2)]{LosevTotally}, 
$\Hom_{\Hk(W)}(\mathfrak{H}_{\kappa}, M)= (eM)^{C_+}$
for $M\in \euls{O}_{\kappa}$. 
On the other hand $\eG_0= \Ak(W)\otimes_C (C/C_+)$ and so, by adjunction, 
\[\begin{aligned}
\Hom_{\Ak(W)}(\eG_0, eM) \ =& \ \Hom_{\Ak(W)}(\Ak(W)\otimes_C(C/C^+),\, eM) \\
&=\Hom_{C}(C/C_+,eM)= (eM)^{C_+},
\end{aligned}\]
as required.

(3)	Assume that $\mathfrak{H}_{\kappa}$ is projective in $\euls{O}_{\kappa}$. Then for any  short exact sequence $0 \to N_1 \to N_2 \to N_3 \to 0$ in $\Osph_{\kappa,0}$, 
we get a long exact sequence
	$$
	0 \to T \to \Hk(W) e \otimes_{\Ak(W)} N_1 \to \Hk(W) e \otimes_{\Ak(W)} N_2 \to \Hk(W) e \otimes_{\Ak(W)} N_3 \to 0,
	$$
	where $e T = 0$.  
 Now  apply $\Hom_{\Hk(W)}(\mathfrak{H}_{\kappa} ,-)$ to this long exact sequence. Using the identity from (2)  recovers a short exact sequence
	$$
	0 \to \Hom_{\Ak(W)}(\eQ_0,N_1) \to \Hom_{\Ak(W)}(\eQ_0,N_2) \to \Hom_{\Ak(W)}(\eQ_0,N_3) \to 0.
	$$
	Thus, $\eQ_0$ is projective in $\Osph_{\kappa,0}$. 
	
	Conversely, assume that $\eQ_0$ is projective  in $\Osph_{\kappa,0}$. Then applying  the   isomorphism 
	\[\Hom_{\Hk(W)}(\mathfrak{H}_{\kappa} ,-) \cong \Hom_{\Ak(W)}(\eQ_0,e(-))\] shows  that $\Hom_{\Hk(W)}(\mathfrak{H}_{\kappa} ,-)$ is the composition of
	 two exact functors, and hence is exact. Thus, $\mathfrak{H}_{\kappa}$ is projective in $\euls{O}_{\kappa}$.
\end{proof}

We will deduce the following  fact from the main result of \cite{LosevTotally}.

\begin{proposition}\label{prop:Qlambdaprojiffsimple}
	The module $\eQ_{\lambda}$ is projective in $\Osph_{\kappa,\llambda}(W)$ if and only if the algebra $\Ak(W_{\lambda})$ is simple.
\end{proposition}

\begin{proof}
	By Lemma~\ref{lem:genOsphparabolicequi}, $\eQ_{\lambda}$ is projective in $\Osph_{\kappa,\llambda}(W)$ if and only if 
	$\eQ_{0}$ is projective in $\Osph_{\kappa,0}(W_{\lambda})$. By Kashiwara's Lemma \cite[Theorem~1.6.1]{HTT},
	we may assume that $\h$ is an irreducible $W_{\lambda}$-module.
	
	By \cite[Proposition~2.7]{LosevTotally}, and in the notation of that paper,  
	$\Ak(W_{\lambda})$ is simple if and only if $\kappa$ is totally aspherical for $W_{\lambda}$. 
 By	 \cite[Theorem~1.2]{LosevTotally} the latter holds if and only if $\mathfrak{H}_{\kappa}$ is isomorphic to the projective module $P_{\mathrm{KZ}}$. In particular,
  if  $\Ak(W_{\lambda})$ is simple then $\mathfrak{H}_{\kappa}$ is projective 
	 and hence $\eQ_{\lambda}$ is projective by 
	 Lemmata~\ref{lem:HCprojiffQproj}(3) and~\ref{lem:genOsphparabolicequi}.
	  If $\Ak(W_{\lambda})$ is not simple then \cite[Remark~2.6]{LosevTotally} says that $\mathfrak{H}_{\kappa}$ is not projective; hence
	   $\eQ_{\lambda}$ is not projective by Lemmata~\ref{lem:HCprojiffQproj} and~\ref{lem:genOsphparabolicequi}, again.
\end{proof}

\begin{lemma}\label{lem:simpleinductionLosev}
	The algebra $\Ak(W)$ is simple if and only if $\Ak(W_{\lambda})$ is simple for all $\lambda \in \h^*$.
\end{lemma}

\begin{proof}
	We just need to show that $\Ak(W_{\lambda})$ is simple if $\Ak(W)$ is simple. This follows directly from \cite[Lemma~2.9]{LosevTotally} since if property (*) of \cite{LosevTotally} holds for $W$ and $\kappa$ then it certainly holds for $W_{\lambda}$ and $\kappa$. 
\end{proof}

The significance  of these results  will become apparent when we combine them  with the 
Knizhnik-Zamolodchikov ($\mr{KZ}$)  functor. In order to explain  this,  we first need to define the  Hecke algebra, for which we will use the presentation from \cite[Section~5.2.5]{GGOR}, except for a change of sign as explained in \cite[Section~5.2.1]{RouquierQSchur}

\begin{definition}\label{Hecke-defn} 
 Let  $\euls{A}$ denote the set of reflecting hyperplanes for the complex reflection group $W$. Then the pointwise stabiliser of $H \in \euls{A}$ is $W_H=\langle s_H\rangle$, where $s_H \in W_H$ is uniquely specified by $\det_{\h}(s_H) = \exp(2 \pi i / \ell_H)$ and $\ell_H := |W_H|$. 

 	Let $B_W$ be the braid group associated to $W$, with generators $\{T_{H} : H \in \euls{A}\}$, where $T_H$ is an $s_H$-generator of the monodromy around $H$, as defined in \cite{BMR}. The 
	\emph{Hecke algebra   $\euls{H}_q(W_{\lambda})$ associated to $W$ at parameter $q= \{ q_{H,j} \}$} is the quotient of the group algebra $\C B_W$ by the $|\euls{A}|$ relations
	\begin{equation} \label{eq:Heckecrg1}
		\bigg\{\prod_{j=0}^{\ell_H-1} \left(T_{H} - q_{H,j}\right) \,\, :\,\,  H\in \euls{A}\bigg\},
	\end{equation}
	where the parameters $q_{H,j} \in \mathbb{C}^{\times}$ satisfy $q_{w(H),j} = q_{H,j}$ for all $w \in W$. We will always define $q$ in terms of the $\kappa_{H,j}$ by setting
	\begin{equation} \label{eq:Heckecrg2}
		  q_{H,j} \ = \ \det(s_H)^{-j} \exp(-2\pi  i \kappa_{H,j}).	\end{equation}

\end{definition}

The $\mr{KZ}$-functor is then an exact functor $\euls{O}_{\kappa,0}(W_{\lambda}) \to \euls{H}_q(W_{\lambda})\lmod$ which factors through $\Osph_{\kappa,0}(W_{\lambda})$; see \cite[p.217]{BellamySRAlecturenotes}, noting that our presentation of $\Hk(W)$ agrees with \cite[(4.F)]{BellamySRAlecturenotes} if the vectors $\alpha_{H},\alpha_{H}^{\vee}$ are normalised so that $\langle \alpha_H^{\vee}, \alpha_H \rangle = -1$. Composing with the equivalence of Lemma~\ref{lem:genOsphparabolicequi} gives a functor $\Osph_{\kappa,\lambda}(W) \to \euls{H}_q(W_{\lambda})\lmod$.  By a slight abuse of notation,   we will also refer to the latter as the $\mr{KZ}$-functor.

\begin{corollary}\label{A-projectives}
 \begin{enumerate}
\item 	The algebra $\Ak(W)$ is simple if and only if every $\eQ_{\lambda}$ is projective in $\Osph$.
\item 	Assume that  $\Ak(W)$ is simple and let $\lambda\in \h^*$. Then  $\eQ_{\lambda}$ is a progenerator in $\Osph_{\lambda}$ and the $\mr{KZ}$-functor is an 
equivalence of categories 
$  \Osph_\lambda \to \euls{H}_{q}(W_{\lambda})\lmod$. 
\item If $\Ak(W)$ is simple, then $\End_{\Ak(W)}(\eQ_0) =  \euls{H}_q(W)$.
\end{enumerate}
\end{corollary}

\begin{proof} (1) Combine  Proposition~\ref{prop:Qlambdaprojiffsimple} and Lemma~\ref{lem:simpleinductionLosev}.
 
 (2,3) By Lemma~\ref{lem:genOsphparabolicequi} and Lemma~\ref{lem:simpleinductionLosev}, it suffices to consider the case where $\lambda = 0$. In this case, \cite[Theorem~1.2]{LosevTotally} implies that
 $\mathfrak{H}_\kappa$  is projective in category $\euls{O}_{\kappa}$ with 
 $\End_{\Hk}(W)(\mathfrak{H}_{\kappa})= \euls{H}_q(W)$.   
 
 By Lemma~\ref{lem:HCprojiffQproj}(3), $\eQ_0 =e\mathfrak{H}_{\kappa}$ is therefore projective in $\Osph_{0}$. Moreover, by Lemma~\ref{lem:HCprojiffQproj}(1,2), $\End_{\Ak(W)}(\eQ_0) =  \euls{H}_q(W)$, showing statement (3). Finally, \cite[Proposition~2.7]{LosevTotally} implies that $\kappa$ is totally aspherical. By definition, this means that $\Hom_{\Hk}(\mathfrak{H}_{\kappa},M) = 0$ if and only if $\mr{KZ}(M) = 0$ if and only if $eM =0$. Since $\Hom_{\Ak}(\eQ_0,eM) = \Hom_{\Hk}(\mathfrak{H}_{\kappa},M)$, this implies that $\eQ_0$ is a generator for $\Osph_0$ and 
 $$
 \Hom_{\Ak}(\eQ_0,-) \colon \Osph_{0} \to \euls{H}_{q}(W)\lmod
 $$
 is an equivalence. 
\end{proof}

\begin{corollary}\label{P'-description}
	Assume that ${\Aak}=\Ak(W)$ is simple and let $\eP \in \Osph$ be projective. Then $\BD_{\Aak}(\eP) =\Ext^n_{\Aak}(\eP,{\Aak})$ is a  projective
	object in  $ \Osphop$. 
\end{corollary}

\begin{proof}
	We may assume that $\eP$ is indecomposable. Then Corollary~\ref{A-projectives}(2) says that $\eP$ is a direct summand of $\eQ_{\lambda}$ for some $\lambda$. 
	Hence,   by Corollary~\ref{Q'-description}, $\BD_{\Aak}(\eP)$ is a direct summand of $\BD_{\Aak}(\eQ_{\lambda})\cong \eQ_{\lambda}'$. 
	Since $\eQ_{\lambda}'$ is projective in $\Osphop$ by the left-right analogue of Proposition~\ref{prop:Qlambdaprojiffsimple}, the result follows. 
	\end{proof}

 \subsection*{Semi-invariants}
We end this section by recalling some standard facts regarding complex reflection groups that will be needed later. 

\begin{notation}\label{not:regularelementh} 
  Define  $\euls{A}$ and $\{W_H=\langle s_H\rangle: H\in \euls{A}\}$  as in Definition~\ref{Hecke-defn}, with 
   $\ell_H =|W_H|$.  As in the definition of the Cherednik algebra in \cite[Definition~2.1]{BLNS}, we choose
  $\alpha_H \in \h^*$ such that $H = \alpha_{H}^{\perp}$.  The element 
 $\prod_{H\in \euls{A}} \alpha_H$ acts as the inverse determinant character $\det_{\h}^{-1} \colon W \to \C^{\times}$; 
 see \cite[Theorem~2.3]{StanleySemiInv}. Moreover, the \textit{discriminant} is defined to be 
\begin{equation}\label{eq:discriminant}
 \ \deltah \ = \  \prod_{H \in \euls{A}} \alpha_H^{\ell_H};
\end{equation}
 this polynomial is $W$-invariant, and reduced as an element of $\C[\h]^W$. 
  	Let $\h_{\reg}$ denote the open subset of $\h$ where $W$ acts freely. It is a consequence of Steinberg's Theorem that $\h_{\reg}$ is precisely the non-vanishing locus of $\deltah $.  \end{notation}

The abelianization $W/[W,W]$ can be identified with $\prod_{[H] \in \euls{A}/W} (\mathbb{Z} / \ell_H \mathbb{Z})$. Dually, it was explained by Stanley \cite{StanleySemiInv} how one can use the semi-invariant factors of $\deltah $ to construct all linear characters of $W$. We recall this construction. 

\begin{remark}\label{rem:Wsemisummary}
	Let $\chi$ be a linear character of $W$. For each $[H] \in \euls{A}/W$, let  $\ell_{\chi, H}$ be the unique integer with $0 \le \ell_{\chi,H} \le \ell_H-1$ satisfying  $\chi(s_H) = \det_{\h}(s_H)^{\ell_{\chi,H}}$. Set 
	\begin{equation}\label{eq:Wsemiinvariantdelta}
			\deltah_{\chi} = \prod_{H \in \euls{A}} \alpha_H^{\ell_{\chi,H}}. 
	\end{equation}
	Then \cite[Theorem~3.1]{StanleySemiInv} says that $\deltah_{\chi}$ is a $W$-semi-invariant with associated linear character $\chi$. Moreover, the space $\C[\h]^{\chi}$ of $\chi$-semi-invariants is free of rank one over $\C[\h]^W$, generated by $\deltah_{\chi}$. 
	
	Conversely, given any choice of integers $0 \le \ell_{\eta,H} \le \ell_H-1$ for $[H] \in \euls{A}/W$, we may define an element $\deltah_{\eta}$ by the rule  \eqref{eq:Wsemiinvariantdelta}. Then $\deltah_{\eta}$ is a $W$-semi-invariant with associated linear character $\eta$. This defines a bijection between the linear characters of $W$ and the semi-invariants of $W$ (up to scalar) that properly divide $\deltah$. 
\end{remark}

\section{Background  on   $\D$-modules}    \label{Sec:polarreps}

In this section we give   basic definitions and results about  $\eD$-modules, most especially  equivariant $\eD$-modules. 
  We remind the reader that we always work over the field $\C$ of complex numbers.

Fix a connected reductive group $G$.  We will always assume that $V$ is a polar $G$-representation, as defined in the introduction and write $\ddd=\dd(V)$ for the algebra of $\C$-linear (algebraic) differential operators on $V$.

The  basic properties of   a polar representation $V$  are given in \cite[Section~3]{BLNS} and we will not repeat them here, except to recall that $V$ has a \emph{Cartan subspace} $\h\subset V$ with associated Weyl group $W$ such that restriction of functions induces an isomorphism $\rr \colon \C[V]^G\isom \C[\h]^W$. Let $\euls{N}(V) := \pi^{-1}(0)$ be the \emph{nilcone}\label{defn:nilcone} of $V$, where $\pi : V \rightarrow V\git G$ is the categorical quotient. The representation $V$ is said to be \textit{visible} 
\label{defn:visible} if $\euls{N}(V)$ consists of finitely many $G$-orbits; this implies that every fibre of $\pi$ consists of finitely many $G$-orbits. We will use several times the fact that $V$ is visible if and only if $V^*$ is visible. This follows, for instance, from the fact that there is a bijection between
the closed orbits in $V$ and $V^*$, as in the proof of \cite[Proposition~3.2]{BLT}. 

A representation is said to be \textit{locally free}\label{defn:locally-free} if the stabiliser of a generic element is finite. Finally, $V$ is said to be \textit{stable}\label{defn:stable} if there is a closed orbit whose dimension is maximal amongst all orbits.  

Let   $\deltah\in \C[\h]^W$ denote  the discriminant from \eqref{eq:discriminant} and put 
$\deltav = \rr^{-1}(\deltah)$\label{deltaV-defn}   for the corresponding  element in $\C[V]^G$, which by a slight abuse of notation will also be called the discriminant.  Since it will not cause confusion we will write $\delta$ for both $\deltah$ and $\deltav$ from here onwards. 
For a further discussion of the discriminants and the corresponding    \emph{regular loci}  $V_{\reg}  \subseteq V$ and 
$\h_{\reg}\subset \h$ \label{Vreg-defn}
 see \cite[Section~3]{BLNS}.  The following fact, proved in \cite[Corollary~6.9]{BLNS}, will be used frequently.
\begin{equation}\label{lem:stablepolaropen}
\text{If $V$ is stable then $V_{\reg} \cong G \times_{N_G(\h)}\h_{\reg}$. }   
\end{equation}

We now turn to some basic facts about equivariant $\eD$-modules. Let $\g=\Lie(G)$ and recall that
differentiating the action of $G$ on $V$ defines a  Lie algebra  morphism  $\tau \colon \mf{g} \rightarrow \ddd$. \label{tau-defn}    

\begin{definition}\label{defn:monodromic}    
Let $\chi$ be a linear character of $\g$. A   \emph{$(G,\chi)$-monodromic
 left $\ddd$-module}, or simply a monodromic left $\ddd$-module, 
 is defined to be a rational  left $G$-module $M$ with the structure of a left $\ddd$-module such that  
\begin{enumerate}
	\item[(i)] the action map $\ddd \otimes M \rightarrow M$ is $G$-equivariant; and
	\item[(ii)] for all $x \in \mf{g}$,  \ $\tau(x) = \tau_M(x) + \chi(x) \mr{Id}_M$ as endomorphisms of $M$,  
\end{enumerate}
where   $\tau_M: \g\to \End_{\C}(M)$ denotes the differential of the action map $G\times M\to M$.
The category of  finitely generated $(G,\chi)$-monodromic left $\ddd$-modules is denoted 
$\left(G,\chi,\ddd\right)\lmod$,   with  morphisms being $G$-equivariant $\ddd$-module morphisms. 

   The  category of  finitely generated $(G,\chi)$-monodromic right $\eD$-modules is defined analogously and denoted $\rmod \left(G,\chi,\ddd\right)$.  \end{definition}

 When $\chi=0$,  a $(G,\chi)$-monodromic left $\eD$-module is simply  a $G$-equi\-variant module. For this reason, $(G,\chi)$-monodromic left $\eD$-modules are sometimes called  $\chi$-twisted equivariant  left $\ddd$-modules \cite[(4.1)]{MN}.

 If $\chi$ is a linear character of $\g$, set 
\begin{equation}\label{chi-defn}
\g_{\chi} = \{ \tau(x) - \chi(x) \, | \, x \in \mf{g} \} \subset \ddd.
\end{equation}
 For a $(G,\chi)$-monodromic
left $\eD$-module $M$, the elements of $\g_{\chi}$ annihilate $M^G$.

Define the \emph{moment map}
\label{defn:moment}  
 $\mu \colon T^* V \to \mf{g}^*
$ by  $\mu(v,w)(x) = w(x \cdot v)$ for $(v,w)\in V\oplus V^*=T^*V$.  

\begin{remark}\label{Lambda-significance}
The significance of $\mu$ to $\ddd$-modules comes the following observations about associated varieties. 
Give $\ddd$ the order filtration $\ddd=\bigcup_{j} \ddd_j$ with associated graded ring $\gr\ddd\cong \C[V\times V^*]=\C[T^*V]$. If $d\in \ddd$ we write $\sigma(d)$ for its \emph{principal symbol}. A simple computation shows that $\mu^{-1}(0)$ is the   variety of zeros  in $T^*V$ of $\sigma(\tau(\g))$, written $\mu^{-1}(0)= \euls{V}\left(\sigma\left(\tau(\g)\right)\right)$. Note also that $\sigma(\tau(\g))=\sigma(\g_{\chi})$. 
\end{remark}

By definition $\euls{N}(V^*) = \euls{V}((\Sym V)_+^G)$ inside $V^*$. Hence 
$$
V\times \euls{N}(V^*) = \euls{V}\bigl(\sigma((\Sym V)_+^G)\bigr) \subset T^*V.
$$
Set 
$$
\Lambda = \mu^{-1}(0) \cap (V \times \euls{N}(V^*)).
$$
In order to understand the space $\Lambda$, we first recall the properties of the Fourier transform. For each $v \in V$, write $\partial_v$ for the associated derivation so that $[\partial_v,x] = x(v)$ for $x \in V^* \subset \C[V]$. The \emph{Fourier transform}
is the isomorphism 
\begin{equation}\label{defn:Fourier}
	\mathbb{F}_V \colon \ddd \stackrel{\sim}{\longrightarrow} \dd(V^*); \quad 	\mathbb{F}_V(x) = \partial_x, \quad \mathbb{F}_V(\partial_y) = - y. 
\end{equation}
If we (temporarily) write $\tau_V \colon \g \to \ddd$ and $\tau_{V^*} \colon \g \to \dd(V^*)$ for the two differentials of the $G$ action,  then a direct calculation shows that 
\begin{equation}\label{eq:Fouriergaction}
	\mathbb{F}_V(\tau_V(X)) = \tau_{V^*}(X) + \mr{Tr}_V(X),
\end{equation}
where $\mr{Tr}_V$\label{defn:trace} denotes the trace function. At the associated graded level, the Fourier transform is a $G$-equivariant symplectomorphism $\mathbb{F} \colon T^* V \stackrel{\sim}{\rightarrow} T^* V^*$, defined by  $\mathbb{F}(v,\lambda) = (\lambda,-v)$. 

\begin{lemma}\label{lem:regularlocalsystem}
	Assume that $V$ is polar. 
	\begin{enumerate}
		\item $V$ is visible if and only if $\Lambda$ is Lagrangian. 
		\item If $V$ is stable then $\Lambda \cap T^* V_{\reg} = V_{\reg}\times \{0\}$. 
	\end{enumerate}
\end{lemma}

\begin{proof}
	(1)  Since $\mathbb{F} \circ \mu_V = \mu_{V^*}$, we have $\mathbb{F}(\mu^{-1}_V(0)) = \mu_{V^*}^{-1}(0)$. Similarly, $\mathbb{F}(V \times \euls{N}(V^*)) = V^* \times \euls{N}(V)$. Thus, since $\mathbb{F}$ is a symplectomorphism, it suffices to show that the space
	$$
	\mathbb{F}(\Lambda) = \mu^{-1}_{V^*}(0) \cap ( V^* \times \euls{N}(V)) =: \Lambda^*
	$$
	is Lagrangian in $T^* V^*$ if and only if $V$ is visible. As noted in the the proof of \cite[Proposition~2.4]{GanGinzburg}, $\Lambda^* = \bigsqcup_{\euls{O}} T^*_{\euls{O}} V^*$, where the union is over the $G$-orbits $\euls{O}$ in $\euls{N}(V^*)$. 
	
	If $V$ is visible, this is a finite union because $V^*$ is visible. Since each irreducible component $\overline{T^*_{\euls{O}} V^*}$ is Lagrangian, $\Lambda^* = \bigcup_{\euls{O}} \overline{T^*_{\euls{O}} V^*}$ is also Lagrangian. If $V$ is not visible then $\Lambda^*$ can be written as the disjoint union of infinitely many locally closed subvarieties $T^*_{\euls{O}} V^*$, all of dimension $\dim V$. This means that $\dim \Lambda^* > \dim V$ and thus $\Lambda^*$ cannot be Lagrangian in $T^* V^*$.
	
	(2) Define   $
	\mf{h}^{\vee} = \{ \lambda \in V^* \, | \, \lambda(\mf{g} \cdot \mf{h}) = 0 \}. $
	Since $V$ is stable, \cite[Proposition~3.2]{BLT} implies that $V^*$ is also a polar $G$-representation with Cartan subspace $\h^{\vee}$. 
	As in \textit{loc. cit.}, set 
	$$
	C_0 = \overline{G \cdot (\h_{\reg} \times \h^{\vee})},
	$$
	and let $\mathrm{pr} \colon T^* V \to V$ denote the projection onto the first component. 
	
	We claim that
	$\mu^{-1}(0) \cap T^* V_{\reg} = G \cdot (\h_{\reg} \times \h^{\vee})$. Indeed, as explained in \cite[p.654]{BLT},    
	since $V$ is stable $G \cdot \h_{\reg}$ is contained in the regular sheet $S_0$ of $V$. If $y \in \mu^{-1}(0) \smallsetminus C_0$
	then $\mathrm{pr}(y) \notin S_0$ by \cite[Proposition~4.3(iii)]{BLT}. Thus
	$$
	\mu^{-1}(0) \cap T^* V_{\reg} = C_0 \cap T^* V_{\reg} = G \cdot (\h_{\reg} \times \h^{\vee}),
	$$
	which proves the claim.
	
	Since $\h^{\vee}$ consists of semisimple elements, $\h^{\vee} \cap \euls{N}(V^*) = \{ 0 \}$ and hence,
	by \eqref{lem:stablepolaropen}, 
	\[ \Lambda \cap T^* V_{\reg}  \ = \  \mu^{-1}(0) \cap (V_{\reg} \times \euls{N}(V^*)) \ = \  
	G \cdot \h_{\reg} \times \{ 0 \}  \ = \  V_{\reg} \times \{0\},\]
	as required.
\end{proof}

\begin{remark}\label{rem:notstablelocalfail}
	In general, Lemma~\ref{lem:regularlocalsystem}(2) can fail if   $V$ is not stable.  For instance, if $G = GL(V)$ 
	acts on $V$ in the canonical way, then $V$ is polar simply because $V \git  G = \{ 0 \}$. It is not a stable representation. As $V^*\git G=\{0\}$, clearly $\euls{N}(V^*)=V^* $ and so  
		\[ \Lambda \cap T^*V_{\reg}  \ = \ \mu^{-1}(0) \cap (V_{\reg} \times \euls{N}(V^*))  \ = \  \mu^{-1}(0) \neq  V_{\reg}\times\{0\} .\]
	Thus, Lemma~\ref{lem:regularlocalsystem}(2) fails.  
\end{remark}

We next record some basic facts about $(G,\chi,\dd)\lmod$. If needed, we distinguish the morphisms in the category $(G,\chi,\dd) \lmod$  from the  $\ddd$-module morphisms by writing the former as  $\Hom_{(G,\chi,\dd)}(M,M')$.  The next result is standard, but we could not find a precise reference.

\begin{lemma}\label{lem:charvar}
	Let $V$ be  a visible polar representation and $0\not=M \in (G,\chi,\dd) \lmod$.
	\begin{enumerate}		
		\item The \emph{characteristic variety $\Ch M$} \label{singular-defn} satisfies $\Ch M \subset \mu^{-1}(0)$. In particular,  
		\[
		\mathrm{GK. dim}_{\ddd}(M) \le \dim \mu^{-1}(0) = \dim V+\dim \h.
		\]		
		\item  If $(\Sym V)^G$ acts locally finitely on $M$, then $\Ch M \subset \Lambda$ and  $M$ is holonomic. 
		\item   $\Hom_{(G,\chi,\dd)}(M,M') = \Hom_{\dd}(M,M')$ for all $M, M'\in (G,\chi,\dd) \lmod$. 
	\end{enumerate}
 \end{lemma}

\begin{proof} (1) Pick a  finite dimensional, $G$-stable subspace $M_0 \subset M$  with $M=\ddd M_0$
 and take the good  filtration  $M=\bigcup \Gamma_jM$, where   $\Gamma_{j}M = \dd_{j}(V) \cdot M_0$. 	The associated graded module 
 $\gr_\Gamma M = \bigoplus   \Gamma_{j}M/\Gamma_{j-1}M$ is then a finitely generated, graded module over $\gr \ddd=\C[T^*V]$   with graded pieces  that are  stable under the action of $\sigma(\g_{\chi}) =\sigma(\tau(\g))$. 
Since $\tau(\g)\subseteq \Der(V)$, this means that  $\sigma(\tau(\g))\cdot M_0=0$.   Thus, by Remark~\ref{Lambda-significance}, $\Ch M\subseteq \euls{V}\bigl(\sigma(\tau(\g))\bigr) =\mu^{-1}(0)$.
Since  $V$ is visible,  it then follows from \cite[Proposition~4.3(i)]{BLT} that $\dim \mu^{-1}(0) = \dim V+\dim \h$.

(2) In this case pick  a finite dimensional, $G$-stable $(\Sym V)^G$-submodule $M_0 \subset M$ with $M=\ddd M_0$ and again filter $M$ by  $\{\Gamma_{j}(M) = \dd_{j}(V) \cdot M_0\}$.  In this case the graded
module $\gr_\Gamma M$ still has a locally finite action of $(\Sym V)^G$ and hence has  a locally nilpotent action of $(\Sym V)^G_+$.  Thus $\Ch M\subseteq \euls{V}(( \Sym V)_+^G) =V\times \euls{N}(V^*)$, by Remark~\ref{Lambda-significance}, again. Combined with Part~(1), this implies that  $\Ch M\subseteq \Lambda$.

Finally,  $M$ is holonomic by  Lemma~\ref{lem:regularlocalsystem}(1). 

	(3)  	We just need to show that any $\phi \in \Hom_{\dd}(M,M')$ is $G$-equivariant. Since $G$ is assumed to be connected, 
	this is the same as saying that $\phi$ is $\mf{g}$-equivariant. But the $\mf{g}$-action on both $M$ and $M'$ is 
	given by $\tau- \chi$; that is,  $x \cdot m = \tau(x) m- \chi(x) m$ for all $x \in \mf{g}$. Thus, 
	$$
	\phi(x \cdot m) = \phi(\tau(x) m) - \chi (x) \phi(m) = \tau(x) \phi(m) - \chi(x) \phi(m) = x \cdot \phi(m),
	$$
	as required.
\end{proof}

In this article we will mainly be interested in admissible $\dd$-modules; the definition given below is motivated by the admissible modules introduced by Ginzburg \cite{Gi}. 

\begin{definition}\label{defn:admissible}
	A  $(G,\chi)$-monodromic left  $\ddd$-module $N$ is defined to be  \emph{admissible} if  $(\Sym  V)^G$ acts locally finitely on $N$. The category of all such modules is written $\eC$. 
	The analogous category of right modules is written $\eCop$. 
	 \end{definition}
	 
We note that  $\eC$ is an abelian category. Since $G$ is connected, Lemma~\ref{lem:charvar}(3) implies that $\eC$ is a full subcategory of $\ddd\lmod$. The following result is immediate from Lemma~\ref{lem:charvar}(2).

\begin{corollary}\label{cor:admissible-holonomic}
	Let $V$ be a visible polar representation.  Then every admissible left $\dd(V)$-module is holonomic.\qed
\end{corollary}

Thus, when $V$ is visible, the category $\eC$ is a \emph{length category}\label{length-defn} in the sense that every object has finite length; see \cite[Proposition~I.5.3]{Bj}.
 
Given an object $M\in \dd(V^*) \lmod$, we write $\mathbb{F}^*_V(M)$ for the corresponding $\ddd$-module where, as a vector space, $\mathbb{F}^*_V(M) = M$ but  
$$
D \cdot m = \mathbb{F}_V(D) m, \quad \forall \, m \in M, \, D \in \ddd. 
$$
This defines an equivalence $\mathbb{F}^*_V \colon \dd(V^*) \lmod \to \ddd \lmod$. 

\begin{lemma}\label{lem:FourierGequiv}
	\begin{enumerate}
	\item 
	The equivalence $\mathbb{F}^*_V$ restricts to give an equivalence 
	$$
	\mathbb{F}^*_V \colon \bigl( G,\chi,\dd(V^*)\bigr) \lmod \to \bigl(G,\chi+\mr{Tr}_V, \ddd\bigr) \lmod.
	$$
\item 	 The morphism  $\mathbb{G}:= (\mathbb{F}^*_{V})^{-1}$  defines an equivalence 
$$\mathbb{G}\colon \bigl(G,\chi, \ddd\bigr) \lmod \to \bigl(G,\chi+\mr{Tr}_V,\dd(V^*)\bigr) \lmod. $$
\end{enumerate} 
\end{lemma}

\begin{proof}  (1) 
	Let $M \in (\dd(V^*),G,\chi) \lmod$ and write $\tau_{M} \colon \mf{g} \to \End_{\C}(M)$ for  the differential of the $G$-action. Then, by definition, $\tau_{V^*}(x) = \tau_{M}(x) + \chi(x) \mr{Id}_{M}$. By \eqref{eq:Fouriergaction}, this implies that 
	$$
	\tau_V(x) \cdot m = \mathbb{F}_V(\tau_V(x)) m = \tau_{V^*}(x)m + \mr{Tr}_V(x) m = \tau_{M}(x) m  + (\chi(x) + \mr{Tr}_V(x)) m,
	$$
for $m\in M$.	Therefore, $\mathbb{F}^*_V(M)$ is $(G,\chi +\mr{Tr}_V)$-monodromic.  
	
	(2)  This follows from (1) combined with the observation that  $\mr{Tr}_{V^*} = - \mr{Tr}_V$.
\end{proof}

\begin{lemma}\label{lem:decomposeadmissible} 
	The category $\eC$ admits a decomposition 
	$$
	\eC = \bigoplus_{\lambda \in V^* \git G} \euls{C}_{\chi,\lambda}
	$$
	where $\euls{C}_{\chi,\lambda}$ is the full subcategory of all modules on which the maximal ideal $\mf{m}_{\lambda} $  of $  (\Sym V)^G=  (\Sym \h)^W$ acts locally nilpotently. 
	
	If  $V$ is visible then $\euls{C}_{\chi,\lambda}$ contains only finitely many irreducible objects. 
\end{lemma}
 
\begin{proof}
	Since each element of $(\Sym V)^G$ acts locally ad-nilpotently on $\ddd$, the decomposition is immediate. 
	
	Now assume that $V$ is visible and note that $V^*$ is also visible. Let  $\varpi \colon V^* \to V^* \git G$ be  the quotient map.  By Lemma~\ref{lem:FourierGequiv}(2),  
  $\mathbb{G}(\euls{C}_{\chi,\lambda})$ is the category of $(G,\chi+\mr{Tr}_V)$-monodromic $\eD(V^*)$-modules supported on $\varpi^{-1}(\lambda)$.  Since $V^*$ is visible,   this subvariety has finitely many $G$-orbits. It follows, as in \cite[Theorem~11.6.1]{HTT}, that $\mathbb{G}(\euls{C}_{\chi,\lambda})$ has finitely many irreducible objects.
\end{proof}

\begin{lemma}\label{lem:Vregintconnection}
	Assume that $V$ is a stable polar representation. If $N$ is an admissible left $\ddd$-module then 
	$N |_{V_{\reg}} = \dd(V_{\reg})\otimes_{\ddd}N$ is an integrable connection. 
\end{lemma}

\begin{proof}
	Recall that a $\ddd$-module $\euls{L}$ is an integrable system if it is finitely generated as an $\C[V]$-module. Let $\euls{L} := N |_{V_{\reg}}$. By Lemmata~\ref{lem:charvar}(2) and~\ref{lem:regularlocalsystem}(2), the characteristic variety $\Ch \euls{L}$ is contained in the zero section $V_{\reg}\times\{0\}$. By \cite[Proposition~2.2.5]{HTT}, this implies that $\euls{L}$ is an integrable connection.  
\end{proof}

The example from Remark~\ref{rem:notstablelocalfail} also shows that  Lemma~\ref{lem:Vregintconnection} can fail  without the stability assumption.

  Finally, we  note that certain  cohomology groups of  monodromic modules are also monodromic. As this is presumably well-known to the experts, and is implicit in \cite[(1.2)]{BL}, we 
 relegate the proof to the appendix. 
 
\begin{lemma}\label{Ext-equivariance}
	Let $N$ be a finitely generated, monodromic left $\ddd$-module. For each $i \ge 0$, the right $\ddd$-module $\Ext^i_{\ddd}(N,\ddd)$ has a canonical $(G,\chi)$-monodromic structure.\qed
\end{lemma}



\section{Ad-Nilpotence}\label{Sec:nilpotence} 

In understanding quantum Hamiltonian reduction, one is interested in relating monodromic modules over $\dd(V)$ to those over (a factor of) $ \dd(V)^G$ through the bimodule $\eMt=\ddd/ \ddd\g_{\chi}$. In this section we clarify how locally nilpotent actions of subrings ``pass through'' $\eMt$. This is actually quite general and so the module $\eMt$ will only be defined later; see Definition~\ref{M-new-definition}, in particular. Instead,  we fix the following hypotheses. 

\begin{hypotheses}\label{ad-hypotheses}
	In this section, fix $\C$-algebras $S$ and $A$, with $S$ noetherian, and a commutative ring $C$ that embeds into both $A$ and $S$ and acts locally ad-nilpotently on both  rings.  
\end{hypotheses}

To fix notation, if $d\in C$ and $m$ belongs to an $(S,A)$-bimodule $B$,  then  we define $\ad(d)(m)=dm-md$. Then, as for rings,  the action of $d$ on $B$ is \emph{locally ad-nilpotent}  if,  for any  $m\in B$,
there exists $n\in \mathbb{N}$ such that $\ad(d)^n(m)=0$.   Recall that an element $0\not=r \in S$ is  \emph{regular}  if it is neither a left nor a right zero divisor. 
 
\begin{definition}\label{Bi-definition}   
	Let $\Bi(S,d)$ denote the full subcategory of $(S,\C[d])$-bimodules that are finitely generated as left $S$-modules and, similarly, let 
	$\Bi(d,S)$ denote the category of $(\C[d], S )$-bimodules that are finitely generated as  right $S$-modules. Let 
	$\Biad(S,d)$ denote the set of modules $B\in \Bi(S,d)$ on which the action of $d$ is locally ad-nilpotent. The category $\Biad(d,S)$ is defined similarly.

	For each $k \ge 1$ consider the $(\C[d], S)$-bimodule 
	$$
	F_k \ = \  (\C[d]\otimes_\C S) /(d \otimes 1 - 1 \otimes d)^k (\C[d]\otimes_\C S) \ \in \ \Biad(d,S).
	$$ 
	As a right $S$-module, $F_k\cong S[d] / (d^k)$ is free. 
\end{definition}

\begin{lemma}\label{ad-nilp1}
	Let $B\in \Biad(d,S)$. 
	Then there exist $k,\ell \ge 1$ and a surjection $F_k^{\oplus \ell} \twoheadrightarrow B$ in the category of $(\C[d],S)$-bimodules. 
\end{lemma}

\begin{proof}
	Let $B_0 \subset B$ be a finite dimensional $\ad(d)$-stable subspace of $B$ such that $B_0S = B$. Then there 
	exists some $k \gg 0$ such that $\ad(d)^k(n) = 0$ for all $n \in B_0$. If we take $\ell = \dim B_0$ then fixing a basis 
	of $B_0$ defines a surjection of $(\C[d], S )$-bimodules $F_k^{\oplus \ell} \twoheadrightarrow B$. 
\end{proof}

By induction we immediately obtain:

\begin{corollary}\label{ad-nilp11}  Let $B\in \Biad(d,S)$.  Then  there exists a resolution $(P_*,\dalpha_*)$  of $B$ by  
	bimodules $P_j=\bigoplus F_i^{(j_i)}\in \Biad(d,S)$ that are   free right $S$-modules. Moreover,   each differential
	$\dalpha_i : P_{i+1} \to P_i$ is a morphism of $(\C[d],S)$-bimodules. 
	\qed
\end{corollary}

The next result is a key application of ad-nilpotence to homology groups. 

\begin{proposition}\label{ad-nilp2}
	Let $B\in \Biad(d,S)$  and pick  a resolution  $P_* \to B$ by Corollary~\ref{ad-nilp11}. Then the following hold.
	\begin{enumerate}
		\item[(1)] For each $i$, $\Hom_S(P_i,S)\in \Biad(S,d)$. 
		\item[(2)] For each $i$, the induced morphism $\dalpha_i^* : \Hom_S(P_i,S) \to  \Hom_S(P_{i+1},S)$ is a morphism of $(S,\C[d])$-bimodules. 
		\item[(3)] In particular, the cohomology of $(\Hom_S(P_*,S),\dalpha_i^*)$ belongs to $\Biad(S,d)$.
	\end{enumerate}
\end{proposition}

\begin{proof}  (1) That $\Hom_S(P_i,S)\in \Bi(S,d)$ is routine. Let $\phi \in \Hom_S(P_i,S)$ and $x \in P_i$. Then 
	$$
	\bigl(\ad(d) \phi\bigr)(x) = d\phi(x) - \phi(dx) =       \ad(d)(\phi(x))     -    \phi\left(\ad(d)(x)\right) 
	$$
	and hence, by induction,
	$$
	\bigl(\ad(d)^n \phi)(x) = \sum_{i = 0}^n \lambda_i \ad(d)^{n-i}\left(\phi\bigl(\ad(d)^i(x)\bigr)\right)
	$$
	for suitable integers $\lambda_i$.
	Pick $n_0$ such that $\ad(d)^{n}(x) = 0$  for all $n \geq n_0$. Then  
	\begin{equation}\label{eq:ad-nilp2}
		\bigl(\ad(d)^n \phi\bigr)(x) = \sum_{i < n_0} \lambda_i \ad(d)^{n-i}\left(\phi\bigl(\ad(d)^i(x)\bigr)\right).
	\end{equation}
	By ad-nilpotence,  there exists $n \gg 0$ such that 
	$\ad(d)^{n-n_0}(\phi(\ad(d)^i(x))) = 0$ for all $0\leq i < n_0$. Therefore, by \eqref{eq:ad-nilp2},
	$(\ad(d)^n \phi)(x) = 0$ for such $n$. 
	
	As $P_i$ is a  finitely generated $S$-module, say by a finite dimensional vector space $Z$, we can therefore find  $n\geq 1$ such that $(\ad(d)^n \phi)(Z) = 0$.  By inspection,  $\ad(d)^n \phi$ is   a $(\C[d],S)$-bimodule map, from which it follows that 
	$ (\ad(d)^n \phi)(P_i) = 0$. In other words, $\ad(d)^n \phi = 0$.  

	(2)  By construction, the $\dalpha_i^*$ are morphisms of left $S$-modules. Now
	let   $\phi \in \Hom_S(P_i,S), x \in P_i$ and $g \in \C[d]$. Then, 
	unravelling the definitions, and using that $\dalpha_i$ is a map of left $\C[d]$-modules, one finds that
	\[ \left(\dalpha_i^*(\phi)\cdot g\right)(x)= \phi(\dalpha_i(gx)) = \left((\dalpha_i^*\phi)\cdot g\right)(x).\]
	Thus  $\dalpha_i^*(\phi)\cdot g =  \dalpha_i^*(\phi\cdot g)$ and  $\dalpha_i^*$ is a morphism of right $\C[d]$-modules. Since the  rings $S$ and $\C[d]$ act from different sides, $\dalpha_i^*$ is 
	clearly a morphism of bimodules. 
	
	(3) This is immediate from (1,2).
\end{proof}

\begin{corollary}\label{ad-nil prop} For $m\geq 0$, 
	$\Ext_S^m( - , {}_SS)$ defines a functor $\Bi(S,d) \to \Bi(d,S)$ sending:  
	\begin{enumerate}
		\item $\Biad(S,d)$  to  $\Biad(d,S)$; 
		\item $\ad(d)$-locally finite modules to $\ad(d)$-locally finite modules. 
	\end{enumerate}  
\end{corollary}

\begin{proof}
	(1)  This is immediate from Proposition~\ref{ad-nilp2}.
	
	(2) Assume that $N \in \Bi(S,d)$ is $\ad(d)$-locally finite. Since the action of $\ad(d)$ on $S$ is locally nilpotent by assumption and $S$ is finitely generated, $N$ 
	admits a finite decomposition $N = \bigoplus_{\lambda \in \C} N_{\lambda}$, where $N_{\lambda}$ is the $(S,\C[d])$-submodule 
	$$
	N_{\lambda} = \{ n \in N \, | \, (\ad(d) - \lambda)^k \cdot n = 0 \textrm{ for some } k > 0 \}.
	$$
	Since  $\Ext_S^m(N, {}_SS) = \bigoplus_{\lambda} \Ext_S^m( N_{\lambda} , {}_S S)$,  it suffices to assume that $N = N_{\lambda}$ for some 
	$\lambda$.  
	Now define a
	$(S,\C[d])$-module $N'$, which equals $N$ as a left $S$-module, but has the new right action given by 
	$n \star d := n (d - \lambda)$. Then applying Proposition~\ref{ad-nilp2} to $N'$ shows that 
	$(\ad(d) - \lambda)$ acts locally  ad-nilpotently on $\Ext_S^m( N, {}_SS)$, as is required.
\end{proof}

A variant of the proof of Proposition~\ref{ad-nilp2} gives the following useful fact.  

\begin{proposition}
	\label{torsion-homs}  Let $B$ be an $(A,S)$-bimodule that is finitely generated as a right $S$-module.
	Let $\eS$ be an Ore set of elements in both $A$ and $S$ acting locally ad-nilpotently on $S,A$ and on  the bimodule 
	$B$. Let $L$ an $\eS$-torsion  right $S$-module. Then $\Hom_S(B, L)$ is also $\eS$-torsion.
\end{proposition}

\begin{proof}  Pick $0\not=d\in \eS$. By ad-nilpotence, $\eS'= \{d^j: j\in \mathbb{N}\}$ is also an Ore set in $S$ and $A$,  and so it is enough to prove the result when 
	$\eS=\{d^j\}$.  If $\theta\in H=\Hom_S(B, L)$, then the action of $d$  on $H$ is given by $(\theta d  )(x)=\theta(dx)$ for $x\in B$.  
	This can be rewritten as 
	\[\begin{aligned} (\theta d)(x)  \ = &\ \theta(dx) \ =\ \theta\left(\ad(d)(x)\right) + \theta(xd)  = \theta\left(\ad(d)(x)\right) + \theta(x)d.
	\end{aligned}
	\] 
	Using the obvious  induction, this leads to the following formula: for any  $\theta\in H$, $x\in B$ and  $j\geq 1$,
	\begin{equation}\label{torsion-homs-equ2}
		(\theta d^j)(m)  \ =\ \sum_{i=0}^j \lambda_i  \theta\left(\ad(d)^i(m)\right)d^{j-i} \qquad\text{for some }\ \lambda_i\in \mathbb{Z}.
	\end{equation}

	Since $B$ is finitely generated as an $S$-module, we can choose a finite dimensional, $\ad(d)$-stable subspace 
	$Z\subseteq B$ for which $ZS=B$. Choose $j_0$ such that $\ad(d)^{j_0}(Z)=0$. Then \eqref{torsion-homs-equ2} implies that, for all $j\geq j_0$,
	\begin{equation}\label{torsion-homs-equ3}  
		(\theta d^j)(Z)  \ \subseteq  \ \sum_{i=0}^j \lambda_i  \theta\left(\ad(d)^i(Z)\right)d^{j-i}     
		\	\subseteq \ \sum_{i=0}^{j_0} \theta(Z)d^{j-i}.
	\end{equation}
	
	As $L$ is $\eS$-torsion, there exists $j_1\geq j_0$ such that $\theta(Z)d^{j_1-j_0}=0$. Hence,
	\eqref{torsion-homs-equ3} implies that  $\left(\theta d^j\right)(Z)=0$ for all $j\geq j_1$. Finally, as $\theta d^j$ is  a right $S$-module map, this in turn implies that $\theta d^j(B)=0$ and hence that $\theta d^j=0$.
\end{proof}

We end the section with some further results on ad-nilpotence that will be needed later. For the first result we slightly weaken our hypotheses  so 
that they are left-right symmetric.

\begin{lemma}\label{lem:lefttfree}  Let   $0\not=B$ be an $(X,Y)$-bimodule  for rings $X$ and $Y$ and  suppose that $C=\mathbb{Z}[d]$ embeds into both rings so that $d$
	acts locally ad-nilpotently on both rings and on the bimodule $B$. Set $\eS=\{d^j ; j\geq 1\}$. Then
	\begin{enumerate} 
		\item  if $\eS$ acts torsionfreely on $B$ from one side, then $B$ is $\eS$-torsionfree on the the other side.
		\item  if   $B$ is $\eS$-torsion from one side, then $B$ is $\eS$-torsion  on  the other side.
	\end{enumerate}
\end{lemma}

\begin{proof}  (1)  If the result fails then, by symmetry, we may suppose that $B$ is right $\eS$-torsionfree but not left 
	$\eS$-torsionfree.  Then  $B$ has a left $\eS$-torsion left submodule, $T$ and this  is again an $(X,Y)$-bimodule. Pick $0\not=x\in T$, say with $d^kx=0$. 
	By  ad-nilpotence,  $T_0 := \mathbb{Z}[\mathrm{ad}(d)]x$ is a finite $\mathbb{Z}$-submodule of $T$. Moreover,  as $d^k x = 0$,  clearly $d^k T_0 = 0$. Since the 
	action of $\mathrm{ad}(d)$ on $T_0$ is nilpotent, there exists  $0\not=y\in T_0$ such that $\mathrm{ad}(d)(y) = 0$. But this means that $d y = y d$. Thus
	$yd^k = d^k y = 0$ and hence $y$ is right $d$-torsion,  contradicting our 
	starting assumption.
	
	(2) Assume that $B$ is left $\eS$-torsion and write $T'$ for the $\eS$-torsion submodule of $B$ as a right $Y$-module. 
	Clearly $T'$ is also an $(X,Y)$-bisubmodule of $B$. Thus, if $T'\not=B$ then $B'=B/T'$ is a nonzero $(X,Y)$-bimodule that is 
	$\eS$-torsionfree on the right. But, as $B$ is left $\eS$-torsion, so is $B'$. This contradicts Part~(1).  
\end{proof}

The lemma becomes stronger when we add the hypothesis that $A$ be simple. 

\begin{corollary}\label{lem:leftrightdfreepolar} In Hypotheses~\ref{ad-hypotheses} assume that 
	$A$ is a  simple right Goldie ring and let $0\not=B$ be a $(S,A)$-bimodule that is finitely generated as a left $S$-module. Pick 
	$0\not=d\in C$ and assume that  the image of $d$  in   $A$ is regular and that $d$ acts locally ad-nilpotently on     $B$. 
	Then $B$ is $d$-torsionfree on both the left and the right.  
\end{corollary}

\begin{proof}   
	Set $\eS=\{d^j\}$ and suppose first that $B$  has $\eS$-torsion as a right $A$-module. Since $\eS$ acts locally ad-nilpotently on $A$,  it is an Ore set and 
	hence  $B$ has a nonzero right $\eS$-torsion submodule, say $T$.
	As $T$ is trivially a left $S$-module,  it is therefore    finitely generated on the  left; say $T=\sum_{i=1}^w  St_i$. Therefore, as $A$ is prime right Goldie, 
	$r$-ann$_{A}(T)=
	\bigcap_{i=1}^w r\text{-ann}(t_i) \not=0$. This contradicts the simplicity of $A$.  By Lemma~\ref{lem:lefttfree} $B$ is also $d$-torsionfree on the left. 
\end{proof}

Finally, we need a result stated in \cite[Lemma~6.3(2)]{AJM}. 

\begin{lemma}\label{ore91}  Keep Hypotheses~\ref{ad-hypotheses} and assume that $C$ is a domain. Let  $B$ be an    $(S,A)$-bimodule such that    $C$  acts  locally ad-nilpotently   on $B$.    
	Let $\eS$ be a multiplicatively closed subset of $C\smallsetminus\{0\}$ and assume that $\eS$ consists of regular elements in both $S$ and $
	A$ and acts torsionfreely on $B$ from both sides.
	
	 Then the two localisations $B_{\eS}:=B\otimes_SS_{\eS}$ and 
	${}_{\eS}B:= A_{\eS}\otimes_ AB$ are naturally isomorphic as $(S,A)$-bimodules or, indeed, 
	as $(S_{\eS}, A_{\eS})$-bimodules.\end{lemma}

\begin{proof}   	Define $R=\left(\begin{smallmatrix}S&B\\ 0&A\end{smallmatrix}\right)$ with the usual matrix multiplication.
	Our hypotheses   ensure that
	$\eS'=\left\{\left(\begin{smallmatrix}c&0\\ 0&c\end{smallmatrix}\right) : c\in \eS\right\}$ consists of regular elements that act locally ad-nilpotently on $R$. Hence $\eS$ is    denominator    set in $R$, in the notation of \cite[p.168]{GW}. By universality,  the right and left Ore localisations $R_{\eS'}$ and $ {}_{\eS'}R$ are
	canonically  isomorphic	\cite[Proposition~10.6]{GW}.  Looking at the $(1,2)$ entry this reduces to the assertion that 
$B_{\eS}$ and $ {}_{\eS}B$   are both  $(S_{\eS}, A_{\eS})$-bimodules and, moreover, that 	$B_{\eS}\cong {}_{\eS}B$   as   $(S_{\eS}, A_{\eS})$-bimodules and hence as $(S,A)$-bimodules. \end{proof}

\begin{remark}\label{ore-remark}
(1) We have given a detailed proof to Lemma~\ref{ore91}, since it can fail   if $\eS$ is not regular. For example, take $R=\C[x]=C$, regarded as a right $C$ module in the usual way, but regarded as a left $C$-module via the map $C\to C/xC=\C\subset R$. If $\euls{S}=\C[x]\smallsetminus \{0\}$, then clearly ${}_{\eS}R=0 \not=R_{\eS}=\C(x)$.
 
(2) The proof  of  \cite[Lemma~6.3(2)]{AJM}  does still hold once one makes  the following observation: Note that  the module $\mathcal{N}=\dd(\g)/\dd(\g)\tau(\g)$ in that proof is $d$-torsion-free for any $d\in \euls{O}(\g)^G$ by \cite[Theorem~1.1]{LS}. Thus  \cite[Equation~6.2]{AJM} follows  from Lemma~\ref{ore91}. 
\end{remark}

\section{The Bimodule Controlling Quantum Hamiltonian Reduction}\label{Sec:HC-bimodule}

We now turn to the detailed study of differential operators on polar representations and the resulting quantum Hamiltonian reduction.  The main aim of this section is define and describe the  main properties of the bimodule $\eM$ that controls the interplay between  $\eD(V)^G$ and $\Aak= \Ak(W)$. We will also describe  some of the easier consequences of these results to the category of   admissible modules, as defined in Definition~\ref{defn:admissible}.  We begin by making this precise.   

\begin{notation}\label{notation4.11}  \emph{The following setup will be fixed for the rest of the paper.}
	Fix a visible, stable, polar representation $V$ for a connected reductive algebraic group $G$ and set $\mathfrak{g} = \mathrm{Lie}(G)$. Without loss of generality we will also assume that $V$ is faithful. 
	Set $\ddd=\eD(V)$ for the ring of differential operators on  $V$ with its induced action of $G$. 
	We will  reserve $n$ and $m$ to denote $n=\dim \h$, respectively  $m= \dim V-\dim \h$.
\end{notation} 

As in \cite[Equation~3.10]{BLNS}, let	 $\delta_1, \ds, \delta_k$ denote the pairwise distinct irreducible factors of $\delta$ in $\C[V]$, where $\delta_i$ has weight, say,  $\theta_i$. Choose $\vs_i \in \C$ for $i = 1, \ds, k$ and let $\chi =\sum_{i=1}^k \vs_i d \theta_i$ be the associated linear character of $\mf{g}$, as defined in \cite[Equation~4.1]{BLNS}. 
By \cite[Theorem~5.1]{BLNS}, and in the notation of \eqref{chi-defn},  there is 
\emph{a radial parts map}\label{radial-defn2}
\begin{equation}\label{eq:notation4.11}
\rad_{\vs} \colon \eD(V)^G \rightarrow (\eD(V) / \eD(V)\mf{g}_{\chi})^G  \to {\Aak} ,  
\end{equation}
for some parameter $\kappa$.  The precise  relationship between 
$\vs=(\vs_1,\dots,\vs_k)\in \C^k$ and  the resulting parameter $\kappa$ will not be needed in this paper, but it is  described in detail in \cite[Section~5]{BLNS} (see \cite[Theorem~5.21]{BLNS} in particular).  With this notation, we also have the following generalisation of Theorem~\ref{thm:intro-radial-exists}.

\begin{theorem} {\rm (\cite[Corollary~6.10, Theorem~7.8, Theorem~7.10]{BLNS})}  \label{thm:radial-exists}  
Let $V$ be a stable polar representation. 
  The radial parts map $\rad_{\vs}:  \dd(V)^G \to \Aak$  has the following properties.
\begin{enumerate}
 \item The restriction of $\rad_{\vs}$ to $V_{\reg}$ induces the isomorphism
  \begin{equation}\label{eq:localizatioradiso}
  \bigl(\dd(V_{\reg})\bigl/\dd(V_{\reg})\g_{\chi}\bigr)^G \isom \dd(\h_{\reg})^W.
  \end{equation}
\item The   map
$\rad_{\vs}  $ restricts to give
filtered isomorphisms $\rr \colon \C[V]^G
\stackrel{\sim}{\longrightarrow} \C[\h]^W$ and $\rrpp \colon
(\Sym \, V)^G \stackrel{\sim}{\longrightarrow} (\Sym \, \h)^W$. 
\item If    $\Aak$ is simple, then $\rad_{\vs}$ is surjective. \end{enumerate}
\end{theorem}

\emph{In this section  we will always  assume that $\rad_{\vs}$ is surjective.}
  We do not know of any situation where this does not hold, and as Theorem~\ref{thm:radial-exists}(3)  and its generalisations in \cite[Theorem~7.1]{BLNS} show, it certainly holds in many cases.

\subsection*{The bimodule}
 While  $\rad_{\vs}$ is known to be surjective in many circumstances,  much less is known about its  kernel $\ker(\rad_{\vs})$.  It is readily  shown that  $\ker(\rad_{\vs}) \supseteq (\ddd\g_\chi)^G$ but it is not known whether this is always an equality. We circumvent this by means of the following  definition.  This also describes the modules $\eM$ and $\eMt$ that will play  a  fundamental r\^ole in the interplay between  $\eD(V)$-modules and   $\Ak$-modules.

\begin{definition}\label{M-new-definition}  
	Assume that  $\rad_{\vs}$ is surjective. Set  $\ddd=\eD(V)$ and write 
	\begin{equation}\label{R-defn}
	R\ = \ \frac{\ddd^G}{\left(\ddd\g_{\chi}\right)^G} \ \supset \ P = \frac{\ker(\rad_{\vs})}{\left(\ddd\g_{\chi}\right)^G};
	\quad\text{thus} \quad \Aak = R/P.
	\end{equation}
	Define  
	\begin{equation}\label{M-defn}
 \eMt=\ddd/ \ddd\g_{\chi} \qquad\text{and}\qquad  	\eM = \eMt/\eMt P.
	\end{equation}
  
It is important to note that $R$ is a ring. It follows that  $\eMt$ is naturally  an $(\ddd,R)$-module under the actions induced from multiplication inside $\ddd$.  Hence $\eM$ a right $R/P$-module and thus is an $(\ddd,{\Aak})$-bimodule. We note that the associated graded algebra  of $\ddd^G$ with respect to the order filtration is noetherian. This implies that $\ddd^G$ and its quotients $R,{\Aak}$ are noetherian algebras.    
 
\end{definition}

\begin{remark}\label{rem:M-tilde}
In a number of important examples it is known that $\ker(\rad_{\vs})= (\ddd\g_\chi)^G$ in which case $R={\Aak}$ and $\eMt=\eM$; see, for example, \cite{LS, LS3, GordonCyclicQuiver} and  Section~\ref{Sec:examples}. In those cases the Harish-Chandra bimodule  
 $\eMt$ is fundamental to the interplay between $\ddd$ and ${\Aak}$; see, for example, \cite{HK} and \cite{AJM}.
However, proving that $\ker(\rad_{\vs})= (\ddd\g_\chi)^G$  in   general seems to be hard, which is why we 
have to distinguish between $\eMt$ and $\eM$.  This allows us to circumvent the problem since, as will be shown in Proposition~\ref{semiprime2}, when ${\Aak}$ is simple $\eMt$ splits as a direct sum of $\eM$ and a largely irrelevant     $\delta$-torsion module.   \end{remark}

For completeness, we record the question alluded to above.

\begin{question} \label{Question:M-tilde}
	Is there an example of a stable visible polar representation where $R\not={\Aak}$? Or, indeed, an example where $\eMt\not=\eM$? In Remark~\ref{curio} we will discuss an example from \cite{LS3}  that illustrates the subtlety of these questions. 
\end{question}

Recall  the category $\left(G,\chi,\ddd\right)\lmod$ of  $(G,\chi)$-monodromic left $\ddd$-modules from Definition~\ref{defn:monodromic}.
 It is  standard that $\Hom_\ddd(\ddd/ \ddd\g, L)=L^G$ for a $G$-equivariant module $L$. The next result shows that 
the  analogue for $(G,\chi)$-monodromic modules  is also true.

\begin{lemma}\label{hom=invariants} 
If  $L$ is a monodromic $\ddd$-module, then $\Hom_\ddd(\eMt,\, L) \cong L^G$. 
\end{lemma}

\begin{proof}  Set $Y=  \Hom_\ddd(\ddd/\ddd\g_\chi,\, L)$.
	Since $\g_\chi = \{\tau(x)-\chi(x) : x\in \g\}$,  
	\[ Y  \ = \ \{\ell \in L \ | \  \tau(x)  \ell = \chi(x)\ell \quad \text{for all } x\in \g\}.\]
	Since  $L$ is monodromic, the morphism $\tau_L : \g \to \End_{\C} (L)$ given by differentiating the $G$-action, 
	satisfies $\tau_L(x)=\tau(x)-\chi(x) $,  for all $x\in \g$.   
	Thus
	\[  
	Y\ = \  \{\ell\in L \  | \ \tau_L(x)(\ell) = 0 \, \text{for all } x\in \g\}  
	= \  \ L^{^{\scriptstyle \g}} \ = \ L^G,
	\]
	by the connectedness of $G$.
\end{proof}

The following consequence of   Lemma~\ref{hom=invariants} will be used frequently. 
 
\begin{lemma}\label{summand-endo21} 
	Let $R$   be defined by \eqref{R-defn} and assume that $\rad_{\vs}$ is surjective. 
	\begin{enumerate}
	  \item As an  $(\ddd^G,R)$-bimodule  $\eMt$ decomposes as $\eMt=R\oplus \widetilde{T}$, where $\widetilde{T}$ is the sum of the nontrivial $G$-submodules of $\eMt$.
	  
	  \item Consequently, given a  left $R$-module $L\not=0$, then $\eMt\otimes_RL\not=0$.
	\item   $\End_\ddd(\eMt)=R=\left(\ddd/\ddd \mf{g}_{\chi}\right)^G$. 
	\item Given a left $R$-module $L\not=0$, then  
	  $\Hom_\ddd(\eMt,\eMt\otimes_RL)\cong L$. 
  \end{enumerate}
\end{lemma}  

\begin{proof} (1) As an $\ddd^G$-bimodule, $\ddd$ splits as $\ddd=\ddd^G\oplus \ddd_G$ where $\ddd_G$ is the sum of the nontrivial $G$-submodules. Since $G$ is reductive, $R=\ddd^G/(\ddd \g_{\chi})^G = \eMt^G$.  Thus the decomposition of $\ddd$  descends to a decomposition $\eMt=\ddd/\ddd\g_{\chi} = R \oplus \widetilde{T}$ of   $(\ddd^G,R)$-bimodules, where $\widetilde{T}$ is the sum of the nontrivial $G$-submodules of $\eMt$.  

(2) is immediate from (1).
 
 (3,4)  These follow  immediately from the decomposition $\eMt = R\oplus  \widetilde{T}$
    described above, combined with  Lemma~\ref{hom=invariants}.  \end{proof}

The analogous result for $\eM$ states

\begin{lemma}\label{summand-endo21M} 
	Let ${\Aak}$ be defined by \eqref{R-defn} and assume that $\rad_{\vs}$ is surjective. 
	\begin{enumerate}
		\item As an  $(\ddd^G,{\Aak})$-bimodule  $\eM$ decomposes as $\eM={\Aak}\oplus T$, where $T$ is the sum of the nontrivial $G$-submodules of $\eM$.
		
		\item Consequently, given a  left ${\Aak}$-module $L\not=0$, then $\eM\otimes_{\Aak} L\not=0$. 
	\end{enumerate}
\end{lemma}  
 
\begin{proof} (1)  Consider the splitting $\ddd=\ddd^G\oplus \ddd_G$ from the proof of Lemma~\ref{summand-endo21}. Since $G$ is reductive, and $(\ddd\g_{\chi})^G\subseteq \ker( \rad_{\vs})$,  
	$$\eM^G \ = \  \left( \frac{\eMt}{\eMt P}\right)^G \ = \ \left(\frac{\ddd}{\ddd\g_{\chi} + \ddd \ker (\rad_{\vs})}\right)^G
	\ = \  \frac{\ddd^G}{ \ddd^G \ker (\rad_{\vs})} \ = \ \frac{R}{P} \ = \ {\Aak}.$$
	Thus, the decomposition of $\ddd$  descends to a decomposition  	$\eM =  {\Aak} \oplus T, $  of   $(\ddd^G,{\Aak})$-bimodules
   where $T$ is now the sum of the nontrivial $G$-submodules of $\eM$.  
	
	(2) is immediate from (1).
\end{proof}

Before stating the next lemma we want be precise about the actions of $\C[V]^G$ on $\eM$ in order to avoid any possible confusion.
By construction, $\eMt = \ddd/\ddd\g_{\chi}$ is an $(\ddd,\ddd^G)$-bimodule and, as such,  
has left and right actions of the subring $\C[V]^G$. On the other hand, the module $\eM=\eMt/\eMt P$ is naturally an $(\ddd,{\Aak})$-bimodule and hence has a right action of $\C[\h]^W$. 
By Theorem~\ref{thm:radial-exists}(2),
$\C[\h]^W$ is precisely the image of $\C[V]^G$  in ${\Aak}=R/P=\ddd^G/ \ker (\rad_{\vs})$. As such, the right actions on $\eM$  of $\C[\h]^W$ (thought of as a subalgebra of ${\Aak}$) and $\C[V]^G$ (thought of as a subalgebra of $\ddd^G$) are identical.   Thus, there is no possible ambiguity about the adjoint action of $\C[V]^G$ on $\eM$ under the identification $\C[V]^G=\C[\h]^W$.

\begin{lemma}\label{lem:Gequivisfullsub}  Assume that $\rad_{\vs}$ is surjective. 
	\begin{enumerate}
		
		\item    $\eMt$ is  $(G,\chi)$-monodromic and  is a projective object in $(G,\chi,\ddd)\lmod$.
		Similarly,  $\eM$  is  monodromic.
		
		\item The adjoint action of $\C[V]^G$  on the $(\ddd,\ddd^G)$-bimodule $\eMt$ is locally ad-nilpotent. Similarly,
		the adjoint action of $\C[V]^G$  on the  $(\ddd,{\Aak})$-bimodule   $\eM$  is locally ad-nilpotent.
		
		\item For any $j\geq 0$, the right $\ddd$-module $\Ext^j_\ddd(\eM,\ddd)$  is  monodromic.

		\item    $\Ext^j_\ddd(\eM,\ddd)$ has an induced $({\Aak}, \ddd)$-bimodule structure  under which the induced action of $\C[V]^G$   is locally ad-nilpotent. 
	\end{enumerate}  
\end{lemma}

\begin{proof} 
	(1)   It follows from the definition of $\g_{\chi}$ that  $\eMt$  is monodromic.
	Now use   \cite[Lemma~6.1]{AJM} to show that the factor $\eM$ is   also monodromic.
	By Lemma~\ref{hom=invariants}, 
	the functor $\Hom_\ddd(\eMt,-) = (-)^G$ is exact in  $(G,\chi,\ddd)\lmod$  and so  $\eMt$ is a projective object in that category.

	 	(2) The subring $\C[V]^G$ of $\ddd$ acts locally ad-nilpotently on $\ddd$ and hence on the $(\ddd,\ddd^G)$-bimodule $\eMt=\ddd/\ddd\g_\chi$.  
	The result  for $\eM$  then follows from  the comments made before the statement of the lemma.  
	
	(3) Apply Lemma~\ref{Ext-equivariance}.

	(4) This follows from Part~(2) and  Lemma~\ref{ad-nil prop}(1).   
\end{proof}
 
  The following standard result will be used frequently.

\begin{lemma}\label{equivariant tensors}  Assume that $\rad_{\vs}$ is surjective. 
\begin{enumerate}
\item  If $L$ is a finitely generated left ${\Aak}$-module then  the $\ddd$-module $\eM\otimes_{\Aak}L$ is $(G,\chi)$-monodromic. Similarly, if $L'$ is
 a  finitely generated  right ${\Aak}$-module then 
 $L'\otimes_{\Aak}  \Ext^j_\ddd(\eM, \ddd)$ is monodromic for any $j\geq 0$. 
 
\item   Any $G$-stable subfactor of a monodromic  $\dd(V)$-module $M$ is also monodromic. 
\end{enumerate}
 \end{lemma}
 
 \begin{proof}  (1) By Lemma~\ref{lem:Gequivisfullsub},  $\eM$ is monodromic.
  Thus for any  natural number $u$,   $\eM^{(u)}\cong \eM\otimes_{\Aak}{\Aak}^{(u)}$ is also 
  monodromic, simply because  $G$ acts trivially on  ${\Aak}$. Now the argument of \cite[Lemma~6.1]{AJM}
 shows that any factor of $\eM\otimes_{\Aak}{\Aak}^{(u)}$, and in particular $\eM\otimes_{\Aak}L$, is monodromic. 
 The same argument works for $\Ext^j_\ddd(\eM,\ddd)$ since it is    monodromic  by Lemma~\ref{lem:Gequivisfullsub}.
 
 (2) The fact that a $G$-stable submodule $N$  of $M$ is  monodromic is immediate from the definition. 
 The proof  for factors of $N$ then follows as in (1). 
\end{proof} 

\subsection*{Admissible modules}
For the rest of the section we will be interested in the categories of admissible   $\ddd$-modules $\eC$  and $\eCop$,  as defined in  Definition~\ref{defn:admissible}, and their relationship to ${\Aak}$-modules. Key to understanding these modules is to understand $\delta$-torsion more generally, where $\delta$ is the discriminant defined  at the beginning of  Section~\ref{Sec:polarreps}.  Thus, we start with   general results on $\delta$-torsion, for which the  following   localisation result will be useful. 

\begin{lemma}\label{Brown-Levasseur} {\rm (}See \cite[Proposition~1.6]{BrL}.)
	Let $X\subseteq Y$ be rings such that $Y$ is flat as both a left and a right $X$-module and let $N$ be a finitely generated left $X$-module. Then
	\[ \Ext^j_Y(Y\otimes_XN,Y)\cong \Ext^j_X(N,\, X)\otimes_XY\]
	as right $Y$-modules. If $N$ is an $X$-bimodule, this
	is an isomorphism of $(X,Y)$-bimodules.
	
	In particular, if $\eS$ is a two-sided Ore set of regular elements in $X$ then
	\[  \Ext^j_{X_{\eS}} ({}_{\eS}N,\, X_{\eS})\cong \Ext^j_X(N,\, X)\otimes_{X}X_{\eS}\]
	as right $X_{\eS}$-modules.  If $N$ is an $X$-bimodule, this
	is an isomorphism of $X_{\eS}$-bimodules. \qed
\end{lemma}

We will need the following  consequence of Proposition~\ref{torsion-homs}.     We recall that $n$ and $m$ are defined in Definition~\ref{notation4.11}, and we set 
\begin{equation}\label{M'-defn}  \eM'=\Ext^m_\ddd(\eM,\ddd).
\end{equation}

\begin{lemma}\label{torsion-consequence}  Assume that $\rad_{\vs}$ is surjective. 
	Pick  $0\not= d \in \C[V]^G$, or $d\in \Sym(V)^G$. Then, the following hold.
	\begin{enumerate}
		\item If $L$ is a $d$-torsion left $\ddd$-module then each of $\Hom_\ddd(\eMt, L)$,   $\Hom_\ddd(\eM, L)$ and $\Hom_\ddd(\eM',L)$  is  $d$-torsion.
		\item If $L\in \eC $, then $ \Hom_\ddd(\eM,L)   \in \Osph$.
		\item  The left $R$-module  structure of $\Hom_\ddd(\eM,L) $  is induced from its ${\Aak}$-module structure. Moreover,  
		$\Hom_\ddd(\eM,L) $ is a left  $R$-module submodule of  $L^G\cong \Hom_\ddd(\eMt,L).$
		
		\item Let $L\in \eCop$. If $Z'\subseteq \Hom_\ddd(\eM',L)$ is a finitely generated ${\Aak}$-submodule, then $Z' \in \Osphop$.
	\end{enumerate}
\end{lemma}

\begin{remark}\label{rem:torsion-consequence} The point of Part~(4) of the lemma is that there is no obvious reason why $ \Hom_\ddd(\eM',L)$ should be a finitely generated ${\Aak}$-module.
\end{remark}

\begin{proof}  (1) By Lemma~\ref{lem:Gequivisfullsub},  $d$ acts locally ad-nilpotently on both $\eMt$ and $\eM'$ . Thus the result follows from 
	Proposition~\ref{torsion-homs}.
	
	(2) By (1) it remains  to prove that $\Hom_\ddd(\eM,L)$ is a finitely generated left ${\Aak}$-module. To begin with, consider 
	$\Hom_\ddd(\eMt,L)$. 
	
	Lemma~\ref{hom=invariants} implies that  
	$
	\Hom_\ddd(\eMt,L) \ =  L^G.$   Set $N=\ddd L^G\subseteq L$ and note that, as $N^G\supseteq  L^G$ we have $N^G=L^G$, but now  $N=\ddd N^G$. 
	Moreover, as $L$ and hence $N$ is noetherian, $\ddd N^G=\ddd N_0$ for some  finite dimensional subspace $N_0\subseteq N^G$. 
	(We are not excluding the possibility that $L^G=0=N$.)
	
	We claim that $N^G =\ddd^GN_0$. To see this, set $F=\ddd^GN_0$. Then we have a short exact sequence $0\to U\to \ddd\otimes_{\C}F\to \ddd F\to 0$, for some subspace $U$. 
	Now $G$ acts rationally on $\ddd$ and hence on both $\ddd\otimes_{\C}F$ and $\ddd F=\ddd N^G=N$. Thus, as $G$ is reductive, we obtain a short exact sequence 
	\[ 0\too U^G \too (\ddd\otimes_{\C}F)^G \buildrel{\phi}\over{\too} N^G\too 0.\]
	But $(\ddd\otimes_{\C}F)^G = \ddd^G\otimes_{\C}F$  and hence $N^G=\Im(\phi)=\ddd^GF=\ddd^GN_0$. In other words, since $N_0$ is finite-dimensional,
	$$\Hom_\ddd(\eMt,L) = L^G=N^G = \ddd^GN_0$$ is a finitely generated   left  $\ddd^G$-module.  By Lemma~\ref{summand-endo21}(3), 
	the $\ddd^G$-module structure on $\eMt$ and hence on  $\Hom_\ddd(\eMt,L) $ comes via the right $R$-module structure of $\eMt$. Therefore, 
	$\Hom_\ddd(\eMt,L) $ is a finitely generated   left  $R$-module.  
	
	Now  consider $\eM$. Applying $\Hom_\ddd( - , L)$ to the sequence $\eMt \to \eM \to 0$ shows that $\Hom_\ddd(\eM,L)$ is the submodule of the $R$-module $\Hom_\ddd(\eMt,L)$ consisting of all sections killed by the prime ideal $P$ of \eqref{R-defn}. Therefore, the left $R$-module  structure of $\Hom_\ddd(\eM,L) $  is induced from its ${\Aak}$-module structure. Since $R$ is noetherian and $\Hom_\ddd(\eMt,L)$ is finitely generated, it follows that $\Hom_\ddd(\eM,L) $ is a finitely generated ${\Aak}$-module.  
	
	(3)  This was proved in the course of proving (2)
	
	(4)   By Part~(1)  $Z'$ is $d$-torsion, and so the result is immediate.
\end{proof}

We now turn to  the relationship between the category $\eC$ and the category $\Osph$ of ${\Aak}$-modules from Definition~\ref{defn:Osph}. 

\begin{lemma}\label{torsion-consequence2}  Assume that $\rad_{\vs}$ is surjective. 
	Let $T= \eM\otimes_{\Aak}\eP$ for some $\eP\in \Osph$. Then:
	\begin{enumerate}
		
		\item $T\in \eC$ while $\BD_\ddd(T) = \Ext_\ddd^{n+m}(T,\,\ddd) \in \eCop$.
		
		\item If $0\not=d\in (\Sym  \h)^W$ acts locally  nilpotently on $\eP$ then it also acts locally nilpotently on $T$ and $\BD_\ddd(T)$.
	\end{enumerate}
\end{lemma}

\begin{proof}  (1)  We first consider $T$.  Since   $T$ is $(G,\chi)$-monodromic by Lemma~\ref{equivariant tensors}, it remains to check 
	the action of $(\Sym  V)^G$. Set $E=(\Sym  \h)^W$. Write $\eP={\Aak}U$ for some  finite dimensional subspace $U$  and  let $I=\lann_EU$. 
	By definition, $E$ acts locally finitely on $\eP$, and so  $E/I$  is finite dimensional. Pick $0\not=d\in I$.
	Then, once we identify $E$ with $(\Sym V)^G$, the element   $d$ acts locally ad-nilpotently  on ${\Aak}$,  $\ddd$ and $\eM$. Hence $d$ acts  locally nilpotently on  $T$. Consequently,  $(\Sym V)^G$ acts locally finitely on $T$, and so $T\in \eC$.
	
	On the other hand, $\BD_\ddd(T)$ is $(G,\chi)$-monodromic by Lemma~\ref{Ext-equivariance}.  Moreover, 
	$\eS=\{d^j\}$ is an Ore set in  both 	${\Aak}$ and  $\ddd$ and $T_{\eS}=0$ by the last paragraph.  Thus, by  Lemma~\ref{Brown-Levasseur}, 
	\[ \BD_\ddd(T)\otimes_\ddd \ddd_{\eS}  \ = \ \Ext^{n+m}(T,\, \ddd)\otimes_\ddd \ddd_{\eS}  \ = \ \Ext^{n+m}_{\ddd_{\eS}}(T_{\eS},\, \ddd_{\eS} ) \ = \ 0 .\]
	Thus, $\BD_\ddd(T)$ is $d$-torsion and 
	$\BD_\ddd(T)   \in \eCop$.

	(2) The  proof from Part~(1) works here as well.	\end{proof}

 The most important modules in $\eC$ are the following. 
 
\begin{definition}\label{defn:Glambda}   Let $\lambda\in \h^*$  with  corresponding maximal ideal  
$ \mf{m}_{\lambda} \subset (\Sym  \h)^W$, 
	as in Definition~\ref{defn:Osph}. Let   $\eQ_{\lambda} = {\Aak} / {\Aak} \mf{m}_{\lambda}$, as in Equation~\ref{M-definition}, and define
	\[
	\eGt_{\lambda} = \ddd / (\ddd \mf{g}_{\chi} + \ddd \mf{m}_{\lambda}), \qquad\text{and} \qquad \eG_{\lambda} = \eM \otimes_{{\Aak}} \eQ_{\lambda}.
	\]
In particular, $\eGt_0 = \ddd/(\ddd\mf{g}_{\chi}+\ddd (\Sym \h)^W_+)$. The modules $\eGt_\lambda$ are called the \emph{Harish-Chandra} modules since, for $V=\g$, the module $\eGt_0=\eG_0$ is  fundamental to Harish-Chandra's study of $G$-invariant eigendistributions  \cite{HC2, HC3}.  
\end{definition}

\begin{remark}\label{rem:eGpres}
	Assume that $\rad_{\vs}$ is surjective so that $\eM$ is well-defined. It follows from the definition \eqref{M-defn} that $\eM = \ddd/ (\ddd \g_{\chi} + \ddd \Ker \rad_{\vs})$. Therefore, 
	$$
	\eG_{\lambda} = \ddd / (\ddd \mf{g}_{\chi} + \ddd  \Ker \rad_{\vs} + \ddd \mf{m}_{\lambda}). 
	$$
	In particular, $\eGt_{\lambda} = \eG_{\lambda}$ if $\Ker \rad_{\vs} = (\ddd \g_{\chi})^G$. 
\end{remark}

We will show later that $\eG_{\lambda}$ is very well-behaved when ${\Aak}=\Ak(W)$ is simple; see Corollary~\ref{torsionfree-corollary} and Theorem~\ref{G-is-injective} in particular. Here we consider the case where ${\Aak}$ is not simple and show that $\eG_\lambda$ has some less pleasant properties.  We start with an easy consequence of Ginzburg's Generalised Duflo Theorem.

\begin{lemma}\label{lem:deltatorsionsimple}
	Let $\lambda \in \h^*$. If $\Ak(W_{\lambda})$ is not simple then there exists a simple module $L\in \Osph_{\lambda}$ such that $L$ is $\delta$-torsion; 
\end{lemma}

\begin{proof}
 Since $\Ak(W_{\lambda})$ is prime, the fact that  it is   not simple implies that there exists a primitive ideal $\mf{p}_0 \lhd \Ak(W_{\lambda})$ with $\mr{GKdim}(\Ak(W_{\lambda}) / \mf{p}_0) < \GKdim \Ak(W_{\lambda}) $. By the Generalised Duflo Theorem \cite[Theorem~2.3]{Primitive}, there exists a simple module $L_0 \in \Osph_{\kappa,0}(W_{\lambda})$ whose annihilator equals 
  $\mf{p}_0$.  Recall the definition  from \eqref{defn:completion} of  $\widehat{A}_{\kappa}(W_{\lambda})_{0} =\Ak(W_{\lambda})\otimes_E \widehat{E}_0$, where now $E=(\Sym \h)^{W_{\lambda}}$.  Clearly  the action of $\Ak(W_{\lambda})$ on $L_0$ extends to $\widehat{A}_{\kappa}(W_{\lambda})_{0}$,
  and so  the
   (necessarily primitive) annihilator of $L_0$  in  $\widehat{A}_{\kappa}(W_{\lambda})_{0}$ is also non-zero. 
   Thus, via the     equivalence of Lemma~\ref{lem:genOsphparabolicequi}, there exists 
   $L \in \Osph_{\kappa,\lambda}(W)$ whose 
   annihilator $\widehat{\mf{p}}$ in $\widehat{A}_{\kappa}(W)_{\lambda}$ is non-zero. The algebra $\Ak(W)$ is dense in the latter, 
   hence $\mf{p} = \widehat{\mf{p}} \cap \Ak(W)\not=0$. This is the annihilator of the simple $\Ak(W)$-module $L$. Finally, by \cite[Lemma~2.9(1)]{BLNS}, $\Ak(W)[\delta^{-1}] = \dd(\h_{\reg})^W$ is a simple ring, and so    some power of $\delta$ belongs to $\mf{p}$. In other words, $L$ is $\delta$-torsion. 
\end{proof}

Recall the definition of the regular loci $V_{\reg}$ defined at the beginning  of Section~\ref{Sec:polarreps}. 

\begin{definition}\label{minextn-defn}
	If $\euls{L}$ is a left module over $\eD(V_{\reg})$, then the \emph{minimal extension} of $\euls{L}$ to $V$ is, by definition, the unique $\eD(V)$-module $L$ whose restriction to  $V_{\mathrm{reg}}$ equals $\euls{L}$, and such that $L$ has no submodules or quotients supported on the complement $V\smallsetminus V_{\reg}$. 
\end{definition}

\begin{proposition}\label{torsionfree-converse}
	Assume that $\rad_{\vs}$ is surjective. Let $\lambda \in \h^*$ and assume that the algebra $\Ak(W_{\lambda})$ is not simple. Then:
	\begin{enumerate}
			\item both $\eG_{\lambda}$ and $\eQ_{\lambda}$ have a nonzero factor module that is $\delta$-torsion; and
			\item the Harish-Chandra module $\eG_{\lambda} = \eM \o_{\Aak} \eQ_{\lambda}$   is not the minimal extension of $\euls{L}_{\lambda} =\eG_{\lambda} {|_{V_{\reg}}}$.
	\end{enumerate}
\end{proposition}

\begin{remark}\label{rem: tf-converse} As an aside we note that, under the hypotheses of the proposition, $\eG_{\lambda} = \eM \o_{\Aak} \eQ_{\lambda}$ is a quotient of $\eGt_{{\lambda}}=\ddd / (\ddd {\g_\chi} + \ddd \mf{m}_{\lambda})$,   although it need not be a summand. 
\end{remark} 
\begin{proof}
	(1) By Lemma~\ref{lem:deltatorsionsimple} there exists a simple $\delta$-torsion module $0\not=L \in \Osph_{\lambda}$. Since $L={\Aak}\ell$ for some element $\ell $ with $\mf{m}_{\lambda}\ell=0$, the module $\eQ_{\lambda}$ surjects onto~$L$. 
		
	It follows that $\eG_{\lambda}=\eM\otimes_{\Aak}\eQ_{\lambda}$ surjects onto $\eM\o_{\Aak} L$. Moreover, as $\eM$  is a factor of $\ddd$,   $\delta$ still acts locally ad-nilpotently on $\eM$. Since 
	$\delta$ acts locally nilpotently on $L$, it follows  that $\delta$ acts locally nilpotently on $\eM\otimes_{\Aak}L$; thus, 
	$\eM\otimes_{\Aak}L$ is a $\delta$-torsion factor of $\eG_{\lambda}$. As $L$ is a non-zero $\Aak$-module, Lemma~\ref{summand-endo21M}(2) therefore  implies that $\eM\otimes_{\Aak} L\not=0$.
	
	(2) This is immediate from (1) since, by definition, the minimal extension has no quotients or submodules that are $\delta$-torsion.  \end{proof}

This proposition   has the following useful consequence. 

\begin{corollary}\label{torsionfree-converse2} Assume that $\rad_{\vs}$ is surjective. Let $\lambda \in \h^*$ and assume that the algebra $\Ak(W_{\lambda})$ is not simple. Then:
	\begin{enumerate}
		\item there exists a non-split $\Ak(W)$-module surjection $\eQ_{\lambda} \twoheadrightarrow L$ for a $\delta$-torsion $\Ak(W)$-module $L\not=0$; and 
		\item there exist a non-split $\ddd$-module surjection $\eG_{\lambda}\twoheadrightarrow N$ for a $\delta$-torsion $\ddd$-module $N\not=0$.	
	\end{enumerate} 
	In particular, the Harish-Chandra $\eD$-module $\eG_{\lambda} = \eM \o_{\Ak(W)} \eQ_{\lambda}$ is not semisimple. 
\end{corollary}  

\begin{proof}
	Let $\Ak = \Ak(W)$. As in the proof of Proposition~\ref{torsionfree-converse}(1), there exists a surjection $\theta \colon \eQ_{\lambda} \twoheadrightarrow L$  where $0\not=L$ is  $\delta$-torsion. The proof of Proposition~\ref{pdQ}(2) shows that $\eQ_{\lambda}$ is a free left $\C[\h]^W$-module and so $\eQ_{\lambda}$ is necessarily torsionfree as a $\C[\delta]$-module. Thus $\theta$ is nonsplit.

	It remains to prove that $\vartheta=1\otimes\theta: \eG_{\lambda} \twoheadrightarrow N=\eM\otimes_{\Ak} L$ is also nonsplit.  Recall from Lemma~\ref{summand-endo21M}(1) that we have a decomposition   $\eM = \Ak\oplus T$ as $(\ddd^G,\Ak)$-bimodules, 
	where $ T  $ is the sum of all nontrivial $G$-submodules of $\eM$.   Now suppose that $\phi$ is an $\ddd$-module splitting (and hence an $\ddd^G$-module splitting) of $\vartheta$. If $\alpha$ is the projection of 
	$
	\eM \otimes_{\Ak}\eQ_{\lambda} = (\Ak\otimes_{\Ak}\eQ_{\lambda})\oplus (T \otimes_{\Ak} \eQ_{\lambda})
	$
	onto $({\Ak}\otimes_{\Ak}\eQ_{\lambda})$,  and $\iota: L\to  L\oplus T\otimes_{\Ak} L=\eM\otimes_{\Ak} L$ is the natural injection,  then $\alpha\circ \phi\circ \iota$ is  an $\ddd^G$-module splitting of 
	$$
	\theta \colon \eQ_{\lambda}={\Ak}\otimes_{\Ak} \eQ_{\lambda} \ \to \  {\Ak}\otimes_{\Ak} L=L.
	$$
	This contradicts the conclusion of the last paragraph.   
\end{proof}

 \section{An Intertwining Theorem for  Quantum Hamiltonian Reduction}\label{cherednik-intertwining}

 We  continue  to  consider a visible, stable polar representation $V$, as described in Notation~\ref{notation4.11} but in this section,
 \begin{center}
 	 \emph{we always  assume that the spherical  algebra ${\Aak}=\Ak(W)$ is simple.}
 \end{center}
 By Theorem~\ref{thm:radial-exists}(3), this    implies  that $\rad_{\vs}$ is surjective.  
 
Using this assumption, in this section  we prove   Theorem~\ref{intro-intertwining} from the introduction; see Theorem~\ref{intertwining}.  
This is a strong intertwining result that relates the cohomology of   $\Aak$-modules with that of  $\dd(V)$-modules.

Recall   the modules $\eM$ and $\eMt$ from Definition~\ref{M-new-definition}.
We first prove that the simplicity of ${\Aak}$   implies that $\eM$ is  a summand of $\eMt$, thereby strengthening the results from Section~\ref{Sec:HC-bimodule}. We begin with a subsidiary result.

\begin{lemma}\label{semiprime}
	Let $R$ be a noetherian ring   with  a maximal  ideal $P$. Assume that $0\not=d\in R$ is such that 
	$[d+P]$ is regular in $R/P$ and that $d$ acts locally ad-nilpotently on $R$. Finally, if $\eS=\{d^j : j\geq 0\}$ assume that 
	$P$ is $\eS$-torsion.  Then
	
	\begin{enumerate}
		\item $P$ is a minimal prime ideal of $R$ and $R=R/P\oplus R/J$ for some ideal $J$. Moreover, 
		$R/J$ is (left and right) $\eS$-torsion.
		\item Any right $R$-module $M$ decomposes as $M\cong M/MP\oplus M/MJ$.
	\end{enumerate}
\end{lemma}

\begin{proof} (1)  Note that  left and right $\eS$-torsion are the same by Lemma~\ref{lem:lefttfree}. 
	The following standard sublemma will be used frequently.
	
	\begin{sublemma}\label{sub-semiprime}
		Keep the hypotheses of the lemma and suppose that $N$ is an $R$-bimodule that is finitely generated as a left $R$-module. 
		If $N$ is $\eS$-torsion on the right,  then $Nd^j=0$ for some $j$.  In particular, if  $NP=0$, then $N=0$.  
	\end{sublemma}
	
	\begin{proof} Write $N=\sum_{i=1}^r Rn_i$. Then $ \rann(N)=\bigcap_{i=1}^r \rann(n_i)$. Since $N$ is right $d$-torsion there therefore exists 
		some $d^j$ with $n_id^j=0$ for all $i$ and hence $Nd^j=0$. 
		
		If $NP=0$, then  the ideal $\rann(N)\supsetneq  P$. As $P$ is maximal this implies that $\rann(N)=R$ and $N=0$. 
	\end{proof}
	
	We return to the proof of the lemma. First  suppose that $P\supsetneq P'$ for some other prime ideal $P'$. By the sublemma 
	$P(Rd^j)=Pd^j=0\subseteq P'$ for some $j\geq 1$ which, since  $P\not\subseteq P'$, forces  $d^j\in P'$. This contradicts the fact that $d$ is regular modulo $P$. 
	Hence   $P$ is a minimal prime ideal. 
	
	Write $I$ for the intersection of   minimal prime ideals of $R$, other than $P$. 
	For some $j\geq 1$ set   $N=(I^j\cap P)/(I^j\cap P)P$. Then $N$   is a subfactor of $P$ and hence $N_\eS=0$. Since $NP=0$  the sublemma implies 
	that $N=0$ and $(I^j\cap P)P=I^j\cap P$.  In particular, $IP\supseteq (I\cap P)P=I\cap P$ and so $IP=I\cap P$. By symmetry, $IP=I\cap P= PI$.
	An application of the sublemma to $N=P/P^2$ shows that $P^2=P$.
	Thus,  by induction,  \[(I\cap P)^k=(I\cap P)^{k-1}IP = (I^{k-1}P^{k-1})IP =I^kP^k=I^kP \quad \text{for all } k\geq 1.\]
	
	As $I\cap P=N(R)$, the nilradical of $R$,  there then exists  $k\geq 1$ such that 
	$0=(I\cap P)^k = I^k P$.  We will show that the lemma holds for $J=I^k$. Indeed,  the identity 
	$(I^k\cap P)P=I^k\cap P$ implies that $0=I^kP\supseteq (I^k\cap P)P=I^k\cap P$.	
	On the other hand, as $P$ is maximal and $I^k\not\subseteq P$ certainly $I^k+P=R$. Thus $R=I^k\oplus P=R/P\oplus R/I^k$.
	Finally, since $R/I^k\cong P$, certainly it is $\eS$-torsion.
	
	(2) Obvious from (1).
\end{proof}

We next show that Lemma~\ref{semiprime} does apply in our situation. 
Recall from \eqref{eq:discriminant} that $\delta$ denotes the discriminant, so that 
$\h_{\reg} = (\delta\not=0)\subset \h $ and $V_{\reg}=( \delta\not=0)\subset V$.

\begin{proposition}\label{semiprime2} Assume that   ${\Aak}$ is a simple algebra and define $R$, $P$,  $\eMt$ and $\eM$ as in  Definition~\ref{M-new-definition}.	Then
	\begin{enumerate}
		\item  $\{R,P\}$ satisfy the hypotheses of Lemma~\ref{semiprime}, for  $d=\delta$   the discriminant.
		\item  As an $(\ddd,R)$-bimodule, 
		$$\eMt\ = \ \eMt/\eMt P\oplus \eMt/\eMt J \  \cong \  \eMt/\eMt P\oplus \eMt P.$$
		 \item   $\eM= \eMt/\eMt P$ is an $(\ddd,{\Aak})$-bimodule
		that   is  projective in $(G,\chi, \ddd)\lmod$. On the other hand,  $\eMt/\eMt J$ is $\delta$-torsion on both left and right.
	\end{enumerate}
\end{proposition}

\begin{proof}   
	(1)  Recall from Theorem~\ref{thm:radial-exists}
		  that $\rad_{\vs}$ restricts to give an isomorphism 
	\[ 
	Y\ := \ \left(\frac{ \eD(V_{\reg} )}{\eD(V_{\reg}) \g_{\chi}} \right)^G\ \cong \ \eD(\h_{\reg})^W . 
	\]
	Since  $Y=\C[V_{\reg}]^G\otimes_{\C[V]^G} R  = R_{\eS} $, for  $\eS=\{\delta^r: r\geq 1\}$, this implies that  $R_{\eS}$ is a domain.
	
	As $R/P=\Aak\hookrightarrow \eD(\h_{\reg})^W = R_{\eS}$, it also  follows that $P_{\eS}=0$.
Since $R/P   = \Aak$ is a domain, 
  $[\delta+P]$ is automatically 
	regular in $R/P$.  
	Thus all the hypotheses of Lemma~\ref{semiprime} are satisfied.

	(2,3) The first two sentences follow immediately from    (1) combined with Lemmata~\ref{semiprime} and \ref{lem:Gequivisfullsub}. 
	The fact that $\eMt/\eMt J$ is $\delta$-torsion on the  right follows   from Lemma~\ref{semiprime}. 
	It is then $\delta$-torsion on the left by Lemma~\ref{lem:lefttfree}.
\end{proof}

\begin{notation}\label{eG-torsion} Note that Proposition~\ref{semiprime2} implies that 
$
\eGt_{\lambda} = \eG_{\lambda} \oplus \eGt_{\lambda,\mr{tor}},
$
where  $\eGt_{\lambda,\mr{tor}}$  is a $\delta$-torsion summand of $\eGt_{\lambda}$. 
\end{notation}

\begin{remark}\label{rem:semiprime2} The following observation will be  used frequently without further comment. As ${\Aak}=R/P$, we can, and will, 
	regard any left ${\Aak}$-module $L$ as a left $R$-module. Moreover as $P=P^2$ by Lemma~\ref{semiprime}, 
	$\eMt P\otimes_R L = \eMt P\otimes_R R/P\otimes_{R/P}L = 0$.  Thus, by Proposition~\ref{semiprime2},
	\[\eMt\otimes_R L = (\eM\otimes_RL) \oplus (\eMt P\otimes_R L) = \eM\otimes_R L=\eM\otimes_{\Aak} L.\]
\end{remark}

\subsection*{The intertwining theorem} 

Recall   that we are always assuming that ${\Aak}=\Ak(W)$ is a simple ring. 
The  aim of   this subsection is show that the intertwining theorem from Corollary~\ref{maincorollary} applies to the present context. The first main step in the proof is to show  that $\eM$ is CM as a left $\ddd$-module, which we achieve by mimicking the proof of \cite[Corollary~2.3]{LS}.

We start with a definition, where we recall from
 Notation~\ref {notation4.11}  that $n=\dim \h$ and $m=\dim V-n$.   

\begin{definition} \label{star}
	Define a nonzero $(\ddd,\, {\Aak})$-bimodule $L$ to 
	satisfy property ~\textrm{(*)} if $L$ is a finitely generated
	left $\ddd$-module with
	\[\GKdim_\ddd L \leq \dim \h+\dim V=2n+m.\]
\end{definition}

The \emph{Krull dimension}, \label{defn:Krull}  in the sense of Rentschler and
Gabriel, of a module $L$ over a ring $R$ will be denoted by $\Kdim_R L$.
The nonzero module $L$ will be called {\it GK-homogeneous}\label{defn:homog} (respectively {\it
	Krull-homogeneous}) if $\GKdim L'=\GKdim L$ (respectively $\Kdim L'=\Kdim
L$) for all nonzero submodules $L'$ of $L$.   
More details about these concepts can be found in \cite[Chapters 6 and  8]{MR}.

\begin{lemma} \label{homog}   
	Let $L$ be an $(\ddd,\,{\Aak})$-bimodule that  
	satisfies property \textrm{(*)}. 
	Then, $$\GKdim_{\ddd} L= \dim \h +\dim V=\frac{1}{2}\bigl(\GKdim(\ddd)+\GKdim({\Aak})\bigr)$$ and
	$\Kdim_\ddd L= n =\Kdim {\Aak}$.  
	The analogous result holds for $({\Aak},\ddd)$-bimodules.
\end{lemma}

\begin{proof}
	Let $\Cdim T $ denote the maximal Krull dimension of a commutative, finitely generated subring of a ring $T$. Clearly, $\Cdim {\Aak}\geq \dim \h$ since $\C[\h]^W \subset \Aak$. 
	As ${\Aak}$ is  simple, the map ${\Aak}\to \text{End}_{\ddd}(L)$ induced by the right action of ${\Aak}$ on $L$ is an injection.
	Thus, by  \cite[Proposition~1.2]{Jo}, 
	$$ 
	\dim \h \leq \Cdim {\Aak} \leq \Cdim \text{End}_{\ddd}(L) \leq \Kdim L.
	$$
	Now, by \cite[Corollary~8.5.6]{MR} and (*), $\Kdim L  \leq \GKdim L-\dim V \leq \dim \h,$ as required.
\end{proof}

\begin{corollary}\label{maincor}   
	Let $L$ be an $(\ddd,{\Aak})$-bimodule satisfying property \textrm{(*)}. Then:
	\begin{enumerate}
	    \item As a left $\ddd$-module, 
	$L$ is Krull and GK-homogeneous. Moreover, $L$ has finite length as 
	an  $(\ddd,\, {\Aak})$-bimodule.
	\item $L$ is an $m$-CM  left $\ddd$-module,  with  $\pd_{\ddd} L = m.$
	\end{enumerate}
\end{corollary}

\begin{proof} (1) If $L$ is not GK-homogeneous, write $T$ for the unique
	largest $\ddd$-sub\-module 
	of $L$ with $\GKdim T<\GKdim L$. 
	Since $T$ is mapped to itself by any $\ddd$-endomorphism of $L$, 
	$T$ is an
	$(\ddd,\, {\Aak})$-bisubmodule of $L$. Thus, $T$ satisfies (*) and so, by Lemma~\ref{homog}, 
	$\GKdim T=\dim \h+\dim V=\GKdim L$, a contradiction. 
	Hence, $L$ is GK-homogeneous and, similarly, $L$ is Krull-homogeneous.
	
	Next, let $L=L_0\supsetneq L_1 \supsetneq \cdots$. Then each $L_i/L_{i+1}$ sa\-tis\-fies
	(*) and so Lemma~\ref{homog} implies that
	$$\Kdim L_i/L_{i+1}=\dim \h = \Kdim L \; \  \text{for each $i$.}$$
	By the definition of  Krull dimension, this forces the chain to have
	finite length. 
	
	(2)  
	By \cite[Theorem~II.5.15]{Bj} and the Cohen-Macaulay property,  $\grade(L)$  satisfies
	$$
	\grade(L)= \GKdim \ddd-\GKdim_{\ddd}L=m,
	$$ 
	and so $\Ext^j_\ddd(L,\ddd)=0$ for $j<m$.  Also $\Ext^{\pd L}_{\ddd}(L,\,\ddd)\not= 0$.  Thus, if either
	assertion of Part~(2) of the corollary is false, there exists an
	integer $s>m$ such that $\Ext^{s}_{\ddd}(L,\,
	\ddd)\not= 0$.  However, $\Ext^{s}_{\ddd}(L,\,
	\ddd)$ is naturally a 
	$({\Aak},\,\ddd)$-bimodule, finitely generated as a right
	$\ddd$-module and, by \cite[Proposition~II.5.16]{Bj},
	$$\GKdim_{\ddd} \Ext^{s}_{\ddd}(L,\,
	\ddd)\ \leq \  \GKdim \ddd -s \  < \  m  \ =  \ \dim V+\dim \h.$$ 
	This contradicts Lemma~\ref{homog}.
\end{proof}

It is now immediate that $\eM$ is CM. 

\begin{corollary} \label{m-CM}     
	As a  left $\ddd$-module,  $\eM$ is $m$-CM for  $m = \dim V - \dim \h$.   
\end{corollary}

\begin{proof}   
	By  Lemmata~\ref{lem:Gequivisfullsub} and~\ref{lem:charvar},  $\eM$ satisfies \textrm{(*)}. Now, apply Corollary~\ref{maincor}.   
\end{proof}

 Before proving the intertwining result, we need  a better understanding of various (co)homology modules.

\begin{lemma}\label{visible2}      
	Pick $0\not=d\in (\Sym \h)^W_+$. The  $(\ddd,{\Aak})$-bimodule $\eM$, as well as    the $({\Aak},\ddd)$-bimodule 
	$\eM'=\Ext^m_\ddd(\eM,\ddd)$ are $d$-torsion\-free on both sides. 
\end{lemma}

\begin{proof}
	By Theorem~\ref{thm:radial-exists}(2), 
	we identify $(\Sym \h)^W_+ \cong (\Sym  V)^G_+ $.   
	Clearly, $\eM'$ is   a finitely generated right $\ddd$-module, while the induced action of $d$ on  both $\eM$  and $\eM'$ is   locally ad-nilpotent by Lemma~\ref{lem:Gequivisfullsub}. Since  ${\Aak}$ is a simple ring, the lemma
	now  follows from Corollary~\ref{lem:leftrightdfreepolar} and its left-right analogue.
\end{proof}

\begin{proposition}\label{admissible}   
	Let  $L \in \euls{O}^{sph}$ and  $j\in \mathbb{N}$.
	\begin{enumerate}
		\item   $\Tor_j^{\Aak}(\eM,L)$ is a $(G, \chi)$-monodromic $\ddd$-module on which  $(\Sym  V)^G$ acts locally
		finitely. Hence, $\Tor_j^{\Aak}(\eM,L)$ is either zero or a holonomic $\ddd$-module.  
		
		\item If  $L \in \euls{O}^{sph}_0$, then $(\Sym  V)^G_+$ acts locally
		nilpotently on $\Tor_j^{\Aak}(\eM,L)$.
	\end{enumerate}
\end{proposition}

\begin{proof}  (1) By Lemma~\ref{equivariant tensors}(1), a left ${\Aak}$-module $K$ gives rise to a 
	$(G, \chi)$-monodromic $\ddd$-module $\eM \otimes_{\Aak} K$. Moreover, if $\phi : K_1 \rightarrow K_2$ is an 
	${\Aak}$-module map then $1 \otimes \phi  : \eM \otimes_{\Aak}K_1 \rightarrow \eM \otimes_{\Aak} K_2$ is an equivariant map; 
	thus   its kernel, image,  and so on, are also  $(G, \chi)$-monodromic $\ddd$-modules by Lemma~\ref{equivariant tensors}(2).
	Now,
	$$
	\Tor_j^{\Aak}(\eM,L) = \Ker (1 \otimes \dalpha_j ) /  \Im (1 \otimes \dalpha_{j+1})
	$$
	where $\dalpha_j : P_j \rightarrow P_{j-1}$ is part of a projective resolution of ${}_{\Aak}L$. It follows that $\Tor_j^{\Aak}(\eM,L)$ is 
	$(G, \chi)$-monodromic. 
	
	Set $E=(\Sym  \h)^W$, which we identify with $(\Sym  V)^G$.
	Let $L={\Aak}L_0$, for some finite dimensional $E$-submodule $L_0$ and set $I=\lann_EL_0$. Since $E$ acts locally finitely on $L$,
	$E/I$ is finite dimensional. Similarly, as $E$ acts locally ad-nilpotently on ${\Aak}$, each $d\in I$ acts locally nilpotently on $L$.

	Take $0\not=d\in I$
	and set $\eS=\{d^k\}$.  By Lemma~\ref{visible2}, $\eM$ is $d$-torsionfree on both sides and hence, 
	by Lemma~\ref{ore91}, the two localisations ${}_{\eS}\eM$ and $\eM_{\eS}$ are isomorphic.
	In particular, $ {}_{\eS}\eM\otimes_{\Aak}N \cong \eM_{\eS}\otimes_{\Aak}N
	= \eM_{\eS}\otimes_{{\Aak}_{\eS}}{}_{\eS}N$ for any left ${\Aak}$-module $N$. 
	Now follow the usual commutative proof, say of \cite[Proposition~7.17]{Rot}, to conclude  that 
	$$
	\ddd_{\eS}\otimes_\ddd\Tor_j^{\Aak}(\eM,L) \cong \Tor_j^{{(\Aak)}_{\eS}}(\eM_{\eS}, {}_{\eS} L) = 0, 
	$$
	where the final equality follows from the fact that  ${}_{\eS}L=0$.  
	Hence $d$ acts locally nilpotently on $\Tor_j^{\Aak}(\eM,L)$ and  $(\Sym  V)^G$ acts locally finitely on $\Tor_j^{\Aak}(\eM,L)$.

	Finally,   Lemma~\ref{lem:charvar}(2)  implies that  $ \Tor_j^{\Aak}(\eM,L)$ is holonomic (or zero).  
	
	(2) The above proof, \emph{mutatis mutandis},  shows that each $d\in (\Sym  \h)^W_+$ acts locally nilpotently on $\Tor_j^{\Aak}(\eM,L)$.
\end{proof}

\begin{corollary}\label{vanishing lemma}   
	Let  $L \in \Osph$. Then:
	\begin{enumerate}
		\item  $ \eM\otimes_{\Aak}L$ is a monodromic   $\ddd$-module on which  $(\Sym  V)^G$ acts locally
		finitely. 
		\item $\Ext^i_\ddd(\Tor_j^{\Aak}(\eM,L),\ddd) = 0$ for all $i \neq \dim V= n + m$.  
	\end{enumerate}
\end{corollary}

\begin{proof} (1) is immediate from  Proposition~\ref{admissible}.
	
	(2) Since $\Tor_j^{\Aak}(\eM,L)$ is a left $\eD(V)$-module, it suffices to show that it is a holonomic $\eD(V)$-module (or zero). This follows from  the proposition.   
\end{proof}

Combining the  above  results   we can prove the main result of this section. This requires the following minor refinement of 
Notation~\ref{main-notation}.

\begin{notation}\label{main-notation2} Keep the notation from Definition~\ref{M-new-definition}
and set $ \eM'=\Ext^m_\ddd(\eM,\, \ddd)$. Write
	\[\HHleft(J)=\eM\otimes_{\Aak}J \qquad\text{and}\qquad \HHright(L) = L\otimes_{\Aak}\eM'\] 
	for a left ${\Aak}$-module $J$ and right ${\Aak}$-module $L$. Similarly, write
	\[   \BD_\ddd(J_1)=\Ext^{n+m}_\ddd(J_1,\, \ddd) \qquad\text{and}\qquad \BD_{\Aak}(L_1) = \Ext_{\Aak}^n(L_1,\, {\Aak})\]
	for a left $\ddd$-module $J_1$ and left ${\Aak}$-module $L_1$.  
\end{notation}

\begin{theorem}\label{intertwining}
	Assume that   ${\Aak}$ is a simple algebra.
	Let $N={\Aak}\otimes_{(\Sym  \h)^W}L$ for some finite dimensional $(\Sym  \h)^W$-module $L$, or let 
	$N$ be a projective object in $\Osph$. Then
	\[ \left( \BD_\ddd\circ{}^\perp\BH\right) (N) \ \cong \ \left(\BH^\perp\circ \BD_{\Aak}\right)(N). \]
\end{theorem}

\begin{proof}
	The result will follow from Corollary~\ref{maincorollary} once we check that the assumptions of  Hypotheses~\ref{notation-vanishing}  for the given rings and modules. By Notation~\ref {notation4.11} and Lemma~\ref{CM for spherical},  $\ddd$ and ${\Aak}$ satisfy the requisite conditions, while $N$ satisfies \eqref{vanishing0} by Proposition~\ref{pdQ}(2,3). By Corollary~\ref{m-CM}, \eqref{vanishing0.5} holds while, by  Corollary~\ref{vanishing lemma},  \eqref{vanishing1}  holds.
\end{proof}

One  reason why Theorem~\ref{intertwining} will be important in the later parts of this paper is that
one often finds that particular properties  of  classes of $\dd(V)$-modules are easy to prove for submodules, 
 but are needed 
for factor modules (or vice versa). The fact that the functors   $\mathbb{D}$ are contravariant allows one to switch 
between the two cases.

\section{The  Action Of The Discriminant On Harish-Chandra Modules}\label{sec:simplicityhypothesis}

We  continue  to  consider a visible, stable polar representation $V$ as in Notation~\ref{notation4.11},
   and  \emph{will always  assume that the spherical  algebra ${\Aak}=\Ak(W)$ is simple, and hence that  $\rad_{\vs}$ is surjective.}  
   
 \medskip
 
In this section  we prove  Theorem~\ref{intro-torsionfree} from the introduction, which is one of the major results of the paper. This shows that  the discriminant $\delta$ 
acts torsion-freely on the 
Harish-Chandra modules  $\eG_\lambda = \eM\otimes_{\Aak} ({\Aak}/{\Aak}\mf{m}_\lambda)$, as defined in Definition~\ref{defn:Glambda}. Moreover, these modules  have no  nonzero $\delta$-torsion factor module.     This, in turn, allows us  to prove strong results on the structure  of these modules; 
for example  it follows that  $\eG_{\lambda}$  is the   minimal extension of its restriction  $\eG_\lambda |_{V_{\reg}}$ to the regular locus; see Corollary~\ref{torsionfree-corollary}.

We first need to show that the passage  from $\eMt$ to $\eM$ does not affect the basic properties of that module.  

\begin{lemma}\label{summand-endo}   
	Let  $M$ be a left $\ddd$-module and set  $B=\End_\ddd(M)$. Assume that $B=Be\oplus B(1-e)$ for a central idempotent $e$. Then $\End_\ddd(Me)=Be$.
\end{lemma}

\begin{proof}  This is  routine exercise  is left to the reader.
\end{proof}

The next result  give some useful observations about   $\eM$ that are analogous to the results proved about $\eMt$ in Lemma~\ref{summand-endo21}.

\begin{lemma}\label{summand-endo22}    
	\begin{enumerate}
		\item $\Hom_\ddd(\eM,\,-)$ is an exact functor on $(G,\chi,\ddd)\lmod$.
		\item  
		$\Hom_\ddd(\eM,\eM\otimes_{\Aak}L)=L$ for any left ${\Aak}$-module $L$.  
		\item  In particular, $\End_\ddd(\eM)={\Aak}$.
	\end{enumerate}
\end{lemma}

\begin{proof} 	(1) By Proposition~\ref{semiprime2}, $\Hom_\ddd(\eM,-)$ is a summand of  $\Hom_\ddd(\eMt,-)$ and hence is exact by     Lemma~\ref{hom=invariants}.
	
	(2)   By Lemma~\ref{summand-endo21M}, $\eM\otimes_{\Aak}L\not=0$. 
	 As in Remark~\ref{rem:semiprime2}, we regard $L$ as a left  $R$-module via the identification ${\Aak}=R/P$ and then 
	$ \eMt\otimes_RL =   \eM\otimes_{\Aak}L.$
	Therefore, by Lemma~\ref{summand-endo21}(4),
	\begin{equation}\label{summand-endo-equ}\begin{aligned}
	L\ = \ & \Hom_\ddd(\eMt,\eMt\otimes_RL) \ = \ \Hom_\ddd(\eM\oplus \eMt P , \eM \otimes_{\Aak}L) \\
	\ =   \ &  \Hom_\ddd(\eM  , \eM  \otimes_{\Aak}L) \oplus  \Hom_\ddd( \eMt P , \eM \otimes_{\Aak}L  ) \\
	\end{aligned}\end{equation}
	
	The  isomorphism  $L\cong \Hom_\ddd(\eMt, \eMt\otimes_RL)$ is given by mapping 
	$\ell\in L$ to the morphism $\theta=\theta_\ell: [1+(\eD(V)\g_{\chi})^G] \mapsto 1 \otimes \ell$. 
	This is well-defined since  $1 \otimes \ell$ is $G$-invariant and hence 
	$\tau(x)\cdot(1 \o \ell)= \chi(x) (1 \o \ell)$
	since $\eM\otimes_{\Aak}L$ is monodromic. 
	However, this implies that $\theta_{\ell} (\eMt P) = \eMt P\otimes_R \ell =0$ as $\eMt P\otimes_RL=0$. Therefore, in \eqref{summand-endo-equ}, 
	$\theta_\ell\in \Hom_\ddd(\eM  , \eM  \otimes_{\Aak}L) $ and hence $L= \Hom_\ddd(\eM  , \eM  \otimes_{\Aak}L)$.
	
	(3) It follows from Definition~\ref{M-new-definition} that $R = \End_{\ddd}(\eMt)$. Therefore, (3) follows from Proposition~\ref{semiprime2} and Lemma~\ref{summand-endo}.  
\end{proof}

As an easy consequence we obtain:

\begin{corollary} \label{summand-endo3}  Assume that   ${\Aak}$ is a simple algebra. Define a functor 
$$\BH \colon (G,\chi,\ddd)\lmod \to {\Aak}\lmod$$  by $
	\BH(X)= \Hom_\ddd(\eM,X). $ Then $\BH$ is an  exact functor  with left adjoint $$\HHleft \colon {\Aak}\lmod \to (G,\chi,\ddd)\lmod$$ given 	 by $\HHleft(L) = \eM \otimes_{\Aak} L$.   
	The  adjunction 
	$
	\mr{Id} \to \BH \circ \HHleft
	$
	is an isomorphism.   
\end{corollary}

\begin{proof}  As observed in the proof of Lemma~\ref{summand-endo21M}, given  any   left ${\Aak}$-module $L\not=0$ we have 
	$\eM\otimes L = L\oplus L'$ where $L'$ is a sum of  non-trivial $G$-module summands of $\eM\otimes L$. Thus $L=(\eM\otimes_{\Aak}L)^G =\BH \circ \HHleft(L)$.
	Now use  Lemma~\ref{summand-endo22}(1).
\end{proof}


 We are now ready to prove the main result of this section.
  
\begin{theorem}\label{torsionfree}
	  Assume that   ${\Aak}$ is a simple algebra. Let $\eG =\eG_\lambda$ or, more generally, let $\eG=\eM\otimes_{\Aak}\eP$ for some projective object $\eP\in \Osph$. Pick  $0\not= d \in \C[V]^G\cong \C[\h]^W$. Then:
	\begin{enumerate}
		\item  $\BD_\ddd(\eG)= \Ext^{n+m}_\ddd(\eG,\,\ddd)$ has no nonzero $d$-torsion factor module;
		\item  $\eG$ is $d$-torsionfree.
		\item  $\eG$ has no nonzero $d$-torsion factor module.
		\item   The module $\eGt_{\lambda,\mr{tor}}$ defined in \eqref{eG-torsion}  is the 
	full 	$\delta$-torsion submodule of $\eGt_{\lambda}$.
	\end{enumerate}
\end{theorem}

\begin{remark} \label{LS-C5}  Suppose that $V$  is a symmetric  space that satisfies  Sekiguchi's niceness condition (or more generally a \gainly \ symmetric space); see Definition~\ref{nice-space}.  Then Theorem~\ref{torsionfree}  
	completely answers \cite[Conjecture~C.5]{LS3}.  The relevant definitions and further comments will be given in Section~\ref{Sec:examples}.
\end{remark}

\begin{proof} (1) We use the functors defined in   Notation~\ref{main-notation2}. Suppose that  there is a nonzero morphism $\phi: \BD_\ddd(\eG)  \to L$, for some $d$-torsion module $L$. Then 
	\[\begin{array}{rll}
	0\ \not= \ \phi & \in\ \Hom_\ddd(\BD_\ddd(\eG),\, L) \\ \noalign{\vskip 5pt} 
	& =\ \Hom_\ddd\bigl(\BD_\ddd\circ {}^\perp\BH(\eP),\, L\bigr) &\quad \text{by definition} \\  \noalign{\vskip 5pt} 
	& = \ \Hom_\ddd\bigl(  \BH^\perp \circ\BD_{\Aak}(\eP),\, L\bigr) &\quad \text{by Theorem~\ref{intertwining} } \\  \noalign{\vskip 5pt} 
	& = \ \Hom_\ddd\bigl(\DD_{\Aak}(\eP)\otimes_{\Aak} \eM' ,\, L\bigr) &\quad \text{by definition} \\  \noalign{\vskip 5pt} 
	& =\ \Hom_{\Aak}\bigl(\BD_{\Aak}(\eP),\, \Hom_\ddd(\eM', L)\bigr) &\quad \text{by adjunction.} \\  \noalign{\vskip 5pt} 
	\end{array}\]
	In particular, $Z:= \Hom_\ddd(\eM', L) \not=0.$ By Lemma~\ref{torsion-consequence}(4), every finitely generated submodule $Z'\subseteq Z$ satisfies $Z'\in \Osphop$. Moreover, by Proposition~\ref{torsion-homs}, $Z'$ is $d$-torsion. Thus, by Lemma~\ref{CM for spherical2}(3), $Z'=0$. Hence $Z=0$, a contradiction.

	(2) As usual, set $\eS=\{d^r : r\in \NN\} $ and assume that $\eG$ has a nonzero $d$-torsion submodule $T$.
	As $\gldim{\ddd}=\dim V=n+m$, we obtain a surjection  
	$\BD_\ddd(\eG) \twoheadrightarrow \BD_\ddd(T) $.  On the other hand,  
	$\eG$ and hence  $T$ are holonomic left $\ddd$-modules by Proposition~\ref{admissible}  and hence 
	both $\eG$ and $T$ are $(n+m)$-CM. Thus $ \BD_\ddd(T)\not=0$.   
	
	Finally,  by Lemma~\ref{Brown-Levasseur}, 
	$\Ext^{n+m}_{\ddd_{\eS}}(T_{\eS},\, \ddd_{\eS} ) \ = \ \Ext^{n+m}(T,\, \ddd)\otimes_\ddd \ddd_{\eS}.$
	Since $T_{\eS}=0$ by hypothesis, this implies that $\BD_\ddd(T)\otimes_\ddd \ddd_{\eS}=0$ and hence 
	that  $\BD_\ddd(T)$ is a nonzero $d$-torsion factor of $\BD_\ddd(\eG$). This contradicts Part~(1).
	
	(3) Now suppose that there is a nonzero surjection  $\phi: \eG \twoheadrightarrow T$ for a  $d$-torsion module $T$. 
	Applying the exact functor $\BH=\Hom_\ddd(\eM,-)$   
	to	  $\phi$ gives a surjection $\BH(\eG) \twoheadrightarrow \BH(T)$. The adjunction
	$\mr{Id} \to \BH \circ \HHleft$   from  Corollary~\ref{summand-endo3} implies that 
	$\BH(\eG) = \eP$ and hence that $\BH(T)\in \Osph$. Moreover, by adjunction, 
	\[	\phi \in \Hom_\ddd(\eG,T) = \Hom_\ddd(\eP, \, \BH(T)   ),\]
	and so $\BH(T) \neq 0$. By Lemma~\ref{CM for spherical2}(3),   $\BH(T)$ is also $\delta$-torsionfree. This contradicts 
	Lemma~\ref{torsion-consequence}(1) which says that $\Hom_\ddd(\eM,T)$ is $\delta$-torsion.   
	
	(4) This is immediate from Part~(2) and the discussion in Definition~\ref{defn:Glambda}.  
\end{proof}

The following result clarifies which modules Theorem~\ref{torsionfree} applies to. 

\begin{lemma}\label{Glasgow-cor}
	
	Assume that   ${\Aak}$ is a simple algebra and let $J$ be an ideal of  finite codimension $k$ in $(\Sym  \h)^W$  for which  $\sqrt{J}$ 
	is a maximal ideal $\mathfrak{m}_\lambda$.  Then $\Psi = \eM\otimes_{\Aak}({\Aak}/{\Aak}J)$ decomposes as $\Psi \cong \eG_\lambda ^{\oplus k}$.
\end{lemma}

\begin{proof}
	Let $U = (\Sym  \h)^W/J$, a space of dimension $k$, and let $\C_{\lambda} = (\Sym  \h)^W / \mf{m}_{\lambda}$. Note that $\eQ_{\lambda} = \Ak \o_{(\Sym \h)^W} \C_{\lambda}$. We show by induction on $k$ that ${\Aak}/{\Aak}J \cong \eQ_{\lambda}^{\oplus k}$. This is vacuously true for $k = 1$. Forming the short exact sequence $0 \to U' \to U \to \C_{\lambda} \to 0$ and tensoring by $\Ak \o_{(\Sym \h)^W} - $ gives
	$$
	0 \to \Ak \o_{(\Sym \h)^W} U' \to \Ak \o_{(\Sym \h)^W} U \to \eQ_{\lambda} \to 0. 
	$$
	This sequence is exact because $\Ak$ is free over $(\Sym \h)^W$ by Proposition~\ref{pdQ}(1). Moreover, since $\Ak(W_{\lambda})$ is simple, Proposition~\ref{prop:Qlambdaprojiffsimple} says that $\eQ_{\lambda}$ is projective in $\Osph_{\kappa}$ Therefore, the sequence splits. It follows by induction that ${\Aak}/{\Aak}J \cong \eQ_{\lambda}^{\oplus k}$. Then 
	$$
	\Psi = \eM\otimes_{\Aak}({\Aak}/{\Aak}J) \cong \eM\otimes_{\Aak} \eQ_{\lambda}^{\oplus k} = \eG_{\lambda}^{\oplus k},
	$$
as required.	 
\end{proof}  

As in Section~\ref{Sec:polarreps},  the regular locus is the open set  $V_{\reg} =(\delta\not=0)\subset V$  defined by the non-vanishing of the discriminant $\delta\in \C[\h]^W=\C[V]^G$. Recall from Definition~\ref{minextn-defn} the notion of minimal extensions of $\dd(V_{\reg})$-modules. 

\begin{corollary}\label{torsionfree-corollary}   
	Let $\eG=\eM\otimes_{\Aak}\eP$ for a  projective object $\eP\in \Osph$. Then $\eG$ is the minimal extension of $\euls{L} := \eG |_{V_{\mathrm{reg}}}$; 
	equivalently,  $\eG = j_{!*} \euls{L}$   for the inclusion $j \colon V_{\mathrm{reg}} \hookrightarrow V$.  
\end{corollary}

\begin{proof}  
	By  Theorem~\ref{torsionfree},    $\eG$ has the properties required of a minimal extension.
\end{proof}

\begin{remarks}\label{torsionfree-remark}
	(1) By Lemma~\ref{lem:Vregintconnection}  $\euls{L} = \eG |_{V_{\mathrm{reg}}}$ is also an integrable connection. 
	
	(2) Suppose that  $V$ is a \gainly \ symmetric space for $G$, as in Remark~\ref{LS-C5}. Then  $\eMt=\eM$ and hence $R={\Aak}$ by \cite[Theorem~A]{LS3} and \cite[Theorem~9.10]{BLNS}. Thus,  Corollary~\ref{torsionfree-corollary} answers \cite[Conjecture~7.2]{Se}.  For further details see 
	Section~\ref{Sec:examples}. Similarly, if $V$ is a space of representations of the cyclic quiver, as discussed in Section~\ref{Sec:Quivers},  then again   
	$\eMt=\eM$ and so we obtain the analogous conclusion.  
	
	(3) It seems possible that that Corollary~\ref{torsionfree-corollary} is related to the work of Grinberg\cite{GrinbergSymmetric} and others on nearby cycles. We will expand upon this in Section~\ref{Sec:Shiftfunctors}; see, in particular, Conjecture~\ref{conj:nearby-cycle} and the discussion preceding it.  \end{remarks}

 \begin{corollary}\label{cor:eniso}    
	If $\eG = \eM \o_{{\Aak}} \euls{P}$ for a  projective object  $\euls{P} \in \Osph$,  then the map $\End_{\dd}(\eG) \to \End_{\dd(V_{\reg})}(\eG_{\reg})$ is an isomorphism.
\end{corollary}

\begin{proof}
	This is a standard fact about minimal extensions, but we include a proof since we do not know of a reference. Let $\phi \in \End_{\dd}(\eG)$ and $\phi |_{V_{\reg}}$ its image in $\End_{\dd(V_{\reg})}(\eG_{\reg})$. If $\phi |_{V_{\reg}} = 0$ then the image of $\phi$ is a $\delta$-torsion submodule of $\eG$. This forces $\phi = 0$ by Theorem~\ref{torsionfree}. 
	
	Let $\eta \in \End_{\dd(V_{\reg})}(\eG_{\reg})$. To show that $\eta$ is in the image of the restriction map, it suffices to show that $\eta(\eG) \subset \eG$. The $\dd(V)$-module $(\eta(\eG) + \eG) / \eG$ is a quotient of $\eta(\eG)$ and thus of $\eG$ too. This quotient is $\delta$-torsion. Therefore, Theorem~\ref{torsionfree} again implies that it is zero.  
\end{proof}

 \subsection*{$G$-invariant eigendistributions} As was remarked in the introduction, the original  motivation for Theorem~\ref{torsionfree}  comes from Harish-Chandra's work on $G$-invariant eigendistributions and we note that Theorem~\ref{torsionfree} has immediate consequences to such objects. In this discussion, we will always take $\vs=0$. Fix a real form $V_0$ of $V$ and write $\mathrm{Dist}(V_0)$ for the space of distributions on $V_0$.   One can then take as the definition that the space of \emph{$G$-invariant eigendistributions} (corresponding to eigenvalue $\lambda$) is the space 
 $$\Hom_{\dd(V)}\left( \eGt_{\lambda},\, \mathrm{Dist}(V_0)\right).$$
   To apply our results to these objects, we will need $\eGt_{\lambda}=\eG_{\lambda}$, which at least holds when $\eMt=\eM$. As remarked earlier, we know of no examples  where this does not hold, but it is known to hold in a number of important cases. In particular, it holds for \gainly \ symmetric spaces $V$; see Section~\ref{Sec:examples} for the definitions and Theorem~\ref{LS3-theorem} for the assertion. It also holds when $V=\mathrm{Rep}(Q_{\ell}, n\mathfrak{d})$ is the appropriate representation space for the cyclic quiver $Q_{\ell}$; see Section~\ref{Sec:Quivers} for the definitions and 
 Theorem~\ref{thm:cylicquiverradialparts0} for the result.  Then we have the following result.
 
 \begin{corollary}\label{cor:distributions}
 Assume that $\vs=0$ and we are in the situation where $\Ak(W)$ is simple and $\eMt=\eM$. Suppose that 
 $T$ is a $G$-invariant eigendistribution  supported on the singular  locus $(d=0)\subset V, $ for some $0\not=d \in \C[V]^G$ (for example, $d=\delta$).   Then $T=0$.
 \end{corollary}
 
 \begin{proof} If $T$ is such a distribution then $d^jT=0$ for some $j$ and so $T$  would lie in 
 $\Hom_{\dd(V)}\left( L,\, \mathrm{Dist}(V_0)\right)$, for some $d$-torsion factor $L$ of $\eG_{\lambda}$. By Theorem~\ref{torsionfree}(3), $L$ and hence $T$ is zero.
 \end{proof}
 
 For nice symmetric spaces and $d=\delta$ this result is  due to Sekiguchi \cite[Theorem~5.2]{Se} which in turn  generalises 
 Harish-Chandra's result stating that there are no nonzero  $G$-invariant eigendistributions on $V_0=\g_0$ supported on the singular locus; see \cite{HC3} or \cite[Theorem~8.3.5]{Wa1}.  For general $0\not=d\in \C[V]^G$ the result can be regarded as a generalisation of Harish-Chandra's theorem \cite[Theorem~5]{HC2} on invariant eigendistributions on $\g_0$  with nilpotent support; see the comments after \cite[Theorem~5.2]{LS}.

We end the section by noting that, for ``generic'' values of $\lambda$,  results like Theorem~\ref{torsionfree} (and 
the analogue  of  Corollary~\ref{thm:semi-simplicity2}  below)  are easy to prove. Thus it is for special values of $\lambda$, most notably $\lambda=0$, where these results become most  significant.

\begin{proposition}\label{prop:generic-lambda} Assume that $\rad_{\vs}$ is surjective and that 
 $\lambda \in \h^*$ has trivial stabiliser under $W$. Then:
\begin{enumerate}
\item both $\eGt_{{\lambda}}$ and $\eG_{\lambda}$ are semisimple;
\item $\eG_{\lambda}$ has a unique (necessarily simple) summand $L$ such that $L^G \neq 0$;
\item $L$ is $\delta$-torsionfree; and
\item if $Z$ is connected then $\eGt_{\lambda} = \eG_{\lambda} = L$ is simple.
\end{enumerate}
\end{proposition}

\begin{proof}
We apply the Fourier transform $\mathbb{F}^*_V$ of \eqref{defn:Fourier}. Then 
Lemma~\ref{lem:FourierGequiv} implies that 
$
\mathbb{F}_V^*(M) = \eGt_{\lambda},  $  where 
$$ M :=
 \dd(V^*) / (\dd(V^*) \g_{\chi - \mr{Tr}} + \dd(V^*) \mf{m}_{-\lambda})
$$
and $\mf{m}_{-\lambda}$ is the maximal ideal of  $ (\Sym V)^G$ defined by $-\lambda$. Since 
$-\lambda \in \h^*_{\reg}$, the $(G,\chi - \mr{Tr})$-monodromic module $M$ is supported on the closed orbit 
$G \cdot (-\lambda) \cong G / Z$. Since $(\chi - \mr{Tr}) |_{\mf{z}} = 0$, \cite[Proposition~9.1.1(i)]{BG} says
 that the category of finitely generated $(G,\chi - \mr{Tr})$-monodromic $\dd$-modules on $G / Z$ is equivalent
  to the category of $Z$-equivariant modules on a point. This is the same as the category of finite-dimensional 
  $Z / Z^{\circ}$-modules. In particular, it is semisimple. Therefore, $\eGt_{\lambda}$ is semisimple. Since 
  $\eG_{\lambda} = \eM \o_{\Aak} \eQ_{\lambda}$ is a quotient of $\eGt_{\lambda}$, it  is also semisimple.
Since the stabiliser $W_{\lambda}$ of $\lambda$ in $W$ is trivial, Lemma~\ref{lem:genOsphparabolicequi} 
says that $\eQ_{\lambda}$ is a simple $\Ak$-module. Then Lemma~\ref{summand-endo21M} says that 
$\eG_{\lambda}^G = \eM^{G} \o_{\Aak} \eQ_{\lambda} = \eQ_{\lambda}$. This implies that there must be a 
simple summand $L$ of $\eG_{\lambda}$ (occurring with multiplicity one) such that $L^G = \eQ_{\lambda}$ 
and $K^G = 0$ for every other simple summand of $\eG_{\lambda}$. Proposition~\ref{pdQ}(1) implies that 
$\eQ_{\lambda}$ is $\delta$-torsionfree, hence so too is $L$.
Finally if $Z$ is connected then the category of finitely generated $(G,\chi - \mr{Tr})$-monodromic 
$\dd$-modules on $G / Z$ is semisimple with one simple object ($U$ say). This implies that 
$\mathbb{F}_V^*(U) \cong L$ and hence $\eGt_{\lambda} = L^{\oplus k}$ for some $k \ge 1$. 
By the same reasoning, $\eG_{\lambda} = L^{\oplus r}$ for some $r\leq k$. Now,
 by Lemma~\ref{hom=invariants},  
$$
\Hom_{\dd}(\eGt_{\lambda}, L) = \Hom_{R}(\eQ_{\lambda}, \Hom_{\dd}(\eMt,L)) =
 \Hom_{R}(\eQ_{\lambda}, L^G) = \Hom_{{\Aak}}(\eQ_{\lambda}, \eQ_{\lambda})
$$
is one-dimensional. Therefore, $k = r =  1$.
\end{proof}


 \section{When Are Harish-Chandra Modules Semisimple?}\label{Sec:admissible}

In this section we apply the earlier results of the paper   to give a much more detailed understanding of the category $\eC$ of admissible modules. In particular, we will  determine precisely  when  the module $\eG_{\lambda} $ from Definition~\ref{defn:Glambda} is semisimple; see Theorem~\ref{thm:semi-simplicity}.
We continue to  assume  that $V$ is a visible, stable, polar $G$-representation, as in Notation~\ref{notation4.11}.   Except when we explicitly say otherwise,   \emph{we assume  that ${\Aak}=\Ak(W)$ is a simple algebra.}
 
\medskip
Recall that the Hecke algebra $\euls{H}_{q}(W_{\lambda})$ associated to $W_\lambda$ is defined in Definition~\ref{Hecke-defn}. The next lemma can be used to provide a more 
concrete realisation of this algebra.

\begin{lemma}\label{lem:endGHeckealg}
	Let $\lambda \in \h^*$ and assume that $\Ak(W_{\lambda})$ is simple. Then $\End_\ddd(\eG_{\lambda}) \cong \euls{H}_{q}(W_{\lambda})$. 
\end{lemma}
 
 \begin{proof} 
	Recall from Definition~\ref{defn:Glambda}  that $\eG_\lambda  = \HHleft (\eQ_{\lambda})=\eM\otimes_{\Aak}\eQ_{\lambda}$. 
	In the notation of   Corollary~\ref{summand-endo3},  $\BH(Y) =\Hom_\ddd(\eM,Y)$ for $Y\in \eC$ and, by that result,    
	 the adjunction $\mathrm{Id} \to \BH \circ\HHleft$ is an isomorphism. Thus, 
	\begin{align*}
	\End_\ddd(\eG_{\lambda}) & = \Hom_\ddd(\HHleft (\eQ_{\lambda}),\HHleft (\eQ_{\lambda})) \\
	& = \Hom_{\Aak}(\eQ_{\lambda}, \BH\circ \HHleft (\eQ_{\lambda}))  \\
	& = \Hom_{\Aak}(\eQ_{\lambda},\eQ_{\lambda}).
	\end{align*}
	  By Lemma~\ref{lem:genOsphparabolicequi}, $\End_{{\Aak}}(\eQ_{\lambda}) \cong \End_{\Ak(W_{\lambda})} (\eQ_0)$. 
	  Since  $\Ak(W_{\lambda})$ is simple, it follows from 
	Lemma~\ref{lem:HCprojiffQproj}(1,2)  and  \cite[Theorem~1.2]{LosevTotally} that $\End_{\Ak(W_{\lambda})}(\eQ_{0}) \cong \euls{H}_{q}(W_{\lambda})$. 	
\end{proof}

\begin{corollary}\label{cor:semi-simplicityGtoHecke}   
	If $\eG_{\lambda}$ is semisimple then (even without assuming that $\Ak(W)$ is simple) the Hecke algebra $\euls{H}_{q}(W_{\lambda})$ is semisimple. 
\end{corollary}

\begin{proof}   
	If $\eG_{\lambda}$ is semisimple then Corollary~\ref{torsionfree-converse2}  implies that $\Ak(W_{\lambda})$ is a simple algebra. Therefore, it follows from Lemma~\ref{lem:endGHeckealg} that the Hecke algebra $\euls{H}_{q}(W_{\lambda})$ is semisimple.
\end{proof}

\begin{lemma} \label{grad}
	Let $\{u_1,\dots, u_n \}$ be algebraically independent generators of $\BC[V]^G$. Then there 
	exist $G$-invariant  derivations $\theta_i\in \mathrm{Der}(V)$
	 for $1\leq i\leq n =\dim\h$ such that 
	$$
	\Der (V_{\reg}) = \BC[V_{\reg}] \, \tau(\g) \; \oplus \;
	\Bigl(\bigoplus_{i=1}^n\,   \BC[V_{\reg}] \, \theta_i \Bigr).  
	$$
	Consequently, for $\eS = \{ \delta^k \}$, $\eMt_{\eS} = \eD(V_{\reg})/  \eD(V_{\reg})\tau(\g)$ is a free left
	$\BC[V_{\reg}]$-module with basis 
	$$
	\left\{\Theta^{\mathbf{i}} =\theta^{i_1}_1\cdots \theta^{i_\ell}_\ell \; : \;
	\mathbf{i} = (i_1, \dots,i_n)\in \mathbb{N}^n  \right\}.
	$$
\end{lemma}

\begin{proof}
	Write  $W = N / Z$ where $N=N_G(\h)\supset Z_G(\h)$ as in \cite[Equation~3.2]{BLNS}. 
	Let $\tau_G\colon \mf{g} \to \Der G$ be the differential of the left regular 
	action of $G$ on itself. Then $\tau_G$ is injective, with image
	being the space of \textit{right} invariant vector fields on $G$ and $\Der ( G) = \euls{O}_G \o_{\C} \tau_G(\mf{g})$. 
	
	The proof of \cite[Proposition~4.4.1]{BellamySRAlecturenotes} shows that there exist $W$-invariant derivations $\theta_1', \ds, \theta_{n}'$ in $ \Der ( \h_{\reg})$ such
	   that the latter is a free $\C[\h_{\reg}]$-module with basis $\{ \theta_1', \ds, \theta_{n}' \}$. If $\Theta$ is the space spanned by these $n$ derivations then  
	\[
	 \Der  (G \times \h_{\reg}) = B\tau_G(\mf{g}) \oplus B\Theta,  
	\]
	where $B := \C[G \times \h_{\reg}]$. Here $G$ acts  on $G \times \h_{\reg}$ by acting on the first factor only. Thus $\Theta$ consists of $G$-invariant derivations (in contrast, $G$ acts non-trivially on $\tau_G(\mf{g})$). Recall that $N^{\circ} = Z^{\circ}$, with Lie algebra $\mf{z}$. The group $N$ acts diagonally on $G \times \h_{\reg}$, on the right on $G$, and this is a free action. By, for example,  \cite[Corollary~4.5]{Schwarz}, this implies that 
	\begin{align*}
	 \Der  (G \times_N \h_{\reg}) & = \left( \left( B \otimes_{\euls{O}(G)} \frac{\Der (G)}{( \Der  (G)) \mf{z}} \right) \oplus (B \o_{\C[\h_{\reg}]} \Der ( \h_{\reg})) \right)^N  \\
	& = \left(B\tau_G(\mf{g}) \oplus B\Theta\right)^N  \\
	& = \C[G \times_N \h_{\reg}]\tau_G(\mf{g}) \oplus \C[G \times_N \h_{\reg}]\Theta.
	\end{align*}
	The $G$-equivariant identification $G \times_N\h_{\reg}\cong V_{\reg}$ from ~\eqref{lem:stablepolaropen} induces an isomorphism 
	$
	 \Der  (G \times_N \h_{\reg}) \cong \Der( V_{\reg}).
	$
	Under this isomorphism, $\tau_G(\mf{g})$ is identifies with $\tau(\mf{g})$. The image of the $\theta_i'$ under this identification may have poles along the complement to $V_{\reg}$. We choose $k > 0$ such that the image of $\delta^k \theta_i'$ is regular on the whole of $V$. Let $\theta_i$ be the image of $\delta^k \theta_i'$.  
\end{proof}

 We are now ready to  prove the following very useful result, which generalises \cite[(6.4)]{AJM}, although the proof is essentially the same. 

\begin{lemma}\label{lem:keynonvanishG}
	Let  $Q\varsubsetneq P$ be $\D(V_{\reg})$-submodules of  $\eM_\eS$ such that $P/Q$ is torsion-free as a $\C[V_{\reg}]$-module. Then $(P/Q)^G \neq 0$. 
\end{lemma}

\begin{proof}
	By Proposition~\ref{semiprime2}(2), $\eM_{\eS}=\eMt_{\eS}  = \eD(V)_{\eS}/\eD(V)_{\eS}\tau(\g)$.  Fix a lexicographic
	ordering $\preceq $ on $\NN^\ell$. Then, by Lemma~\ref{grad},  and in the notation of that result, 
	any
	element $\alpha\in \eM_\eS$ can be uniquely written $\alpha= a_{\mathbf{j}}
	\Theta^{\mathbf{j}} + \sum_{\mathbf{i} \prec {\mathbf{j}}} a_{\mathbf{i}} \Theta^{\mathbf{i}}$, for some
	$a_{\mathbf{i}} \in \BC[V_{\reg}]$.  For $\mathbf{i} \in \NN^\ell$, set
	$$
	F_{\mathbf{i}} \eM_\eS= \sum_{\mathbf{j} \preceq {\mathbf{i}} }\BC[V_{\reg}] \Theta^{\mathbf{j}} \,
	\subset \, \eM_\eS.
	$$
	Since the $\{ \theta_i \}$ are $G$-invariant, each $F_{\mathbf{i}} \eM_\eS$ is
	an $(\BC[V_{\reg}],G)$-module.  It follows from Lemma~\ref{grad} that the
	$\{F_{\mathbf{i}} \eM_\eS\}$ form an exhaustive filtration on $\eM_\eS$. Let $P$ be a
	$\eD(V)_\eS$-submodule of $\eM_\eS$ and set $F_{\mathbf{i}} P = P \cap F_{\mathbf{i}}
	\eM_\eS$. By \cite[Lemma~6.1]{AJM}, the $\{F_{\mathbf{i}} P\}$ provide a
	filtration of $P$ by $(\BC[V_{\reg}],G)$-modules that are locally finite as
	$G$-modules. For ${\mathbf{j}} \in \NN^\ell$, set $\gr_{\mathbf{j}} P = F_{\mathbf{j}} P/
	\bigl( \sum_{{\mathbf{i}} \prec {\mathbf{j}}} F_{\mathbf{i}} P\bigr)$. Thus, $\gr_{\mathbf{j}} P
	\subseteq \gr_{\mathbf{j}} \eM_\eS \cong \BC[V_{\reg}] \Theta^{\mathbf{j}} $. 
	
	 Since $G$ is reductive, $ F_\bsq P \cong \gr_\bsq P \, \bigoplus
	\, \sum_{{\mathbf{i}} \prec \bsq} F_{\mathbf{i}} P$ as $G$-modules.  Therefore, $(F_\bsq
	P)^G \ne 0$ whenever $(\gr_\bsq P)^G \ne 0 $.  Suppose that $\gr_\bsq
	P \ne 0$.  Then $\gr_\bsq P = I \, \Theta^\bsq $ for some nonzero,
	$G$-stable ideal $I$ of $\BC[V_{\reg}]$. Since $V$ is a stable $G$-module, \cite[Theorem~4.(2)]{PopovStability} implies that
	$I^G \ne 0$, and so Lemma~\ref{grad} implies that $(\gr_\bsq P)^G =
	I^G \, \Theta^\bsq \ne 0$. To summarise, we have proved:
	\begin{equation} \label{eq3}
	\gr_\bsq P \ne 0  \;  \Longrightarrow \;  (F_\bsq P)^G \ne 0
	\end{equation}

	Let $Q \varsubsetneq P $ be submodules of $\ \eM_\eS$   that satisfy the hypotheses  of the lemma. Pick $\bsq \in \NN^\ell$
	minimal such that $F_\bsq Q \varsubsetneq F_\bsq P$.  We claim that
	$F_\bsq Q = 0$.  If not, we may pick $b = \beta \Theta^\bsq + b' \in
	F_\bsq Q$ and $f = \alpha \Theta^\bsq + f' \in F_\bsq P \setminus
	F_\bsq Q $, with $0 \ne \beta, \alpha \in \BC[V_{\reg}]$ and $b', f' \in
	\sum_{{\mathbf{i}}\prec \bsq} F_{\mathbf{i}} \eM'$. Then, 
	$$\beta f - \alpha b \ = \  \beta
	f' - \alpha b' \  \in \  F_{\mathbf{j}} \eM' \cap P \ = \  F_{\mathbf{j}} P,$$ for some ${\mathbf{j}} \prec
	\bsq$. Thus, $\beta f \in \alpha b + F_{\mathbf{j}} Q \subseteq Q$.  Since
	$P/Q$ is a torsionfree $\C[V_{\reg}]$-module, this forces $\beta = 0$ and a
	contradiction. Thus, $F_\bsq Q = 0$. Finally, since $F_\bsq Q = 0$ and $\gr_\bsq P \ne 0$, \eqref{eq3} implies that $(P/Q)^G \supseteq (F_\bsq P)^G \ne 0 $.
\end{proof}	
	
\begin{proposition}\label{bigg}
	Let $L$ be a nonzero  submodule of $\eG_{\lambda}$. Then $L^G \neq 0$. 
\end{proposition}

\begin{proof}
	By Theorem~\ref{torsionfree}(2), $L\hookrightarrow L_{\eS}$. Moreover, since $\delta$ is $G$-invariant, the action of
	$G$ on $\eD(V)$ extends to an action on $\eD(V)_{\eS}=\eD(V_{\reg})$ and hence to
	$L_{\eS}$.  If $0\not=f\in L_{\eS}^G$, then some $\delta^jf\in
	L^G$ and hence $L^G\not=0$, again by Theorem~\ref{torsionfree}(2). Thus, it
	suffices to prove: 
	\begin{equation}\label{invar2}
	\text{If } L \text{ is a nonzero submodule of }  \eG_{\lambda},
	\text{ then } L^G_\eS\not=0.
	\end{equation}
	
	Let $\eG=\eG_{\lambda}$ and write $\eG_{\eS} = \eD(V)_\eS / F$,
	where $F$ denotes the annihilator of the
	natural generator of $\eG_{\eS}$, noting that $F \supseteq \eD(V)_\eS \g_{\chi}$. That is, $\eG_{\eS}$ is a quotient of $\eM_{\eS}$. Therefore, $L$ can be expressed as a subquotient $P/Q$ of $\eM_{\eS}$. By
	Theorem~\ref{torsionfree}(2), $L_\eS \subseteq \eG_{\lambda,\eS}$ is 
	torsionfree as a $\BC[V_{\reg}]$-module and so Lemma~\ref{lem:keynonvanishG} applies. 
	This says that $L_\eS^G \neq 0$ which, by
	\eqref{invar2},   completes the proof of the proposition.
\end{proof}

In general, it is not true that if $M$ is a  monodromic $\dd(V)$-module with $M|_{V_{\reg}} \neq 0$ then $M^G \neq 0$. However, we do have the following positive result that complements Proposition~\ref{bigg}. 
This result will not be used elsewhere in the paper. 

\begin{lemma}\label{lem:invariantsvstfpolar}
	Assume that the stabiliser subgroup $Z \subset G$ is connected. If $M$ is a $(G,\chi)$-monodromic $\dd(V)$-module with $M|_{V_{\reg}} \neq 0$ then $M^G \neq 0$. 
\end{lemma}

\begin{proof}
	Recall that $\chi |_{\mf{z}} = 0$. If $\Upsilon \colon \h_{\reg} \hookrightarrow V_{\reg}$ is the closed embedding then
	\cite[Proposition~9.1.1(i)]{BG} says that $\Upsilon^*$ is an equivalence between the category of 
	$(G,\chi)$-monodromic $\dd(V_{\reg})$-modules and the category of $N$-equivariant $\dd(\h_{\reg})$-modules. 
	In particular, $\Upsilon^* M \neq 0$.
	
	If $K$ is a $N$-equivariant $\dd(\h_{\reg})$-module then it is also $Z$-equivariant. The group $Z$ 
	acts trivially on $\h_{\reg}$, so that the action of $N$ factors through $N/Z = W$. This implies that 
	$\tau_{\mf{h}}(\mf{z}) = 0$ and thus $\mf{z}$ acts trivially on $K$. Since $Z$ is assumed connected, 
	we deduce that $Z$ acts trivially on $K$. In other words, the category of $N$-equivariant 
	$\dd(\h_{\reg})$-modules is equivalent to the category of $W$-equivariant $\dd(\h_{\reg})$-modules. 
	The group $W$ acts freely on $\h_{\reg}$. Therefore, if $L$ is a nonzero $W$-equivariant 
	$\dd(\h_{\reg})$-module,  then $L^W \neq 0$. Finally, by \cite[Proposition~9.1.1(ii)]{BG}, 
	$M^G = (\Upsilon^* M)^W~\neq~0$.  
\end{proof}

Recall that  $\BH(N) =\Hom_\ddd(\eM,\, N)$ for $N \in (G,\chi,\ddd)\lmod$. 

 \begin{corollary}\label{cor:hom(ML)}   Assume that   ${\Aak}$ is a simple algebra. If $L$ is a left $\ddd$-submodule of $\eG_\lambda$, then $\BH(L)=L^G\not=0$.
\end{corollary}

\begin{proof} 
	By Proposition~\ref{semiprime2},
	$$
	\Hom_\ddd(\eMt, L) =\Hom_\ddd(\eM,L)\oplus \Hom_\ddd(\eMt/\eMt J, L).  
	$$
	However, $L$ is $\delta$-torsionfree by Theorem~\ref{torsionfree} while $\eMt/\eMt J$ is left $\delta$-torsion
	by Proposition~\ref{semiprime2}. Hence $ \Hom_\ddd(\eMt/\eMt J, L)=0$  and so 
	\[\Hom_\ddd(\eM, L) =\Hom_\ddd(\eMt,L) = L^G,\] by Lemma~\ref{hom=invariants}. Therefore, Proposition~\ref{bigg} implies that $\BH(L) \neq 0$. 
\end{proof}

Combining these results we obtain the following statement which, even in the case of  a symmetric space,
gives a strikingly complete answer to \cite[Conjecture~C4]{LS3}.  See Remark~\ref{LS-C4} 
and Theorem~\ref{cor:symmetricspacessHecke} for further comments.
 
\begin{theorem}\label{thm:semi-simplicity}   
	 Assume that   ${\Aak}$ is a simple algebra. Then $\eG_{\lambda}$ is semisimple if and only if the Hecke algebra $\euls{H}_{q}(W_{\lambda})$ is semisimple. 
	 
	 When $\euls{H}_q(W_{\lambda})$ is semisimple, each simple summand of  $\eG_{\lambda} $ has the form $L = \HHleft(\mathbb{H}(L)) =  \HHleft(L^G)$. 
\end{theorem}

\begin{proof}   
	If $\eG_{\lambda}$ is semisimple then Corollary~\ref{cor:semi-simplicityGtoHecke} implies that the Hecke algebra $\euls{H}_{q}(W_{\lambda})$ is semisimple.
	
	Conversely, assume that $\euls{H}_{q}(W_{\lambda})$ is semisimple; thus $\Osph_{\lambda}$ is also semisimple by Corollary~\ref{A-projectives}.
	Fix an irreducible submodule $L$ of $\eG_{\lambda}$ and note  that $\BH(L)  = L^G\neq 0$
	  by Corollary~\ref{cor:hom(ML)}.  By  Lemma~\ref{summand-endo22}(1),  $\BH$ is exact and so
		\[
	0\not=\BH(L) \hookrightarrow\BH(\eG_{\lambda}) = \eQ_{\lambda}.
	\]
	Since $\Osph$ is semisimple,   $L^G=\BH(L)$ is therefore a summand of $\eQ_{\lambda}$, and so   
	  $ \HHleft(L^G)=\eM\otimes_{\Aak}L^G$ is a  
	summand of $ \HHleft(\eQ_{\lambda}) = \eG_{\lambda}$.  By
	Lemma~\ref{summand-endo21M}(1),  $ \eM={\Aak}\oplus T$ as a $(\ddd^G,{\Aak})$-bimodule and so 
	the natural map      
		\[ 
	 \phi: \HHleft(L^G)  \ = \ \eM\otimes_{\Aak} L^G \ \to \ \eM\cdot L^G \ = \ {\Aak}\cdot L^G + T\cdot L^G 
	 \  \hookrightarrow \  L.
	 \]
	    has image containing $L^G$. 
	Since $L$ is irreducible, $\phi$ is therefore surjective.  In particular,  $ \HHleft(L^G)=\eM\otimes_{\Aak}L^G$ is a  
	nonzero summand of $  \eG_{\lambda}$.
	
	It remains to prove that $\phi$ is injective. 	Suppose that $K=\ker(\phi)\not=0$. Then  $K\hookrightarrow \HHleft(L^G)\hookrightarrow
	\HHleft(\eQ_{\lambda}) = \eG_{\lambda}$. By Corollary~\ref{cor:hom(ML)},   again,  $\BH(K)=K^G\not=0$.
	But this is absurd: applying  the exact functor $\BH$ to the exact sequence $0\to K \to \HHleft(L^G)\to L\to 0$ gives the exact sequence 
	\[0\to \BH(K) \to \BH\circ\HHleft\circ \BH(L)\to \BH(L)\to 0.\] By Corollary~\ref{summand-endo3}, $\BH\circ\HHleft\circ\BH(L) =\BH(L)$ 
	which  forces $\BH(K)=0$. This contradiction shows that $K=0$ and hence that $ L=\HHleft(L^G)$ is a summand of $\eG_\lambda$.
\end{proof}

As shown in \cite[Theorem~3.1]{BEG}, if the Hecke algebra $\euls{H}_q(W)$ associated to $\Aak$ is semisimple, then the algebra $\Aak$ is simple. At least for $\lambda=0$, this allows us to drop  the simplicity hypothesis on ${\Aak}$  in Theorem~\ref{thm:semi-simplicity}. 
   
\begin{corollary}\label{thm:semi-simplicity2}   
	Assume that $\rad_{\vs}$ is surjective but make no assumption about the simplicity of ${\Aak}$. Then $\eG_{0}$ is semisimple if and only if the Hecke algebra 
	$\euls{H}_{q}(W)$ is semisimple. 
\end{corollary}

\begin{proof}  We need only consider the case when  ${\Aak}$ is not simple. Then, by  \cite[Theorem~3.1]{BEG}, the Hecke algebra $\eH_q(W)$ is not semisimple, while by Corollary~\ref{torsionfree-converse2} $\eG_0$ is not semisimple.
\end{proof}
 
 Much is known about the semisimplicity of Hecke algebras, as we explain at the end of the section. But we first want to describe another consequence of Theorem~\ref{thm:semi-simplicity}.
This is the  following  exact analogue of \cite[Theorem~B]{AJM} and \cite[Theorem~5.3]{HK}, both of which were concerned with the special case $V=\g$.

\begin{corollary}\label{cor:simplicity}  
	 Assume that   ${\Aak}$ is a simple algebra.  If $\euls{H}_{q}(W_{\lambda})$ is semisimple then 
	$$
	\eG_{\lambda} = \bigoplus_{\rho \in \mr{Irr} \euls{H}_{q}(W_{\lambda})} \eG_{\lambda,\rho} \o \rho^*
	$$
	as a $(\eD(V),\euls{H}_{q}(W_{\lambda}))$-bimodule. Moreover,  each $\eG_{\lambda,\rho}$ is an irreducible $\eD(V)$-module and the $\eG_{\lambda,\rho}$ are non-isomorphic for distinct $(\lambda,\rho).$ 
\end{corollary}

\begin{proof}  
By  Corollary~\ref{A-projectives}(2), $\eQ_{\lambda}$ is a projective object in $\Osph_{\lambda}$ 
with   $\End_{\Aak}(\eQ_{\lambda}) = \euls{H}_{q}(W_{\lambda})$.   Moreover, $\Osph_{\lambda}$ is a semisimple category. This implies that 
\begin{equation}\label{eq:simplicity}
	\eQ_{\lambda} = \bigoplus_{\rho \in \mr{Irr} \euls{H}_{q}(W_{\lambda})} \eQ_{\lambda,\rho} \o \rho^*
	\end{equation}
	as a $({\Aak},\euls{H}_{q}(W_{\lambda}))$-bimodule. Here the $\eQ_{\lambda,\rho}$ run through  all simple objects in $\Osph_{\lambda}$.

	By  Theorem~\ref{thm:semi-simplicity}, $\eG_{\lambda}
	=\HHleft(\eQ_{\lambda})$ is semisimple, with each simple summand being of the form $L = \HHleft(\mathbb{H}(L))$.  Conversely, if $0\not=N\subsetneq \eQ_{\lambda}$, then $\HHleft(N)\not=0$, by 
	Lemma~\ref{summand-endo21M}(1). Therefore, applying $\HHleft$ to   \eqref{eq:simplicity}  gives the required decomposition. 
\end{proof}
  
  When $\euls{H}_{q}(W_{\lambda})$ is not   semisimple the structure of $\eG_{\lambda}$ gets considerably more complicated. First, of course, the decomposition of 
  $\eQ_{\lambda}$ into a direct sum of indecomposable  summands is reflected in the decomposition of $\eG_{\lambda}$. However, one may also obtain 
  extra $\delta$-torsion subfactors inserted into these modules, as happens with the example described in Section~\ref{HC-example}; see also Proposition~\ref{prop:quiversystem} for a generalisation of that example.
  
   \begin{remark}\label{LS-C4} As in Remark~\ref{LS-C5}, let $V$ be a symmetric space  that satisfies Sekiguchi's niceness condition.
  Then \cite[Conjecture~C.4]{LS3} asks whether $\eG_{\lambda}$ is semisimple and of course  
  Theorem~\ref{thm:semi-simplicity} says that this happens if and only if $\euls{H}_{q}(W_{\lambda})$ is  semisimple.
   There is, in turn, a standard procedure for determining when a given Hecke algebra is semisimple.  The procedure is described   in  \cite{Chloubook} but, in outline, is as follows.  Consider  the Hecke algebra $\euls{H}(W)$ corresponding to  a  complex reflection group $W$. Then  \cite[Theorem 2.4.10]{Chloubook} shows that $\euls{H}(W)$ is semisimple at $q$ if and only if the Schur element associated to each irreducible $W$-module is nonzero when evaluated at $q$. These Schur elements   are described in \cite{Chloubook}. 
    
When $V$ is a  symmetric space, this algorithm is worked out in detail   in Theorem~\ref{cor:symmetricspacessHecke}.  
    
     \end{remark}
  

\section{Projectivity and Injectivity  of  Certain Monodromic  Modules}\label{Sec:projectivity}

   Throughout the section we  assume   that $V$ is a visible, stable, polar representation, as in Notation~\ref{notation4.11},   and that $\Aak$ is simple. 
    
This section continues the study of monodromic $\ddd$-modules,  and in particular the Harish-Chandra modules 
$\eG_{\lambda}$.   The  main result, Theorem~\ref{G-is-injective},  proves that  $\eG_{\lambda}$   is both injective and projective in the category $\eC $.   The proof of this is quite involved and   we will  also need to show that 
the right $\ddd$-module $\eM' := \Ext^m_\ddd(\eM,\ddd)$ is a projective object in $\rmod (G,\chi, \ddd)$; see 
Theorem~\ref{projectivity-theorem-stable}.  Since we therefore need to understand both left and right $\ddd$-modules, we will  eventually need the stronger assumption  that both $\Aak=\Ak(W)$ and its left-right analogue are simple; see Hypothesis~\ref{K-hyp} for the precise requirement.

    The fact that $\eG_{\lambda}$ is projective is an easy consequence of the earlier results.

\begin{proposition}\label{G-projective} 
	 Assume that   ${\Aak}$ is a simple algebra and let $\eG=\eM\otimes_{\Aak}\eP$ for some projective object $\eP\in \Osph$. Then $\eG$ is projective in $\eC$.
\end{proposition}
 
\begin{proof} 
	 By Lemma~\ref{torsion-consequence2}(1), $\eG\in \eC $. By Lemma~\ref{lem:charvar}(3), the category $\eC$ is a full
	  subcategory of the category $(G,\chi,\ddd)\lmod$. Thus,  for any admissible $\ddd$-module $L$, 
	 adjunction implies that  
	\[\begin{aligned} 
	\Hom_{\eC}(\eG,\, L) \  \cong&  \ 
	\Hom_{\ddd}(\eM\otimes_{\Aak}\eP,\, L)   \cong \ 
	\Hom_{{\Aak}}\left(\eP,\, \Hom_\ddd(\eM, L)\right). 
	\end{aligned}\]
	By Lemma~\ref{summand-endo22}(1),  $\Hom_\ddd(\eM,-)$ is an exact functor on $(G,\chi,\ddd)\lmod$. On the other hand, by  
	Lemma~\ref{torsion-consequence}(2)  $\Hom_\ddd(\eM, L)\in \Osph$ and  $\Hom_{\Aak}(\eP,-)$ is  exact on $\Osph$.
	  Combining these observation with the  displayed equation shows that $\Hom_{\eC}(\eG,\, -)$ is 
	exact on $\eC$. Hence $\eG$ is projective in $\eC$.
\end{proof}

\subsection*{The dual module $\eM' =\Ext^m_\ddd(\eM,\ddd)$. } 
We now turn  to the study of $\eM'$, with the main aim being  to show that it is a projective object  in $\rmod (G,\chi, \ddd)$.   We  begin with an elementary observation.
 
\begin{lemma}\label{left-simplicity}
	For any character $\chi$, 
	\begin{equation} \label{left-simplicity-equ1}
		(\ddd / \g_{\chi} \ddd)^G \cong \End_{\ddd}(\ddd / \g_{\chi} \ddd) \cong \End_{\ddd}(\ddd/ \ddd \g_{-\chi})^{\mathrm{op}} \cong (\ddd/ \ddd \g_{-\chi})^{G,op}.
	\end{equation}
	In particular, $(\ddd/ \g_{\chi} \ddd)^G$ is simple if and only if $(\ddd/ \ddd \g_{-\chi})^{G}$ is simple. 
\end{lemma}

\begin{proof}  
	Note first  that, by Lemma~\ref{hom=invariants},
	\[\End_\ddd(\ddd/\ddd\g_{-\chi}) = \Hom_\ddd(\ddd/\ddd\g_{-\chi},\, \ddd/\ddd\g_{-\chi}) \cong  \left(\ddd/\ddd\g_{-\chi}\right)^{G}.\]  Similarly $\End_\ddd(\ddd / \g_{\chi} \ddd) \cong  \left(\ddd /\g_{\chi} \ddd \right)^G $, giving the outside isomorphisms of \eqref{left-simplicity-equ1}. 
	
	Suppose that $R$ is a ring with an anti-automorphism $( - )^t$ and a left ideal $J$. 
	Let  $\mathbb{I}(J) :=\{\theta\in R : J\theta\subseteq J\}$ be the idealiser of $J$. Then   
	$$
	\mathbb{I}(J^t)/J^t   \cong \End_{R}(R/J^t)^{\mathrm{op}} \cong \End_{R}(R/J) \cong \left(\mathbb{I}(J)/J\right)^{op}.
	$$
	In our case, the antiautomorphism is defined by $f^t = f$ for $f \in \C[V]$ and $v^t = -v$ for 
	$v \in V\subset \Sym  V$.  If we also define   $( - )^t$ on  $\g$  by $x^t = - x$, then 
	$\tau$ intertwines the anti-automorphisms, implying that $(\ddd \g_{-\chi})^t = \g_{\chi} \ddd$
	and hence that $ \End_{\ddd}(\ddd / \g_{\chi} \ddd) \cong \End_{\ddd}(\ddd/ \ddd \g_{-\chi})^{\mathrm{op}}$.
\end{proof}

We note that the action of $G$ on $\ddd$ commutes with the antiautomorphism in the proof of Lemma~\ref{left-simplicity}. This shows that $(\ddd \g_{-\chi})^G = (\g_{\chi} \ddd)^G$ as ideals in $\ddd^G$. As in Equation~\ref{eq:notation4.11}, there exists a parameter $\cdag$ for the spherical rational Cherednik algebra such that the radial parts map is 
\begin{equation}\label{eq:left-rad}
\rad_{-\vs} \colon \ddd^G \to A_{\cdag}(W),
\end{equation} 
with kernel containing $(\ddd\g_{-\chi})^G=(\g_{\chi} \ddd)^G$.

The following hypothesis will be assumed for the rest of this section. 

\begin{hypothesis} \label{K-hyp} 
	The algebras $\Ak(W)$ and $A_{\cdag}(W)$ are both simple. 
\end{hypothesis} 

\begin{remark}\label{rem:K-defn} 
 By Theorem~\ref{thm:radial-exists}(3), Hypothesis~\ref{K-hyp} implies that both $\rad_{\vs}$ and $\rad_{-\vs}$ are surjective. Consequently, the results of the earlier sections also apply on the left. 
Since we will swop sides fairly frequently in the section,  we will write 
\[ \Ak := \Ak(W)\qquad \text{and} \qquad A_{\cdag} := A_{\cdag}(W).\]
\end{remark}

\begin{corollary}\label{left-simplicity2}
	Suppose that $\vs=0$. Then $A_{\cdag}\cong A^{op}_{\kappa}$ and so 
	$\Ak$ is simple if and only if $A_{\cdag}$ is simple. 
\end{corollary}

\begin{proof}   
	By Lemma~\ref{left-simplicity}, $(\ddd / \g \ddd)^G\cong R^{op}=\left((\ddd/ \ddd \g)^{G}\right)^{op}$. Moreover, by the proof of that lemma,  $\delta\in \C[V]$ is fixed by the given antiautomorphism. Also, by Proposition~\ref{semiprime2}, the kernel of the surjection $R\twoheadrightarrow \Ak $ is precisely the  $\delta$-torsion submodule of $R$. The analogous comment applies to $(\ddd / \g \ddd)^G$.
	
	Therefore, by factoring out the $\delta$-torsion submodules from the isomorphism $(\ddd / \g \ddd)^G\cong  \left((\ddd/ \ddd \g)^{G}\right)^{op}$ 
	gives $A_{\cdag}\cong \Ak^{op}$. 
\end{proof}

\begin{remark}\label{rem:K-defn1}
	When $\vs\not=0$ the algebras   $\Ak$ and $A_{\cdag}$ can be non-isomorphic and so the simplicity of $A_{\cdag}$ does not   follow from that of $\Ak$.  An explicit example, where  one of the rings is simple but the other is not,
	will be  given in Example~\ref{ex:strangerankonepolarrad}.
	\end{remark}

We also need the left-right analogues of $\eMt$ and $\eM$. The following result is the left-right analogue of Proposition~\ref{semiprime2}.

\begin{lemma}\label{lem:K-hyp}
	Under Hypothesis~\ref{K-hyp}, the $(A_{\cdag},\ddd)$-bimodule $\eKt :=  \ddd/\g_{\chi} \ddd $ has a bimodule decomposition  
	$\eKt \cong \eK\oplus L_\eK$  with the following properties.
	\begin{enumerate} \item $\End_\ddd(\eK) \cong A_{\cdag}$ is a simple domain.
		\item The discriminant $\delta\in \C[V]^G$ acts locally ad-nilpotently on $\eK$. 
		\item The complementary summand $L_\eK$ is $  \delta$-torsion.\qed
	\end{enumerate}
\end{lemma}

\begin{remark}\label{rem:K-hyp}
Note  that,  by  Corollary~\ref{lem:leftrightdfreepolar}, $\eK$ is automatically  $\delta$-torsionfree on both sides  while, by the left-right analogue of Lemma~\ref{lem:Gequivisfullsub},  $\eKt$ and hence $\eK$ are projective  objects in $\rmod (G,\chi,\ddd)$.
\end{remark}

The goal of the next several  subsections is to prove:

\begin{theorem} \label{projectivity-theorem-stable}
	Assume that Hypothesis~\ref{K-hyp} holds.  Then the right $\ddd$-module $\eM' = \Ext^m_\ddd(\eM,\ddd)$ is a projective object in $\rmod (G,\chi,\ddd)$. 
\end{theorem}

\subsection*{Vector bundles on homogeneous spaces}

Let $Y$ be a smooth affine variety  over $\C$, with a free action of our given  finite group  $W$.  Thus $\eD(Y/W)=\eD(Y)^W$.   
We can identify the category $(\eD(Y),W)\lmod$ of finitely generated  left $(\eD(Y),W)$-modules with  $(\eD(Y) \rtimes W)\lmod$. This leads to   the following
routine result.

\begin{lemma}\label{lem:finitegroupDequiv}  
	Keep  the above assumptions on $Y$. Then,    $M \mapsto M^W$, gives an equivalence $(\eD(Y),W)\lmod  \stackrel{\sim}{\longrightarrow} 
	\eD(Y)^W\lmod.$ Hence
	$
	\Ext^i_{ \eD(Y/W)}\left(M^W,N^W\right) = \Ext^i_{\eD(Y)}(M,N)^W,  $	for   $M,N \in (\ddd,W)\lmod$. \qed
\end{lemma}

As in Section~\ref{Sec:polarreps}, let $\h_{\reg} $ and $V_{\reg}$ denote  the regular loci defined by the non-vanishing of $\delta$. Recall that $W = N / Z$ where $N=N_G(\h)\supset Z_G(\h) = Z$ as in
 \cite[Equation~3.2]{BLNS}. Also,  $\mf{z}=\Lie(Z)$ can be identified with $\Lie(N)$ since $|W|<\infty$. Finally, using \eqref{lem:stablepolaropen}, set 
\[ \eX = V_{\reg} \cong G \times_N \h_{\reg}. \]
The group $G$ acts on itself by multiplication  on  either side   (with the right action given by $g \mapsto g h^{-1}$). 
With a slight generalisation of our earlier notation, differentiating these actions gives moment 
maps $\tau_L \colon \g \rightarrow \eD(G)$ \label{tau-left-defn}
and $\tau_R \colon \g \rightarrow \eD(G)$ with images the \emph{right}, respectively \emph{left}, invariant vector fields on $G$. Let 
$\euls{O}_G^{\chi} = \eD(G) \otimes_{\tau_{{L}}(\mf{g})} \C_{\chi}$, a rank one integrable connection on $G$. It is a consequence of \cite[Proposition~9.1.1(i)]{BG}
that every coherent $(G,\chi)$-monodromic $\eD(G)$-module is a finite direct sum of copies of $\euls{O}_G^{\chi}$; in particular, $\euls{O}_G^{\chi}$ is the unique 
$(G,\chi)$-monodromic rank one integrable connection on $G$. For any character $\psi$ of $\mf{z}$, the associated ring of twisted differential operators 
on $G/Z$ is defined to be 
$$
\eD_{\psi}(G/Z) = ( \eD(G) / \eD(G) \tau_{R}(\mf{z})_{\psi})^Z. 
$$  
For the next few results we work with an arbitrary linear character $\eta$ of $\g$.

\begin{lemma}\label{general-char1}
	Let $\eta$ be a linear character of $\g$ and set  $\psi = - \eta |_{\mf{z}}$. The space $\euls{O}^{\eta}_{G/Z} := (\euls{O}_G^{\eta})^Z$ is a left $\eD_{\psi}(G/Z)$-module,
	free of rank one over $\euls{O}(G/Z)$.
\end{lemma} 

\begin{proof}
	The space $\euls{O}^{\eta}_{G/Z}$ is a module over $\eD(G)^Z$. As explained in the proof of \cite[Proposition~9.1.2]{BG}, $$
	\eD(G) (\tau_{L} - \eta)(\g) = \eD(G)(\tau_R + \eta)(\g). 
	$$
	Thus, if $x \in \mf{z}$, then $\tau_R(x) - \psi(x)$ belongs to $\eD(G) (\tau_{L} - \eta)(\g)$. Hence $(\tau_R - \psi)(x) \cdot 1 = 0$ in 
	$\euls{O}_G^{\eta}$. If $s \in \euls{O}_{G/Z}^{\eta}$, then it can be represented as $f.1$ for some $f \in \euls{O}(G)^Z$. Since $[\tau_R(x),f] = 0$ for $x \in \mf{z}$, 
	we deduce that $(\tau_R - \psi)(x) s = 0$. In other words,   $(\eD(G)(\tau_R + \eta)(\mf{z}))^Z$ annihilates $(\euls{O}_G^{\eta})^Z$ and so $\euls{O}^{\eta}_{G/Z} $ is a left $\dd_\eta (G/Z)$-module.
	
The fact that $\euls{O}^{\eta}_{G/Z}$ is a free  $\euls{O}(G/Z)$-module of rank one is obvious.

\end{proof}

\begin{remark}\label{rem:general-char}  
	Set $\eY = G /Z \times \h_{\reg}$. Then $W$ acts freely on $\eY$, the action commuting with the $G$-action. Moreover,  $\eY / W = \eX$. This implies that $\eD(\eX) = \eD(\eY)^W$ and $\D(\eY) = \eD(G/Z) \boxtimes \D(\h_{\reg})$. 
\end{remark}

\begin{lemma}\label{lem:Xchiquotientdd}
	Keep the above notation. For each character $\eta$, 
	$$
	\eD(\eY) / \eD(\eY) \mf{g}_{\eta} = \left\{ \begin{array}{ll}
		\euls{O}_{G/Z}^{\eta} \boxtimes \eD(\h_{\reg}) & \text{if } \eta |_{\mf{z}} = 0 \\
		& \\
		\ 0 & \textrm{otherwise.} 
	\end{array} \right.
	$$
\end{lemma}   

\begin{proof} 
	The proof of \cite[Proposition~9.1.2]{BG} shows that  
	\begin{align*}
		\eD(\eY) / \eD(\eY) \mf{g}_{\eta} & \cong \left( \euls{O}^{\eta}_G \boxtimes (\eD(\h_{\reg}) / \eD(\h_{\reg}) \mf{z}_{\eta}) \right)^Z\\
		& = \euls{O}^{\eta}_{G/Z} \boxtimes (\eD(\h_{\reg}) / \eD(\h_{\reg}) \mf{z}_{\eta}) .
	\end{align*}
	Since $Z$ acts trivially on $\h_{\reg}$, $\mf{z}_{\eta}=0$ if $\eta(\mf{z}) = 0$, and $\mf{z}_{\eta}=\C$ otherwise. 
\end{proof}

\begin{remark}\label{rem:Xchiquotientdd} 
	This lemma shows that $\eD(V)/\eD(V)\g_{\eta}$ is very degenerate if $\eta_{| \mf{z}} \neq 0$. This situation does not occur for us because of \cite[Lemma~6.2(1)]{BLNS}. 
\end{remark} 

Recall the definition of monodromic right $\eD(V)$-modules from Definition~\ref{defn:monodromic}. It is a consequence of \cite[Corollary~6.9]{BLNS} that $m := \dim V - \dim \h$ (as in  Notation~\ref{notation4.11}) equals $\dim G - \dim Z$. Then $\Omega_{G/Z} := \wedge^m \Omega_{G/Z}^1$ is a $G$-equivariant right $\eD_{G/Z}$-module. The next lemma is routine, the fact that $\Omega_{G/Z} \otimes_{\euls{O}_{G/Z}} M$ is a right $\dd_{G/Z}$-module with the given action following from \cite[Proposition~1.2.9(ii)]{HTT}.  

\begin{lemma}\label{lem:leftrightDGmono}
	Tensoring by $\Omega_X$ defines an equivalence 
	$$
	(G,\eta,\eD_{G/Z}) \lmod \stackrel{\sim}{\longrightarrow} \rmod (G,-\eta,\eD_{G/Z}), \quad M \mapsto \Omega_{G/Z} \otimes_{\euls{O}_{G/Z}} M, 
	$$
	where the action of $G$ on $\Omega_{G/Z} \otimes_{\euls{O}_{G/Z}} M$ is given by 
	$$
	(\omega \otimes m) \cdot g = (\omega \cdot g) \otimes (g^{-1} \cdot m)
	$$
	and  the action of $\dd_{G/Z}$ is given by 
	$(\omega \otimes m) \cdot v = (\omega \cdot v) \otimes m - \omega \otimes (v \cdot m)$. \qed
\end{lemma}

We similarly define the right $\eD_{\psi}(G/Z)$-module $\Omega_{G/Z}^{\eta} := (\Omega_G^{\eta})^Z$, where $\Omega_G^{\eta} = \eD(G) / \mf{g}_{\eta} \eD(G)$. The right version of Lemma~\ref{lem:Xchiquotientdd} says that: 
\begin{equation}\label{eq:rightddgchiquotient}
	\eD(\eX) / \mf{g}_{\eta} \eD(\eX) = \left\{ \begin{array}{ll}
		\left( \Omega_{G/Z}^{\eta} \boxtimes \eD(\h_{\reg}) \right)^W & \eta |_{\mf{z}} = 0 \\
		& \\
		0 & \textrm{otherwise.} 
	\end{array} \right.
\end{equation}

We now return to our specific character $\chi$, which satisfies $\chi |_{\mf{z}} = 0$.   

\begin{lemma}\label{lem:extcomputationOGZchi}
	$
	\Ext^m_{\eD (G / Z)} \bigl(\euls{O}_{G/Z}^{\chi},\eD(G / Z)\bigr) = \Omega_{G/Z}^{\chi}.
	$ 
\end{lemma}

\begin{proof}
	For clarity, let $\Ver_{G/Z}$ denote Verdier duality on holonomic    $\eD_{G/Z}$-modules; thus, by definition,
	\[ \Ver_{G/Z}(M) = \Ext^{m}_{\eD(G/Z)}\bigl(M,\, \eD(G/Z)\bigr) \o_{\eD(G/Z)} \Omega_{G/Z} .\]
	Since  Lemma~\ref{lem:leftrightDGmono} implies that 
	$\Omega_{G/Z}^{\chi} \o_{\euls{O}} \Omega_{G/Z} \cong \euls{O}_{G/Z}^{-\chi}$, it suffices to prove that 
	$\Ver_{G/Z}(\euls{O}_{G/Z}^{\chi}) = \euls{O}_{G/Z}^{-\chi}$. 
	
	Let $\mr{pt}$ denote the coset of $1$ in $G/Z$ and $j \colon \{ \mr{pt} \} \hookrightarrow G/Z$ the closed embedding. By \cite[Proposition~9.1.2]{BG} $(G,\chi)$-monodromic modules on $G/Z$ are non-characteristic for $j$ and pull-back defines an equivalence between $(G,\chi)$-monodromic modules on $G/Z$ and $Z$-equivariant modules on $\{ \mr{pt} \}$. Therefore, it suffices to show that   
	$$
	j^* \Ver_{G/Z}(\euls{O}_{G/Z}^{\chi}) \cong j^* \euls{O}_{G/Z}^{-\chi}
	$$
	as $Z$-equivariant modules. Since $(G,\chi)$-monodromic modules on $G/Z$ are non-characteristic for $j$, \cite[Theorem 2.7.1(ii)]{HTT} implies that 
	$$
	j^* \Ver_{G/Z}(\euls{O}_{G/Z}^{\chi}) \cong  \Ver_{\{ \mr{pt} \}}(j^* \euls{O}_{G/Z}^{\chi}).
	$$
	We have $j^* \euls{O}_{G/Z}^{\chi} \cong \euls{O}_{\{ \mr{pt} \}}$. Here $\euls{O}_{\{ \mr{pt} \}}$ is the one-dimensional $Z$-equivariant module that is isomorphic to the trivial module as a $Z$-module. Thus the Verdier dual is again $\euls{O}_{\{ \mr{pt} \}}$. But this is isomorphic to $j^* \euls{O}_{G/Z}^{-\chi}$ as a $Z$-equivariant module, as required. 
\end{proof}

\begin{proposition}\label{prop:isoofExtmonopenset}
As right $\D(\eX)$-modules,  
\[
	\Ext^m_{\D(\eX)} (\D(\eX) / \D(\eX) \g_{\chi},\D(\eX)) \  \cong \  \eD(\eX) / \g_{\chi} \D(\eX).
	\]
\end{proposition}

\begin{proof}  We have
	$$	\begin{array}{rll}
		&  \Ext^m_{\eD(\eX)} \bigl(\eD(\eX)  / \eD(\eX) \g_{\chi}, \, \eD(\eX) \bigr) =    \\    \noalign{\vskip 4pt}
		& \qquad = \Ext^m_{\eD(\eY)^W} \bigl(\eD(\eY)^W  / \eD(\eY)^W \g_{\chi}, \, \eD(\eY)^W\bigr) & \text{by Remark~\ref{rem:general-char}} \\ \noalign{\vskip 4pt}
		& \qquad  = \Ext^m_{\eD(\eY)} \bigl(\eD(\eY) / \eD(\eY) \g_{\chi}, \, \eD(\eY)\bigr)^W   &\text{by Lemma~\ref{lem:finitegroupDequiv}}\hfill \\  \noalign{\vskip 4pt}
		& \qquad = \Ext^m_{\eD(\eY)} \bigl(\euls{O}_{G/Z}^{\chi} \boxtimes \eD(\h_{\reg}),\,
		\eD(G / Z) \boxtimes \eD(\h_{\reg})\bigr)^W   &
		\text{by Lemma~\ref{lem:Xchiquotientdd}}\hfill \\   \noalign{\vskip 4pt}
		& \qquad  = \Bigl(\Ext^m_{\eD(G / Z)} \bigl(\euls{O}_{G/Z}^{\chi}, \, \eD(G / Z)\bigr) \boxtimes \eD(\h_{\reg})\Bigr)^W & \\   \noalign{\vskip 4pt}
		& \qquad  = \bigl(\Omega_{G/H}^{\chi} \boxtimes \eD(\h_{\reg})\bigr)^W   &\text{by Lemma~\ref{lem:extcomputationOGZchi}. }\\    \noalign{\vskip 4pt}
	\end{array}$$
	The isomorphism now follows from \eqref{eq:rightddgchiquotient}. 
\end{proof}

\subsection*{Preliminary results on projective objects}

Before we can give a proof of Theorem~\ref{projectivity-theorem-stable}, we require a number of preliminary results. 

\begin{notation}\label{yet-more-notation} 
	Set $\eS=\{\delta^k : k\in \mathbb{N}\}$, regarded as a subset of each of the rings $\ddd$, $\Ak$ and $A_{\cdag}$. As $\delta$ acts locally ad-nilpotently in each ring, $\eS$ is a left and right Ore set in those rings. 
	For a left, respectively right $\ddd$-module $L$, write $\Dvar(L) = \Ext^m_\ddd(L,\ddd)\in \rmod \ddd$,  respectively 
	$\Dvar^{op}(L) = \Ext^m_\ddd(L,\ddd)\in \ddd\lmod $, where $m=\dim V-\dim \h$, as in Notation~\ref {notation4.11}.  	Note that $\Dvar$ is not quite the functor $\DD_\ddd$ from Notation~\ref{main-notation2} since they have different exponents. 
\end{notation}

\begin{lemma}\label{double-dual}
	Let  $N$ be an  $m$-CM left  $\ddd$-module.  Then:
	\begin{enumerate}
		\item  $\Dvar^{op} \circ \Dvar(N) \cong N$; 
		\item $\End_\ddd(\Dvar(N)) \cong \End_\ddd(N)^{op}$.  In particular, $\Ak=\End_\ddd(\eM)=\End_\ddd(\eM').$
	\end{enumerate}
\end{lemma}

\begin{proof}  (1) Since $\Ext^j_\ddd(N,\,\ddd)=0$ for $j\not=m$,  this  is a standard consequence of 
\cite[Theorem~II.4.15, p.61]{Bj}.       
	
	(2) By functoriality, we have canonical maps
	$$
	\End_\ddd(N) \rightarrow \End_\ddd(\Dvar(N))^{op} \rightarrow \End_\ddd(N) 
	$$
	and 
	$$
	\End_\ddd(\Dvar(N)) \rightarrow \End_\ddd(N)^{op} \rightarrow \End_\ddd(\Dvar(N)) 
	$$
	such that both the composites are the identity. This implies that the morphism $\End_\ddd(N) \rightarrow \End_\ddd(\Dvar(N))^{op}$
	is an isomorphism. The final  assertion of Part~(2) follows from Lemma~\ref{summand-endo22}(3) and Corollary~\ref{m-CM}.
\end{proof}

\medskip 

We next want to fix a progenerator for  the category $\rmod (G,\chi,\ddd)$ as follows.

\begin{notation}\label{module-notation}
By  \cite[Theorem~1]{BGu}, the category $\rmod (G,\chi,\ddd)$ has a progenerator $Q$. By Remark~\ref{rem:K-hyp},  the right $\ddd$-module  $\eKt=\ddd/\g_{\chi}\ddd$  
	 is a projective object in $\rmod (G,\chi,\ddd)$ and so one can assume that $\eKt$  is a  direct summand of 
	$Q$ in $\rmod (G,\chi,\ddd)$. Also, after possibly replacing $Q$ by some $Q^{(m)}$ and using Lemma~\ref{lem:Gequivisfullsub}(3), we can also assume 
	that there is a surjection $\alpha: Q\twoheadrightarrow \eM'$ in the category $\rmod (G,\chi,\ddd)$. We fix a progenerator $Q$ satisfying these properties. 
	
	Set $U = \End_\ddd(Q)$. Write $ H_Q=\Hom_\ddd(Q,\eM')$, which is an $(\Ak,U)$-bimodule by Lemma~\ref{summand-endo22}(3), 
	and ${\widehat{H}}_Q  = \Hom_{U}(H_Q,U)$, which is a $(U,\Ak)$-bimodule.  We write $H_Q\widehat{H}_Q$ for the natural image of $H_Q\otimes_U\widehat{H}_Q$ inside $\End_U(H_Q)$.
	Finally, let $\eK$ be the summand of $\eKt$ defined in Lemma~\ref{lem:K-hyp} and 
	set $H_\eK =\Hom_\ddd(\eK,\, \eM')$	 which is naturally an $(\Ak,A_{\cdag})$-bimodule. 
\end{notation}

The following result is a standard consequence of  the fact that $Q$ is a progenerator; see \cite[Exercise, p. 55]{BassKtheory}.

\begin{lemma}\label{equiv-cor}
	\begin{enumerate}
		\item The map $F \colon \rmod (G,\chi,\ddd) \rightarrow \rmod U$ given by $F(N) = \Hom_\ddd(Q,N)$ is an equivalence of categories.
		\item The  module $\eM'$ is projective in $\rmod (G,\ddd)$ if and only if $H_Q$
		is a projective right $U$-module.
		\item  The map 
		$
		F_{N_1,N_2} \colon \Hom_{\ddd}(N_1,N_2) \rightarrow \Hom_U(F(N_1),F(N_2))
		$
		is an isomorphism for all $N_1,N_2 \in \rmod (G,\chi,\ddd)$.   \qed
	\end{enumerate}  
\end{lemma}

The next few results   provide equivalent conditions to the projectivity  of $\eM'$.   
\begin{lemma}\label{whenisHproj}   \begin{enumerate}
		\item   $\End_U(H_Q)\cong \Ak$. 
		\item Moreover, 
		$H_Q$ is a projective right $U$-module if and only if $H_Q \widehat{H}_Q \neq 0$.  
	\end{enumerate}
\end{lemma}

\begin{proof}  (1)   Lemma~\ref{double-dual} implies that $\Ak =\End_\ddd(\eM')$ while  Lemma~\ref{equiv-cor} implies that 
\[   
		\Ak  = \End_\ddd(\eM')= \End_{\rmod (G,\ddd)}(\eM')  \cong  \End_U(H_Q)  .
\]  
	
	(2)  The implication 	 ($\Rightarrow$)   is obvious by the Dual Basis Lemma.
	
	Conversely,  by Part~(1) and Hypothesis~\ref{K-hyp},   $\Ak = \End_U(H_Q)  $	is a simple ring. Thus, if  $ H_Q\widehat{H}_Q\not=0$,  then 
	it  is a nonzero ideal of $ \End_U(H_Q)  $ and hence  equals  $ \End_U(H_Q)  $.  
	By the Dual Basis Lemma, $H_Q$ is then a projective $U$-module.   \end{proof}

\begin{lemma}\label{newlemma8.14}
	Assume that there exists $0\not= \xi\in \Hom_\ddd(\eM',\, Q)$. Then: 
	\begin{enumerate}
		\item $\widehat{H}_Q = \Hom_{U}(H_Q,U) \not=0$;
		
		\item if   there exists $\xi' \in \Hom_\ddd(Q,\eM')$ such that $\xi' \circ \xi \neq 0$, then $H_Q \widehat{H}_Q\neq 0$.
	\end{enumerate}
	
\end{lemma}

\begin{proof}
	(1) We keep the notation of  Lemma~\ref{equiv-cor}(3). Since  the map $F_{\eM',Q} $  
	is an isomorphism,  
	\[ 
	\begin{array}{rl}0\not=  \Hom_U\bigl(F(\eM'),F(Q)\bigr) =& \Hom_U\bigl(\Hom_\ddd(Q,\eM') ,\,  \Hom_\ddd(Q,Q)\bigr)  \\
		\noalign{\vskip 5pt}
		= &   \Hom_U(H_Q,U) = \widehat{H}_Q   ,\end{array} \]
	as desired.

	(2) Recall that $H_Q$ is a $(\Ak ,U)$-bimodule while $\Hom_\ddd(\eM',Q)$ is a $(U,\Ak )$-bimodule. Thus there is a multiplication map of $\Ak $-bimodules 
	$$
	\alpha: \Hom_\ddd(Q,\eM') \otimes_U \Hom_\ddd(\eM',Q) \rightarrow \End_\ddd(\eM'), \quad \theta \otimes \eta \mapsto \theta \circ \eta.
	$$
	Here, $\alpha\not=0$ since $\alpha(\xi' \otimes \xi ) =  \xi' \circ \xi\not=0$. We now use   the fact that $F$ is an equivalence. In more detail, by  Lemma~\ref{whenisHproj}(1),  
	$\Ak \cong  \End_U(H_Q)$. Now, as in the proof of Part~(1) we have  the following commutative diagram, 
	in which  the vertical maps  are isomorphisms.
	$$
	\begin{tikzcd}
		H_Q \otimes_U \Hom_\ddd(\eM',Q) \ar[r, "\alpha" ] \ar[d,"F_{Q,\eM'} \otimes F_{\eM',Q}"] & \End_\ddd(\eM') \ar[d,"F_{\eM',\eM'}"] \\
		\Hom_U(U,H_Q) \otimes_U \Hom_U(H_Q,U) \ar[r,] \ar[d,equal] & \End_U(H_Q) \ar[d,equal] \\
		H_Q \otimes \widehat{H}_Q \ar[r,  "\beta "] & \Ak .
	\end{tikzcd}
	$$
	Since $\alpha\not=0$ it follows that $\beta\not=0$.
\end{proof}

\begin{corollary} \label{new-equiv3}
	Assume that  $0\not=\xi' \in \Hom_\ddd(Q,\eM')$ and $0\not= \xi\in \Hom_\ddd(\eM',\, Q)$ satisfy $\xi'\circ\xi\not=0$. 
	Then $\eM'$ is projective in $\rmod (G,\ddd)$. 
\end{corollary}

\begin{proof}
	By Lemmata~\ref{equiv-cor}(2) and~\ref{whenisHproj} it suffices to show that $H_Q \widehat{H}_Q \neq 0$ in $\Ak $. This follows from Lemma~\ref{newlemma8.14}(2). 
\end{proof}

The problem with working with $Q$ is that, since it   can have  many summands, its endomorphism ring is unlikely to simple. 
This means we  do not have an analogue of Lemma~\ref{visible2} for $Q$. However, we can restrict to $\eK$ since, as the next result shows, 
it is easy to move between $Q$ and $\eK$.

\begin{corollary}\label{submoduletest}  
	Assume that there exist functions $\phi\in H_\eK=\Hom_\ddd(\eK,\, \eM')$ and $\psi\in \Hom_\ddd(\eM',\, \eK)$ with $\phi\circ\psi \not=0$.
	Then $\eM'$ is projective in $\rmod (G,\chi, \ddd)$.     
\end{corollary}

\begin{proof}  
	By Notation~\ref{module-notation}, $\eK$ is an $\ddd$-module summand of $Q$, say $Q=\eK\oplus \euls{L}$ and so we can regard $\phi$ as an element 
	of $H_Q$ by defining $\phi(\euls{L})=0$. Similarly, $\psi\in \Hom_\ddd(\eM',\, Q)$. Thus,  the map $\phi\circ \psi$ is still nonzero
	when regarded as an element of $H_Q\circ  \Hom_\ddd(\eM',\, Q)$.  Now apply Corollary~\ref{new-equiv3}.
\end{proof}

The advantage of using $H_\eK$ comes from the following result.   

\begin{lemma}\label{lem:subbimodulestorsionfree2}
	The left $\Ak$-module	$H_\eK$ is $\delta$-torsionfree. Similarly $\Hom_\ddd(\eM',\eK)$ is $\delta$-torsionfree  as a  left  $A_{\cdag}$-module.
\end{lemma}

\begin{proof}
	If some $0\not=\theta\in H_\eK$ is   left $\delta$-torsion, then so is   $\theta(k)$ for $k\in \eK$,  since 
	$ (\delta^j\theta)(k)=\delta^j \theta(k)$ for $j\in\mathbb{N}$. This    contradicts the fact that,
	by Lemma~\ref{visible2},    $\eM'$ is $\delta$-torsionfree as a left $\Ak $-module. 
	Using the fact that $\eK$ is $\delta$-torsionfree as a left $A_{\cdag}$-module (see Remark~\ref{rem:K-hyp}),   	the same proof works for  $ \Hom_\ddd(\eM',\eK)$.
\end{proof}

\begin{lemma}\label{ore101}    We have 
	$    \eM'_{\eS}=\Ext^m_\ddd(\eM,\, \ddd)\otimes_\ddd\ddd_{\eS}      \cong  \Ext^m_{\ddd_{\eS}} ({}_{\eS}\eM ,\ddd_{\eS}),   $
	as right $\ddd_{\eS}$-modules. 
\end{lemma}

\begin{proof}  
	Recall that, by Lemma~\ref{visible2},   $\eM$  is $\delta$-torsionfree on both sides and so, by Lemma~\ref{ore91}, the localisations ${}_\eS \eM$ and 
	$ \eM_{\eS}$  are isomorphic  bimodules. Thus the lemma is a special case of Lemma~\ref{Brown-Levasseur}.
\end{proof}

\begin{remark}\label{ore102} By Proposition~\ref{semiprime2},  $\eMt_{\eS} = \eM_{\eS}$ and, similarly, Lemma~\ref{lem:K-hyp} implies that   $\eKt_{\eS} = \eK_{\eS}$. Therefore, any results stated for $\eMt_{\eS}$ also hold for $\eM_{\eS}$, and conversely.  For example, by Lemma~\ref{ore101}, 
	\begin{equation}\label{equ:ore102}
		\eM'_{\eS} \cong \Ext^m_{\ddd_{\eS}}({}_{\eS}\eM , \ddd_{\eS}) 
		\cong \Ext^m_{\ddd_{\eS}}({}_{\eS}\eMt , \ddd_{\eS})
		= \Ext^m_{\ddd_{\eS}}( (\ddd_{\eS}/\ddd_{\eS}\g_{\chi}) , \ddd_{\eS}).
	\end{equation}
	Similarly, results like Proposition~\ref{prop:isoofExtmonopenset} can be applied to $\eM'_{\eS}$. 
\end{remark}

The next result   is standard and similar to Lemma~\ref{Brown-Levasseur},  but   we do not know a suitable  reference.

\begin{lemma}\label{ore100}  
	{\rm (1)}  As left $A_{\cdag}$-modules, $ {}_{\eS}\hskip -2pt\Hom_\ddd(\eM',\eK) 
	\cong  \Hom_{\ddd_{\eS}}(\eM'_{\eS},\eK_{\eS}).$
	\begin{enumerate}
		\item[(2)]  Similarly,  $ {}_{\eS}\hskip -2pt\Hom_\ddd(\eK,\eM')   \cong  \Hom_{\ddd_{\eS}}(\eK_{\eS},\eM'_{\eS})$ as left $\Ak $-modules.
		
		\item[(3)] If 
		$  \Hom_{\ddd_{\eS}}(\eK_{\eS},\eM'_{\eS})\circ 
		\Hom_{\ddd_{\eS}}(\eM'_{\eS},\eK_{\eS})\not=0$,  then 
		$$    \Hom_{\ddd}(\eK,\eM')\circ 
		\Hom_{\ddd}(\eM',\eK)\not=0.$$
	\end{enumerate}
\end{lemma}

\begin{proof}      (1)  By definition,   $ {}_{\eS}\hskip -2pt\Hom_\ddd(\eM',\eK)=  (A_{\cdag})_{\eS}\otimes_{A_{\cdag}} \Hom_\ddd(\eM',\eK)$.
	By adjointness \cite[Theorem~2.75]{Rot} we have 
	\[\begin{array}{rl}
		\Hom_{\ddd_{\eS}}(\eM'_{\eS},\, \eK_{\eS}) &\cong 
		\Hom_{\ddd_{\eS}}(\eM'\otimes_\ddd\ddd_{\eS},\, \eK_{\eS})\\
		&  \cong \Hom_{\ddd }(\eM',\, \Hom_{\ddd_{\eS}}(\ddd_{\eS},\, \eK_{\eS}))
		=   \Hom_{\ddd }(\eM',\, \eK_{\eS}).\end{array}\]
	As $\delta$ acts torsionfreely on $\eK$  by Remark~\ref{rem:K-hyp},  
	$\eK\subseteq \eK_{\eS}$, and  it follows that    
	\[  \Hom_\ddd(\eM', \eK)\hookrightarrow 
	\Hom_{\ddd }(\eM',\, \eK_{\eS})\cong \Hom_{\ddd_{\eS}}(\eM'_{\eS},\, \eK_{\eS}) .\]
	Note that $\Hom_\ddd(\eM', \eK)$ is a left $A_{\cdag}$-module and, similarly, 
	$\Hom_{\ddd_{\eS}}(\eM'_{\eS},\, \eK_{\eS}) $ is a left $(A_{\cdag})_{\eS}$-module, both 
	by left multiplication. Moreover,  $\Hom_\ddd(\eM', \eK)$ is left $\eS$-torsionfree
	by Lemma~\ref{lem:subbimodulestorsionfree2}.
	Thus  by universality, 
	\[{}_{\eS}\hskip -2pt \Hom_\ddd(\eM', \eK) =  {(A_{\cdag})_{\eS}}\otimes_{A_{\cdag}}\Hom_\ddd(\eM', \eK) \hookrightarrow \Hom_{\ddd_{\eS}}(\eM'_{\eS},\, \eK_{\eS}) .\]

	For the opposite   inclusion,  let $\theta\in \Hom_{\ddd}(\eM',\, \eK_{\eS})$   and  write $\eM'=\sum_{i=1}^p x_i\ddd$ for some $x_i\in \eM'$.  Then
	${}_{\eS}\eK=\eK_{\eS}$ by Lemma~\ref{ore91}, and so  we can find $c\in \eS$ such that $\theta(x_i)=c^{-1}y_i$ for some $y_i\in \eK$ and all 
	$1\leq i\leq p$. Thus $\phi=c\theta\in \Hom_\ddd(\eM', \eK)$ and so 
	$ \theta=c^{-1}\phi \in (A_{\cdag})_{\eS}\otimes_{A_{\cdag}}\Hom_\ddd(\eM', \eK) = {}_{\eS}\hskip -2pt \Hom_\ddd(\eM', \eK).$  
	
	(2)  By Lemma~\ref{lem:Gequivisfullsub}(4), $\eS$ acts locally ad-nilpotently on $\eM'$.  Moreover $\delta$ acts torsionfreely on $\eM'$ from both sides
	by Lemma~\ref{visible2} 
	and on $\Hom_\ddd(\eK,\eM')$ from the left  by  Lemma~\ref{lem:subbimodulestorsionfree2}. Thus    the proof of (1) also works here.

	(3)    
	Obviously if there exists $\theta'\in \Hom_\ddd(\eK,\eM')$ and $\phi'\in \Hom_\ddd(\eM',\eK)$ such that $\theta'\circ \phi' (\eM')\not=0$, 
	we are done, so we suppose not throughout the proof.  
	
	By hypothesis, and using  Parts~(1) and~(2), we can find 
	\[c_1^{-1}\theta\in (\Ak)_{\eS}\otimes_{\Ak } \Hom_\ddd(\eK,\eM')\quad \text{and}\quad
	c_2^{-1}\phi\in (A_{\cdag})_{\eS}\otimes_{A_{\cdag}} \Hom_\ddd(\eM',\eK)\]
	 such that $\bigl(c_1^{-1}\theta\circ c_2^{-1}\phi\bigr)(mc_3^{-1})\not =0$.
	Here,  each $c_i\in \eS$, $\theta\in \Hom(\eK,\eM')$,  and
	$\phi\in \Hom(\eM',\eK)$, while $m\in \eM'$. 
	Now $\eM'$ is $\eS$-torsionfree on both sides by Lemma~\ref{visible2}, so we can multiply by $c_1$ and $c_3$ on the two sides 
	of this equation  to get  
	\begin{equation}\label{lastly}
		\theta( c_2^{-1}\phi(m ))\not =0\qquad\text{and so}\qquad \phi(\eM')\not=0.
	\end{equation}

	Set
	\[J \ =   \left\{ \sum \phi'(\eM') : \phi'\in  \Hom_\ddd(\eM',\eK).\right\}\] 
	The space $J$ is a right $\ddd$-submodule of $\eK$.  
	Moreover, by the first paragraph of the proof necessarily $J\subseteq \ker_{\eK_{\eS}}(\theta),$
	while  $J\supseteq \phi(\eM')\not=0$  by \eqref{lastly}.  
	On the other hand,   $\Hom(\eM',\eK)$ is a 
	$(A_{\cdag},\Ak )$-bimodule and hence a $\mathbb{C}[\delta]$-bimodule,  and so $ c'\phi'\in \Hom_\ddd(\eM',\eK)$ for any 
	$\phi'\in \Hom_\ddd(\eM',\eK)$ and $c'\in \mathbb{C}[\delta]$. 
	Therefore $J$ is a left	$\C[\delta]$-module.  As $J\subseteq \eK$, it   is $\eS$-torsionfree on both sides  and hence ${}_{\eS}J=J_{\eS}$ by   Lemma~\ref{ore91}. But   under the embedding of $\eS$ into $\ddd$, and regarding $\theta$ as lying 
	in $ \Hom_{\ddd_{\eS}}(\eK_{\eS},\eM'_{\eS})$    we have 
	\[ \theta(J_{\eS})= \theta( J\otimes_\ddd\ddd_{\eS} ) = \theta(J) \otimes_\ddd\ddd_{\eS}  = 0.\]

	Finally,   since  $c_2^{-1}\phi(\eM')\in  {}_{\eS} J=  J_{\eS}$  this implies that 
	$\theta(c_2^{-1}\phi)(\eM')=0$, contradicting \eqref{lastly}.
\end{proof}

We remark that results like Lemma~\ref{ore100} do need a regularity assumption,  as the following   easy example shows.

\begin{example}\label{Ore-example} Let $U=U(\mathfrak{sl}(2))$, and take the usual basis $\{e,h,f\}  $ 
	for $\mathfrak{sl}(2)$. Let $\nabla$ be the injective hull of the trivial $\mathfrak{sl}(2)$-module $\mathbb{C}$  in
	category $\euls{O}$; thus $\nabla$ is indecomposable of length 2 with $\nabla/\euls{O}=M$ being a simple Verma module. 
	Then 
	\begin{enumerate} 
		\item $x=f$ acts locally ad-nilpotently on $U$ and so $\eS=\{x^j\}$ is an Ore set in $U$;
		\item $\Hom_U(M,\nabla)=0$  but 
		\item$\Hom_{U_{\eS}}(M_{\eS},\,\nabla_{\eS})\not=0.$
	\end{enumerate}
\end{example}

\begin{proof} (1) and (2) are obvious. For  (3) note  that 
	$M_{\eS} = \nabla_{\eS}$  since $x$ kills the trivial module $\C$. On the other hand,  
	$M_{\eS}\not=0$ since
	$M$ is a torsionfree $ \mathbb{C}[x]$-module. 
\end{proof}

\subsection*{The proof of Theorem~\ref{projectivity-theorem-stable}: $\eM'$ is projective in $\rmod(G,\chi, \ddd)$.}         

Finally, we are in a position to give a proof of Theorem~\ref{projectivity-theorem-stable}. 

By   \eqref{lem:stablepolaropen}, $\eX= V_{\reg}  \cong G \times_N \h_{\reg}$.  Set
$H_\eK = \Hom_\ddd(\eK,\eM')$ as in Notation~\ref{module-notation}. By Corollary~\ref{submoduletest}, it suffices to prove that $$
H_\eK \circ \Hom_\ddd(\eM',\eK) \not=0.
$$

By Lemma~\ref{ore100}(3),  it therefore  suffices to prove that  	 
\begin{equation}\label{eq:nonvanishingpolarproduct}
	\Hom_{\eD(\eX)}(\eK_{\eS},\eM'_{\eS}) \circ\Hom_{\eD(\eX)}(\eM'_{\eS},\eK_{\eS})\not=0.
\end{equation}
By Remark~\ref{ore102},   
$\eK_{\eS} = \eKt_{\eS} =  \eD(\eX ) / \mf{g}_{\chi} \eD( \eX )$. 
Combining Equation~\ref{equ:ore102} and Proposition~\ref{prop:isoofExtmonopenset} therefore shows that 
\begin{align*}
	\Hom_{\eD(\eX)}  &  \bigl(\eM'_{\eS},\eK_{\eS}\bigr) 
	\ \cong\   \Hom_{\eD(\eX)}  \bigl(\Ext^m_{\eD(\eX)} (\eD(\eX) / \eD(\eX) \g_\chi,  \eD(\eX)),\, \eK_{\eS}\bigr) \\
	& = \Hom_{\eD(\eX)} (\eK_{\eS}, \eK_{\eS})
\end{align*}
and, similarly, 
$
\Hom_{\eD(\eX)}(\eK_{\eS},\eM'_{\eS}) = \Hom_{\D(\eX)} (\eK_{\eS}, \eK_{\eS}).
$  
However, by \eqref{eq:rightddgchiquotient}, $\eK_{\eS} \neq 0$. Thus, 
$$
\Hom_{\D(\eX)} (\eK_{\eS}, \eK_{\eS}) \circ \Hom_{\D(\eX)} (\eK_{\eS}, \eK_{\eS}) \neq 0
$$
holds automatically. Therefore \eqref{eq:nonvanishingpolarproduct} also holds and the theorem follows. \qed

   \subsection*{The injectivity of $\eG_{\lambda}$ and related admissible modules. }
 
Since we now know that $\eM'=\Ext_\ddd^{m}(\eM,\ddd)$ is a projective object, it is fairly easy to deduce the main result of this section, the injectivity of $\eG_{\lambda}$,  from the projectivity of  related objects. We first need to 
compute the images of the various functors defined in Notation~\ref{main-notation2}. 

\begin{lemma}\label{G-projective2}
	\begin{enumerate} 
		\item $\BH^\perp= -\otimes_{\Ak} \eM'$ maps $\Osphop$ to $\eCop$.
		\item $\BD_\ddd=\Ext_\ddd^{m+n}(-,\ddd)$ defines a   contravariant duality $\eC\to \eCop$.
	\end{enumerate}
\end{lemma} 

\begin{proof}  
	(1) If $N\in \Osphop$, then $\BH^\perp(N)=N\otimes_{\Ak} \eM'$ is monodromic by Lemma~\ref{equivariant tensors}.
	Since $(\Sym  V)^G_+\cong (\Sym  \h)_+^W$ acts locally ad-nilpotently on $\eM'$ by Lemma~\ref{lem:Gequivisfullsub}(4),
	while  $(\Sym  \h)^W$ acts locally finitely on $N$ by assumption, the subalgebra $(\Sym  V)^G$ of $\ddd$ will act locally finitely on $N\otimes_{\Ak}  \eM'$. 
	Thus $\BH^\perp(N) \in \eCop$.    
	
	(2) Let $N\in \eC$. Then  $\BD_\ddd(N)= \Ext^{n+m}(N,\ddd)$ is a monodromic  module  by Lemma~\ref{Ext-equivariance}. Choose $0\not=d \in (\Sym  V)^G$
	and  regard $N\in \Bi(\ddd,d)$, in  the sense of Definition~\ref{Bi-definition},  by making $d$ act trivially on the right. Then, Corollary~\ref{ad-nil prop}(2) 
	shows that 
	$(\Sym  V)^G$ acts locally finitely on $\BD_\ddd(N)$. Thus $\BD_\ddd$  defines a contravariant functor  $\eC\to \eCop$.  
	Every $N\in \eC$ is holonomic and so $N\cong \BD_{\ddd^{\mathrm{op}}}\circ\BD_\ddd(N)$ by  \cite[Theorem~II.4.15, p 61]{Bj}.
	Thus $\BD_\ddd$ is a duality, and the lemma is proven.   
\end{proof}

\begin{lemma}\label{G-projective3} 
	Assume that Hypothesis~\ref{K-hyp} holds. Let $\eP \in \Osph$ be a projective object and set $\eG=\eM\otimes_{\Ak}\eP$. Then  
	 $\BD_\ddd( G )=\Ext^{m+n}_\ddd(\eG, \ddd)$ is projective in $\eCop$.
\end{lemma}

\begin{proof}   By Lemma~\ref{G-projective2},  $\BD_\ddd(\eG)\in \eCop$. Set $ \eP' =\BD_{\Ak} (\eP).$
	Then  Theorem~\ref{intertwining} implies that  
	\begin{equation}\label{equ:G-projective3}
		\BD_\ddd(\eG)\ = \ \BD_\ddd\circ \HHleft (\eP) \ = \  \HHright \circ  \BD_{\Ak} (\eP) \ = \ \BH^\perp(\eP').
	\end{equation}
	By Corollary~\ref{P'-description}, $\eP'  $ is  projective in $\Osphop$. Now, using adjunction,
	\begin{equation}\label{equ:G-projective33}
		\Hom_\ddd(\BH^\perp(\eP'), -)  =  \Hom_\ddd(\eP'\otimes_{\Ak}  \eM',\,-)  =  \Hom_{\Ak} \left(\eP',\, \Hom_\ddd(\eM', -)\right).
	\end{equation}
	
	Let $L_1\to L_2\to 0$ be an exact sequence in $\eCop$.  Since we are assuming that Hypothesis~\ref{K-hyp}  holds, Theorem~\ref{projectivity-theorem-stable} says that $\eM'$ is a projective monodromic $\dd(V)$-module. Therefore, $\Hom_\ddd(\eM',L_1)\to \Hom_\ddd(\eM',L_2)\to 0$ is also  exact.  We do not know that  the $\Hom_\ddd(\eM',L_\ell)$ are finitely
	generated $\Ak $-modules and so they may not lie in   $\Osphop$.  However, by  Lemma~\ref{torsion-consequence}(4) 
	any finitely generated $\Ak $-submodule $Z'\subseteq \Hom_\ddd(\eM', L_2)$,   does lie in $\Osphop$  and so any $\ddd$-module morphism $\psi: \eP'\to Z'$
	does lift to a morphism  $\phi:\eP'\to  \Hom_\ddd(\eM',L_1).$  Thus any $\ddd$-module morphism
	$\psi: \eP'\to    \Hom_\ddd(\eM', L_2)  $ also  lifts and so  the complex 
	\[\Hom_{\Ak} \left(\eP',\, \Hom_\ddd(\eM', L_1)\right)\  \to \Hom_{\Ak} \left(\eP',\, \Hom_\ddd(\eM', L_2)\right) \ \to \  0\] is  exact.
	In other words,   $\Hom_{\Ak} (\eP',\, \Hom_\ddd(\eM', -))$ is  an exact functor on $\eCop$.
	
	From Equations~\ref{equ:G-projective3} and \ref{equ:G-projective33}, the functor $\Hom_{\Ak} (\eP',\, \Hom_\ddd(\eM', -))$ is simply $\Hom_\ddd(\BD_\ddd(\eG),-),$  and so  $\BD_\ddd(\eG)$ is  indeed projective in $\eCop$.  
\end{proof} 

Finally, we can combine the  results of this section  to prove  one of our main results on the structure of $\eG_{\lambda}$.  Not only does this result give very useful information about the structure of category $\eC$, but it is also forms a crucial  ingredient in the proof of  Theorem~\ref{thm:adextensionsHC} in the next section; see also Question~\ref{block-question}. The fact that $\eG_{0}$ is injective also plays an important, but subtle, part in the proof of the fact (Theorem~\ref{thm:KZGtwistsimples}) that the shift functor applied to the summands of $\eG_{0}$ is governed by Opdam's KZ-twist. 

\begin{theorem}\label{G-is-injective}
	Assume that Hypothesis~\ref{K-hyp} holds. If $\eG = \eM\otimes_{\Ak}  \eP$ for a projective object $\eP \in \Osph$, then $\eG$ is both projective and injective 
	as an object in $\eC$. 
	
	In particular, this holds when $\eG=\eG_\lambda  = \eM\otimes_{\Ak}  (\Ak /\Ak \mathfrak{m}_\lambda)$ for  $\lambda\in\h^*$.
\end{theorem}

\begin{proof}
	The projectivity of  $\eG$ is  given by Proposition~\ref{G-projective}. On the other hand, by Lemma~\ref{G-projective3},
	$\BD_\ddd(\eG)$ is a projective object in $\eCop$.  By Lemma~\ref{G-projective2},  $\BD_\ddd$ defines a contravariant duality $\eC\to \eCop$ and 
	so  the projectivity of $\BD_\ddd(\eG)$  in $\eCop$  is equivalent 
	to the injectivity of $\eG$  in~$\eC$. 
	
	The claim for $\eG_\lambda$  follows by combining this result with Corollary~\ref{A-projectives}.
\end{proof}

As was noted after Definition~\ref{defn:admissible},  $\eC$ is a length category, and so  it admits a decomposition into blocks. Let $\euls{T}$ denote the Serre subcategory of $ \eC$   generated by all irreducible modules $L$ such that $L$ is $\delta$-torsionfree with $L^G \neq 0$.   

\begin{corollary}\label{cor:euler3}
	Assume that   Hypothesis~\ref{K-hyp} holds and that $\euls{H}_q(W)$ is semisimple. 	
	Then the  category $\euls{T}$ is a union of semisimple blocks of $\eC$, and is the subcategory generated by the $\{\eG_{\lambda} :  \lambda \in \h^*\}$. In particular, there are no extensions between objects of $\euls{T}$ and any other objects of $\eC$.
\end{corollary}

\begin{proof}  By Theorem~\ref{thm:semi-simplicity},  each $\eG_{\lambda}$ is semisimple while, 
  by Proposition~\ref{bigg} combined with  Theorem~\ref{torsionfree}, each irreducible summand $L$ of 
  $\eG_{\lambda}$ belongs to $\euls{T}$.  Moreover, 
   Theorem~\ref{G-is-injective} implies that $L$ is projective and injective in $\eC$. Thus, $L$ generates a semisimple block of $\eC$. 
\end{proof}


\section{Knizhnik-Zamolodchikov Functors And Shift Functors} \label{Sec:Shiftfunctors}

The main goal of this section is to define and study shift functors between categories of admissible $\dd$-modules  for differing   $\chi$. This is a direct analogy of the shift functors defined by Berest-Chalykh \cite{BerestChalykhQuasi} between category $\euls{O}$s for rational Cherednik algebras.  
 In particular, we will show that the image of the Harish-Chandra 
 module under these shift functors is governed by Opdam's KZ-twist.   On the way we will also define an analogue of the Knizhnik-Zamolodchikov functor that acts on admissible modules.
 
  Except where explicitly stated otherwise, {we assume throughout the section that $V$ is a visible, stable, polar represent\-ation, as in Notation~\ref{notation4.11}, and that Hypothesis~\ref{K-hyp} holds.} 
We continue to write    $\Ak=\Ak(W)$.   

\subsection*{Strongly admissible modules}

The KZ-functor to be introduced below is defined on the category of strongly admissible modules, where an admissible $\ddd$-module $L\in \eC$
 is said to be \textit{strongly admissible}\label{defn:strongly} if $L \in \eC_{,0}$, in the notation of Lemma~\ref{lem:decomposeadmissible}; equivalently, $L$ is a $(G,\chi)$-monodromic $\dd(V)$-module on which the action of $(\Sym  V)^G_+$ is locally nilpotent. Note that all subfactors of  $\eG_0$ are strongly admissible.
 
 In this section, which is the only place in the article were we use the notion of regular singularities,  we do so in the standard algebraic sense,  as opposed to the local analytic definition  implicit in Conjecture  C1 from the introduction. Explicitly, we say that an algebraic, holonomic D-module is \emph{regular holonomic}  if it  satisfies  the hypotheses of  \cite[Definition 6.1.1]{HTT}.

Let $\eu_V\in \dd(V)$ denote the \emph{Euler operator}\label{notation:gradings} on $V$,
so that  $[\eu_V,x_i] = x_i$ for coordinate vectors    $x_i \in V^*\subset    \C[V]$.

\begin{proposition}\label{prop:euler}  Assume that Hypothesis~\ref{K-hyp} holds and 
	let $L \in \eC_{,0}$. Then $L$ is   regular holonomic   and $\eu_V$ acts locally finitely on $L$. 
\end{proposition}

\begin{proof}
	 Let      
	$
	\mathbb{G} \colon \left(G,\chi,\dd(V)\right)\lmod \stackrel{\sim}{\longrightarrow} (G,\psi,\dd(V^*)) \lmod,
	$
 for  $\psi = \chi+\mr{Tr}_V$, 	be the Fourier transform defined by Lemma~\ref{lem:FourierGequiv}. If 
	$(G,\psi,\eD(V^*))_{\euls{N}}\lmod$ denotes the full subcategory of $(G,\psi,\eD(V^*))\lmod$ consisting of all modules supported on the nilcone $\euls{N}(V^*)$, then 
	$\mathbb{G}$ restricts to an equivalence $\eC_{,0} \stackrel{\sim}{\longrightarrow} (G,\psi,\eD(V^*))_{\euls{N}}\lmod$  (see the discussion in \cite[Section~7]{AJM}). If $L \in \eC_{,0}$ is
	irreducible then $\mathbb{G}(L) = IC(\euls{O},\euls{L})$ for some nilpotent $G$-orbit $\euls{O}$ and $(G,\psi)$-monodromic local system $\euls{L}$; see \cite[Theorem~3.4.2]{HTT}. In particular, $\mathbb{G}(L)$ is   regular holonomic  by \cite[Lemma~2.2.2(ii)]{BG}. Let $\C^{\times}$ act on $V^*$ by dilations, so that the vector field $\eu_{V^*}$ is the differential of this action. Then $\euls{O}$ is $\C^{\times}$-stable, which implies that the irreducible local system $\euls{L}$ is $\C^{\times}$-monodromic. Hence, $\eu_{V^*}$ acts locally finitely on $IC(\euls{O},\euls{L}) = \mathbb{G}(L)$ by \cite[Proposition~7.12]{BrylinskiFourier}. This implies that $\eu_V$ acts locally finitely on $L$. Moreover, by \cite[Th\'eor\`eme~7.24]{BrylinskiFourier}, this implies that $L$ is also  regular holonomic.  
	
	Thus, we have shown that the statement of the proposition holds for the irreducible objects in $\eC_{,0}$. But the properties of being $\eu_V$-locally finite or  being   regular holonomic    are both closed under extensions and so   these properties hold for all modules in $\eC_{,0}$.  
\end{proof}
 
\begin{remark}\label{rem:euler} Proposition~\ref{prop:euler}  implies   that, from the geometric point of view, strongly admissible $\dd(V)$-modules are the natural category of modules to consider since they are mapped, via the Riemann-Hilbert correspondence, to the category of equivariant perverse sheaves on a polar representation that generalise Lusztig's character sheaves on the adjoint representation (see, for example \cite{MirkovicVilonen, VilonenSymVanishing, VilonenSymChar, VX2}). 
\end{remark}

We do not know in general whether  $\eC$ has enough projective (or injective) objects. However, 
Lemma~\ref{lem:euler} gives  one case where this is true. Recall from Section~\ref{Sec:polarreps} that $\mu: T^*V\to \g^*$ denotes the moment map. 
 
\begin{lemma}\label{lem:euler}
	Assume that the Hamiltonian reduction $\mu^{-1}(0) \git G$ is irreducible (but not necessarily reduced). 
	Then the category $\eC_{,0}$ has enough projectives and injectives. In particular, every object in $\eC_{,0}$admits a projective cover and an  injective hull. 
\end{lemma}

\begin{proof}
	We wish to apply the results from \cite{BB}, which in turn build on results of Ginzburg \cite{Primitive}. To do so, we need to check that hypothesis (F1) of \cite{BB} holds; that is, we much show that 
	$\C[\mu^{-1}(0) \git G] = \C[\mu^{-1}(0)]^G$ is a finite module over
	$\C[V]^G \o (\Sym V)^G$. It suffices to show that $\C[(\mu^{-1}(0) \git G)_{\mr{red}}]$
	 is a finite module over  $\C[V]^G \o (\Sym V)^G$. By \cite[Lemma~4.5]{BLT}, the normalisation 
	 of $\C[(\mu^{-1}(0) \git G)_{\mr{red}}]$ can be  identified, via the map $\rr \o \rrp$ of
	  \cite[Theorem~7.18]{BLNS}, with $\C[\h \times \h^*]^W$. Since the latter is finite over 
	  $\C[\h]^W \o (\Sym \h)^W \cong \C[V]^G \o (\Sym V)^G$, it follows that 
	  $\C[\mu^{-1}(0) \git G] = \C[\mu^{-1}(0)]^G$  is a finite $\C[V]^G \o (\Sym V)^G$-module.  
	
	Since we have shown in Proposition~\ref{prop:euler} that $\eu_V$ acts locally finitely on every object of $\eC_{,0}$, it follows that 
	strongly admissibility in our situation is the same as $\chi$-admissibility as defined in \cite[Section~4.1]{BB}. 
	Therefore, it follows from \cite[Proposition~1.4]{BB} that $\eC_{,0}$ has enough projectives. The same result holds for the corresponding category $\euls{C}_{\chi,0}^{\mathrm{op}}$ of right modules.  As noted in Definition~\ref{defn:admissible}, 
	$\eC_{,0}$ is a $\Hom$-finite, length category. Thus,   by 
Lemma~\ref{Ext-equivariance} and 	Corollary~\ref{ad-nil prop},  $\mathbb{D}_\ddd=\Ext_\ddd^{n+m}(-,\, \ddd)$ is a contravariant functor between these two categories, and so  $\eC_{,0}$ also  has enough injectives. Thus injective hulls exist. As $\eC_{,0}$ is a $\Hom$-finite, length category,   it follows that projective covers exist too.  
\end{proof}	

As before, let $j \colon V_{\reg} \hookrightarrow V$ denote the open embedding. Given a holonomic module $N$ on $V_{\reg}$ there are (underived) extensions of $N$ to a holonomic module on $V$, related as follows
$$
j_! N \twoheadrightarrow j_{!*} N \hookrightarrow j_* N;
$$
thus $j_!N$ has no $\delta$-torsion factor module while $j_*N$ has no $\delta$-torsion submodule. Each of these operations maps  $(G,\chi)$-monodromic modules to $(G,\chi)$-monodromic modules.

 In general, none of these extensions will be an admissible module even if we begin with a $(G,\chi)$-monodromic integrable connection.  To circumvent this problem, given a $(G,\chi)$-monodromic holonomic $\dd$-module $N$ on $V_{\reg}$, we define $j_{*,\ad} N$ to be the largest submodule of $j_* N$ on which $(\Sym V)^G_+$ acts locally nilpotently. Similarly, define $j_{!,\ad} N$ to be the largest quotient of $j_! N$ on which $(\Sym V)^G_+$ acts locally nilpotently (this exists since $j_! N$ has finite length). The following result  is immediate. 

\begin{lemma}\label{j-star-admissible}
	Let $N$ be a $(G,\chi)$-monodromic holonomic $\dd$-module on $V_{\reg}$. 
	\begin{enumerate}
		\item Both $j_{*,\ad} N$ and $j_{!,\ad} N$ are strongly admissible modules. 
		\item There is a canonical map $j_{!,\ad} N \to j_{*,\ad} N$ whose image $j_{!*,\ad} N$ is 
		contained in $j_{!*} N$.   \qed
	\end{enumerate}  
\end{lemma}

  If   $\mathscr{L} = \eG_0 |_{V_{\reg}}$, then Corollary~\ref{torsionfree-corollary} implies that $j_{!*,\ad} \mathscr{L} = j_{!*}\mathscr{L} \cong \eG_0$.  Theorem~\ref{G-is-injective} allows us to prove  the following much stronger statement.

\begin{theorem}\label{thm:adextensionsHC}
	Assume that Hypothesis~\ref{K-hyp} holds and set  $\mathscr{L} = \eG_0 |_{V_{\reg}}$. Then 
	$$
	j_{!,\ad} \mathscr{L} \ \isom \ j_{!*,\ad} \mathscr{L}  
	 \ \isom  \ j_{*,\ad} \mathscr{L}  \ \cong \  \eG_0.
	$$
	More generally, these isomorphisms hold if $\eP \in \Osph_{\kappa,0}$ is projective, $\eG = \eM \o_{\Aak} \eP$ and $\mathscr{L} = \eG|_{V_{\reg}}$. 
\end{theorem}
	
\begin{proof}  It suffices to prove the result for $\eG= \eM \o_{\Aak} \eP$. 
		By  Corollary~\ref{torsionfree-corollary},   $j_{!*,\ad} \mathscr{L} \cong \eG$.

	 We next show that the embedding  $\alpha: j_{!*,\ad} \mathscr{L}  \hookrightarrow j_{*,\ad} \mathscr{L}$ is an isomorphism. 		
	By  Theorem~\ref{G-is-injective},   $ \eG$
	 is an injective object  in $\eC_{,0}$. Since  the cokernel $C$ of the embedding $\alpha$ 
  also belongs to $\eC_{,0}$, it follows that 
	 $j_{*,\ad} \mathscr{L} \cong j_{!*,\ad} \mathscr{L} \oplus C$.  Since    
	 $\left(j_{!*,\ad} \mathscr{L}\right)|_{V_{\reg}}  =\left( j_{*,\ad} \mathscr{L}\right)|_{V_{\reg}} = \mathscr{L}$, it follows that $C$ is 
	 a $\delta$-torison submodule of $j_{*,\ad} \mathscr{L}$. Finally,  $j_{*,\ad} \mathscr{L}$ is $\delta$-torsionfree
	 by Theorem~\ref{torsionfree} and so  $C = 0$.  
	
	The proof for the  first isomorphism is similar, and is  left to the reader.
 \end{proof}

Even under Hypothesis~\ref{K-hyp},  one typically has $\eG_0 \not= j_* \mathscr{L}$, although Theorem~\ref{thm:adextensionsHC} does imply that $\eG_0$ is the largest admissible submodule of 
$j_* \mathscr{L}$.

For any visible polar representation $V$, Grinberg \cite{GrinbergSymmetric} associates to the quotient map 
$\pi \colon V \to V\git G$ the nearby cycles sheaf $P$. Remarkably, the Fourier transform $\mathbb{F}^*_V(P)$ 
is shown in  \cite[Theorem~3.1]{GrinbergSymmetric}  to be the minimal extension $j_{!*} \mathbb{L}$ of a local system 
$\mathbb{L}$ on $V_{\reg}$ of rank $|W|$. Via the Riemann-Hilbert correspondence, $j_{!*} \mathbb{L}$ 
may be considered as a  regular holonomic $\dd$-module on $V$. Based on 
Corollary~\ref{torsionfree-corollary} and Lemma~\ref{lem:endGHeckealg}, it is natural to expect that  the 
following holds.

\begin{conjecture}\label{conj:nearby-cycle}
	Assume that $(G,V)$ is a visible stable polar representation for a connected reductive group $G$ such that  ${\Aak}=\Ak(W)$ is simple when $\vs = 0$. Then there is an isomorphism  $j_{!*}\mathbb{L} \cong \eG_0$ of $G$-equivariant $\dd(V)$-modules. 
\end{conjecture}  

By Corollary~\ref{torsionfree-corollary}, in order to prove the conjecture  it would suffice to identify  the local systems $\euls{L} :=\eG_0 |_{V_{\reg}} \cong \mathbb{L}$ on the regular locus.   
  
\begin{remark}\label{Grinberg-remark}   
In contrast,  regardless of   Conjecture~\ref{conj:nearby-cycle}, it would be very unlikely that Grinberg's results could be used to prove  torsion-freeness results like Theorem~\ref{torsionfree}. One reason is simply that Grinberg's results hold whether or not $\Aak$ is simple. In contrast, when 
$\Aak$ is not simple,  Proposition~\ref{torsionfree-converse} shows that $\eG_0$ is not torsionfree and so it is certainly not a minimal extension.   Furthermore, Grinberg's  results also hold in the non-stable case where the analogue of Theorem~\ref{torsionfree}  fails rather badly (see Section~\ref{sec:otherexamples}).
\end{remark}
\subsection*{The Knizhnik-Zamolodchikov functor}

We continue to assume that Hypothesis~\ref{K-hyp} holds.	In this subsection we define a version of the $\mr{KZ}$-functor on the category of strongly admissible $\dd$-modules. Combining our results with those of Losev shows that the Harish-Chandra $\dd(V)$-module  $\eG_0$ represents (for Weil generic $\chi$) this geometric $\mr{KZ}$-functor. 

\begin{definition}\label{defn:KZG}
Let $\mr{KZ}: \Osph_0\to \euls{H}_q(W) \lmod$ be the usual KZ-functor, as defined in \cite[Section~5.3]{GGOR}. (This is usually defined as acting on the category $\euls{O}_{\kappa,0}$ for the whole Cherednik algebra but, as noted   after Definition~\ref{Hecke-defn},  this induces a functor on $\Osph_0$.) Recall that the exact  functor 
$\BH \colon (G,\chi,\ddd)\lmod \to \Aak\lmod$  is defined by $	\BH(X)= \Hom_\ddd(\eM,X) $; see Corollary~\ref{summand-endo3}. Then   the \emph{geometric Knizhnik-Zamolodchikov functor}
  $\mr{KZ}_G$ is defined to be the composition 
  \begin{equation}\label{eq:KZGcommute}
\mr{KZ}_G:  \eC_{,0} \ \buildrel{\mathbb{H}}\over{\longrightarrow} \ \Osph_0 
   \ \buildrel{\mr{KZ}}\over{\longrightarrow} \  \euls{H}_q(W). 
  \end{equation}
  As both factors are exact, it follows that  $\mr{KZ}_G$ is an exact functor. 
\end{definition}

\begin{remark}\label{KZ-old-defn} 
One can also  define $\mr{KZ}_G$ directly and  without reference to $\mr{KZ}$ as the composition
$$
	\begin{tikzcd}
	\eC_{,0} \ar[r,"\mathrm{loc}"] \ar[rrrd,bend right=7,"\mr{KZ}_G"'] & 
	(\dd(V_{\reg}),G,\chi) \lmod \ar[rr,"( \gamma_{\vs}(-))^G"] 
	& & \mr{Loc} \, (\h_{\reg}/W) \ar[d,"(( -)^{\rm{an}})^\nabla"] \\
	& & & \pi_1(\h_{\reg}/W) \lmod,
	\end{tikzcd}
	$$ 
where the individual maps are defined as follows. The morphism $\mathrm{loc}$ is just the localisation functor 
$\dd(V)\lmod\to \dd(V_{\reg})\lmod$. The map $\gamma_{\vs}$ is the conjugation $x\mapsto \delta^{-\vs} x \delta^{\vs}$, as in \cite[Equation~(5.2)]{BLNS}, with  $\mr{Loc} \, (\h_{\reg}/W)$ being the category   of integrable connections on $\h_{\reg} / W$ with regular singularities.
The final map is obtained by passing to the analytic coefficients and applying the horizontal sections  functor $(( -)^{\rm{an}})^\nabla$,  in the sense of \cite[p.214]{BellamySRAlecturenotes}. 

The fact that this composition does indeed give the   morphism  $\mr{KZ}_G$ from Definition~\ref{defn:KZG}  requires unpacking the various definitions and is left to the reader.  
\end{remark}

We can now prove our  analogue of Losev's Theorem \cite[Theorem~1.2]{LosevTotally}.

\begin{proposition}\label{prop:KZ-represented}  Assume that $\rad_{\vs}$ is surjective. Then the  functor $\mr{KZ}_G$ is represented by the object $\eG_0\in \eC_{,0}$ if and only if $\Ak(W)$ is simple. 
\end{proposition}

\begin{proof} 
	Assume first that $\Ak(W)$ is simple. Recall that  $\eG_0 = \eM \o_{\Aak}\eQ_0$ for $\eQ_0 = {\Aak}/{\Aak} \mf{m}_0$, and that $\eG_0$ is projective by 
Proposition~\ref{G-projective}. Let $M$ be a strongly admissible module. There are functorial isomorphisms
	$$\begin{aligned}
	\Hom_{\eC}(\eG_0,M) \ =& \ \Hom_{\dd(V)}(\eG_0,M)  \ \isom\   \Hom_{{\Aak}}(\eQ,\mathbb{H}(M))\\
	& \  \isom\  \mr{KZ} \circ \mathbb{H}(M)
	\ \isom \ \mr{KZ}_G(M).
	\end{aligned}
	$$
	Here, the second isomorphism is adjunction, the third is \cite[Theorem~1.2]{LosevTotally} (using the fact that ${\Aak}$ is simple) and the final isomorphism is just  the definition of the functor $\mr{KZ}_G$.
	
	If $\Ak(W)$ is not simple then, by Corollary~\ref{torsionfree-converse2}, $\eG_0$ admits a non-zero torsion quotient $L$. Thus, $\Hom_{\ddd}(\eG_0,L) \neq 0$. But $\mr{KZ}_G(L) = 0$, so $\mr{KZ}_G$ cannot be represented by $\eG_0$. 
\end{proof}

As Losev notes  in the introduction to \cite{LosevTotally}, it is standard that KZ functors are represented by  projective objects, but it is usually hard to determine those objects. Indeed,  in 2005,  Rouquier and Ginzburg independently asked whether $KZ$ is represented by  the corresponding Harish-Chandra module for the Cherednik algebra, and this was answered by   \cite[Theorem~1.2]{LosevTotally}: yes if and only if $\kappa$ is a totally aspherical parameter (equivalently, if and only if $\Ak(W)$ is simple). Proposition~\ref{prop:KZ-represented} answers the analogous question for the functor $\mr{KZ}_G$.
 
\begin{question}\label{block-question}
 Let $\euls{T}$ denote the (direct) sum of all blocks containing an indecomposable summand of $\eG_0$. Since $\eG_0$ is a projective-injective object in $\eC_{,0}$ with endomorphism ring   $\euls{H}_q(W)$
(see   Lemma~\ref{lem:endGHeckealg}),  it is natural to ask if $\mr{KZ}_G \colon \euls{T} \lmod \to \euls{H}_q(W)\lmod$ is a cover. In other words, is $\mr{KZ}_G$ fully faithful on projectives? 
\end{question}

\subsection*{Definition and properties of the shift functors}

 We next want to   define the promised  shift functors between categories $\eC$ for differing characters $\chi$.
For this we  need  some notation   from \cite[Section~4]{BLNS}, where the reader is referred for the details. First,  decompose  $\delta=\prod_{i=1}^k \delta_{i}^{m_i}$ into irreducible factors as in \cite[Equation~(3.10)]{BLNS} and, for each  $i$, write $\theta_i$ for  the character of $G$ corresponding to $\delta_i$.  Then our character $\chi$ is given by $\chi=\sum \vs_i d\theta_i$, where  $\vs=(\vs_1,\dots,\vs_k)\in \C^k$.
 Given a second $k$-tuple $\vs'=(\vs_1',\dots,\vs_k')$ write $\chi'= \sum \vs_j' d \theta_j$ and fix the parameter $\kappa'$ accordingly. Finally, recall that  $\delta^\vs = \delta_1^{\vs_1}\cdots\delta_k^{\vs_k}$. 
 
Given $\vs, \vs' \in \C^k$ as above, let $r_i = \vs_i - \vs_i'$ and  $\mathbf{r} = (r_1, \ds, r_k)\in \C^k$. Then, just as in remark~\ref{KZ-old-defn}, we define the automorphism $\gamma=\gamma_{\mathbf{r}}$ of $\dd(V_{\reg})$, which is the identity on $\C[V_{\reg}]$ and acts on derivations by 
\begin{equation}\label{eq:sigmachchiauto}
	\gamma(v) = v + \sum_{i = 1}^k (\vs_i - \vs_i') \frac{v(\delta_i)}{\delta_i}\qquad\text{for} \ v\in \Der (V_{\reg}).
\end{equation}
The significance of this definition is that, for $x \in \g$,  
\begin{equation}\label{eq:tausigmaauto}
	\gamma(\tau(x)) = \tau(x) + (\chi - \chi')(x), 
\end{equation}
and so  $\gamma$ defines an isomorphism 
$$
\gamma \colon (\dd(V_{\reg}) / \dd(V_{\reg}) \mf{g}_{\chi})^G  \ \isom \  (\dd(V_{\reg}) / \dd(V_{\reg}) \mf{g}_{\chi'})^G.
$$
Informally, $\gamma$ is the automorphism given by conjugation  $D \mapsto \delta^{-{\mathbf{r}}} D \delta^{\mathbf{r}}$.  By 
\eqref{eq:localizatioradiso}, we obtain a commutative diagram of isomorphisms
\begin{equation}\label{eq:sigmaradchicommute1}
	\begin{tikzcd}
		(\dd(V_{\reg}) / \dd(V_{\reg}) \mf{g}_{\chi})^G \ar[dr,"\rad_{\vs}"'] \ar[rr,"\gamma"] & & (\dd(V_{\reg}) / \dd(V_{\reg}) \mf{g}_{\chi'})^G \ar[dl,"\rad_{\vs'}"] \\
		& \dd(\h_{\reg})^W. &  
	\end{tikzcd}
\end{equation}
Indeed, if $D \in \dd(V_{\reg})^G$ and $z \in \C[\h_{\reg}]^W$ then 
\begin{align*}
\rad_{\vs'}(\gamma(D))(z) & = (\delta^{-\vs'} \gamma(D)(\varrho^{-1}(z) \delta^{\vs'}))|_{\h_{\reg}} \\
& = (\delta^{-\vs'} \delta^{\vs' - \vs} D \delta^{\vs - \vs'} (\varrho^{-1}(z) \delta^{\vs'}))|_{\h_{\reg}} \\
& = (\delta^{-\vs} D (\varrho^{-1}(z) \delta^{\vs}))|_{\h_{\reg}} = \rad_{\vs}(D)(z).
\end{align*}

Given a $\dd(V_{\reg})$-module $N$, we can  also twist the action of $\dd(V_{\reg})$ on $N$ by the automorphism $\gamma$; that is, $D \cdot n = \gamma^{-1}(D) n$ for $n \in N^{\gamma} = N$ and $D \in \dd(V_{\reg})$. This allows us to define a shift functor  by   
\begin{equation}\label{defn:T-shift}
\shT_{\vs,\vs'}(M) = j_{*,\ad} (M|_{V_{\reg}})^{\gamma}
\end{equation}
for  $M\in \bigl(G,\chi,\dd(V)\bigr)\lmod$. The functor $\shT_{\vs,\vs'}$ is left exact. 

\begin{lemma} \label{lem:T-property}  
	If $M \in \eC_{,0}$, then $\shT_{\vs,\vs'}(M) \in  \eCdash_{,0}$. 
\end{lemma}

\begin{proof}
	We must check that $\shT_{\vs,\vs'}(M)$ is a finitely generated $(G,\chi')$-monodromic $\dd(V)$-module. First,  since $\delta$ is $G$-invariant, $M|_{V_{\reg}}$ is a rational $G$-module and the $\dd(V)$-submodule $\shT_{\vs,\vs'}(M)$ is a $G$-submodule. Equation \eqref{eq:tausigmaauto} therefore  implies that $\shT_{\vs,\vs'}(M)$ is $(G,\chi')$-monodromic. 
	
	 It remains to show that   $\shT_{\vs,\vs'}(M)$ is finitely generated. By Lemma~\ref{lem:charvar}(2), the module $M$ is holonomic. Therefore, $M|_{V_{\reg}}$ and $(M|_{V_{\reg}})^{\gamma}$ are both holonomic $\dd(V_{\reg})$-modules. By \cite[Theorem~3.2.3(i)]{HTT}, $(M|_{V_{\reg}})^{\gamma}$ is even a holonomic $\dd(V)$-module and, in particular,  is finitely generated. Thus the same is true of  the submodule $\shT_{\vs,\vs'}(M)$.  \end{proof}

By \cite[Proposition~7.1]{BerestChalykhQuasi}, there is also a shift functor 
$\euls{T}_{\kappa,\kappa'} \colon \Osph_{\kappa,0} \to \Osph_{\kappa',0}$. (In that paper the functor is defined 
between the  full categories $\euls{O}$ for $\Hk(W)$, but it is easily seen to   induce a functor on the 
spherical category $\Osph$.) Since any $\delta$-torsion module in $\Osph_{\kappa,0}$  is killed by 
$\euls{T}_{\kappa,\kappa'}$, this functor factors through a functor $T_{\kappa,\kappa'} \colon \euls{H}_q(W) \lmod
 \to \Osph_{\kappa',0}$. We define $\mr{kz}_{\kappa,\kappa'} \colon \euls{H}_q(W)\lmod \to \euls{H}_{q'}(W)\lmod$ 
 to be the composite $\mr{KZ} \circ T_{\kappa,\kappa'}$.

\begin{proposition}\label{prop:KZtwistdiagram}  Let $\vs, \vs'\in \C^k$  and assume that Hypothesis~\ref{K-hyp} holds for the corresponding parameters $\kappa$ and  $\kappa'$. 
	Then there is a commutative diagram 
	$$
	\begin{tikzcd}
	\eC_{,0} \ar[r,"\mathbb{H}"] \ar[d,"\shT_{\vs,\vs'}"] \ar[rr,bend left=30,"\mr{KZ}_G"] & \Osph_{\kappa,0} \ar[d,"\euls{T}_{\kappa,\kappa'}"] \ar[r,"\mr{KZ}"] & \euls{H}_q(W) \lmod \ar[d,"\mr{kz}_{\kappa,\kappa'}"] \\
	\eCdash_{,0} \ar[r,"\mathbb{H}"] \ar[rr,bend right=30,"\mr{KZ}_G"'] & \Osph_{\kappa',0} \ar[r,"\mr{KZ}"] & \euls{H}_{q'}(W) \lmod .
	\end{tikzcd}
	$$
\end{proposition}

\begin{proof}  
	We recall from \cite[Lemma~2.9(1)]{BLNS} the identifications $\Ak(W)[\delta^{-1}] = \ddd(\h_{\reg})^W = A_{\kappa'}(W)[\delta^{-1}]$. Therefore, given $N \in \Osph_{\kappa,0}$, the space $N[\delta^{-1}]$ can be considered as an $A_{\kappa'}(W)$-module. Then $\euls{T}_{\kappa,\kappa'}(N)$ is the submodule of $N[\delta^{-1}]$ consisting of all sections on which $(\Sym \h)^W_+$ acts locally nilpotently. As shown in \cite[Proposition~7.1]{BerestChalykhQuasi}, this submodule is finitely generated and thus lives in $\Osph_{\kappa',0}$.
	
	For clarity, let us write $B_{\kappa} = (\Sym \h)^W \subset \Ak(W) \subset \dd(\h_{\reg})^W$, so that $B_{\kappa} \cong B_{\kappa'}$, but $B_{\kappa} \neq B_{\kappa'}$ in $\dd(\h_{\reg})^W$. 
	
	Given $M \in \eC_{,0}$, the space $\mathbb{H}(\shT_{\vs,\vs'}(M))$ is all sections in $M[\delta^{-1}]^G$ that are locally nilpotent under the action of $(\Sym V)_+^G$ since the action of the latter commutes with taking invariants. But we must unpack how $(\Sym V)_+^G$ acts on this space. We have $M[\delta^{-1}]^G \subset M[\delta^{-1}]^{\gamma}$, where the latter is the $\dd(V_{\reg})$-module with action 
	$$
	D \cdot (m \delta^{-k}) = \gamma^{-1}(D)(m \delta^{-k}).
	$$
	On the other hand, $\euls{T}_{\kappa,\kappa'}(\mathbb{H}(M))$ is the space of $B_{\kappa',+}$-locally nilpotent elements in $M^G[\delta^{-1}]$. The action of $B_{\kappa'}$ comes from its embedding in $\dd(\h_{\reg})^W$, but the action of the latter algebra comes from its realisation as the image of $\dd(V_{\reg})^G$ under $\rad_{\vs}$ (and not under $\rad_{\vs'}$). Thus, given $d \in B_{\kappa',+}$, we choose $D_0 \in \dd(V_{\reg})$ such that $\rad_{\vs}(D_0) = d$ so that $d \cdot (m \o \delta^{-k}) := D_0(m \delta^{-k})$. Note that $\rad_{\vs} = \rad_{\vs'} \circ \, \gamma$ and recall from Theorem~\ref{thm:radial-exists}(2) that $\rad_{\vs'}$ restricts to a graded isomorphism $\rrpp \colon (\Sym V)^G \to B_{\kappa'}$. Therefore, we may (without loss of generality) choose $D_0$ such that $\gamma(D_0) = D_1 \in (\Sym V)_+^G$. Thus, $d \cdot (m \o \delta^{-k}) = \gamma^{-1}(D_1)(m \delta^{-k})$. It follows that $\mathbb{H}(\shT_{\vs,\vs'}(M)) = \euls{T}_{\kappa,\kappa'}(\mathbb{H}(M))$ as vector spaces. Repeating the above for $d \in A_{\kappa'}(W)$ rather than $B_{\kappa'}$ shows that this is an identification of $A_{\kappa'}(W)$-modules.     
	
	The right most square commutes by definition and the upper and lower parts commute because the diagram \eqref{eq:KZGcommute} is commutative.  
\end{proof}

\begin{definition}\label{regular-parameter} The parameter $q$ is said to be \emph{regular}\label{regular-param}
  if $\euls{H}_q(W)$ is semisimple. 
	The parameter $\kappa$, respectively $\vs$, is \emph{regular} if the associated parameter $q$ is regular. 
\end{definition}

By Corollary~\ref{thm:semi-simplicity2}, $\vs$ is regular if and only if the module $\eG_{0}$ is semisimple. Similarly, we say that $\kappa$, respectively $\vs$, is \textit{integral} if the associated parameter $q$ is everywhere one; that is, if $\euls{H}_q(W) = \C W$. We now prove the analogues of the results of \cite[Section~7]{BerestChalykhQuasi}. We note that, although the statements are similar, the proofs are quite different since our results  depend upon Theorem~\ref{G-is-injective}.  
 
\begin{lemma}\label{lem:BCquasiDlemma}  Let $\vs, \vs', \vs'' \in \C^k$  and assume that Hypothesis~\ref{K-hyp} holds for  $\kappa$, $\kappa'$ and  $\kappa''$. 
	\begin{enumerate}
		\item If $\vs$ is regular and $L$ an irreducible summand of $\eG_0$ then $\shT_{\vs,\vs}(L) = L$. 
		\item If $\vs,\vs'$ are regular and $L$ an irreducible summand of $\eG_{\chi,0}$ then $\shT_{\vs,\vs'}(L)$ is either a simple summand of $\eG_{\chi',0}$ or zero. 
		\item If $\vs$ is regular and $L$ an irreducible summand of $\eG_0$ with $\shT_{\vs,\vs'}(L) \neq 0$ then $\shT_{\vs,\vs''}(L) \cong [\shT_{\vs',\vs''} \circ \shT_{\vs,\vs'}](L)$. 
	\end{enumerate}
\end{lemma} 

\begin{proof}
	(1) Since $L$ is $\delta$-torsion free by Theorem~\ref{torsionfree}, $L \subset L|_{V_{\reg}}$ and hence $L \subseteq \shT_{\vs,\vs}(L)$. Consider the short exact sequence 
	$$
	0 \to L \to \shT_{\vs,\vs}(L) \to \shT_{\vs,\vs}(L) / L \to  0
	$$
 By  Corollary~\ref{thm:semi-simplicity2} and Theorem~\ref{G-is-injective}, $L$ is an injective object in $\eC_{,0}$. Therefore, the sequence splits and  $C=\shT_{\vs,\vs}(L) / L$ is a summand of $\shT_{\vs,\vs}(L)$. However, the fact that $L|_{V_{\reg}} = \shT_{\vs,\vs}(L)|_{V_{\reg}}$ shows that $C$ is a $\delta$-torsion summand of the torsion free module $\shT_{\vs,\vs}(L)$. Thus, $C = 0$. 
	
	(2) Assume  that $K$  is a simple submodule of  $ \shT_{\vs,\vs'}(L)$. Then, as $G$-modules, $K|_{V_{\reg}} = L|_{V_{\reg}}$. Since $L$ is a summand of $\eG_0$, Proposition~\ref{bigg} says that $L^G \neq 0$.  Since $\delta$ is $G$-invariant, we deduce that $K^G \neq 0$. Thus, $K$ is a quotient of $\eG_{\chi',0}$. Since $\vs'$ is regular, Theorem~\ref{G-is-injective} therefore  implies that $K$ is a summand of $\eG_{\chi',0}$. In particular,  
	$K$ is injective as an object in $\euls{C }_{\chi',0}$. Just as in Part~(1), we deduce that $K = \shT_{\vs,\vs'}(L)$. 
	
(3) Let $\gamma'$ be the automorphism associated to $(\vs',\vs'')$ and $\gamma''$ the automorphism associated
 to the pair $(\vs,\vs'')$. Since $L$ is irreducible, $L|_{V_{\reg}}$ and $(L|_{V_{\reg}})^{\gamma}$ are irreducible. 
 Since $N := \shT_{\vs,\vs'}(L)$ is non-zero, it follows that $N|_{V_{\reg}} = (L|_{V_{\reg}})^{\gamma}$. Then 
 $(N|_{V_{\reg}})^{\gamma'} = (L|_{V_{\reg}})^{\gamma \gamma'} = (L|_{V_{\reg}})^{\gamma''}$. By taking 
 $(\Sym V)^G_+$-locally nilpotent vectors, we deduce that 
 $\shT_{\vs,\vs''}(L) \cong 
 \left(\shT_{\vs',\vs''} \circ \shT_{\vs,\vs'}\right)(L)$. 	
\end{proof} 

When we combine our setup with the non-vanishing result of \cite[Theorem~7.11]{BerestChalykhQuasi}, we can deduce much stronger results. Recall from Corollary~\ref{cor:simplicity} that when $\euls{H}_q(W)$ is semisimple, so too is $\eG_0$ and the simple summands of $\eG_0$ are $\eG_{0, \rho}$ for $\rho \in \mr{Irr} \euls{H}_q(W)$.

\begin{theorem}\label{thm:KZGtwistsimples} Let $\vs, \vs'  \in \C^k$ be regular and 
assume that Hypothesis~\ref{K-hyp} holds  for both $\kappa$ and  $\kappa'$. Then 
	$$
	\shT_{\vs,\vs'}(\eG_{0,\rho}) \cong \eG_{0,\rho'}, \quad \textrm{where} \quad \mr{kz}_{\kappa,\kappa'}(\rho) \cong \rho'.
	$$
\end{theorem}

\begin{proof}
	By construction, $\mr{KZ}_G(\eG_{0,\rho}) \cong \rho$ and so   Proposition~\ref{prop:KZtwistdiagram}
 implies that 
 $$
 \mr{KZ}_G(\shT_{\vs,\vs'}(\eG_{0,\rho})) \cong \mr{kz}_{\kappa,\kappa'} (\mr{KZ}_G(\eG_{0,\rho})) \cong \mr{kz}_{\kappa,\kappa'} (\rho).
 $$
 Since $\vs, \vs'$ are regular, so too are $\kappa,\kappa'$. Therefore, \cite[Theorem~7.11]{BerestChalykhQuasi} says that $\rho' := \mr{kz}_{\kappa,\kappa'}(\rho)$ is an irreducible $\euls{H}_{q'}(W)$-module. In particular, this implies that $\shT_{\vs,\vs'}(\eG_{0,\rho})\not= 0$. The  result now follows from Lemma~\ref{lem:BCquasiDlemma}(2).   
\end{proof}

When $q = q' = 1$, so that $\euls{H}_q(W) = \euls{H}_{q'}(W) = \C W$, it is shown in \cite[Corollary~7.18]{BerestChalykhQuasi} that the permutation 
 $\mr{Irr} \, W \to \mr{Irr} \, W$ induced by $\mr{kz}_{\kappa,\kappa'}$ is the KZ-twist originally defined by 
 Opdam \cite[Part II,~Corollary~3.8(v)]{Opdamlectures}. Thus, as an immediate consequence we obtain: 

\begin{corollary}\label{cor:KZGtwist}
Suppose that  $q = q' = 1$ and 
  that Hypothesis~\ref{K-hyp} holds  for both $\kappa$ and  $\kappa'$.  Then   $\shT_{\vs,\vs'}(\eG_{0,\rho})$ is the irreducible $  \dd(V)$-module $\eG_{0,\rho'}$ labelled by the Opdam KZ-twist $\rho'$ of $\rho$. \qed
\end{corollary}
 

\section{Examples I: Symmetric Spaces}\label{Sec:examples}
  
In the next three sections we discuss various classes of examples to which the main results of this paper can be applied. There are two main 
classes: symmetric spaces, which will be treated in this section and quiver representations, which are the focus of Section~\ref{Sec:Quivers}. 
There are various other, more specific, examples which will be discussed in Section~\ref{sec:otherexamples}. 

\bigskip

We begin with the definitions. Let $\Gtilde$ be a connected, complex reductive
algebraic group with Lie algebra $\gtilde$.   Fix a non-degenerate, $\Gtilde$-invariant symmetric bilinear form
$\varkappa$ on   $\gtilde$ that reduces to 
the Killing form on   $[\gtilde,\gtilde]$.  Let
$\vt$ be an involutive automorphism  $\gtilde$ preserving $\varkappa$ and set
$\g = \ker(\vt -I)$, $\p = \ker(\vt +I)$. Then, $\gtilde= \g \oplus
\p$ and the pair $(\gtilde,\g)$, or $(\gtilde,\vt)$, is called a {\it
  symmetric pair} with \emph{symmetric space} $V=\p$. \label{defn:symmetric} 
  The Lie algebra  $\g$   is always reductive. 
Let  $G$  be  the connected reductive subgroup of $\Gtilde$  such that 
  $\g=\Lie(G)$. The group $G$ acts on $\p$ via the adjoint
action.  For a classification of symmetric spaces, see the  tables in \cite[Chapter~X]{He1} or, in a form more convenient 
for this paper, those in \cite[Appendix~B]{BLNS}. 

Symmetric spaces   provide  a natural generalisation of the adjoint representation of $G$ on $\g$. Specifically,  consider the  {diagonal case}
where $\Gtilde = G \times G$ with $\vt(x,y) = (y,x)$. Then $(\gtilde,\g) = (\g \oplus \g, \g)$  with the natural  adjoint action of $G$ on $V = \g$.  
  These spaces  also  fit into the general framework considered in this article. 

\begin{lemma}\label{lem:symmetricstablepolar}
The symmetric space representations $(G : \p)$ are   visible stable polar representations. 
\end{lemma}

\begin{proof} 
	This is proved in \cite[Lemma~8.2]{BLNS}, except for the fact that $(G : \p)$ is visible, which follows from \cite[Theorem~2]{KR}. 
\end{proof}

By \cite[Theorem~7.18]{BLNS} \fff (which generalises Theorem~\ref{thm:radial-exists}(3)), for every symmetric space $V=\mathfrak{p}$ one has a surjective radial parts map $\rad_{\vs}$.  However, $\Ak(W) = \Im(\rad_{\vs})$ need not be simple; indeed for $\vs\not=0$  the question of simplicity can be particularly difficult. For example, in the very special case when $(\gtilde,\g) = (\mathfrak{gl}(2), \mathfrak{gl}(1)\times \mathfrak{gl}(1)) $, all possible infinite dimensional primitive factors of $U(\mathfrak{sl}_2)$ appear as $\Im(\rad_{\vs})$ for a suitable choice of $\vs$ (use Remark~\ref{quivern=2} and Corollary~\ref{cor:cyclicrank1rad}).  So, for the rest of this section we will   stick to the case $\varsigma=0$.
This is also  the case of most interest in the literature. Moreover, in the case where $G$ is semisimple, this is no loss of generality, since then  every irreducible factor of $\delta$ is  $G$-invariant and hence  \cite[Corollary~6.8]{BLNS} says that  (up to isomorphism) $\Ak(W)$ is independent  of the twist. 
  
Since $\vs=0$, we can, and will, write  $\rad=\rad_0$ and $\g=\g_0$ throughout the section.  In this case,  by Corollary~\ref{left-simplicity2},   Hypotheses~\ref{K-hyp} is equivalent to the hypothesis that ${\Aak}=\Ak(W)$ be simple. The main result of \cite[Section~8 and Appendix~A]{BLNS} gives a complete classification of the symmetric pairs for which $\Aak$ is simple. In order to state the classification, we must first recall the relevant definitions.

\begin{definition} \label{nice-space}
	Suppose first that  $\gtilde$ is semisimple and let $R$ be the reduced root system associated to a \emph{Cartan subspace} $\h\subset \p$, as in \cite[Section~8]{BLNS}. For each $\alpha \in R^+$, define
	$$
	k_{\alpha} = \frac{1}{2}\left( \dim \gtilde^\alpha + \dim \gtilde^{2\alpha} \right). 
	$$
	The symmetric pair $(\gtilde,\g)$ (or the symmetric space $\p$)  is said to be:
	\begin{enumerate}
		\item {\it nice}  if $k_{\alpha} \leq 1$ for all $\alpha\in R^+$;  
	
	\item {\it integral} if $k_{\alpha} \in \mathbb{Z}$ for all $\alpha\in R^+$;
	
	\item {\it \gainly} \ if each simple summand is either nice or integral. 
\end{enumerate}
	
	 If $\gtilde$ is reductive, then $(\gtilde, \g)$ is {\it nice} (respectively {\it integral, \gainly}) provided that the semisimple pair $\bigl([\gtilde,\gtilde],\,\g\cap [\gtilde,\gtilde]\bigr)$ is nice (respectively integral, \gainly). 
\end{definition}

The definition of nice pairs was introduced by Sekiguchi \cite{Se} and the reader is referred to \cite[Section~6]{Se} and \cite{LS3} for further details on such spaces. In particular, the diagonal case is nice.  A complete list of the irreducible   \gainly\  symmetric pairs is given by \cite[Appendix~B]{BLNS}, combined with the next theorem.

\begin{theorem}\label{thm:symmetricsimple} \cite[Theorem~8.23]{BLNS} 
	The algebra $\Aak$ associated to the symmetric pair $(\gtilde,\vt)$ is simple if and only if $(\gtilde,\vt)$ is \gainly. \qed
\end{theorem}

In order to state further relevant results from \cite{LS3} and \cite{BLNS} we need an additional definition. Let 
 $(\gtilde,\vt)$ be  
  symmetric pair with symmetric space  $V=\p$. Then 
  set 
  \[ \eK(V)= \bigl\{d\in \eD(V) : d(f)=0 \text{ for all } f\in \C[V]^G\bigr\}.
  \]
Then    \cite[Theorem~9.10]{BLNS}, building on the case of nice pairs treated in \cite{LS3} gives:
  
\begin{theorem}\label{LS3-theorem}   
 Suppose that $(\gtilde,\vt)$, is a \gainly \ symmetric pair with symmetric space $V=\p$ and assume that $\vs=0$. Then $\eK(V)=\eD(V)\tau(\g) $ and $\ker(\rad)=\bigl(\eD(V)\tau(\g)\bigr)^G$. Moreover,  the ring $R = \eD(V)^G/\ker(\rad)$  is simple.  \qed
\end{theorem}

\begin{corollary}\label{LS3-corollary}
	 Suppose that $(\gtilde,\vt)$, is a \gainly \ symmetric pair with symmetric space $V=\p$ and assume that $\vs=0$. Then, in the notation of Definition~\ref{M-new-definition},  $R= {\Aak}$ and $\eMt=\eM$. In particular, $\eG = \eGt$. 
\end{corollary}

\begin{proof}
	The first statement follow from Theorem~\ref{LS3-theorem}, combined with Theorem~\ref{thm:intro-radial-surjective}. The second statement then follows as in Remark~\ref{rem:eGpres}. 
\end{proof}

\begin{remark}\label{curio}    In the simplest case, where the symmetric space has rank one (thus $\dim \h=1$), the structure of 
	$\eK(V)$ and ${\Aak}$   is studied in more detail in \cite[Section~6]{LS3}.    In particular, it is shown there that 
	$\ker(\rad)=\bigl(\eD(V)\tau(\g)\bigr)^G$ always holds in rank one. Curiously, however, one can have $\eK(V)\supsetneqq \eD(V)\tau(\g) $ in which case
	$\eK(V)=\eD(V)\tau(\g')$ for a strictly larger Lie algebra $\g'\supsetneqq \g$. As noted in \cite[Corollary~6.5]{LS3},
	this happens for the symmetric pair  $(\mathfrak{sl}(3), \mathfrak{gl}(2))$.  One should note that   this example is  not directly relevant to the main results of  this paper 
	since  the ring  ${\Aak}$ is not simple---see  case $AIV_n$  in \cite[Appendix~B]{BLNS}.   \end{remark}

When the symmetric pair has rank greater than one, much less is known about $\ker(\rad)$ and $\eK(V)$  and a series of questions are raised in \cite[Section~7]{LS3}. For example,  we know of no example of a symmetric 
	space where   $\ker(\rad) \not= \bigl( \eD(V)\tau(\g)\bigr)^G$, see \cite[Conjecture~C2]{LS3}.  It is precisely to circumvent this problem that we introduced the objects $\eMt$ and $\eM$  in Definition~\ref{M-new-definition}.
 
\subsection*{The Harish-Chandra module}

In this subsection we amplify the comments made in Remarks~\ref{LS-C5} and~\ref{LS-C4}.
Of the questions raised in \cite[Section~7]{LS3}, the ones of most relevance to this   paper are the following two conjectures. In both cases the conjecture is made under the assumption that  the  symmetric space $V=\mf{p}$ is nice, and we are still assuming that $\varsigma=0$. 

\medskip
\cite[Conjecture~C4]{LS3}    $\eG_\lambda$ is a semisimple $\eD(V)$-module.

\medskip
\cite[Conjecture~C5]{LS3}   $\eG_\lambda$  is  $\delta$-torsionfree as a left  $\eD(V)$-module.
\medskip 

  Proposition~\ref{torsionfree-converse} and Theorem~\ref{torsionfree} give the following  complete answers to Conjecture~C5 for nice symmetric spaces and, more generally, for \gainly\ symmetric spaces:
 
 \begin{corollary}\label{cor:C5}  {\rm (1)}  
 	If the symmetric space $V$ is \gainly \ then $\eG_\lambda$  has no nonzero submodule (or factor module) that is $\delta$-torsion. 
 
{\rm (2)} If $V$ is a symmetric space for which ${\Aak}(W_{\lambda})$ is not simple then $\eG_{\lambda}$ has a non-zero, $\delta$-torsion factor module. \qed
\end{corollary}

By Theorem~\ref{thm:semi-simplicity}  and  Corollary~\ref{thm:semi-simplicity2} we also get a comprehensive answer to Conjecture~C4:

\begin{corollary}\label{cor:C4}  {\rm (1)}	If the symmetric space $V$ is \gainly \ then $\eG_\lambda$ is a semisimple $\eD(V)$-module if and only if the Hecke algebra  $\euls{H}_q(W_\lambda)$ is semisimple.

{\rm (2)} If $V$ is any symmetric space, then the same conclusion holds for $\eG_{0}$.   \qed
\end{corollary}

In order to  answer   Conjecture~C4, we need to determine exactly which of the corresponding Hecke algebras  is semisimple. As noted in Remark~\ref{LS-C4}, for any given group $W_\lambda$ this can be deduced from  the literature. Here, for simplicity we concentrate on the case $\lambda=0$; thus we are in the context of Corollary~\ref{cor:C4}(2).  The answer is then given by the following result, which also 
  proves   Corollary~\ref{intro-cor:symmetricspacessHecke} from the introduction.

\begin{theorem}\label{cor:symmetricspacessHecke}  Set $\vs=0$. 
	The Harish-Chandra module $\eG_0$ of an arbitrary  irreducible  symmetric space $V$ is semi\-simple if and only if each simple factor of $V$ is either of adjoint type or of type 
	\[\begin{aligned} \text{$\mathsf{AII}_n$} \  & \ \text{
	where $(\widetilde{\g},\g)= \mathfrak{sl}(2n), \mathfrak{sp}(n)$, }\\
	\text{$\mathsf{DII}_{p}$}  \ &\ \text{ 
where $(\widetilde{\g},\g)= \mathfrak{so}(2p), \mathfrak{so}(2p-1)$, or }\\
 \text{$\mathsf{EIV}$} \, \ &\ \text{  where $(\widetilde{\g},\g)= \mathfrak{e}(6), \mathfrak{f}(4)$.  }
 \end{aligned}\]
\end{theorem} 

\begin{remark}\label{Brion}  The four cases appearing in the theorem  also appear in other situations; notably they occur as the ``split rank case'' in  \cite[Theorem~6.1, p.429]{He1}.  We would like to thank Michel Brion for this observation.  
\end{remark}
 
\begin{proof}   If ${\Aak}$ is not simple then, as noted in the proof of Corollary~\ref{thm:semi-simplicity2},  neither $\eG_0$ nor the  Hecke algebra $\euls{H}_q(W)$ is semisimple. So, we can restrict to the cases when ${\Aak}$ is simple.  These are completely determined by the 
  tables in   \cite[Appendix~B]{BLNS} and are listed in Table~1.   We have ignored the diagonal case since the  corollary is known there; see \cite[Theorem~B]{AJM}. The rest of the notation for this table will be described later in the proof.

\begin{table}[t]\label{eq:symHeckess}
\vskip 10pt
{\small{	$\begin{array}{|l|l|l|l|l|l|l|l|l|} 
		\hline
		\strut  \textrm{Type } (\widetilde{\g},\theta)  & \mathsf{AI}_n   & \mathsf{AII}_n  & \mathsf{AIII}_{n,n} &
	\mathsf{BI}_{p=n,q} 	& \mathsf{BII}_{n=1}  & \mathsf{CI}_{n} &\mathsf{DI}_{p-1,p+1}  & \mathsf{DI}_{p,p}  \\ 
		 \textrm{Type } W & \mathsf{A}_{n-1} &  \mathsf{A}_{n-1}  & \mathsf{C}_n &
	\mathsf{B}_n	& \mathsf{A}_1  & \mathsf{C}_n   & \mathsf{B}_{p-1}& \mathsf{D}_p  \\
		\hline 
		\textrm{parameters}  &  x= -1 & x=1 & x = -1 & x=-1 & x=-1 & x=-1  & x=-1 & x=-1 \\
		x=u_s, & &  & y=1  & y=-1& & y=-1& y=1&  \\  y=u_t  &&&&&&&& \\
\hline	\strut  \euls{H}_q(W) \textrm{ ss}  &  \mr{N} &  \mr{Y}  &  \mr{N} &  \mr{N} & \mr{N}& \mr{N}& \mr{N} & \mr{N}\\
		\hline
	\end{array}$ }
\vskip 15pt
$\begin{array}{|l|l|l|l|l|l|l|l|l|} 
		\hline
		\strut  \textrm{Type } (\widetilde{\g},\theta)    &  \mathsf{DII}_{p}   & \mathsf{EI}   &\mathsf{EII} &
	\mathsf{EIV} 	& \mathsf{EV}  & \mathsf{EVIII}  &\mathsf{FI}  & \mathsf{G}   \\ 
		 \textrm{Type } W & \mathsf{A}_1& \mathsf{E}_6 & \mathsf{F}_4&
	\mathsf{A}_2	  & \mathsf{E}_7 & \mathsf{E}_8& \mathsf{F}_4  & \mathsf{G}_2  \\
		\hline 
		\textrm{parameters}  &  x= 1 & x=-1 & x = -1 & x= 1 & x= -1 & x=-1  & x=-1 & x=-1 \\
		x=u_s, & &  & y=1  && & & y=-1& y=-1 \\  y=u_t  &&&&&&&& \\
\hline	\strut  \euls{H}_q(W) \textrm{ ss}  &  \mr{Y} &  \mr{N}  &  \mr{N} &  \mr{Y} & \mr{N}& \mr{N}& \mr{N} & \mr{N}\\
		\hline
	\end{array}$}
	\vskip 10pt
	\caption{Semisimple Hecke Algebras.} \vskip-15pt
\end{table}

In order to determine  whether the  Hecke algebras are semisimple or not, we apply \cite{GeckRouquierCenters}.	To do this, we first describe how the presentation \eqref{eq:Heckecrg1} of the Hecke algebra is related to the description of the Iwahori-Hecke algebra given in \cite{GeckRouquierCenters}. Fixing a set of simple reflections $S \subset W$ and $s,t\in S$, let $u_s \in \C^{\times}$ with $u_s = u_t$ whenever $s$ is conjugate to $t$ in $W$. The (specialised) Iwahori-Hecke algebra $\euls{H}_u(W)$ is the $\C$-algebra generated by elements  $\{L_s :s \in S\}$, subject to the braid relations together with the quadratic relations
$$
(L_s - u_s)(L_s + 1) = 0 \quad \textrm{ for } s \in S. 
$$
If $H_s$ is the hyperplane in $\h$ fixed by $s$ then we write $T_s = T_{H_s}$ and $  \kappa_{s,j} = \kappa_{H_s,j}$ and $q_{s,j} = q_{H_s,j}$. Setting $L_s = - q_{s,1}^{-1} T_s$  defines an isomorphism $\euls{H}_u(W) \isom \euls{H}_q(W)$ for $u_s = - q_{s,0} q_{s,1}^{-1}$. Recall from \eqref{eq:Heckecrg2} that 
$$
q_{s,0} = \exp(-2\pi i \kappa_{s,0}) \quad \textrm{and} \quad q_{s,1} = - \exp(-2 \pi i \kappa_{s,1}).
$$
Then $L_s = \exp(2 \pi i \kappa_{s,1}) T_s$ and $u_s = \exp(2 \pi i (\kappa_{s,1} - \kappa_{s,0}))$.

In \cite[Appendix~A]{BLNS}, the  Cherednik algebra is defined in terms of one  parameter $k_{\alpha}$ for every positive root $\alpha$. Thinking of $\alpha = \alpha_s$ as the non-trivial eigenvector in $\h^*$ associated to the reflection $s$, we have $\kappa_{s,1} - \kappa_{s,0} = k_{\alpha_s}$. Therefore, $u_s = \exp(2 \pi i k_{\alpha})$.   
 
For each irreducible Weyl group, the reflections in $S$ are either all conjugate or belong to two conjugacy classes. In the latter case, let $s_1,s_2$ be fixed representatives in $S$ of these two conjugacy classes such that the root $\alpha_{s_1}$ associated to $s_1$ is a long root (just as in \cite{GeckRouquierCenters}). When there is one class, we set $x = u_s$ and when there are two (in type $\mathsf{B}/\mathsf{C}$), we set $x = u_{s_1}, y = u_{s_2}$. These numbers are displayed in the second line of Table~1.

Finally,   we use    \cite[Proposition~5.3]{GeckRouquierCenters} to determine when the resulting Hecke algebra is semisimple. In each case the   polynomial $Q_W$  from  \cite[(5.2)]{GeckRouquierCenters} is evaluated to determine whether  $Q_W(x,y)$ is zero or not. The answers are recorded in 
  the third line of Table~1, with Y, respectively N, denoting that  the Hecke algebra is, respectively is not semisimple.   \end{proof}

\begin{remark}\label{character_sheaves}
Via the Riemann-Hilbert correspondence, the simple objects in the category $\mathscr{C}_{0,0}$ of  
 strongly admissible $\dd$-modules on the symmetric space $V$ identifies with the set $\mr{Char}(\gtilde,G)$ of character sheaves on $V$, as studied, for example, in \cite{VilonenSymChar,VilonenSymVanishing}. In particular, the results of \cite[Section~7]{VilonenSymChar} can be viewed as classifying the simple objects in $\mathscr{C}_{0,0}$, when the symmetric pair is classical. It would be interesting to give an algebraic construction of these simple objects, thought of as $\dd$-modules. 
\end{remark}
  
 

\section{Examples II: Quiver Representations}  \label{Sec:Quivers}
A second important class of polar representations that are both stable and visible   occur as the   representation
spaces of  the cyclic quiver.  We consider the  corresponding algebras $\dd(V)$   in detail in this section, with our two main aims being to determine precisely when the associated spherical algebra $\Aak=\Ak(W)$ is simple and to understand  the structure of the Harish-Chandra module $\eG_0$.

We begin by fixing our notation.
 
  \begin{definition}\label{defn:quiver} 
 Let $Q=Q_{\ell}$ be the cyclic quiver with vertices $0,1,\ds, \ell-1$ and  arrows $a_i : i \to i+1$ (indices taken modulo
  $\ell$).   Fix the dimension vector $n\mathfrak{d} = (n,\ds, n)$ for  the minimal imaginary root $\mathfrak{d} = (1,\dots, 1)$. We are interested in the representation space $V := \mr{Rep}(Q,n\mathfrak{d})$,
  which we identify with $\bigoplus_{i = 0}^{\ell-1} \mr{Mat}(n,\C)$.
The group $G := GL(n)^{\ell}$ then acts on $V$ by change of basis; thus  $g = \left(g^{(i)}\right) \in G$ acts on
$ X = (X^{(0)},\ds, X^{(\ell-1)})\in V$  by
\begin{equation}\label{eq:quiver}
g \cdot X = \Bigl(g^{(1)} X^{(0)} (g^{(0)})^{-1}, g^{(2)} X^{(1)} (g^{(1)})^{-1}, \ds, g^{(0)} X^{(\ell-1)} (g^{(\ell-1)})^{-1}\Bigr).
\end{equation}
Set $\g=\mathrm{Lie}(G)$.  Identifying $\C^n$ with the space of diagonal matrices in $\mr{Mat}(n,\C)$, the Cartan subalgebra $\h$ 
 of $V$ is the diagonally embedded copy of $\C^n$ in $V$. The associated Weyl group is the wreath product 
  $$W := \Z_{\ell} \wr \mathfrak{S}_n = (\Z_\ell )^n \rtimes \mathfrak{S}_n,$$
 for the symmetric group $ \mathfrak{S}_n $  and $ \Z_{\ell}=\Z/\ell \Z.$

Here, $\mathfrak{S}_n$ acts on $\C^n$ by permuting coordinates, whilst the $i^{\mr{th}}$ factor of  $\Z_\ell $ in $(\Z_\ell)^{n}$ acts by rescaling the $i^{\mr{th}}$  coordinate by powers of a
  fixed  primitive $\ell^{\mathrm th}$ root of unity $\omega$. 
 
 For one of the examples in Section~\ref{sec:otherexamples}, we  will also need the framed quiver $\widetilde{Q}_{\ell}$, where we add a vertex ${\infty}$ and an arrow from ${\infty}$ to $0$. In this case the group $G$ is the same, but the representation space is $V'= V\oplus \C^n=\mr{Rep}(\widetilde{Q}_{\ell},(n\mathfrak{d},1))$ with the action
  \begin{equation}\label{eq:quiver2}
g \cdot X' = \Bigl(g^{(1)} X^{(0)} (g^{(0)})^{-1},   \ds, g^{(0)} X^{(i)} (g^{(\ell-1)})^{-1},\,
g^{(0)} w\Bigr).
\end{equation}
for $g\in G$ acting on $X'=(X, w)$ with $w\in \C^n$. This action agrees with that of~\cite{BB}. 

\smallskip
  \emph{This notation will be fixed throughout the section.}
\end{definition}

We remark that the action of $G$ on $V$ factors through a factor $PG:= G/\C^\times$ and so one can, as is done in \cite{OblomkovHC},  replace $G$ by $PG$. We have chosen not to do so since  it does not affect the results.

 We first show that the $Q_{\ell}$  example  does indeed fit into our general framework, for which it is convenient to use 
 Vinberg's $\theta$-representations.
 
 \begin{example}{\bf Theta representations.}\label{ex:theta-reps}
A large class of polar representations is given by Vinberg's $\theta$-representations \cite{Vinberg}, defined as follows. Let $\Gtilde$ be a connected reductive group with Lie algebra 
$\gtilde$ and an automorphism  $\theta \colon \gtilde \to \gtilde$ of order $2\leq \ell < \infty$. Then $\theta$ defines a 
$\mathbb{Z}_{\ell}$-grading $\gtilde = \gtilde_0 \oplus \gtilde_1 \cdots \oplus \gtilde_{{\ell}-1}$, where $\gtilde_0$ is a reductive 
subalgebra of $\gtilde$. Let $\Gtilde_0 \subset \Gtilde$ be the connected subgroup with Lie algebra $\gtilde_0$. Then $\Gtilde_0$ acts on $V := \gtilde_1$ and we define $G$ to be the image of $\Gtilde_0$ in $GL(V)$. 
The pair $(G,V)$ is the $\theta$-representation associated to $(\mf{k},\theta)$. By \cite[Theorem~4]{Vinberg} and \cite[Lemma~3.2]{KacSome}, $\theta$-representations are visible  polar representations. When ${\ell} = 2$, $\theta$-representations are the same as symmetric spaces and are therefore stable but for  ${\ell} \ge 3$ they need  not be stable; see  \cite[p.~245]{PopovVinberg}. 
\end{example}

\begin{lemma}\label{prop:quiver-polar}
The representation $(G,V)= ( GL(n)^\ell,\, \mr{Rep}(Q_\ell,n\mathfrak{d}))$ is a visible stable polar representation.

Similarly, the representation $(G,V')= ( GL(n)^\ell,\, \mr{Rep}(\widetilde{Q}_\ell,(n\mathfrak{d},1)))$ is a visible   polar representation.  It is not stable.
\end{lemma}

\begin{proof} This is standard. We first note that $V$ is a $\theta$-representation. In more detail, the diagonal matrix $\alpha=(1,\dots,1,\omega,\dots,\omega,\dots, \omega^{\ell-1})$ (where each entry appears $n$ times) acts by conjugation on $\mathfrak{gl}(n\ell)$. Then it is immediate that, in the notation of Example~\ref{ex:theta-reps},
 $\g=\mathfrak{gl}(n\ell)_0$  while $V=\mathfrak{gl}(n\ell)_1$ forms the $\omega$-eigenspace.  By \cite[Theorem~4]{Vinberg} and \cite[Lemma~3.2]{KacSome} $V$ is  therefore a visible polar representation. 
 The fact that $V$ is a stable representation is a consequence of \cite[Lemma 2.4]{GordonCyclicQuiver} which says that set $G \cdot \h_{\reg}$ of closed orbits is open and dense in $V$.

The argument for the framed quiver is similar, except that now we replace $\mathfrak{g}(n\ell)$ by the 
$  (n\ell+1)\times( n\ell)$ matrices of the form 
$\left(\begin{smallmatrix} A&0\\B&1\end{smallmatrix}\right) $, where $A$
 is $n\ell\times n\ell$ and $B$ is $1\times n\ell$. Similarly we replace $\alpha$ by the diagonal 
 matrix $\alpha' = (\alpha,1)$. Since we will only use this result for one example, we leave it to the reader to check that this does indeed make $(G,V')$ into a $\theta$-representation. Thus, as in the last paragraph, $V'$ is a visible polar representation. It is a consequence of \cite[Theorem~1]{Semisimplequiver} that $\C[V']^G \stackrel{\sim}{\rightarrow} \C[V]^G$. Therefore, if $(X,w) \in V' = \mr{Rep}(Q_\ell,n\mathfrak{d}) \oplus \C^n$ then $\overline{G \cdot (X,w)} \supset G \cdot (X,0)$ implying that generic orbits in $V'$ are not closed. 
 \end{proof}

\begin{remark}\label{quivern=2}
	When $\ell = 2$, the pair $(G,V)$ can also  be realised as the symmetric pair $(\mf{gl}(2n),\mf{gl}(n) \oplus \mf{gl}(n))$; this is essentially  Case~(II) of \cite[Theorem 2.5]{LS3}. For $\ell>2$, $V$  is not a symmetric representation. \end{remark}
	
	In the rest of the section we will only consider  the representation $(G,V)= ( GL(n)^\ell,\, \mr{Rep}(Q_\ell,n\mathfrak{d}))$.
In order to understand the possible choices of the character $\chi$ and parameter $\vs$,	as discussed at  the beginning of  \cite[Section~4]{BLNS}, we need to factorise the discriminant $\delta$ of $V$. The reflecting hyperplanes for  $\Z_{\ell} \wr \mathfrak{S}_n$ are $x_i = 0$ and $x_i - \omega^k x_j = 0$ for $1 \le i < j \le n$ and $0 \le k \le \ell-1$. The factor $x_i$ appears with multiplicity $\ell$ in $\delta |_{\h}$ and $x_i - \omega^k x_j$ with multiplicity two since 
\begin{equation}\label{eq:delta-h}
\delta |_{\h} \ = \ (x_1 \cdots x_n)^{\ell} \prod_{i< j} (x_i^{\ell} - x_j^{\ell})^2.
\end{equation}
Note that, when $n=1$,  the final term $\prod_{i< j} (x_i^{\ell} - x_j^{\ell})^2$ also equals $1$ and so formul\ae\ involving  $\delta|_{\h}$  have to be suitably interpreted in that case.

Define $\delta_i(X) = \det (X^{(i)})$. Then $\delta_i$ is an irreducible  $G$-semi-invariant whose restriction to $\h$ equals $x_1 \cdots x_n$. As in Section~\ref{Sec:polarreps},  for  $0\leq i\leq \ell-1$  let $\theta_i$ be the character of $G$ associated to the semi-invariant $\delta_i$ and  write  $\C[V]^\theta_i$ for  the space of $\theta_i$-semi-invariants.
 Finally, let $\delta_{\mf{gl}}$ denote the discriminant for the adjoint representation of $\mf{gl}(n)$ and set  $\delta_{\infty}= \delta_{\mf{gl}}(X^{(0)} \cdots X^{(\ell-1)}) $, noting that $\delta_{\infty}=1$ when $n=1$.

\begin{lemma}
	The polynomial $\delta_{\infty}$ is irreducible in $\C[V]$. 
\end{lemma}

\begin{proof}
We assume that $\delta_{\infty}$ is not irreducible. Each irreducible factor $f_i$ is a $G$-semi-invariant, and hence $f_i |_{\h}$ is a $W$-semi-invariant. But the only $W$-semi-invariant (properly) dividing 
$$
\delta_{\mf{gl}}(X^{(0)} \cdots X^{(\ell-1)}) |_{\h} = 
\prod_{i < j} (x_i^{\ell} - x_j^{\ell})^2
$$
is $u := \prod_{i < j} (x_i^{\ell} - x_j^{\ell})$ since $u$ is the product over the $W$-orbit of any (linear) factor $x_i - \omega^k x_j$. This implies that $\delta_{\infty} = f_1 f_2$, where $f_1 |_{\h} = f_2 |_{\h} = u$. Note that $W$ admits a (split) surjection to $\mathfrak{S}_n$ with kernel $\Z_{\ell}^n$. If $\mr{sgn}$ denote pullback along this surjection of the usual sign character of $\mathfrak{S}_n$, then $u$ is a $\mr{sgn}$-semi-invariant. In particular, $u |_{\mathfrak{S}_n}$ is a sign semi-invariant. To see this, note that $\mathfrak{S}_n$ acts as the permutation representation on the space spanned by the monomials $x_1^{\ell}, \ds, x_n^{\ell}$.

The image of $N_{SL_n^{\ell}}(\h)$ inside $W = N_{GL_n^{\ell}}(\h) / Z_{GL_n^{\ell}}(\h)$ contains $\mathfrak{S}_n$. Since $SL_n^{\ell}$ must necessarily act trivially on each $f_i$, so too does $N_{SL_n^{\ell}}(\h)$. But this implies that $\mathfrak{S}_n$ acts trivially on $u$; a contradiction. 

We deduce that $\delta_{\infty}$ is irreducible.  
\end{proof}

\begin{lemma}\label{lem:existsemi-invariantpolar}
	Let $V$ be a stable polar representation and let $g_1, \ds, g_k$ be the irreducible factors of $\delta$ in $\C[V]$. If $\theta \colon G \to \C^{\times}$ is a linear character with $\theta(Z) = 1$ then there exists $n_i \ge 0$ such that
		$\prod_{i= 1}^k g_i^{n_i} \in \C[V]^{\theta}$.
\end{lemma}

\begin{proof}
	Consider the left regular action of $G$ on itself. If $\theta$
	is a linear character of $G$ then the algebraic Peter-Weyl
	Theorem implies that, up to scalar, there is a unique $\theta$-semi-invariant function
	$F_{\theta} \in \C[G]$. If the group $N$ acts by (inverse)
	multiplication on the right on $G$ then $F_{\theta}$ is a
	$\theta^{-1}|_N$-semi-invariant. Then each
	$\theta$-semi-invariant function on
	$V_{\reg} = G \times_N \h_{\reg}$ is of the form
	$F_{\theta} \o h$ for some $N$-semi-invariant $h$ on
	$\h_{\reg}$ with character $\theta|_N$.
	
	If $\varepsilon$ is the restriction of $\theta$ to $W$ then we
	choose some nonzero $h \in \C[\h]^{\varepsilon}$. We may assume
	that it does not vanish on $\h_{\reg}$ by choosing $h$ dividing
	$\deltah$. Then $F_{\theta} \o h$ is a $\theta$-semi-invariant
	on $V_{\reg}$ which is nowhere vanish. There is some $m \ge 0$
	such that $g = \delta^m(F_{\theta} \o h)$ is regular on
	$V$. This is our desired semi-invariant. By construction, the
	zeros of $g$ are contained in $V \smallsetminus V_{\reg}$. This
	means that the irreducible factors of $g$ divide $\deltav$. That
	is, the irreducible factors of $g$ are all among the
	$g_i$.
\end{proof}

\begin{proposition}\label{factorisingdelta}
	In the ring $\C[V]$, $\delta$ factorises into irreducible polynomials as $\delta =  \delta_0 \cdots \delta_{\ell-1}\delta_{\infty}$.
\end{proposition}

\begin{proof}  
 We first need to  show that  $\delta_i$ divides $\delta$ for $0\leq i\leq \ell-1$. If  $\varepsilon_i$ denotes the restriction of 
 $\theta_i$ to $W$, then  the restriction map induces a morphism of $\C[V]^G $-modules $\sigma: \C[V]^{\theta_i} \to \C[\h]^{\varepsilon_i}$ sending $\delta_i$ to $\mathbf{x}=x_1 \cdots x_n$. Since $V_{\reg}=G\times_N\h_{\reg}$ by
 \eqref{lem:stablepolaropen}, $\sigma$ is injective.  Since $\mathbf{x}$  is a proper factor of 
 $\delta|_{\h}$, Remark~\ref{rem:Wsemisummary} explains that 
 $\C[\h]^{\varepsilon_i} = \mathbf{x}\C[\h]^W$. Therefore,  $\sigma$ is  surjective  and  $\C[V]^{\theta_i} \cong \C[\h]^{\varepsilon_i}$. In particular, 
 $\C[V]^{\theta_i} = \C[V]^G \delta_i$. By Lemma~\ref{lem:existsemi-invariantpolar}, there exist irreducible factors $g_1, \ds, g_k$ of $\delta$ 
 such that $g := \prod_i g_i^{e_i}\in \C[V]^{\theta_i}$. Thus, $g = f \delta_i$ for some $f \in \C[V]^G$.  Since $\delta_i$ is irreducible, this implies that $\delta_i$ is indeed a factor of $\delta$.
	
	Let $ \delta_{\infty}' =  \delta \delta_0^{-1} \cdots \delta_{\ell-1}^{-1}$. Since $\delta_0 \cdots \delta_{\ell-1}$
	 is $G$-invariant, so too is $ \delta_{\infty}'$. Since the restriction map $\C[V]^G \to \C[\h]^W$ is an 
	 isomorphism, the identity  
	 $$
	\delta_{\mf{gl}}(X^{(0)} \cdots X^{(\ell-1)}) |_{\h} = 
	\prod_{i < j} (x_i^{\ell} - x_j^{\ell})^2 = \delta_{\infty} |_{\h}
	$$
	implies that $\delta_{\infty}' = \delta_{\mf{gl}}(X^{(0)} \cdots X^{(\ell-1)}) =\delta_{\infty}$.  
\end{proof}

By Proposition~\ref{factorisingdelta}, we consider parameters $\varsigma$ of the form  
$\vs=(\vs_0, \ds, \vs_{\ell-1},\vs_{\infty})$ corresponding to the  character $\chi=\sum_i\vs_i d\theta_i$,
 as in the set-up to \cite[Section~4]{BLNS}. Since $\delta_{\infty}$ is $G$-invariant,  the character $\chi_{\infty}$ corresponding to $\delta_{\infty}$ is zero. As in  \cite[Section~4]{BLNS},  
we write $\delta^{\vs} =  \delta_0^{\vs_0} \cdots \delta_{\ell-1}^{\vs_{\ell-1}} \delta_{\infty}^{\vs_{\infty}}$. 

The  space of linear characters of $\mf{g}\cong\bigoplus_{j=0}^\ell\mathfrak{gl}_n$ can be identified with $\C^{\ell}$ via
$$
\xi \mapsto \bigg(x=(x_0,\dots,x_{\ell-1})  \mapsto \sum_{i = 0}^{\ell-1} \xi_i \mr{Tr}(x_i) \bigg). 
$$
Since $d \det(x) = \mr{Tr}(x)$ for $GL(n)$,   \eqref{eq:quiver} implies that 
 $(d \theta_i)(x) = \mr{Tr}(x_{i+1}) - \mr{Tr}(x_i)$. Thus  the character $\chi$ is related to $\vs$ by 
 $\chi_i = \vs_{i-1} - \vs_i$. Here, and elsewhere, the indices are interpreted  modulo $\ell$; in particular 
 $\chi_0 = \vs_{\ell-1} -\vs_0$. 

Recall from Section~\ref{Sec:polarreps} that $\mu \colon T^*V \to \g^*$ denotes the moment map. We identify $\g^* = \g$ using the trace pairing and think of $\mu$ as a map to $\g$. Let  $I \subset \C[V]$ be the ideal (possibly non-reduced) defined by the  scheme-theoretic fibre $\mu^{-1}(0)$. 

\begin{theorem}\label{thm:IeqIone}
	The ideal $I^G$ in $\C[V]^G$ is prime. Hence,
	\begin{enumerate}
		\item the scheme $\mu^{-1}(0) \git G$ is reduced and irreducible; and
		\item $\rr \o \rrp \colon \C[\mu^{-1}(0)]^G \to \C[\h \times \h^*]^W$ is an isomorphism. 
	\end{enumerate} 
\end{theorem}

\begin{proof}
	This follows from \cite[Theorem~7.2.3 and Proposition~7.2.5]{EGGO}. 
\end{proof}

Applying our general results, we obtain a variant of Oblomkov~\cite[Theorem~2.5]{OblomkovHC} and Gordon~\cite[Theorem~1.4]{GordonCyclicQuiver}. 

\begin{theorem}\label{thm:cylicquiverradialparts0}
	For any parameter $\vs$,  the radial parts map $\rad_\vs$ defined by \eqref{eq:notation4.11}  induces  a filtered isomorphism 
	$$
	\rad_{\vs} \colon (\dd(V) / \dd(V) \g_{\chi})^G \ \isom \  \Ak.
	$$
\end{theorem}

\begin{proof}
	By \eqref{eq:notation4.11}, there exists some parameter $\kappa$ such that the image of the radial parts map $\rad_{\vs}$ is contained in $\Ak$.   Theorem~\ref{thm:IeqIone} and \cite[Corollary~7.20(1)]{BLNS} then say that this map is a filtered isomorphism. 
\end{proof}

\begin{remark}\label{rem:quiverHC} Since Theorem~\ref{thm:cylicquiverradialparts0} also describes 
$\ker(\rad_\vs)$,  $R={\Aak}$ in the notation of Definition~\ref{M-new-definition}. In particular,
 $\eMt = \eM$  and hence $\eGt_{\lambda} = \eG_{\lambda}$ for all $\lambda \in \h^*$.
 \end{remark}

 We next want to determine precisely when $\Aak $ is a simple ring, for which we need to describe $\kappa$  explicitly in terms of $\varsigma$. Since the setup considered in \cite{OblomkovHC} (both in terms of parameters for the rational Cherednik algebra and the definition of the radial parts map) is somewhat different to ours, we include an outline of the computation.
 
 We determine the parameter $\kappa$ by describing the slices for our general representation, for which there are distinct cases: the case  $n=1$ and the adjoint representation of $\mathfrak{sl}(2)$. Since $\delta_{\infty}$ is a $G$-invariant, \cite[Corollary~6.8]{BLNS} implies that the algebra $\Ak$ does not depend on $\varsigma_{\infty}$ up to isomorphism. Therefore, to streamline the computations we take $\varsigma_{\infty} = 0$ throughout.

\subsection*{Computations for rank one}\label{subsec:quivercomputations}
 In this subsection  we  describe   the radial parts map in detail in case when $n=1$. Thus the dimension 
vector is now $\mathfrak{d}$ and so  $\C[V]= \C[x_0, \ds, x_{\ell-1}]$ is a polynomial ring, where $x_i$ corresponds to arrow $a_i$. The group $G = (\mathbb{C}^{\times})^{\ell}$ acts on $\C[V]$  such that $x_i$ is a character of weight $t_i t_{i+1}^{-1}$, where $t_i \colon G \to \mathbb{C}^{\times}$ picks out the $i^{\mr{th}}$ factor. 

Fix scalars $\vs_i\in \C$ and write 
$\delta^{\vs} = \delta_0^{\vs_0} \cdots \delta_{\ell-1}^{\vs_{\ell-1}}$ (we can ignore the 
adjoint discriminant  $\delta_{\infty}$). Let $z = x_0 \cdots x_{\ell-1}$ and $\Delta = \partial_0 \cdots \partial_{\ell-1}$, for $\partial_j=\frac{\partial}{\partial x_j}$,  
so that $\C[V]^G = \C[z]$ and $(\Sym \, V)^G = \C[\Delta]$. For any $j\in \mathbb{Z}$, 
$$
\Delta(z^j \delta^{\vs}) = (j + \vs_0) \cdots (j + \vs_{\ell-1}) z^{j-1} \delta^{\vs},
$$
and $x_i \partial_i (z^j \delta^{\vs}) = (j+ \vs_i) z^j \delta^{\vs}$. The Cartan subalgebra $\h$ of $V$ is 
now given by the diagonal $\{x_i = x : 0\leq i\leq \ell-1\}$. 
Then $\dd(\h_{\reg}) = \C[x^{\pm 1}]\langle \partial_x \rangle$, where $x^{\ell} = z$. 
Using the fact that $\partial_z =(\ell x^{\ell-1})^{-1} \partial_x$, this implies that 
$$
\rad_{\vs}(z) = z = x^{\ell}, \qquad \rad_{\vs}(x_i \partial_i) = z \partial_z+ \vs_i = \ell^{-1} x \partial_x + \vs_i,
$$
and
\begin{align*}
	\rad_{\vs}(\Delta) & = \frac{1}{z} \prod_{j=0}^{\ell-1}(z\partial_z+\vs_j) 
	 = \frac{1}{\ell^{\ell} x^{\ell}}  \prod_{j=0}^{\ell-1}(x\partial_x+\ell\vs_j) 
\end{align*} 

If $\Ak(\mathbb{Z}_{\ell})$ is defined in terms of the parameters $\kappa_0, \ds, \kappa_{\ell-1}$ as in \cite[Example~2.11]{BLNS}, then \cite[Corollary~4.15]{BLNS} implies the following result because $\lambda_i = \vs_i$ in this case. 
 
\begin{proposition}\label{prop:cyclicrank1rad}
For  $n=1$, there is  an isomorphism 
$$\rad_{\vs} : (\dd(V) / \dd(V) \mf{g}_{\chi})^G \stackrel{\sim}{\longrightarrow} \Ak(\mathbb{Z}_{\ell}),
\quad\text{where} \   \kappa_i=\vs_i +\frac{\ell -i}{\ell} - \delta_{i,0} 
$$
for $0\leq i\leq \ell-1.$ 
\qed
\end{proposition} 
  
  \begin{corollary}\label{cor:cyclicrank1rad} In the special case when $n=1$ and $\ell=2$, the image of 
  $\rad_{\vs}$ as $\vs$ varies  ranges over all (infinite dimensional) primitive factor rings of $U(\mathfrak{sl}_2)$.
  \end{corollary}
 \begin{remark}\label{rem:cyclicrank1rad}  For reference elsewhere, we note that  this representation can also  be regarded as the symmetric 
space   corresponding to the symmetric pair $\left(\mathfrak{sl}(2), \mathfrak{so}(2)\right)$.\end{remark}
 \begin{proof} This follows once one matches up our parameters with those of \cite[Proposition~8.2]{EG}. We omit the details.
  \end{proof}
    
\subsection*{Explicit computations for $\mathbf{\kappa}$}  
 We return to the notation from the beginning of the section and the general notation for slices, as described in \cite[Section~5]{BLNS}.
 
In order to compute the radial parts map, we first need to compute the slices to generic points on the hyperplanes in $\mf{h}$. There are two cases  to consider, which are covered by the next two lemmas.  The notation follows that   given in  the discussion before \cite[Lemma~5.17]{BLNS}. In more detail, given $b\in \h$, we
write $G_b$ for the stabiliser of $b$ and take a $G_b$-stable complement $S$ to $\g\cdot b$ in $V$. Then   the slice $S_H$  is a $G_b$-stable complement to the hyperplane $H=S^{G_b}$ inside $S$. 

\begin{lemma}\label{lem:slice1cyclic}
	Let $b = (0,b_2,b_3,\ds) \in \h$, where $b_i \neq 0$ and $b_i^{\ell} \neq b_j^{\ell}$ for $i \neq j$. 
	\begin{enumerate}
		\item This is a generic point on the hyperplane $x_1 = 0$. 
		\item The  group  $G_b$ is isomorphic to $ (\mathbb{C}^{\times})^{\ell} \times (\mathbb{C}^{\times})^{n-1}$ acting on $S_H \cong \mr{Rep}(Q,\mathfrak{d})$, with the second factor $(\mathbb{C}^{\times})^{n-1}$ acting trivially.
		\item  $\C[S_H] = \C[z_0, \ds, z_{\ell-1}]$ and $\C[\h^{W_H}] = \C[x_2, \ds, x_n]$ such that 
		$\delta_i |_{S_H} = z_i x_2 \cdots x_n$ for $0\leq i\leq \ell-1$ and 
				$$
		\delta_{\infty} |_{S_H} = \prod_{2 \le j \le n} (u - x_j^{\ell})^2 \prod_{2 \le i < j \le n} (x_i^{\ell} - x_j^{\ell})^2,
		$$
		where $u = z_0 \cdots z_{\ell-1}$. 	\end{enumerate}
\end{lemma}

\begin{proof}
	Part~(1) is immediate. 
	
(2)	If $(g^{(i)}) \in G_b$ then $g^{(i+1)} b (g^{(i)})^{-1} = b$ or $g^{(i+1)} b = b g^{(i)}$. This implies that $g^{(i)}$ commutes with $b^{\ell}$. Hence $g^{(i)}$ is a diagonal matrix with $g^{(i)}_{kk} = g^{(j)}_{kk}$ for all $k > 1$ and $0 \le i,j \le \ell-1$. This forces $(g^{(i)})\in (\mathbb{C}^{\times})^{\ell} \times (\mathbb{C}^{\times})^{n-1}$, where the first factor records the $\ell$ entries $g_{11}^{(i)}$ and the second factor records $g_{kk}^{(0)}$ for $k > 1$. 
	Conversely, it is clear that $(\mathbb{C}^{\times})^{\ell} \times (\mathbb{C}^{\times})^{n-1}\subseteq G_b$. 
	
(3)	We compute $\g \cdot b$. If $Y = (Y^{(i)}) \in \g$, then $(Y \cdot b)^{(i)} = Y^{(i+1)} b - b Y^{(i)}$ belongs to the space of matrices $X$ with $X_{11} = 0$. Moreover, if $X^{(i)} = Y^{(i+1)} b - b Y^{(i)}$ then $\sum_i X^{(i)} = \sum_i (Y^{(i)} b - b Y^{(i)})$ is a matrix with zero on the diagonal. Thus, 
	$$
	\g \cdot b \ = \  \left\{ (X^{(i)})  \,\,\Big|\,\, X^{(i)}_{11} = 0,\,\, \Bigl(\sum_i X^{(i)}\Bigr)_{jj} = 0, \, \forall \, j \right\}. 
	$$
	A $G_b$-stable complement to $\g \cdot b$ is $S = \h + \C^{\ell}$, where $\C^{\ell}$ is the span of all $e_{11}^{(i)}$. This implies that $S_H = \C^{\ell} \cong \mr{Rep}(Q,\mathfrak{d})$ under the obvious embedding given by putting all entries of $X^{(i)}$ to zero except $X_{11}^{(i)}$.
	
	Finally, if $z_0 , \ds, z_{\ell-1}$ are the functions on $\C^{\ell}$ dual to the $e_{11}^{(i)}$ then $z_i |_{\h} = x_1$  for $0\leq i\leq \ell-1$ and $\C[S_H] = \C[z_0, \ds, z_{\ell-1}, x_2, \ds, x_n]$.  
\end{proof}

\begin{lemma}\label{lem:slice2cyclic}
	Let $b = (b_2,b_2,b_3,\ds) \in \h$, where $b_i \neq 0$ and $b_i^{\ell} \neq b_j^{\ell}$ for $i \neq j$. 
		\begin{enumerate}
		\item This is a generic point on the hyperplane $x_1 - x_2 = 0$. 
		\item The group $G_b  $ is isomorphic to $ GL(2) \times (\C^{\times})^{n-2}$ acting on $S_H \cong \mf{sl}(2)$, with the second factor $(\mathbb{C}^{\times})^{n-2}$ acting trivially. 
		\item $\C[S_H] = \C[x,y,z]$ and $\C[\h^{W_H}] = \C[x_1 + x_2,x_3 \ds, x_n]$ such that $\delta_i |_{S_H} = ((x_1 + x_2)^2 - u)x_3 \cdots x_n$ and $\delta_{\infty} |_{S_H} = g u$, where $u = xy + z^2$ and $g$ is not divisible by $u$. 
			\end{enumerate}
\end{lemma}

\begin{proof} 
This is similar in style to the proof of Lemma~\ref{lem:slice1cyclic}, and details are  left to the reader. 
\end{proof}

\begin{notation}\label{kappa-quiver}
Given $\vs=(\vs_0,\dots,\vs_{\ell-1}, \vs_{\infty}=0)$ as above,  set $\kappa_{0,0} = \kappa_{0,1} = 1/2$ and 
$$
\kappa_{1,i} = \vs_i + \frac{\ell- i}{\ell} - \delta_{i,0}, \quad \textrm{for }  0\leq i\leq \ell-1. 
$$	
\end{notation}
 
\begin{theorem}\label{thm:cylicquiverradialparts}   If  $V=\mr{Rep}(Q_{\ell}, n\mathfrak{d})$, then
  $\Im(\rad_{\vs}) =\Ak(W)$, where $\kappa$ is defined by Notation~\ref{kappa-quiver}.
\end{theorem}
 
 \begin{remark}\label{rem:cylic} This result   generalises to the case when  $\varsigma_{\infty} \neq 0$, in which case   the $\kappa_{1,i}$ are still defined by  Notation~\ref{kappa-quiver} but $\kappa_{0,0} = \kappa_{0,1} = \varsigma_{\infty} + \frac{1}{2}$.
 
  \end{remark}
 
\begin{proof}
	This is an application of \cite[Theorem~5.21]{BLNS}. Therefore, to each $W$-orbit $[H]$ of hyperplanes in $\mathcal{A}$ we must (a) compute the parameter $\bpsi$ for the rank one slice $S_H$, and (b) compute the roots $\lambda_i$ of the $b$-function associated to the radial parts map $\rad_{\bpsi}$ for $S_H$. For an explanation of the computation of $\bpsi$, see \cite[Lemma~5.17 and Equation~5.18]{BLNS}. The computation of the $\lambda_i$ is explained in \cite[Section~4]{BLNS}. 
	
	Recall that for the wreath product $\mathbb{Z}_{\ell} \wr \mathfrak{S}_n$, there are only two $W$-orbits of reflecting hyperplanes. 
	
	Consider first the parameters $\kappa_{1,0}, \ds, \kappa_{1,\ell-1}$ associated to the $W$-orbit of the reflecting hyperplane $x_1 = 0$.
	 If $S_H$ is the $G_b$-module defined in Lemma~\ref{lem:slice1cyclic}, then we need to compute the parameter 
	 $\bpsi =(\sum_i n_{i,0}\vs_i,\dots, \sum_i n_{i,\ell-1}\vs_{\ell-1})$ as defined in
	 \cite[Equation~(5.18)]{BLNS}  for the rank one polar 
	 representation $(G_b,S_H)$. For this slice, Lemma~\ref{lem:slice1cyclic}(3) implies that the coefficients $n_{i,j}$ are just given by $n_{i,j}=\delta_{ij}$  for all $i,j$ and $n_{\infty, j} = 0$; thus  $\bpsi = (\vs_0, \ds, \vs_{\ell-1})$. Proposition~\ref{prop:cyclicrank1rad} says that $\lambda_i = \vs_i$ in this case. Hence, it follows from \cite[Theorem~5.21]{BLNS} that
	$$ 	
	\kappa_{1,i} = \vs_i + \frac{\ell- i}{\ell} - \delta_{i,0} , \quad \textrm{for } i = 0, \ds, \ell-1. 
	$$	
	Next, consider the parameters $(\kappa_{0,0},\kappa_{0,1})$ associated to the $W$-orbit of the reflecting hyperplane 
	$x_1 - x_2 = 0$. If $S_H$ is the $G_b$-module defined in Lemma~\ref{lem:slice2cyclic}, then, once again, we need to compute the parameter $\bpsi = (\sum_i n_{i,\infty}\vs_i)$ by computing the $n_{i,j}$. This time, Lemma~\ref{lem:slice2cyclic}(3) implies that $n_{i,\infty} = 0$ for $i \neq \infty$ and $n_{\infty,\infty} = 1$. Since we have taken $\vs_{\infty} = 0$, it follows that $\bpsi = (0)$. The analogue of Proposition~\ref{prop:cyclicrank1rad} for the adjoint representation of $\mf{sl}(2)$ (see \cite[Lemma~3.1(2)]{LevWashington}) shows that $\kappa_{0,0} = \kappa_{0,1} = 1/2$. 
\end{proof}

\subsection*{Simplicity} 
Here we determine when the ring $\Ak(W)$ from Theorem~\ref{thm:cylicquiverradialparts} is simple and 
we begin by considering  the spherical  algebra $\Ak( \BZ_{\ell})$; thus $n=1$. 
 In this special case, the 
parameters $\kappa_{0,j}$ do not appear,  for the reasons given after \eqref{eq:delta-h}.
Thus  our $\kappa_{1,j}$ equals the $\kappa_j$ in the explicit presentation of $\Ak( \BZ_{\ell})$ given  in  \cite[Example~2.11]{BLNS}.  

In the next lemma,  the second subscript in  $\kappa_{1,j}$ is interpreted as   belonging  to $\{0,1, \dots ,\ell-1 \}$ mod$\, \ell.$ 
  
\begin{lemma}\label{lem:sphericalsimple}   
\begin{enumerate}
\item    The Cherednik algebra $\Hk(\BZ_\ell)$ is simple if and only if,  for all $m\in \mathbb{Z}_{\geq 1}$ and all $0\leq i\leq \ell-1$, one has 
$\ell (\kappa_{1,i}-\kappa_{1,i+m}) \not=m$. 

\item	The spherical  algebra $\Ak(\BZ_{\ell})$ is simple if and only if, for $0\leq i\leq \ell-1$, 
	$$
 \ell(	\kappa_{1,i+1} - \kappa_{1,i+1+j}) \neq j, \quad \forall j \in \BZ, \ \text{with }  j \ge \ell - i.
	$$
\end{enumerate}
\end{lemma}

\begin{proof} (1)  Combine \cite[Example~2.6, p. 45]{Valethesis}   with \cite[Lemma~2.4]{Valethesis}.
Since we will need the details of the proof in Part~(2),  here are the key ideas from \cite{Valethesis} written in the notation from that thesis.

 Write 
  the \emph{standard modules}\label{standard-module-defn} for $\Hk(\BZ_\ell)$  as $M(\tau_i)=\C[x]\otimes\tau_i$ with unique simple factor $L(\tau_i)$, where $\tau_i$ is the irreducible representation of $W$ on which the generator $s$ acts by $\omega^{-i}$. Then  $\Hk(\BZ_\ell)$ has generators $x,y,s$ with the key relation 
  $[y,x]= 1 +\sum_{p=0}^{\ell -1}(\kappa_{1,p+1}-\kappa_{1,p})\ell e_p$ while the module structure of $M(\tau_i)$ is determined by 
  $$y\cdot (x^a\otimes \tau_i)=(a+\ell(\kappa_{1,i+a}-\kappa_{1,i}))x^{a-1}\otimes \tau_i.$$
  Then  $\Hk(\BZ_\ell)$ is not simple if and only if some $L(\tau_i)$ is finite dimensional, if and only if some $y\cdot(x^a\otimes \tau_i)=0$. This reduces to the assertion of the lemma.

    (2)  Our convention is that the trivial idempotent is $e=e_0$.  By \cite[Lemma~2.8]{LosevTotally} $\Ak$ is not simple if and only if there exists $i$ with $\dim L(\tau_i)<\infty$ and $eL(\tau_i)\not=0$. For this to happen there must exist $a>0$ such that $y\cdot(x^a\otimes\tau_i)=0$ where $a$ is also greater than the smallest degree in $e M(\tau_i)$. Since $sx^a\otimes \tau_i= \omega^{-i-a}x^a\otimes \tau_i$, this degree is $\ell-i$ whenever $i>0$. (For $i=0$, the smallest degree in $eM(\tau_0)$  is zero and one should declare that
    $M(\tau_0)=M(\tau_{\ell})$, instead.) So, once again, this reduces to the  assertion
    of the lemma.
 \end{proof}
 
 The formul\ae\ from Lemma~\ref{lem:sphericalsimple} become more comprehensible when we reinterpret them in terms of the $\vs_i$.  A simple computation shows that the lemma is equivalent to the following result.

\begin{corollary}\label{ex:sphericalsimple}  Assume that $n=1$ and that the $\kappa_{1,j}$ are defined by 
	Notation~\ref{kappa-quiver}. 
	Identify  $\vs_0\equiv \vs_{\ell}$ and $\kappa_{1,0}\equiv \kappa_{1,\ell}$, so that $\kappa_{1,i} = \vs_i - i /\ell +1$ for $i = 1, \ds, \ell$.
\begin{enumerate}
\item  $\Ak(\BZ_{\ell})$ is simple if and only if 
	$
	(\vs_i - \vs_{j}) \not\in \mathbb{Z}_{\geq 1},
	$ for $ 1\leq i,j  \leq \ell.$ 
	
\item	 In particular,  $\Ak(\BZ_{\ell})$ is simple if $\vs_i = 0$ for all $i$.\qed
\end{enumerate}
	\end{corollary}

As an aside, note that the condition in Corollary~\ref{ex:sphericalsimple}(1) is symmetric under $\vs\mapsto -\vs$ and so $\Ak(\BZ_{\ell}) $ is simple if and only if $A_{\cdag}(\BZ_{\ell})$ is simple, in the notation  
of \eqref{eq:left-rad}.

The next two results show that it is easy to pass from $\Ak(\BZ_{\ell})$ to $\Ak(\Z_{\ell} \wr \mathfrak{S}_n )$.

\begin{lemma}\label{lem:cherednik-tensor}
	If $\kappa_{0,0}= \kappa_{0,1}$, then $\Hk(\Z_{\ell} \wr \mathfrak{S}_n ) = \Hk(\Z_{\ell})^{\o n} \rtimes \mathfrak{S}_n$. 
\end{lemma}

\begin{proof} Pick dual bases $\{x_i\}$ of $\h^*$ and $\{y_i\}$ of $\h$ and write 
$\C[\h]=\C[x_1,\dots,  x_n]$ as usual.  Consider the defining relations for  the Cherednik algebra $\Hk(W)$ from
\cite[Definition~2.1]{BLNS}. Using the fact that  $\kappa_{0,0}= \kappa_{0,1}$, these   simplify to give
$$x_i x_j = x_j x_i, \quad\text{and}\quad y_i y_j = y_j y_i \quad\text{for }1 \le i, j \le n,$$ 
 while $ [y_i, x_j] = 0 $ for $1 \le i \neq j \le n,$ and
	$$
	[y_i, x_i] = 1 + \sum_{j = 0}^{\ell-1} (\kappa_{1,j+1} - \kappa_{1,j}) \sum_{r = 0}^{\ell-1} \omega^{jr} \gamma_i^r, \quad  \text{ for }\, 1 \le i \le n.
	$$
 (See, for example, \cite[p.28]{Vale}, noting that his $\kappa_{00}$ is our $\kappa_{00}-\kappa_{01}=0$.) On the other hand $\Hk(\Z_{\ell})$ is generated by elements  $u,v$ and $\gamma$ with 
	$\gamma(v) = \omega v$, $\gamma(u) = \omega^{-1} u$ and satisfying 
	$$
	[v,u] = 1 + \sum_{j = 0}^{\ell-1} (\kappa_{1,j+1} - \kappa_{1,j}) \sum_{r = 0}^{\ell-1} \omega^{jr} \gamma^r.
	$$
	Let $u_i = 1 \o \cdots \o u \o \cdots \o 1$, $v_i = 1 \o \cdots \o v \o \cdots \o 1$ and 
	$\gamma_i = 1 \o \cdots \o \gamma \o \cdots \o 1$,  regarded as elements of $H_{\kappa}(\Z_{\ell})^{\o n}$ with  the nontrivial terms  in the $i^{\mathrm{th}}$ factor.  Then  one obtains an isomorphism $H_{\kappa}(W) \isom H_{\kappa}(\Z_{\ell})^{\o n} \rtimes \mathfrak{S}_n$
	by mapping $y_i \mapsto v_i$, $ x_i \mapsto u_i\in  H_{\kappa}(\Z_{\ell})^{\o n}$ and the natural identification  on $W=\Z_{\ell}  \wr \mathfrak{S}_n$.   	
\end{proof}

\begin{lemma}\label{lem:HZcHsimple} Define $\kappa$  by Notation~\ref{kappa-quiver}. 
	Then $\Ak(\Z_{\ell} \wr \mathfrak{S}_n )$ is simple if and only if $\Ak(\BZ_{\ell})$ is simple. 
\end{lemma}

\begin{proof}     
	Let $e_0 \in \C \BZ_{\ell}$ and $e \in \C W$ denote the trivial idempotents in the respective group algebras. Also,  Lemma~\ref{lem:cherednik-tensor} implies 
that $\Hk(W) = \Hk(\BZ_{\ell})^{\otimes n} \rtimes \mathfrak{S}_n$. 
	
	First  assume that $\Ak(\BZ_{\ell})$ is simple. Then the ring $\bigl(e_{0} \Hk(\BZ_{\ell}) e_{0}\bigr)^{\otimes n}$
	 is simple and hence,  by \cite[Proposition~7.8.12]{MR},  so is the ring 
	$$
	B := \bigl(e_0 \Hk(\BZ_{\ell}) e_0\bigr)^{\otimes n} \rtimes \mathfrak{S}_n.
	$$
  Since $B e B \neq 0$, we deduce that $B e B = B$ 
	and thus $e B e = e \Hk(W) e$ is Morita equivalent to $B$. In particular, $e H_{c}(W) e = \Ak(W)$ is simple. 
	
	If $e_0 \Hk(\BZ_{\ell}) e_0$ is not simple then there exists a finite-dimensional simple $\Hk(\BZ_{\ell})$-module $L$ with $e_0 L \neq 0$ \cite[Lemma~2.8]{LosevTotally}.  Consider 
	  $M := L^{\otimes n}$ as a module over $\Hk(W) = \Hk(\BZ_{\ell})^{\otimes n} \rtimes \mathfrak{S}_n$. Then $0\not=e M $ is finite-dimensional and so $e \Hk(W) e$ cannot be simple.  
\end{proof}

Combining Corollary~\ref{ex:sphericalsimple} and Lemma~\ref{lem:HZcHsimple}
with Theorem~\ref{thm:cylicquiverradialparts}, we can completely determine when $\Ak$ is simple. Here, we again identify $\vs_0\equiv \vs_{\ell}$.

\begin{proposition}\label{prop:cyclcirad0simple}   If  $V=\mr{Rep}(Q_{\ell}, n\mathfrak{d})$, then
	 $\Ak(W)=\Im(\rad_{\vs})$ is simple if and only if  
	$
	(\vs_i - \vs_{j}) \not\in \mathbb{Z}_{\geq 1},
	$ for all $ 1\leq i,j  \leq \ell.$
	
	In particular, $\Ak(W)$ is simple when $\vs_0=\cdots=\vs_{\ell-1} = 0$. 
	\qed \end{proposition} 

In the next two subsections we will examine in detail  the Harish-Chandra module $\eG_0$  for the representation space  $V=\mathrm{Rep}(Q_{\ell},n\mathfrak{d})$.  We will do this for  a particularly simple choice 
of the parameter $\vs$, by taking 
$\vs_i = 0$ for all $i  $.     In this case,  the spherical  algebra $\Ak(W)$ is simple by
Proposition~\ref{prop:cyclcirad0simple}.

For this choice of $\vs$,  Notation~\ref{kappa-quiver} becomes:  
\begin{equation}\label{final-kappa-choice}
	\kappa_{0,0} = \kappa_{0,1}= 1/2, \textrm{ and } \kappa_{1,0}=0 \text{  while  }
	\kappa_{1,i} = 1-i/\ell \text{ for } 1\leq i \leq \ell-1.
\end{equation}

It will also be useful to write out the explicit presentation of the  Hecke algebra in this case, for which we follow \cite[Definition~3.1]{ArikiKoike}.  

\begin{definition}
	The \emph{cyclotomic Hecke algebra} \label{cyclotomic-defn}
	$\euls{H}_q(W)$ for $W=\Z_{\ell} \wr \mathfrak{S}_n$,  is the finite dimensional algebra generated by $T_0, T_1, \ds, T_{n-1}$, satisfying the braid relations
	together with the additional relations
	\begin{equation}\label{eq:braid}
		\prod_{r = 0}^{\ell-1} (T_0 - u_r) = 0, \quad\text{and}\quad  (T_i - q_0)(T_i - q_1) = 0, \quad \forall \ i > 0.
	\end{equation}
	where $q_0,q_1, u_0, \ds, u_{\ell-1} \in \mathbb{C}^{\times}$. 
\end{definition}

Recall from the discussion preceding Corollary~\ref{A-projectives} that the functor $\mr{KZ}$  is an exact functor $\euls{O}_{\kappa,0}(W) \to \euls{H}_q(W)\lmod$ which factors through $\Osph_{\kappa}(W)$.
 Equation~\eqref{eq:Heckecrg2} says that the parameter $q$ is given by 
\begin{equation}\label{eq:qkappa}
	\begin{aligned}
		& q_0 = \exp \left(-2 \pi i\, \kappa_{0,0}\right), \qquad   q_1 = -\exp \left(-2 \pi i\, \kappa_{0,1}\right),
		\ \text{and}  \\
		& u_j = \omega^{-j} \exp\left( -2 \pi i\,\kappa_{1,j}\right)  \qquad  \text{for}\ 0\leq j\leq \ell-1.
	\end{aligned}
\end{equation} 
Thus, by our choice of parameters, \eqref{eq:braid}  
becomes
\begin{equation}\label{eq:braid2}
	(T_0 - 1)^{\ell} = 0, \quad (T_i - 1)(T_i + 1) = 0, \quad \forall \ i > 0.
\end{equation}
In this case $\euls{H}_q(W)\lmod = (\C[t]/(t^{\ell}))^{\otimes n} \rtimes \mathfrak{S}_n$, where $t = T_0 - 1$. 

\begin{remark}\label{rem:cyclotomic}  When $n=1$ the generators $T_i$ for $i>0$ do not appear.
	It is therefore immediate from \eqref{eq:braid2} that $\euls{H}_q(W)\cong \C[t]/(t^{\ell})$ when $n=1$. 
\end{remark}

\medskip
\subsection*{The Harish-Chandra module  $\eG_0$ when $n=1$.}
We start by considering the case of $n=1$, but still with $\ell\geq 2$. As mentioned above, we also  assume 
that $\vs=0$ and this is the parameter for which  $\eG_0$ is most complicated. Nonetheless,  we are still able to  give a fairly 
complete picture of the structure of $\eG_0$  and this  provides an interesting generalisation of the observations from Section~\ref{HC-example}.

Thus, in the notation from Definition~\ref{defn:Glambda}, we are interested in  describing the 
Harish-Chandra module $\eG = \eG_0$  for $V=\mathrm{Rep}(Q,\mf{d})$.
As on page~\pageref{subsec:quivercomputations},  write $\C[V]=\C[x_0,\cdots, x_{\ell-1}]$  with $\partial_i=\frac{\partial}{\partial x_i}$. Then a routine computation shows that
$\eG$ can be written explicitly as
\begin{equation}\label{eq:HC-quiver}
	\eG \ = \ \dd(V)/ \euls{I}_{\eG} \quad \text{for } \euls{I}_{\eG}= \Big( \dd(V)\partial_0\cdots \partial_{\ell-1}  + \sum_{j=1}^{\ell-1} \dd(V) (x_j\partial_j-x_0\partial_0) 
	\Big).
\end{equation}

Let $I := \{ 0, \ds, \ell-1\}$ and, for each (possibly empty) subset $J \subsetneqq I$, define 
$\CJ=    \{ \prod x_i^{n_i} \partial_j^{m_j} : i\in I\smallsetminus J,\, j\in J\}$,
regarded as an Ore set in $\dd(V)$.  
The key to understanding the Harish-Chandra module $\eG=\eG_0$ is to understand it locally, 
for which it suffices to understand $\eG$ once we invert the various $\CJ$.  Finally, set 
$$
L(J) = \dd(V)/ \sum_{j\in J} \dd(V) x_j + \sum_{i\in I\smallsetminus J} \dd(V)\partial_i
$$
for the simple holonomic $\dd(V)$-module that we associate to $J$.

\begin{lemma}\label{lem:localHC} Pick a subset  $J\subsetneqq I$. Then:
\begin{enumerate}
\item  the localised module $\eG_{\CJ}$  has  length $\ell -|J|$ with all simple subfactors isomorphic to 
$L(J)_{\CJ}$. 
\item Moreover, $[\eG_{\CJ} : L(J)_{\CJ}] = [\eG: L(J)]$.
\end{enumerate}
\end{lemma}
 
\begin{proof}
(1)	By  permuting the subscripts, we may assume that $J=\{{r+1},\dots, {\ell-1}\}$ for some $r$  (or $J=\emptyset$
	of course).   
	
	Set $u_i=\partial_i-x_i^{-1}x_0\partial_0$ and $v_i=-x_i$ for $1\leq i\leq r$ while 
	$$u_j=\partial_j^{-1}(x_j\partial_j-x_0\partial_0)= x_j-\partial_j^{-1}x_0\partial_0-\partial_j^{-1}  $$ and $v_i = \partial_j$ for $r+1\leq j\leq \ell-1$. Finally 
	write $u_0=x_0\partial_0$  and $v_0=-x_0$.  A simple computation shows that 
	$v_0u_0=(u_0-1)v_0$  and 
	$$[v_i, u_i]=1 \ \text{and} \ v_0u_i=(u_i \pm v_i^{-1})v_0  	\qquad  \text{ for}\  0\leq i\leq \ell$$
	with all other commutators being zero.
	 	
	Therefore, $C=\C\langle u_0,\dots, u_{\ell-1}\rangle$ is a 
	commutative polynomial subring of $\dd(V)_{\CJ}$ in $\ell$ variables. Moreover, 
	$\dd(V)_{\CJ}$ is a free (left or right)  $C$-module with basis 
	$\big\{ v_0^{m_0}\cdots v_{\ell-1}^{m_{\ell-1}} : m_j\in \mathbb{Z}\big\}$.

Consider the presentation of $\eG_{\euls{C}}$ arising from \eqref{eq:HC-quiver}. 
	Multiplying the first generator by $(\partial_{r+1}\cdots\partial_{\ell-1})^{-1}$, we can replace it 
	by $\partial_0\cdots \partial_{r}$. Then a simple induction using the elements 
	$u_i=\partial_i-x_i^{-1}x_0\partial_0\in \dd(V)_\euls{C}$ for $1\leq i\leq r$
	shows that 
	\begin{equation}\label{eq:localHC1}
		\eG_{\euls{C}} \ = \ \dd(V)_{\euls{C}} /\euls{I} \quad \text{for } 
		\euls{I}= \Big( \dd(V)_{\euls{C}} u_0^r +\sum_{j=1}^{\ell-1} \dd(V)_{\euls{C}}  u_j	\Big).
	\end{equation}
	Finally, $M=C/Cu_0^r+\sum_{i=1}^{\ell-1} C u_i$ is a left $C$-module of length $r$ with all subfactors isomorphic to $N=C/\sum_{i=0}^{\ell-1} C u_i$. Hence $\dd(V)_{\CJ} \otimes_C M$ has a chain of $r$ submodules 
	with each subfactor isomorphic to $\dd(V)_{\CJ} \otimes_C N$. By   \eqref{eq:localHC1}, 
	$\dd(V)_{\CJ} \otimes_C M\cong \eG_{\CJ}$ while, by inspection,  $\dd(V)_{\CJ} \otimes_C N \cong L(J)_{\CJ}$. 
	This proves (1).
	
(2) 	This follows from the following observations: (a) if $L$ is a simple $\dd(V)$-module with $L_{\CJ} \neq 0$ then $L_{\CJ}$ is a simple $\dd(V)_{\CJ}$-module, and (b) if $L,L'$ are simple  $\dd(V)$-modules and $L_{\CJ} \cong L'_{\CJ}\not=0$, then $L \cong L'$.
\end{proof}

\begin{proposition}\label{prop:quiversystem}  Assume that $V=\mathrm{Rep}(Q_{\ell},\mathfrak{d})$ with $\vs=0$. 
	Then the   composition factors of $\eG$ are the $L(J)$, for $J\subsetneqq I$ with $L(J)$ occurring $\ell -|J|$ times.
	Hence the length of $\eG$ is $\sum_{m= 0}^{\ell} \binom{\ell}{m}   m= {\ell} 2^{{\ell}-1}$. 
\end{proposition}

\begin{proof} The direct sum $D := \bigoplus_{J \subseteq I} \dd(V)_{\CJ}$ is a faithfully flat overring of $\dd(V)$ since any module $M$ that is killed by each of these localisations must, in particular, be both $x_i$-torsion and $\partial_i$-torsion for some $i$. Only $M=0$ satisfies that property.

Lemma~\ref{lem:localHC} says that $\ell - |J| = [\eG_{\CJ} : L(J)_{\CJ}] = [\eG: L(J)]$. Hence the length 
$length(\eG)$ of $\eG$ satisfies $length(\eG)\geq \sum_{J\subsetneqq I} (\ell - |J|) = {\ell} 2^{{\ell}-1}$.
 Conversely, since $D$ is faithfully flat over $\dd(V)$,  $length(D \o_{\dd(V)} \eG)\geq length(\eG)$. 
 But  $\eG_{\CI}=0$ which,
combined with   Lemma~\ref{lem:localHC},  implies that 
$$ 
D \o_{\dd(V)} \eG = \bigoplus_{J\subsetneqq I} \eG_{\CJ}
$$
also has length $\sum_{J\subsetneqq I} (\ell - |J|)$. Therefore,   $length(D \o_{\dd(V)} \eG)=length(\eG)$,
 and the proposition follows.   
\end{proof}

We can be more precise about the localisation $\eG|_{V_{\reg}} = \eG_{\euls{C}({\emptyset})}$.

\begin{lemma}\label{lem:localHC2} Assume that $V=\mathrm{Rep}(Q_{\ell},\mathfrak{d})$
	with $\vs=0$.  Then 
	$\eG|_{V_{\reg}} $ is a serial module of length $\ell$ with all simple subfactors isomorphic to $\C[V_{\reg}]$.
\end{lemma}

\begin{proof} Except for proving that $\eG|_{V_{\reg}}$ is serial, this all follows from Lemma~\ref{lem:localHC}.
	
As $\vs=0$,  Proposition~\ref{prop:cyclcirad0simple} implies that  $\Ak(W)$ is simple. Thus, Corollary~\ref{cor:eniso} says that $\End_{\dd(V)}(\eG) \to \End_{\dd(V)}(\eG|_{V_{\reg}})$ is an isomorphism. Remark~\ref{rem:cyclotomic} combined with Lemma~\ref{lem:endGHeckealg}, then says that 
$$\End_{\dd(V)}(\eG|_{V_{\reg}}) \cong \End (\eG)\cong \euls{H}_q(W) \cong \C[t]/(t^\ell).$$ Since  $\eG$ is a projective object in $\eC$ by Proposition~\ref{G-projective}, this is     only possible if $\eG|_{V_{\reg}}$ is  serial.
\end{proof}

\begin{corollary}\label{thm:neq1cyclicminext} Assume that $V=\mathrm{Rep}(Q_{\ell},\mathfrak{d})$
	with $\vs=0$. Then 
	\begin{enumerate}
		\item $\eG_0$ has a simple socle and simple top, both isomorphic to the $\dd(V)$-module
		$\C[V]$.
		\item $\End_{\dd} (\eG_0) \cong \C[t] / (t^{\ell})$. 
	\end{enumerate}
\end{corollary}

\begin{remark}\label{striking-remark}  Corollary~\ref{thm:neq1cyclicminext} shows the striking nature of Theorem~\ref{torsionfree}: despite the fact that a composition series of $\eG$  contains   $(\ell-1) 2^{\ell-1}$ $\delta$-torsion modules,  none of them can appear  as the socle  or  the top of   $\eG$.
\end{remark}

\begin{proof}  (1) By Theorem~\ref{torsionfree}, $\C[V]$ is the  only simple module that can appear as either  a submodule or a factor module of $\eG_0$. That there is only one of each  follows from Lemma~\ref{lem:localHC2}.
	
	(2) This was observed within the proof of  Lemma~\ref{lem:localHC2}.
\end{proof}

Note that when $\ell=2$, Proposition~\ref{prop:quiversystem} and Corollary~\ref{thm:neq1cyclicminext} 
reduce to  Proposition~\ref{HC1-lemma} and so they do indeed provide a generalisation of that result. The one exception to this statement is that, for $\ell>2$, a description of the full lattice of submodules of $\eG_0$ is not 
  easy.

\subsection*{The Harish-Chandra module $\eG_0$ for  $n\geq 1$.}
In this subsection we consider the structure of the Harish-Chandra module $\eG_0$ for the representation space  $V=\mathrm{Rep}(Q_{\ell},n\mathfrak{d})$ with $n\geq1$. Once again, we take
  $\vs_i = 0$ for all $i$ so that $\kappa$ is given by \eqref{final-kappa-choice}.   The value $\vs_i \equiv 0$ can be view as the analogue of the most singular central character for category $\euls{O}$ associated to a simple Lie algebra. 
  
  Recall that the spherical algebra $\Ak(W)$ is simple by
Proposition~\ref{prop:cyclcirad0simple}, while  the Hecke algebra  $\euls{H}_{q}(W)$  has the particularly simple form given by \eqref{eq:braid2}.

   \begin{proposition}\label{lem:Heckecyclicquierequiv} Define $\kappa$  by \eqref{final-kappa-choice}.
	Then   $\eQ = \Ak / \Ak (\Sym  \h)_+^W$
   is a projective object  in $\Osph$ and 
	$$
	\End_{\Ak}(\euls{Q}) \cong\euls{H}_{q}(W) \cong \left( \C[t] / (t^{\ell}) \right)^{\otimes n} \rtimes \mathfrak{S}_n.
	$$
\end{proposition}   

\begin{proof} By our choice of parameters,  $\Ak(W)$ is simple and so  
  Proposition~\ref{prop:Qlambdaprojiffsimple} implies that   $\eQ$ is projective in $\Osph$.   Moreover, Lemma~\ref{lem:endGHeckealg} implies that $\End_{\Ak}(\euls{Q}) \cong \euls{H}_{q}(W)$.   The result therefore follows from the explicit relations \eqref{eq:braid2}, with $t = T_0-1$. \end{proof}

Proposition~\ref{lem:Heckecyclicquierequiv} implies that the irreducible representations of $\euls{H}_{q}(W)$ are all inflated from irreducible representations of the symmetric group $\mathfrak{S}_n$. This  leads to the following result. 

\begin{theorem}\label{HC-quiver-decomp}  Assume that $\kappa$ is defined by \eqref{final-kappa-choice}.
	Then the  Harish-Chandra module $\eG_0$ admits a decomposition 
	$$
	\eG_0 = \bigoplus_{\rho \in \mr{Irr}\, \mathfrak{S}_n} (\eG_{0,\rho})^{\oplus \dim \lambda}, 
	$$
	where the $\eG_{0,\rho}$ are pairwise non-isomorphic, indecomposable summands, satisfying  $\End_{\dd(V)}(\eG_{0,\rho}) \cong \C[t] / (t^{\ell})$.  
\end{theorem}

\begin{proof}  Since	  $\Ak$ is simple,  Corollary~\ref{A-projectives}(2) implies that  the functor 
	  $$
	   \Hom_{\Ak}(\eQ, -) \colon \Osph_{0} \to \euls{H}_q(W) \lmod$$
	 is  an equivalence. By Lemma~\ref{lem:endGHeckealg},  $\Hom_{\Ak}(\eQ,\eQ)$ equals 
	  the regular representation $\euls{H}_q(W)$. Since $\euls{H}_q(W)$ decomposes into a sum of  indecomposables as 
	  \begin{equation}\label{eq:HC-quiver-decomp}
	  \euls{H}_q(W) \cong \bigoplus_{\rho \in \mr{Irr} \, \mathfrak{S}_n} (P_{\rho})^{\oplus \dim \rho},
	  \end{equation}
	   we deduce that
	   $\eQ = \bigoplus_{\rho \in \mr{Irr} \, \mathfrak{S}_n} (\eQ_{\rho})^{\oplus \dim \rho}$, where the 
	   $\eQ_{\rho}$ are 
	   pairwise non-isomorphic indecomposable summands. 
	Arguing as in the proof of Theorem~\ref{thm:semi-simplicity}, it follows from 
	Corollary~\ref{cor:hom(ML)}  that
		$$
	\eG_0 = \bigoplus_{\rho \in \mr{Irr}\, \mathfrak{S}_n} {}^{\perp} \mathbb{H}(\eQ_{\rho})^{\oplus \dim \lambda}, 
	$$
	where the ${}^{\perp} \mathbb{H}(\eQ_{\rho})$ are pairwise non-isomorphic indecomposable modules. 
	
	Finally, by combining \eqref{eq:HC-quiver-decomp} with  the identity  $ \euls{H}_q(W)\cong \End_{\euls{H}_q(W)}(\euls{H}_q(W))$, it is easy to see that $\End_{\euls{H}_q(W)}(P_{\rho}) \cong \C[t] / (t^{\ell})$.
	Thus    $\End_{\dd(V)}(\eG_{0,\rho})\cong \C[t] / (t^{\ell})$.
\end{proof}

\section{Examples III: Further Examples}\label{sec:otherexamples}

In this section we describe other  examples which illustrate some of the earlier results of the paper.

\medskip

\begin{example}\label{ex:strangerankonepolarrad} This is an example  where
  $\Ak(W)$ is simple but its left-right analogue $A_{\cdag}(W)$ is not. This  justifies the comment made in   Remark~\ref{rem:K-defn1}. 

 For $n \ge 1$,   let $G = \C^{\times}$ act on $V = \C^2$ with weights $(n,-1)$. Then $\C[V]^G = \C[xy^n]$. Every orbit $G \cdot (a,b)$ with $ab \neq0$ is closed because it is the set of zeros of $xy^n - ab^n = 0$. For any such $a,b$, the space $\h = \C \cdot (a,b)$ is a Cartan subalgebra of $V$ because $\mf{g} \cdot (a,b) = \C (na,-b)$. If $(\lambda^n a, \lambda^{-1} b) = (ta,tb)$ for some $\lambda,t \neq 0$ then $\lambda^{n+1}~=~1$. Therefore, 
 $W = \Z _{n+1}$ and $V$ is a stable visible locally free polar representation. Under the factorisation 
 $\delta = \delta_1^{m_1} \cdots \delta_k^{m_k}$ from \cite[Equation~(3.10)]{BLNS}, we have 
\begin{equation}\label{eq:deltafactormbig}
	\delta = x y^n, \quad \textrm{ and } \quad \delta_1 = x, \, \delta_2 = y.
\end{equation}
A direct calculation shows that 
$$
\rad_{\vs} (\partial_x \partial_y^n) = \frac{1}{z} (z \partial_z + \vs_1) \prod_{i = 0}^{n-1} (n z \partial_z + \vs_2 - i). 
$$
Therefore, by  \cite[Corollary~4.15 and Corollary~4.21]{BLNS}, the image of $\dd(V)^G$ under $\rad_{\vs}$ equals $\Ak(W)$, where 
$$
\kappa_i = \frac{(\vs_2 -1)}{n} + \frac{i}{n(n+1)}, \quad \text{for } \, 0\leq i\leq  n-1
\qquad \text{and }\kappa_n = \vs_1.$$ 

We now restrict to the case when  $n = 2$. Since $W=\Z_3$,   Lemma~\ref{lem:sphericalsimple} applies here,  even though $V$ is not a representation space for a quiver as in Section~\ref{Sec:Quivers}. By that lemma,  $\Ak(W)$ is simple if and only if all the following inequalities hold:
\begin{align*}
	3\vs_1 - 6\vs_2 & \neq 4 + 6 r,   \quad  3\vs_1 - 6\vs_2  \neq 1 + 6 r, &
	6\vs_2 - 3\vs_1 & \neq 2 + 6 r  & \text{for }r \in \Z_{\ge 1}, \\
	 6\vs_2 - 3\vs_1 & \neq -1 + 6 r  & && \text{for }r\in \Z_{ \ge 0}. \\
	\end{align*}
  If we take  $\vs_1 = -1/3$ and $\vs_2 = -1$ then $\Ak(W)$ is  simple. In contrast, for   $\vs'_1 = 1/3$ and $\vs'_2 = 1$,  the corresponding spherical algebra $\Ak'(W)$ is  not simple.  By  Lemma~\ref{left-simplicity} combined with Remark~\ref{rem:quiverHC}, this ring $\Ak'(W)$ is isomorphic to $A_{\cdag}(W)$ 
for the choice of parameters  $\vs_1 = -1/3$ and $\vs_2 = -1$.  

In conclusion, this is an example where $A_\kappa(W)$ is simple whereas $A_{\cdag}(W)$ is not. 
\end{example}

\begin{example} {\bf The $G_4$ example.}
A beautiful example of a stable $\theta$-represent\-ation, in the sense of Example~\ref{ex:theta-reps}, 
  is given by the representation $(G;V) = (SL_3; S^3 \C^3)$. Indeed, as noted in \cite[Example~5.6]{BLT}, this is a stable, locally free $\theta$-representation for an automorphism of order $m = 3$ on $\mf{so}_8$. Moreover, as with any $\theta$-representation,  $V$ is visible.  If $V = S^3 \C^3$ is viewed as the space of homogeneous polynomials of degree three in $x,y,z$, then a Cartan subspace $\h$ is provided by the Hesse pencil spanned by $\{ x^3+y^3+z^3, xyz \}$. The corresponding Weyl group is the binary tetrahedral group, which is $G_4$, the first of the exceptional complex reflection groups. 

The vector $xyz$ spans the zero weight space in $V$. Therefore, if $g \in G$ fixes $xyz$ then it belongs to $N_{G}(T)$, where $T \subset G$ is the maximal torus of diagonal matrices. The action of $N_{G}(T)$ on $xyz$ factors through $\mf{S}_3 = N_{G}(T)/T$ with $\sigma \in \mf{S}_3$ acting by $\sigma \cdot xyz = (-1)^{\sigma} xyz$. This implies that the stabiliser of $b := xyz$ in $G$ is $G_b = A_3 \ltimes T$. The elements in this subgroup that also fix $a = x^3 + y^3 + z^3$ are $A_3 \ltimes T_a$. The elements $t = (t_1,t_2,t_3)$ in $T_a$ must satisfy $t_i^3 = 1$ and $t_1 t_2 t_3 = 1$. Thus, $T_a \cong (\BZ_3)^2$ and hence $Z = A_3 \ltimes (\BZ _3/ 3)^2$ has order $27$. In particular, it is disconnected. The group $N_G(\h)$ is generated by $N_G(\h) \cap N_G(T)$ together with 
$$
g := \frac{1}{\sqrt{-3}} \left( \begin{array}{ccc}
	1 & 1 & 1 \\
	1 & \omega & \omega^2 \\
	1 & \omega^2 & \omega 
\end{array}
\right) 
$$
where $\omega$ is a primitive $3^{\mathrm rd}$ root of unity. One can check that  $|N_G(\h)| = 648$; see
\cite[Section~4]{HessePencil}.

Let $H=\C\cdot b\subset \h$. We compute $W_H$.  First,  note that the set of  elements $K \subseteq A_3 \ltimes T$ that  fix $b$ and scale  $a$   belong to $N_G(\h)$. An element $(t_1,t_2,t_3) \in T$ scales $a$ if and only if  $t_1^3 = t_2^3 = t_3^3$ and $t_1 t_2 t_3 = 1$. This implies that $t_3 = t_1^{-1} t_2^{-1}$ and $u^2 v = v^2 u = 1$, where $u = t_1^3$ and $v = t_2^3$. Thus, $u^3 = 1$ and $v = u$. In other words, $t_1^9 = 1$ and $t_2^3 = t_1^3$. This is an abelian group of order $27$, with $|K / Z|=3$. Using MAGMA \cite{MAGMA} one can check that $W_H = K/Z$. 

The space $\g \cdot b$ is spanned by $\{ x^2 y, x^2 z,y^2 x,y^2 z,z^2 x,z^2 y \}$ and hence $S_b = \{ x^3,y^3,z^3,xyz \}$ and $S_H = \{ x^3 , y^3, z^3 \}$. If we set $u = x^3, v = y^3$ and $w = z^3$ then $\C[u,v,w]^T = \C[uvw]$ and hence $\C[u,v,w]^{G_b} = \C[uvw]$. Therefore, by  \cite[Corollary~4.15 and Corollary~4.21]{BLNS}, the image of $\rad_{b}$ equals $A_{\kappa}(\BZ_3)$, where $\kappa = (\kappa_{H,0},\kappa_{H,1},\kappa_{H,2}) = (0,2/3,1/3)$. We note that the stabiliser $G_b$ is disconnected.  
\end{example}

\begin{proposition}\label{prop:G4examplerad}  Let $(G;V) = (SL_3; S^3 \C^3)$. 
	Then the radial parts map induces an isomorphism 
	$$
	(\dd(S^3 \C^3) / \dd(S^3 \C^3) \mf{sl}_3)^{SL(3)} \stackrel{\sim}{\longrightarrow} A_{\kappa}(G_4). 
	$$
	Moreover, the algebra $A_{\kappa}(G_4)$ is simple and the Hecke algebra $\euls{H}_q(G_4) = \C G_4$ is semisimple.
\end{proposition}

\begin{proof}
Since $(G;V)$ is a   stable, locally free polar representation,
	the isomorphism follows from \cite[Corollary~7.20(2)]{BLNS}.
				
It follows from \cite[Theorem~3.1]{SunG4} that the Cherednik algebra $H_{\kappa}(G_4)$ does not have any finite-dimensional representations. In fact, the parameters $\{ a_i \}$ from \cite[Section~4.4.3]{SunG4} are $a_0 = -1, a_1 = -1$ and $a_2 = 2$ which in turn implies that $q_0 = q_1 = q_2$. That is, the Hecke algebra $\euls{H}_q(G_4) = \C G_4$ is semisimple. Therefore both $H_{\kappa}(G_4)$ and $\Ak(G_4)$ are simple algebras by \cite[Theorem~3.1]{BEG}.   \end{proof}

A consequence of Proposition~\ref{prop:G4examplerad} is that $\eMt = \eM$ in this case. Hence, $\eGt_{\lambda} = \eG_{\lambda}$ for all $\lambda \in \h^*$. Thus,  Corollary~\ref{thm:semi-simplicity2} has the following immediate consequence.

\begin{corollary}\label{cor:G4examplerad}  Let  $(G;V) := (SL_3; S^3 \C^3)$.  Then the Harish-Chandra module 
$$
\eG_0 = \bigoplus_{\rho \in \mr{Irr} G_4} \eG_{0,\rho} \otimes \rho^*
$$
is semisimple. Here, each simply summand $\eG_{0,\rho}$ is the minimal extension of a local system of rank $\dim \rho$ on $V_{\reg}$. 
\qed
\end{corollary}

\begin{remark}\label{rem:G4} (1)  The commutative analogue of this example, where one is concerned with the Hamiltonian action of $G=SL(2)$ on  $S^3\C^2\times  S^3\C^2$ is examined in detail in \cite[Section~4.1]{Beck}.

 (2) For a family  of polar representations that are visible and stable but not $\theta$-representations, see \cite[Section~8]{BLT}.
\end{remark}
    
 
 We now give several easy examples that illustrate what happens when we drop the stability and visibility conditions.

\begin{example}\label{rem:framed-quiver}   Here we provide an  example that  shows that   Theorem~\ref{thm:semi-simplicity} and Corollary~\ref{thm:semi-simplicity2} can fail for non-stable polar representations.  

Consider  the framed quiver $\widetilde{Q}_{\ell}$ and  the corresponding representation space $V'=\mathrm{Rep}(\widetilde{Q}_\ell, (n\mathfrak{d},1) )$  over  $G=GL(n)^\ell$, as described in Definition~\ref{defn:quiver}. By Lemma~\ref{prop:quiver-polar}, $V'$ is a visible polar representation that is not stable. As explained in the proof of Lemma~\ref{prop:quiver-polar}, $\C[V']^G \stackrel{\sim}{\rightarrow} \C[V]^G$, where $V =\mathrm{Rep}(Q_\ell, (n\mathfrak{d},1) )$. This isomorphism identifies discriminants. It follows from the definition of the radial parts map that the image of $\rad_{\vs}$ for $V'$ agrees with that for $V$. We consider a particular $\vs$, for which the Hecke algebra is semisimple. If $\vs_{\infty} = -1/2$ and 
$$
\vs_i = \frac{i-\ell}{\ell} + \delta_{i,0}, \quad 0 \le i \le \ell-1,
$$
then $\kappa_{0,i} = \kappa_{1,i} = 0$, for all $i$, by Theorem~\ref{thm:cylicquiverradialparts} and Remark~\ref{rem:cylic}. Then equation \eqref{eq:qkappa} implies that $\mathcal{H}_q(W) = \C W$ is semisimple. However, 
$$
\chi_i = \delta_{i,1} - \frac{1}{\ell}, \quad 0 \le i \le \ell-1,
$$
which implies that $\chi \cdot \delta := \chi_0 + \cdots + \chi_{\ell-1} = 0$. By  \cite[Theorem~6.7]{BB}, $\eG_0$ is not semisimple in this case. In fact, the same conclusion holds for any $\vs$ for which the Hecke algebra is semisimple; the condition $\chi \cdot \delta = 0$ holds for all $\vs$ so $\eG_0$ is never semisimple. 
\end{example}

    \begin{example}\label{A non-stable example}
For this example, take $G=SL(V)$ with its natural action on 
a vector space $V$ of dimension $r>1$.  Since $V$ has just two orbits $\{0\}$ and $V\smallsetminus \{0\}$, clearly $V$ is a polar, visible representation  that  is not stable. Setting $\partial_j =\frac{\partial}{\partial x_j}$,   it is easy to see that $\tau(\g)$ is spanned by 
$$\{ x_i\partial_j :   1\leq i<j\leq r\}\cup \{\ x_\ell\partial_{\ell}-x_k\partial_k : 1\le \ell<k\leq r\}.$$
 As the next lemma shows, the modules $\eMt$ and $\eM$ are easily understood. The notation used in the lemma comes from  Definition~\ref{M-new-definition}.

\begin{lemma}\label{lem:nonstable} {\rm (1)} Set $\eMt=  \eD(V)/\eD(V)\tau(\g) $. Then there is an isomorphism 
	$$
	\phi: \eMt   \ \isom \   \C[V] \oplus \Delta_0,
	$$ 
	where we identify $\C[V]=\eD/\sum_j \eD \partial_j$ and, dually,
	$\Delta_0 =\eD/\sum_j \eD x_j$.

	{\rm (2)}  Thus, $R= \left( \eD/\eD \tau(\g) \right)^{SL(V)} =\C\oplus \C$ whereas 
	$A\cong\Im(\rad_0)=\C$.   
		
	{\rm (3) }   $\eM=\C[V]\not=\eMt$.
	 
\end{lemma}

\begin{proof} (1) Using the explicit generators of $\tau(\g)$ it is easy to see that $\eMt$ surjects onto both $\C[V]$ and $\Delta$. Since these are
	non-isomorphic simple $\eD(V)$-modules, it follows that $\phi$ is surjective.
	
	In order to prove that $\phi$ is injective, consider the localisation of $\eD(V)$  at $\eS =\{x_1^s\}$. Then 
	$\partial_j-x_1^{-1}x_1\partial_j\in \eD_{\eS}\tau(\g)$ for all $j>1$ while $\partial_1-x^{-1}x_2\partial_2\in \eD_{\eS}\tau(\g)$. 
	It follows that $\eMt_{\eS} =\C[V]_{\eS}$ and, similarly, $\eMt_{\eS_j} =\C[V]_{\eS_j }$  for $\eS_j =\{x_j^s\}$ and any $1\leq j\leq r$.
	Thus the   surjection  $\phi':\eMt\twoheadrightarrow \C[V]$  has kernel  $K=\ker(\phi')$    supported on $0=\bigcap (x_j=0)$, and so 
	$K$ is a sum of copies of $\Delta$. By a similar argument,  if  $\eT=\{\partial_1^n\}$, then  $K_{\eT} =\eMt_\eT$ and hence  $K_{\eT}=\Delta_\eT$. This suffices
	to show that  $\eMt \cong  \C[V] \oplus \Delta_0$.

	 (2)  There is a natural  isomorphism $\left( \eD/\eD \tau(\g) \right)^{SL(V)} = \End_{\eD}(\eM)$. Now, by Part~(1),  $\End_{\eD}(\eM)\cong \C\oplus \C$ since   $\C[v]$ and $\Delta$ are non-isomorphic simple modules  and hence both have endomorphism ring $\C$.  The assertions about $A$ follow from the fact that 
	 $\C[V]^G\cong \C $ and hence $A\hookrightarrow \dd(\C[V]^G)=\C$; thus  $\rad$ picks out one of the two copies of $\C$ in $R$.
	 
	 (3) 
This follows from Part~(2).
\end{proof}
Note that, in Example~\ref{A non-stable example}, the discriminant is zero and hence  the conclusion of Theorem~\ref{torsionfree} holds  vacuously.  
 
\end{example}

We next give an  easy example showing what happens when the representation is  neither visible  nor stable.
In essence, this example show that the  main results of the paper are not even meaningful when we drop both conditions.

\begin{example}\label{ex:trivial-case}
	Let $G=\C^\times$ act on $V=\C^2$ with weights $(1,1)$.  Then it is immediate that $V\git G = \{0\}  $ and that the non-zero orbits are $G\cdot v_\alpha$ for $v_{\alpha}=(v_{\alpha_1},v_{\alpha_2})$  with $[\alpha_1,\alpha_2]\in \mathbb{P}^1$.  Since $\overline{G\cdot v_{\alpha}} \ni 0$ for each such $v_{\alpha}$, every $v\in V$ is nilpotent and so, for trivial reasons, $V$ is   polar but not stable.  It  is  not  visible since there are infinitely  many nilpotent orbits.  (Non-visibility also follows from the  general   result  [BLLT, Proposition~8.5].) 
	
	Since $V\git G=0$,  clearly $\h=\{ 0 \}$ and hence the associated Hecke algebra is $\C=\C W$ for $W=\{1\}$. 
	Write $\dd(V)=\C\langle x_0,x_1,\partial_0,\partial_y\rangle$, as  usual. 
	Then $\tau(\g) = \C\nabla$ for $\nabla = x_0\partial_0+x_1\partial_1$. 
	Thus the associated Harish-Chandra module is $\eG=\eGt = \dd(V)/\dd(V)\nabla.$ This module has infinite length; for example,   it has the infinite descending chain of submodules $$
	M_n=(\dd(V)(x_0\partial_0)^n+\dd(V)\nabla)/\dd(V)\nabla\ \supsetneq \ M_{n+1}\ \supsetneq \cdots.
	$$
	In particular, $\eG$ is not holonomic and so basic results like Corollary~\ref{cor:admissible-holonomic} fail for this module.  Moreover, $\eG$   cannot be semisimple, despite the fact that the associated Hecke algebra is semisimple. Finally, note that 
	$$\dd(V)\nabla  =\dd(V)(\partial_0x_0+(x_1\partial_1-1)) \ \subsetneq \dd(V)x_0+\dd(V)(x_1\partial_1-1) \
	\not=\dd(V),$$ and so  $\eG$ also has a nonzero $\delta$-torsion factor module.  As in Example~\ref{A non-stable example}, the discriminant is zero and hence Theorem~\ref{torsionfree} is vacuously true. \end{example}

 We remark that we have no example of a polar representation that is stable but not visible for which the results of this paper fail. The problem is that for the standard examples of non-visible, stable polar representations the algebra $\Ak$ is typically non-simple.

\appendix
\section{Technical Results }\label{app-a}
 
 In this appendix we provide the promised proof of Lemma~\ref{Ext-equivariance}. This is proved by using the appropriate sort of projective resolution, so we begin 
 with the relevant definition together with   two subsidiary results.

A   \textit{weakly equivariant}\label{defn:weak}  left $\ddd$-module is a left $\ddd$-module $N$ that is also a rational $G$-module such that the action
  $\ddd \otimes N \rightarrow N$ is equivariant. As usual, differentiating the $G$-action $G \times N \rightarrow N$ defines a Lie algebra morphism 
  $\tau_N\colon \mf{g} \rightarrow \End_{\C}(N)$. 
Recall that $\tau \colon \mf{g} \rightarrow \ddd$ is the usual map. Note that a weakly equivariant module is \emph{strongly $G$-equivariant} or simply 
\emph{$G$-equivariant} if $\tau_N(x) = \tau(x)$ as endomorphisms of $N$, for all $x \in \mf{g}$. 

\begin{lemma}\label{lem:uweakequivmap}
	Let $N$ be a weakly equivariant left $\ddd$-module and define 
	$$
	u_N := \tau_N- \tau
	$$ as a map $\mf{g}  \rightarrow \End_{\C}(N).$ Then $u_N$ is a $G$-equivariant  Lie algebra morphism  $\mf{g} \to \End_\ddd(N)$. 
\end{lemma}

\begin{proof} This is \cite[Lemma~1.8]{BL} in the case $K=F=G$ and $A=\ddd$.
\end{proof}

The following category $C_h(G,\chi,\ddd)$ is based on \cite[(1.2)]{BL}. We remark that  the definition  in \cite{BL} is more restrictive, but the extra conditions will not be needed here. 

\begin{definition}\label{defn:homotopyequivariantcat-appendix}
	The objects $(C^{*}, \dalpha_{*}, j_{*})$ in $C_h(G,\chi,\ddd)$ consist of: 
	\begin{enumerate}
		\item  bounded complexes
		\[\cdots \ \buildrel{\dalpha_{i-1}}\over{\too}\  C^{i}  \ \buildrel{\dalpha_{i}}\over{\too}\  C^{i+1}  \ \buildrel{\dalpha_{i+1}}\over{\too}\cdots\]
		of finitely generated, weakly equivariant $\ddd$-modules $C^i$, with $G$-equivariant boundary maps $\dalpha_i$,  together with  
		\item $G$-equivariant  morphisms $j_i \colon \mf{g} \rightarrow \Hom_\ddd(C^i,C^{i-1})$  such that, for all $x\in \mf{g}$, 
		\begin{equation}\label{cat-equ}
		\dalpha_{i-1} \circ j_i(x) + j_{i+1}(x) \circ \dalpha_i \ =\  (\tau_{C^i}-\tau+\chi)(x),
		\end{equation}\end{enumerate}
	thought of as as elements of $\Hom_\C(\g,\, C^i)$.
	
	Here $G$ acts  by the adjoint action on $\mf{g}$  and by conjugation on $\Hom_\ddd(C^i,C^{i-1})$. The morphisms in the category are the obvious ones, but we 
	will not need them in this paper. 
\end{definition}
 
 The significance of Definition~\ref{defn:homotopyequivariantcat-appendix} is given by the following result.
\begin{lemma}\label{cohom-lemma-appendix}
	If $(C^{*},\dalpha_{*},j_{*}) \in C_h(G,\chi,\ddd)$ then the cohomology groups $H^i(C^{*})$ are $(G,\chi)$-monodromic left $\ddd$-modules. 
\end{lemma}  

\begin{proof}
	For each $x \in \mf{g}$, it follows from \eqref{cat-equ} that $\tau_{C^i}(x) + \chi(x)$ and $\tau(x)$ agree on the subquotient $H^i(C^{*})$. Thus, 
	$H^i(C^{*})$ is monodromic. 
\end{proof}

\begin{lemma}\label{prop:liftofMtohomotopycat-appendix}
Suppose that  $N\in (G,\chi,\ddd) \lmod$. Then there exists a finite projective resolution $P^{*}\to N\to 0$, such that $(P^{*},\dalpha_{*},j_{*}) \in C_h(G,\chi,\ddd)$. 
Here,  $P^i=0$ for $i>0$ and the $j_\ell$ are defined in the proof.
\end{lemma}

\begin{proof}
	By definition, $N$ is a rational $G$-module. Thus, $N$ is generated, as a $\ddd$-module, by some finite dimensional $G$-submodule $V$.  Hence 
	$\pi: \ddd \otimes_{\C} V \rightarrow N$ is a surjective $G$-equivariant $\ddd$-module map from a projective, weakly equivariant $\ddd$-module. The kernel $K$
	 of this morphism is a weakly equivariant, finitely generated $\ddd$-module. Since $G$ acts rationally on $\ddd$ the same is true for $K$  and so $K$ is again
	  generated by a finite dimensional $G$-submodule. Thus, as $\ddd=\eD(V)$ has finite homological dimension, we may repeat the process to  
	construct a finite projective resolution $P^{*}\to N\to 0$ such that the boundary maps $ \dalpha_{*}$ are $G$-equivariant.
	
	It remains to construct the maps $j_i$. Set $u_i := u_{P^i} = \tau_{P^i} - \tau + \chi$ and note that, by Lemma~\ref{lem:uweakequivmap}, $u_i(x)$ is $\ddd$-linear 
	and $G$-equivariant for all $x\in\mf{g}$. We construct $j_i$ by induction on $-i$. For $i > 0$ this is automatic since $P^i = 0$, so consider $P^0$. For each 
	$x \in \mf{g}$, we have homomorphisms
	$$
	\begin{tikzcd}
	\cdots \ar[r] & P^{-1} \ar[r,"\dalpha_{-1}"] \ar[d,"u_{-1}(x)"] & P^0 \ar[r,"\pi"] \ar[d,"u_0(x)"] & N \ar[r] \ar[d,"0"] & 0 \\
	\cdots \ar[r] & P^{-1} \ar[r,"\dalpha_{-1}"]  & P^0 \ar[r,"\pi"]  & N \ar[r]  & 0 .
	\end{tikzcd}
	$$   
	Since $N$ is monodromic, $\Im(u_0(x))\subseteq \ker(\pi)$ and so  the right hand diagram commutes (we will show later that the whole diagram
	 commutes). In particular,  $u_0\in \Hom_G(\mf{g}, \Hom_\ddd(P^0,\Im \dalpha_{-1}))$. Since $P^{0}$ is projective, the natural map 
	$$
	\dalpha_{-1,*} \colon \Hom_{\ddd}(P^0,P^{-1}) \rightarrow \Hom_\ddd(P^0,\Im \dalpha_{-1})
	$$
	is surjective. Thus, as $\Hom_G(\mf{g}, - )$ is exact, the induced map
	$$
	\Hom_G(\mf{g},(\dalpha_{-1})_*) \colon \Hom_G(\mf{g},\Hom_{\ddd}(P^0,P^{-1})) \rightarrow \Hom_G(\mf{g},\Hom_\ddd(P^0,\Im \dalpha_{-1}))
	$$
	is also surjective. Therefore, there exists $j_0 \in \Hom_G(\mf{g},\Hom_\ddd(P^0,P^{-1}))$ such that $\dalpha_{-1,*} \circ j_0 = u_0$. Since $P^1=0$, both $\dalpha_0$ and 
	$j_1$ are automatically  zero, and we  obtain the required equation $\dalpha_{-1} \circ j_0 + j_1 \circ \dalpha_0 = u_0$. 
	
	Assume now that $i < 0$ and we have found $j_{i+1}, \ds, j_{-1}, j_0$ satisfying \eqref{cat-equ}. For each $x \in \mf{g}$, we have the diagram
	$$
	\begin{tikzcd}
	\cdots \ar[r] & P^{i-1} \ar[r,"\dalpha_{i-1}"] \ar[d,"u_{i-1}(x)"] & P^i \ar[r,"\dalpha_i"] \ar[d,"u_i(x)"] & P^{i+1} \ar[r] \ar[d,"u_{i+1}(x)"] & \cdots \ar[r] &P^0\ar[d,"u_{0}(x)"]\\
	\cdots \ar[r] & P^{i-1} \ar[r,"\dalpha_{i-1}"]  & P^i \ar[r,"\dalpha_i"]  & P^{i+1} \ar[r]  & \cdots \ar[r] & P^0
	\end{tikzcd}
	$$
	We first check that the squares in this diagram commute   for all $i\leq \ell \leq 0$. Indeed,
	\begin{align*}
	u_\ell(x) \circ \dalpha_{\ell-1}(p)  & = \tau_{P^\ell}(x)(\dalpha_{\ell-1}(p)) - \tau_X(x) \dalpha_{\ell-1}(p) + \chi(x)\dalpha_{\ell-1}(p) \\
	& = \dalpha_{\ell-1}(\tau_{P^{\ell-1}}(x)(p)) - \dalpha_{\ell-1}(\tau(x) p) + \dalpha_{\ell-1}(\chi(x) p) \\
	& = \dalpha_{\ell-1}(u_{\ell-1}(x)(p))\\
	& = \dalpha_{\ell-1} \circ u_{\ell-1} (p),
	\end{align*}
	where $\tau_{P^\ell}(x)(\dalpha_{\ell-1}(p)) = \dalpha_{\ell-1}(\tau_{P^{\ell-1}}(x)(p))$ holds because $\dalpha_{\ell-1}$ is $G$-equivariant and $\tau(x) \dalpha_{\ell-1}(p) = \dalpha_{\ell-1}(\tau(x) p)$ 
	holds because $\dalpha_{\ell-1}$ is $\ddd$-linear. 
	
	Set $w(x) =u_i(x) - j_{i+1}(x) \circ \dalpha_i$. We next claim  that $\Im w(x)\subseteq \ker  \dalpha_i$. To see this, let $p\in \Im w(x)$ and write $p=  w(x) (q)$, for some $q\in P^i$.
	Then, by induction, 
	\begin{align*}
	\dalpha_i(p) & = \dalpha_i \circ u_i(x) (q) - \dalpha_i \circ j_{i+1}(x) \circ \dalpha_i (q) \\
	& = u_{i+1}(x) \circ \dalpha_i (q) - \dalpha_i \circ j_{i+1}(x) \circ \dalpha_i (q) \\
	& = \Bigl(\dalpha_i \circ j_{i+1}(x) \circ \dalpha_i + j_{i+2} (x) \circ \dalpha_{i+1} \circ \dalpha_i\Bigr)(q) - \dalpha_i \circ j_{i+1}(x) \circ \dalpha_i (q) \\
	& = 0,
	\end{align*}
	as claimed.
	Since $H^{i}(P^{*}) = 0$, it follows that 	$\Im w(x) \subset \Im  \dalpha_{i-1}.$

	Finally, $(\dalpha_{i-1})_* \colon \Hom_\ddd(P^i,P^{i-1}) \rightarrow \Hom_\ddd(P^i, \Im \dalpha_{i-1})$ is surjective
	and so the induced map
	$$  
	\Hom_G(\mf{g},\Hom_\ddd(P^i,P^{i-1})) \rightarrow \Hom_G(\mf{g},\Hom_\ddd(P^i, \Im \dalpha_{i-1}))
	$$
	is also surjective. The inclusion  $\Im w(x) \subset  \Im \dalpha_{i-1}$  implies that  \[u_i - j_{i+1} \circ \dalpha_i\in \Hom_G(\mf{g},\Hom_\ddd(P^i, \Im \dalpha_{i-1})).\] Hence, there exists 
	$j_i \in \Hom_G(\mf{g},\Hom_\ddd(P^i,P^{i-1}))$ such that 
	$$
	(\dalpha_{i-1})_* (j_i) = \dalpha_{i-1} \circ j_i = u_i - j_{i+1} \circ \dalpha_i.
	$$
	In other words, $u_i = \dalpha_{i-1} \circ j_i + j_{i+1} \circ \dalpha_i$ as required. 
\end{proof}

\begin{lemma}\label{Ext-equivariance-appendix}
	Suppose that  $N$ is a  $(G,\chi)$-monodromic left $\dd(V)$-module. For each $i \ge 0$, the right $\ddd$-module $\Ext^i_\ddd(N,\ddd)$ has a canonical  monodromic structure.
\end{lemma}

\begin{proof}
	Using Lemma~\ref{prop:liftofMtohomotopycat-appendix}, we pick a complex $(P^{*},\dalpha_{*},j_{*})\in C_h(G,\chi,\ddd)$ resolving $N$. We set 
	$Q^i = \Hom_\ddd(P^{-i},\ddd)$, $\dalpha_i' = \dalpha_{-i-1}^*$ and $j_i' = j_{-i +1}^*$. Then $\Ext^i_\ddd(N,\ddd) = H^i(Q^{*})$. Therefore the result follows from right-hand 
	version of Lemma~\ref{cohom-lemma-appendix}, if we can show that $(Q^{*},\dalpha_{*}',j_{*}')$ belongs to $C_h(G,\chi,\ddd)^{op}$. 
	
	Recall that we have 
	$$
	\dalpha_{i-1} \circ j_i(x) + j_{i+1}(x) \circ \dalpha_i = (\tau_{P^i} - \tau + \chi)(x), \quad \forall \, x \in \mf{g}. 
	$$
	Next, 
	\begin{align*}
	\dalpha_{i-1}' \circ j_i'(x) + j_{i+1}'(x) \circ \dalpha_i' & = \dalpha_{-i}^* \circ j_{-i+1}(x)^* + j_{-i}(x)^* \circ \dalpha_{-i-1}^* \\
	& = (j_{-i+1}(x) \circ \dalpha_{-i} + \dalpha_{-i-1} \circ j_{-i}(x))^* \\
	& = (\tau_{P^{-i}} - \tau + \chi)(x)^*. 
	\end{align*}
	If $\phi \in Q^i$, $g \in G$, $D \in \ddd$ and $p \in P^{-i}$, then $(\phi \cdot D)(p) = \phi(p) D$ and 	
	$$
	(\phi \cdot g)(p) = g^{-1} \cdot \phi(g \cdot p) = g^{-1} \phi(g \cdot p) g.
	$$
	This implies that 
	\begin{align*}
	(\phi \tau_{Q^i}(x))(p) & = (\tau_{P^{-i}}(x)^* \phi)(p) - [\tau(x),\phi(p)] \\
	& = (\tau_{P^{-i}}(x)^* \phi)(p) - \phi(\tau(x) p) + \phi(p) \tau(x) \\
	& = (\tau_{P^{-i}}(x) - \tau(x))^*(\phi)(p) + (\phi \tau(x))(p).
	\end{align*}
	Thus, 
	$$
	(\tau_{P^{-i}}(x) - \tau(x) + \chi(x))^*(\phi)(p) = [\phi \cdot (\tau_{Q^i}(x) - \tau(x) + \chi(x))] (p),
	$$
	which implies that 
	$$
	\dalpha_{i-1}' \circ j_i'(x) + j_{i+1}'(x) \circ \dalpha_i' = \tau_{Q^i}(x) - \tau(x) + \chi(x)
	$$
	as required. 
\end{proof}
 
 \section*{Index of Notation}\label{index}
  For some of the definitions the reader is referred to the paper \cite{BLNS}.

\begin{multicols}{2}
{\small  \baselineskip 14pt

${\Aak}=\Ak(W)$, spherical  algebra  \hfill\pageref{spherical-defn}

  $\Ak(W_{\lambda})$   \hfill\pageref{spherical-defn2}

    $\euls{A}\subset \h$,  reflecting hyperplanes  of $W$   \hfill\pageref{Hecke-defn}
 
Admissible modules $\eC$, \  $\eCop$ \hfill\pageref{defn:admissible}

local ad-nilpotence  \hfill\pageref{ad-hypotheses}

Auslander conditions  \hfill\pageref{Aus-defn}

$\Bi(S,d),$ \  $\Bi(d,S)$,  \ $\Biad(d,S)$   \hfill\pageref{Bi-definition}

 $B_W$, braid group associated to $W$    \hfill\pageref{Hecke-defn}

$\chi$,  a fixed character of $\g$   \hfill\pageref{chi-defn} 

$\eC_{,\lambda}$,  components of $\eC$ \hfill \pageref{lem:decomposeadmissible}
 
Cartan subspace $\h_v$ \hfill  \cite{BLNS} 

Cherednik algebra $\Hk$  \hfill  \cite{BLNS}

Cohen-Macaulay, $n$-CM module  \hfill\pageref{CM-defn}

Cyclotomic Hecke algebra \hfill\pageref{cyclotomic-defn}

 $\ddd=\eD(V)$ for $V$ polar  \hfill\pageref{R-defn}

$\BD_\ddd(J)   =\Ext^{n+m}_\ddd(J,\, \ddd)$,\  \ $\BD_\ddd^{\mathrm{op}}$    \hfill\pageref{main-notation},\pageref{main-notation2}

  $ \BD_A(L) = \Ext_A^n(L,\, A)$, \  \  $\BD_A^{\mathrm{op}}$ \hfill\pageref{main-notation},\pageref{main-notation2}

$\widetilde{\BD}(L) =\Ext_\ddd^m(L,\ddd) $, \  $\widetilde{\BD}^{op}(L)$  \hfill\pageref{yet-more-notation}

Discriminant $\deltah    \in \C[\h]^W$  \hfill \pageref{eq:discriminant}

Discriminant $\delta   \in \C[V]$  \hfill  \pageref{deltaV-defn}

Euler element $\eu_V$  \hfill\pageref{notation:gradings}

Fourier transform $\mathbb{F}$  \hfill\pageref{defn:Fourier}

$\g_\chi =  \{ \tau(x) - \chi(x) \, | \, x \in \g \}$  \hfill\pageref{notation4.11}

$\eGt _\lambda = \ddd/(\ddd\g_\lambda+\ddd \mf{m}_\lambda)$  \hfill\pageref{defn:Glambda} 

  $\eG_\lambda=\eM\otimes_A\eQ_\lambda$  \hfill\pageref{defn:Glambda}

   $(G,\chi,\ddd)\lmod$,  monodromic modules  \hfill\pageref{defn:monodromic}

  $\grade(L)$,  grade of a module  \hfill\pageref{grade-defn}
  
     $\h$,   a Cartan subspace of $V$ \hfill \cite{BLNS}

$H_Q=\Hom_\ddd(Q,\eM')$, \  $H_{\eK} $ \hfill\pageref{module-notation}

  $\widehat{H}_Q=\Hom_U(H_Q, U)$  \hfill\pageref{module-notation}

$\HHleft(J)=\eM\otimes_AJ$   \hfill\pageref{main-notation},\pageref{main-notation2}

  $ \HHright (L)=L\otimes_A\Ext^m_\ddd(\eM,\, \ddd)$  \hfill\pageref{main-notation},\pageref{main-notation2}
  
$  \BH(X)= \Hom_\ddd(\eM,X)$    \hfill\pageref{summand-endo3}

Harish-Chandra module   $\eG_\lambda$  \hfill\pageref{defn:Glambda}
  
  Hecke algebras  $\euls{H}_q(W_{\lambda})$ 
    \hfill\pageref{Hecke-defn}  
      
  Holonomic  $\Ak(W)$-modules      \hfill\pageref{defn:holonomic}

 Integral symmetric space    \hfill\pageref{nice-space}

  $\eKt \ = \ \ddd/\g_\chi \ddd \ = \  \eK \oplus L_\eK $  \hfill\pageref{lem:K-hyp}
  
  $\kappa$, $\kappa_{H,i}$,   the basic parameters  \hfill  \cite{BLNS}

   Locally free representation \hfill\pageref{defn:locally-free}
   
 $\mr{KZ}_G$, geometric KZ  functor    \hfill\pageref{defn:KZG} 
  
$\mf{m}_\lambda$, a maximal ideal of $(\Sym  \h)^W$  \hfill\pageref{defn:em-lambda}

$\eMt=\ddd/\ddd\g_\chi$,     $\eM =\eMt/\eMt P$  \hfill\pageref{M-defn}

$\eM'=\Ext_\ddd^n(\eM,\,\ddd)$ \hfill\pageref{M'-defn}

$n=\dim\h$,  \ $m=\dim V-\dim \h$ \hfill\pageref{notation4.11}

Minimal extension   \hfill\pageref{minextn-defn}

Moment map $\mu: T^*V\to \g^*$    \hfill\pageref{defn:moment}

Monodromic  modules   \hfill\pageref{defn:monodromic}    

 $\euls{N}(V) := \pi^{-1}(0)$, the nilcone   \hfill\pageref{defn:nilcone}
 
Nice symmetric space    \hfill\pageref{nice-space}

$\Osph=\Osph_{\kappa}(W)$,  
$\Osph_{\llambda } = \Osph_{\kappa,\llambda }(W)$   \hfill\pageref{defn:Osph}

  $ P=\ker(\rad_\vs )/(\ddd\g_\chi )^G$  \hfill\pageref{R-defn}  

 $Q=Q_{\ell} $, cyclic quiver \hfill\pageref{defn:quiver}

$Q$, a progenerator in $(G,\chi,\ddd)\lmod $  \hfill\pageref{module-notation}

Polar representation  \hfill \pageref{defn:polar} 
 
$\eQ_\lambda = A/A\mf{m}_\lambda\in \Osph_\lambda$   \hfill\pageref{M-definition}

$q_{H,j}$,   parameters for $\euls{H}_q(W_{\lambda})$   \hfill\pageref{eq:Heckecrg2}

$R=\ddd^G/(\ddd\g_\chi )^G$   \hfill\pageref{R-defn}

 $\rad_{\vs}$, radial parts map   \hfill\pageref{eq:notation4.11}
 
 Regular loci $V_{\reg}=(\deltav\not=0),\,\  \h_{\reg}$ \hfill\pageref{Vreg-defn}
 
 Regular parameters $p,\kappa,\vs$  \hfill\pageref{regular-parameter} 
 
Robust \ symmetric space \hfill\pageref{nice-space}

$\vs$, parameter    \hfill\pageref{radial-defn2},  [BLNS]
 
  $\Ch M$, characteristic variety of $M$   \hfill   \pageref{singular-defn}

 Stable representation \hfill \pageref{defn:stable}

 Strongly admissible modules $\eC_{,0}$   \hfill\pageref{defn:strongly}
 
 Symmetric pair, symmetric space   \hfill\pageref{defn:symmetric}
 
 $ \shT_{\vs,\vs'}  $, shift functors   \hfill\pageref{defn:T-shift}
  
    $\tau:\g\to \dd(V)$,  $\tau_M : \g\to \End_{\C}(M)$  \hfill\pageref{tau-defn}

$\tau_L,$ $  \tau_R$    \hfill\pageref{tau-left-defn}

Theta representations \hfill\pageref{ex:theta-reps}
    
Trace function     $\mr{Tr}_V$   \hfill\pageref{defn:trace}

$U=\End_\ddd(P) $  \hfill\pageref{module-notation}

Visible representation  \hfill  \pageref{defn:visible}

   Weyl group $W$  of $(V, \h)$  \hfill  [BLNS] 
 
 $W_H$ pointwise stabiliser  of $H\in \mr{A} $   \hfill\pageref{Hecke-defn}

Weakly equivariant \hfill\pageref{defn:weak}

$W_\nu$,  a parabolic subgroup of $W$  \hfill\pageref{defn:c(nu)}   

}
\end{multicols}


\end{document}